\documentclass[10pt,reqno]{amsart}

\usepackage{graphicx}
\usepackage{tabularx}
\graphicspath{ {./images/} }
\usepackage{xargs}
\usepackage{geometry}
\usepackage{fancyhdr} 
\usepackage{lastpage} 
\usepackage{extramarks} 
\usepackage[most]{tcolorbox} 
\usepackage{xcolor} 
\usepackage[hidelinks]{hyperref} 
\DeclareUnicodeCharacter{0345}{\textcompwordmark}
\usepackage[utf8]{inputenc} 
\usepackage[russian,english]{babel}
\usepackage[permil]{overpic}
\usepackage{pict2e} 
\usepackage{listings}
\usepackage{amssymb}
\usepackage{amsthm}
\theoremstyle{plain}
\usepackage[
backend=biber,
style=numeric,
sorting=ynt
]{biblatex}
\usepackage{xargs}
\usepackage{float}
\usepackage{lipsum}
\usepackage{geometry}
\usepackage{fancyhdr} 
\usepackage{lastpage} 
\usepackage{extramarks} 
\usepackage[most]{tcolorbox} 

\usepackage{xcolor} 
\usepackage[hidelinks]{hyperref} 
\usepackage[permil]{overpic}
\usepackage{pict2e} 
\usepackage{listings}

\usepackage{amsmath}
\usepackage{amssymb}
\usepackage{tikz-cd}
\usetikzlibrary{babel}
\providecommand{\Inv}{\operatorname{Inv}}

\newtheorem*{theorem*}{Theorem}

\newtheorem*{assumption}{Assumption}
\theoremstyle{notation}

\numberwithin{equation}{section}
\theoremstyle{plain}
\newtheorem{definition}{Definition}[subsection]
\newtheorem{theorem}[equation]{Theorem}

\newtheorem*{remark}{Remark}
\newtheorem{corollary}[equation]{Corollary}
\newtheorem{lemma}[equation]{Lemma}
\newtheorem*{conjecture*}{Conjecture}

\newtheorem{proposition}[equation]{Proposition} 
\providecommand{\keywords}[1]
{
  \small	
  \textbf{\textit{Keywords---}} #1
}

\definecolor{my-linkcolor}{rgb}{0.7.5,0.15,0.15}
\definecolor{my-citecolor}{rgb}{0.3,0.7,0.1}
\definecolor{my-urlcolor}{rgb}{0,0,0.75}
\hypersetup{
	colorlinks,
	linkcolor={my-linkcolor},
	citecolor={my-citecolor},
	urlcolor={my-urlcolor}
}

\newenvironment{dedication}
  {
   \thispagestyle{empty}
   \itshape             
   \raggedleft          
  }
  {
  }

\title{When Entropy flows: drifting along the route to Chaos}
\author{Eran Igra$^{1,2}$, Valerii Sopin$^{1,2}$, Yanghong Yu$^{3}$}
\address{\newline 1. Shanghai Institute for Mathematics and Interdisciplinary Sciences (SIMIS), Shanghai 200433, China
\newline 2. Research Institute of Intelligent Complex Systems, Fudan University, Shanghai 200433, China
\newline 3. Institue of Science Tokyo, Tokyo, Japan
}
\email{eranigra@simis.cn}
\email{VvS@myself.com, vsopin@simis.cn}
\email{yu.y.2235@m.isct.ac.jp}
\begin{document}

\begin{abstract}
Consider a smooth one--parameter family of vector fields defined over some smooth manifold transitions from order into chaos. Inspired by the Second law of Thermodynamics, one is led to ask: can we find a flow whose dynamics realize this transition? To answer this question, motivated by the Mallet-Yorke Orbit Index theory, the Arnold-Khesin scheme for hydrodynamics and a heuristic argument by Rene Thom, we introduce a construction that transforms any one--parameter family of vector fields into a new object: the "Entropy flow". The Entropy flow is a flow defined on the product of the phase space with the parameter space and is best thought of as a flow generated by the original one--parameter family together with a drift in the parameter space, that pushes the trajectory of a given initial condition into a disordered, more complex state. To exemplify, for the Period Doubling, the Ruelle-Takens-Newhouse and the Intermittency routes to chaos the Entropy flow behaves exactly as expected - that is, it truly pushes trajectories into more complex states. In addition, in the spirit of Forcing Theory, in the paper we use the Conley index to discuss how one can use the Entropy flow to study the connection between topology and bifurcations. Moreover, drawing on the numerical and analytic evidence, we will analyze how the Entropy flow behaves in several examples of famous flows, including the Lorenz system, the Rössler attractor, and the breakup of the Shilnikov homoclinic scenario.
\end{abstract}

\maketitle
\tableofcontents

\begin{dedication}
    We dedicate this paper to James A. Yorke, who taught us to see 
     how periodic orbits push Order into Chaos.
    \vspace{\baselineskip}

  \end{dedication}

\section{Introduction}



Let $M$ be a smooth manifold, let $\dot{s}=F(s)$ be a smooth vector field on it and let $\phi:M\times \mathbb{R}\to M$ be a corresponding flow. By the Second Law of Thermodynamics, one would want that for all unbounded, increasing sequences $\{t_j\}_{j\in\mathbb{N}}$ and all large ensembles of "generic" initial conditions, $x_1,...,x_n$, the finite sequences $\phi(x_1,t_j),...\phi(x_n,t_j)$ become more and more "mixed" as $j\to\infty$. This leads us to ask: can we actually construct a flow having these properties? To better formulate this question, let us formally treat $\dot{s}=F_\tau(s)$ as a $C^k$, $k\geq1$, family of vector fields, defined by a $C^k$-map $F:M\times I\to TM$, where $I$ is some one-dimensional parameter space and $F(s,\tau)=F_\tau(s)$. As $\tau$ is varied in $I$, the dynamics generated by $F_\tau$ should be expected to become increasingly more chaotic (usually by undergoing some route to chaos). As discussed in \cite{RT}, \cite{fei}, \cite{SRT}, \cite{fei2}, \cite{fei1}, \cite{PP}, \cite{fei4} and \cite{shil3} (among others), these transitions from order into chaos could be understood as toy models for the evolution of turbulence (see Chapter 5 in \cite{YLi} and \cite{fei3} for a survey on the modern approach to this problem). We now rephrase our question above as follows: can we find a flow on $M\times I$ whose trajectories behave like the transition of $\dot{s}=F_\tau(s)$ from order into chaos as $\tau$ is varied?\\

It is exactly this problem we tackle in this paper: we prove that given a $C^k$ family of vector fields $\dot{s}=F_\tau(s)$, $s\in M$, $\tau\in I$, there exists a $C^1$-flow $\phi:(M\times I)\times\mathbb{R}\to M\times I$ whose invariant sets include many of the bifurcation sets for $\dot{s}=F_\tau(s)$. In particular, the stable manifolds of these bifurcation sets attract and repel initial conditions in $M\times I$, thus propelling trajectories towards increasingly disordered motion. Specifically, inspired by Mallet-Yorke Index Theory (see \cite{PY}, \cite{PY2}, \cite{PY3}, \cite{Perd}) and the theory of Bifurcations without Parameters (see \cite{fiedler}, \cite{lieb} for a survey), we construct the flow $\phi$ as the solution for the following system, defined on $M\times I$:
\begin{equation*}
    \begin{cases}
        \dot{s}=F_\tau(s)+V(s,\tau),\\
        \dot{\tau}=P(s,\tau),
    \end{cases}
\end{equation*}
where both $P$ and $V$ are smooth functions that vanish away from two sets $Per$ and $Fix$ in $M\times I$, which correspond to the "distinguishable" periodic orbits and fixed points for $\dot{s}=F_\tau(s)$ (respectively). Intuitively, the set $Per$ should be understood as the collection of initial conditions $(s,\tau)\in M\times I$ that lie on attracting or repelling periodic orbits for $F_\tau$. Similarly, $Fix$ should be understood as the collection of points $(s,\tau)\in M\times I$ s.t. $s$ is either a source or sink for $F_\tau$. That being said, the definition for both is more global than that, and both these sets certainly allow some saddle orbits and fixed points inside (see Section \ref{defsect} for the precise details). Therefore, the functions $P$ and $V$ should be understood as follows: if one thinks of each $F_\tau$ as encoding a parameter-dependent law of motion, then $P$ is the force causing the law of motion directing the motion of $s$ to change while in motion. Similarly, $V$ can be interpreted as a function controlling the transition of energy from one bifurcation to another, in a way that increases complexity (see the discussion after Definition \ref{entropyflow} for the precise formulation). \\

To illustrate how this flow on $M\times I$ behaves in practice, let us assume $I=(-1,1)$ and the existence of a period doubling cascade of attracting orbits at parameters $\{\tau_n\}\subseteq (-1,0)$, $\tau_n\to0$. Then, around the subsets $M\times I$ corresponding to this cascade, $\mathcal{D}$, the function $P$ would in general be positive (i.e., the flow above would always push towards $M\times\{0\}$). Similarly, given a period doubling cascade of attractors occurring at a sequence $\{\tau_n\}\subseteq(0,1)$ s.t. $\tau_n\to0$, $P$ would be negative around the subset $\mathcal{T}$ of $M\times I$ corresponding to the cascade (see Section \ref{perioddoubling}). In other words, the flow would push initial conditions $(s,\tau)\in M\times I$ close to $\mathcal{D}$ or $\mathcal{T}$ (respectively) towards $M\times\{0\}$, where there exists some complex motion for $F_0$ corresponding to the "end" of the cascade. To honor the Second Law of Thermodynamics, we name this new flow as the \textbf{Entropy flow}. At this point we explicitly state that a connection between our construction and the Second Law of Thermodynamics does not arise from any known definition of Entropy. Rather, we chose the name since, as stated above, the Entropy flow is designed to push initial conditions in $(s,\tau)$ from an ordered into disordered state. Before moving on, we stress that the Entropy flow for realistic systems can, and many times will, have extremely complex dynamics in $M\times I$. We will exemplify that with several examples throughout this paper, including the Rössler, the Lorenz, and the Michelson systems, among others (see \cite{Ross76}, \cite{Lo}, and \cite{Michh}, respectively).\\

Our motivation to construct the Entropy flow arose from two main sources. The first is an attempt to study bifurcations in a topological setting. Specifically, let us recall that there exists a rich theory for how certain prescribed topological conditions can force certain dynamics to appear, a field often referred to as "Dynamical Forcing Theory" (see \cite{Bo} for a survey). The Entropy flow allows us to study the analogous question - namely, how does topology force bifurcations to appear. To illustrate, let us recall the sets $Per$ and $Fix$ mentioned above. It is the topological configuration of these sets in $M\times I$ that constrains the bifurcation structure of these periodic orbits and fixed points in $M\times I$. As their topology can be hard to analyze directly, by making them the invariant sets for the Entropy flow we can study them as invariant sets for it. This allows us to apply topological tools like the Conley index Theory (see \cite{Mischaikow2002ConleyIndex} for a survey). As such, a central part of this paper develops a Conley index framework for the topology of bifurcations (see Subsection \ref{2dtheory}). In detail, we prove that when a bifurcation set for $\dot{s}=F_\tau(s)$ can be isolated in $M\times I$, we can describe the way it constrains the dynamics of the Entropy flow using the Conley index (see Proposition \ref{decomposition} and Theorem \ref{bifcodesth1}). As the dynamics of the Entropy flow are defined by the bifurcations of $\dot{s}=F_\tau(s)$, these facts should be interpreted as saying that one bifurcation can force others to appear, i.e., that topology forces bifurcations. These results, therefore, should be viewed as bifurcation analogues for the Sharkovskii Theorem (see \cite{Shar}), the Li-Yorke Theorem (see \cite{Yor}), and the Thurston-Nielsen Classification Theorem (see \cite{Fat}), all of which can be interpreted as ways topology forces dynamics (that being said, we stress the analogy should be understood as programmatic rather than a direct equivalence). \\

To further elaborate on the above, our idea is the following: let us assume there exists some isolated invariant set $\mathcal{A}\subseteq M\times I$ for the Entropy flow, corresponding to a collection of some bifurcation sets, each of which can be isolated (again, w.r.t. the Entropy flow). By splitting $\mathcal{A}$ into these components we will construct a filtration of them adapted to the flow (see Definition \ref{def:entropy-index-ladder}), which we use to define transition groups (see Definition \ref{transportmap}) and transport maps between them (see Definition \ref{transportmap}) encoding the bifurcations as the parameter $\tau$ varies via the lens of the Entropy flow. In detail, these objects record how Conley classes are inherited, created, killed or attached as the Entropy flow passes through successive bifurcation blocks (for the details, see Definition \ref{birthdeathcode}). In this sense, our contribution is not to introduce new index filtrations as abstract Conley-theoretic objects, but to select them from the geometry of bifurcation sets in \(M\times I\) w.r.t the Entropy flow and interpret their induced maps as records of bifurcation transitions.\\

The second source of our motivation for constructing the Entropy flow is more philosophical in nature and arose from three different places. The first is the heuristic originally due to Rene Thom, stated in \cite{Thom}. According to it we should not expect to find structurally stable hyperbolic behavior in systems derived from scientific experiments, as such systems often have properties that are essentially unstable. The second is the Chaotic Hypothesis (see \cite{gal}), according to which chaotic behavior should always assumed to have hyperbolic-like properties, even if the dynamics do not satisfy any "nice" splitting condition. The final source is the Arnold-Khesin scheme for fluid dynamics (see \cite{ArKhe}), according to which the solutions of the Euler equations can be seen as a geodesic flow on the space of diffeomorphisms, i.e., a "transition between states". These three different ideas led us to construct the Entropy flow as a flow that behaves like all of these, combined. In detail, as a flow on $M\times I$ the Entropy flow is extremely unstable dynamically, yet, as will be made clear later on, the projection of its flow lines from $M\times I$ to $M$ can certainly appear hyperbolic - thus connecting Rene Thom's heuristic and the Chaotic Hypothesis. At this point we stress that at no stage do we claim the Entropy flow describes a fluid flow like the Arnold-Khesin Scheme was designed to do. That being said, somewhat surprisingly, the projection to $M$ of the flow lines in $M\times I$ for some Entropy flows share a (even if possibly superficial) similarity to turbulent motion. We discuss this in Section \ref{turbulence}, where we will observe this correspondence arising naturally in the Entropy flow for the Shilnikov homoclinic scenario (see \cite{LeS}).\\

This paper is organized as follows: in Section \ref{defsect} we construct the Entropy flow. In order to ensure a large degree of flexibility and a broad applicability for our construction, this Section is highly technical. After that, we discuss two examples showing how the Entropy flow behaves - the first is derived from the van der Pol oscillator (see \cite{vdp}) and the second -- from a cascade of pitchfork bifurcations. Following that, in Section \ref{basicqualitative} we perform basic qualitative analysis and study the behavior of the Entropy flow in the three famous routes to chaos: the Period Doubling route to chaos (see \cite{fei}), the Ruelle-Takens-Newhouse route to chaos (see \cite{RT} and \cite{SRT}), and the Intermittency route to chaos (see Theorem \ref{heteroclinicnets}). Broadly speaking, in all of these routes to chaos the Entropy flow behaves as expected, i.e., pushes initial conditions towards a disordered state (see Theorem \ref{major1}, Corollary \ref{ruelletakens} and Theorem \ref{heteroclinicnets}, respectively, for the precise details). Further, in Section \ref{bifdyn} we discuss how one can use the Entropy flow to study the constrains topology places on bifurcations (see Theorem \ref{bifcodesth1} and Proposition \ref{decomposition}). Finally, in Section \ref{turbulence} we discuss the behavior of the Entropy flow around the Shilnikov homoclinic scenario (see Theorem \ref{shilnikoventropy} and Corollary \ref{turbulententropy}), which leads us to heuristically compare such Entropy flows with the Richardson’s notion of turbulence (see \cite{RichCascade}). We conclude this paper by discussing how our work can possibly continued and to what problems we expect it can possibly applied.\\

Before we begin, we would like to state that even if this may not always appear so, this paper is inspired by many numerical and scientific studies, including \cite{Niel1}, \cite{BBS}, \cite{MBKPS}, \cite{BaIbPer}, \cite{Niel2} (among many others). Moreover, many of our ideas were motivated by the results and approaches of \cite{PY2}, \cite{SY2}, \cite{SY}, \cite{facet}, \cite{transham}, which, together, led us to consider the sets $Per$ and $Fix$ as topological objects. In particular, we started this project keeping in mind the next quote from the preface of \cite{AAIS} (see page 7 in it): "\textit{\textbf{Reformulating the words of Poincare on periodic solutions, one may say that bifurcations, like torches, light the way from well-understood dynamical systems to unstudied ones}}". 

\subsection*{Acknowledgements}
The authors are grateful to James A. Yorke, Warwick Tucker, Tali Pinsky, Zin Arai, Genadi Levin, Vered Rom Kedar, Michał Lipiński, Bowen Chen, Joshua Haim Mamou, Evelyn Sander, Jinxin Xue, Michael Faran, Domenico Lippolis, Jinzi Mac Huang, Bin Shi, Mikhail Zhitomirskii, Gabriel Teixeira Guimarães, Michael Khanevsky and Noy Soffer-Aranov for their helpful comments and suggestions.

\section{Defining the Entropy flow}
\label{defsect}

From now on, let $\dot{s}=F_\tau(s)$ be a $C^1$ one--parameter family of vector fields defined on a smooth manifold $M$, parameterized by $\tau\in I$. Our goal in this Section is to construct and define the Entropy flow on $M\times I$ in a way that captures how the dynamics of $\dot{s}=F_\tau(s)$ transition from simple behavior into chaos, based on the bifurcations occuring as $\tau$ is varied in $I$ (see Definition \ref{generalentropy}). This Section is organized as follows: we begin by recalling several facts about the Implicit Function Theorem which leads us to define two subsets of $M\times I$, $Per$ and $Fix$, corresponding to the "distinguishable" periodic orbits and fixed fixed points of $\dot{s}=F_\tau(s)$ (see Definition \ref{regular}). As will be clear, $Per$ and $Fix$ are non-empty for a large class of flows, including many "real life" versions - in particular, the components of $Per$ and $Fix$ are often separated from one another by bifurcation orbits and fixed points. Following that, to motivate our construction, we heuristically discuss what desired properties we would like the Entropy flow to have. This will lead us to prove several technical facts (see Propositions \ref{def1}, \ref{def11} and Corollary \ref{def2}), that would allow us to transform the heuristic into a rigorous construction (see Subsections \ref{sec1}, \ref{sec2} and \ref{sec3}). At this point we remark that despite our idea being relatively geometric and straightforward, to make it rigorous our arguments will be extremely technical. Therefore, to illustrate how one can actually analyze Entropy flows for $C^1$ families of vector fields at the end of this Section we provide two explicit examples (see Subsections \ref{vpo} and \ref{pitchhh}).\\

To begin, from now on $M$ will always be a connected, locally compact, separable, smooth orientable manifold of dimension at least $1$ - we will also often implicitly assume $M$ is metrizable. $I$ will always denote a connected subset of either $\mathbb{R}$ or $S^1$, that includes some open set (although for the sake of discussion, for most of this Section $I$ would be assumed to be $(-1,1)$). We will always denote by $\dot{s}=F_\tau(s)$ a $C^1$ one--parameter family, where $F_\tau(s)$ will always be defined by a $C^1$-map $F:M\times I\to TM$ satisfying the equation $F(s,\tau)=F_\tau(s)$. Our goal in this Section is to define two sufficiently smooth functions, $V,P:M\times I\to\mathbb{R}$ s.t. the following system of differential equations:
\begin{equation}
\label{reg}
    \begin{cases}
        \dot{s}=F_\tau(s)+V(s,\tau)\\
        \dot{\tau}=P(s,\tau)
    \end{cases}
\end{equation}
is well-defined and at least $C^1$ in $M\times I$, and satisfies the following:
\begin{itemize}
    \item $P(s,\tau)$ points in the direction in $I$ where the trajectory of $s$ is "more complex" w.r.t. $F_\tau$ - where by "more complex" we mean "oscillates between more fixed points and periodic orbits".
    \item The function $V$ is a function ensuring that if $\mathcal{J}_\tau\subseteq M$ is an invariant set for $F_\tau$, then the flow on $M\times I$ has a connected invariant set $\mathcal{J}\subseteq M\times I$ s.t. $\mathcal{J}\cap M\times\{\tau\}=\mathcal{J}_\tau$ - in other words, $V$ is a correction term ensuring tangency to continuously varying invariant sets.\\ 
\end{itemize}

To give a more geometric explanation which types of $P$ and $V$ we are looking for, consider, for example, the case when $I=(-1,1)$ and $s$ lies on some chaotic attractor $A_{0}$ in $M$ for $F_{0}$ which persists in $\tau\in(-\epsilon,\epsilon)$ and whose complexity increases when, say, $\tau$ increases (for some small $\epsilon>0$). In this case, we would like $P$ to be positive while $V$ should ensure the trajectory of $(s,0)$ remains tangent to $A=\cup_{\tau\in(-\epsilon,+\epsilon)}A_\tau\times\{\tau\}$ as we move in $M\times I$. In particular, if there exists some $\tau_0\in(-\epsilon,\epsilon)$ s.t. the chaotic attractor $A_{\tau_0}$ for $F_{\tau_0}$ is the "most complex" in some sense, we would like $P$ to always point towards $A_{\tau_0}\times\{\tau_0\}$, at least around $A$. That being said, we would not like our dynamics to be "too far away" from those of the original system, or, in other words, we would like our flow not to be too far in the $C^1$-metric from the system:
\begin{equation}
\label{lam1}
    \begin{cases}
        \dot{s}=F_\tau(s)\\
        \dot{\tau}=0
    \end{cases}
\end{equation}
These requirements on $P$ and $V$ would ensure the trajectories of this new flow would be projected from $M\times I$ to $M$ they would appear as if the motion becomes more and more irregular around $A$, without changing its dynamical structure "too much". It is this geometric intuition that will guide our construction of both $P$ and $V$. To begin making this idea precise, we now recall the following Corollaries of the Implicit Function Theorem which we state in a topological way more suitable for our needs:
\begin{corollary}
    \label{regularitylemma} Assume $I=(-1,1)$ or $I=S^1$, and that $x_0$ is a fixed point for the vector field $F_{0}$ s.t. $|J_{0}(0)|\ne0$, where $J_{0}$ denotes the Jacobian of $F_{0}$ at $x_0$. Then, there exists some $\epsilon>$ and an arc $\gamma$ in $M\times(-\epsilon,\epsilon)$ passing through $(x_0,0)$ s.t the following \textbf{isolation properties} hold:
    \begin{itemize}
        \item $\gamma$ is a $C^1$ curve homeomorphic to an open interval.
        \item For all $\tau\in(-\epsilon,\epsilon)$, $\gamma$ is transverse to $M\times\{\tau\}$ at a unique point, $(x_\tau,\tau)$ satisfying $F_\tau(x_\tau)=0$. In particular, $\gamma\cap M\times\{0\}=(x_0,0)$.
        \item There exists a connected neighborhood $N$ of $\gamma$ in $M\times (-\epsilon,+\epsilon)$ s.t. for all $\tau|\in (-\epsilon,\epsilon)$, $F_\tau$ does not vanish in the set $N\setminus\gamma$.
    \end{itemize}
    
   Similarly, let $T_0$ be a periodic orbit for $F_{0}$ and let $f_0:S_0\to S_0$ be some first-return map for a cross-section $S_0$ transverse to $T_0$ s.t. $S\cap T_0=\{p_0\}$. Let $D_0(T_0)$ denote the differential of $f_0$ at $p_0$ - then, if all the eigenvalues of $D_0(T_0)$ are not roots of unity, there exists some $\epsilon>0$ and some $C^1$-surface $\mathcal{S}$ in $M\times(-\epsilon,\epsilon)$ passing through $T_0\times\{\tau_0\}$ satisfying the following \textbf{isolation properties}:
    \begin{itemize}
        \item $\mathcal{S}$ is homeomorphic to a cylinder , i.e., to $S^1\times(-\epsilon,\epsilon)$.
        \item $\mathcal{S}$ is transverse to $M\times\{\tau\}$ for all $\tau\in(-\epsilon,\epsilon)$, and $\mathcal{S}\cap M\times\{\tau\}=T_\tau\times\{\tau\}$ where $T_\tau$ is a periodic orbit for the vector field $F_\tau$. 
        \item For all $\tau\in(-\epsilon,\epsilon)$ there exists a cross-section $S_\tau$, transverse to $T_\tau$ w.r.t. to $F_\tau$ and varying $C^1$ with $\tau$, s.t. the following two conditions hold:
        \begin{enumerate}
            \item The first-return map $f_\tau:S_\tau\to S_\tau$ is well-defined and continuous, and varies smoothly with $\tau$. 
            \item  If $\{p_\tau\}=T_\tau\cap S_\tau$, then $\{p_\tau\}=\{y\in S_\tau|f_\tau(y)=y\}$.
            \item For all $n>0$ there exists some $\epsilon'\in(0,\epsilon]$ and some cross-section $S'_\tau\subseteq S_\tau$, varying $C^1$ as we vary $\tau\in(-\epsilon',\epsilon')$ s.t. $\{x_\tau\}=\{y\in S_{\tau'}|f^n_\tau(y)=y\}$.
        \end{enumerate}
    \end{itemize}
    
\end{corollary}
\begin{remark}
    When $M=\mathbb{R}^3$ or $S^3$, one further conclusion is added - namely, that the knot type of the periodic orbit $T_\tau$ in $M$ doesn't change as we vary $\tau$ in $(-\epsilon,\epsilon)$.  Moreover, by Proposition 3.3 in \cite{PY2}, we know that in all dimensions, the Mallet-Yorke Index of $T_\tau$ is well-defined and remains unchanged as $\tau$ is varied in $(-\epsilon,\epsilon)$.
\end{remark}
It is easy to see that every hyperbolic periodic orbit or fixed point satisfies the above, i.e., they are both isolated and persist under small perturbations. Inspired by Corollary \ref{regularitylemma} above, we now introduce the following preliminary definitions for "well-behaved" periodic orbits and fixed points:
\begin{definition}
    \label{regular} Consider a $C^1$ one--parameter family $\dot{s}=F_\tau(s)$ defined on a smooth manifold $M$ as above, where $\tau$ varies $C^1$ in $I$. A \textbf{regular}\textbf{ periodic orbit} for $F_\tau$, $\tau\in I$ is a periodic orbit that satisfies the conclusions of Corollary \ref{regularitylemma}, and \textbf{does not} correspond to a Neimark-Sacker bifurcation orbit. Similarly, a \textbf{regular} \textbf{fixed point} for $F_\tau$ is a fixed point satisfying the conclusions of Corollary \ref{regularitylemma}, which is also \textbf{not} a Hopf bifurcation point. We now define the following sets in the product space $M\times I$:
    \begin{itemize}
        \item $Reg_1=\{(x,\tau)|\text{$x$ lies on a regular periodic orbit for $F_\tau$}\}$.
        \item $Reg_2=\{(x,\tau)|\text{$x$ is a regular fixed point for $F_\tau(x)$}\}$.
    \end{itemize}
    
By definition, every component $\alpha$ in $Reg_1$ is a $C^1$-surface and locally homeomorphic to a cylinder. Similarly, every component $\alpha$ of $Reg_2$ is locally homeomorphic to $(0,1)$. 
\end{definition}
\begin{remark}
 We stress that in Definition \ref{regular} we only require the regular periodic orbits or fixed points to satisfy the conclusions of Corollary \ref{regularitylemma}. We \textbf{do not} require them to also satisfy the assumptions of Corollary \ref{regularitylemma}, or those of the Implicit Function Theorem. The reason for that is because we want our definitions to apply in the broadest setting.
\end{remark}

Before proceeding we remark that despite the abstract setting, generically, at least the set $Reg_1$ is non-empty (see the Appendix of \cite{PY2}). More precisely, as discussed and proven in both \cite{PY} and \cite{PY2}, for a $C^3$-generic choice of a $C^3$ one--parameter family $\dot{s}=F_\tau(s)$, all periodic orbits are either regular, or correspond to either a saddle node, period Doubling, or Neimark-Sacker bifurcation (for more details, see \cite{PY}, \cite{PY2}, \cite{KY2} and the references therein). Consequently, the case where at least $Reg_1\ne\emptyset$ can be thought of as the general rule. The reason we care about regular periodic orbits and fixed points is because of two reasons: the first is that the different components of $Reg_1\cup Reg_2$ are often glued to one another at bifurcations. The second is that these periodic orbits and fixed points are, in some sense, isolated - at least provided their local dynamics have certain behavior. To make this idea precise, we now define:
\begin{definition}
    \label{isolationdef} Consider a $C^1$ family $\dot{s}=F_\tau(s)$ defined over some smooth manifold $M$, where the parameter $\tau$ varies in $I$. With previous notations, given a periodic orbit $T\times\{\tau\}\subseteq Reg_1$ for $F_\tau$, we say it is \textbf{completely isolated} if there exists some neighborhood $N$ of $T\times\{\tau\}$ in $M\times I$ s.t. if $\alpha$ is the component of $Reg_1$ for which $T\times\{\tau\}\subseteq \alpha$, then $(N\cap( \overline{Reg_1\cup Reg_2}))\subseteq \alpha$. Similarly, given a fixed point $(x,\tau)\in Reg_2$ for $F_\tau$ is \textbf{completely isolated} if there exists some neighborhood $N$ s.t. if $\alpha$ is the component of $Reg_2$ for which $(x,\tau)\in \alpha$, then $(N\cap(\overline{Reg_1\cup Reg_2}))\subseteq\alpha$. We define:
    \begin{itemize}
        \item $Per=\{(s,\tau)\in Reg_1|\text{$s$ lies on a completely isolated periodic orbit for the vector field $F_\tau$}\}$.
        \item $Fix = \{(s,\tau)\in Reg_2|\text{$s$ is a completely isolated fixed point for the vector field $F_\tau$}\}$.
    \end{itemize}
\end{definition}
In other words, the sets $Per$ and $Fix$ correspond to the "distinguishable" periodic orbits as $\tau$ is varied in $I$. For example, $Per$ includes all the subsets of $Reg_1$ corresponding to attractors and repellers - and similarly, $Fix$ includes all the fixed points that are sinks or sources (that being said, not all periodic orbits in $Per$ need be attractors or repellers, and not all fixed points in $Fix$ need be sinks and sources). Intuitively, the components of $Fix\cup Per$ are glued to one another at bifurcation periodic orbits and/or bifurcation fixed points. As such, recalling the previous discussion, it makes sense to define the functions $P$ and $V$ described above as directions on the components of $Per\cup Fix$ pointing towards (or away) from bifurcations. With that idea in mind, we are now finally ready to start defining the Entropy flow, which we will do in three stages - see Subsections \ref{sec1}, \ref{sec2} and \ref{sec3}. Recall that given a $C^1$ family $\dot{s}=F_\tau(s)$ defined on a manifold $M$, where$\tau$ varies$ I$, our goal is to find sufficiently smooth functions $P$ and $V$ s.t. the following system:
\begin{equation*}
    \begin{cases}
        \dot{s}=F_\tau(s)+V(s,\tau)\\
        \dot{\tau}=P(s,\tau)
    \end{cases}
\end{equation*}
is well-defined, pushes in the direction of more dynamical "complexity"  and keeps the invariant sets invariant "as best as possible". As explained above, we will do so by perturbing Equations \ref{lam1} around the components of $Per$ and $Fix$. To begin, from now on unless said otherwise, let $\dot{s}=F_\tau(s)$ denote a $C^1$ one--parameter family of vector fields, where $\tau$ varies in an open interval, $(-1,1)$ - we will define the generalization to the case where $I=S^1,(-1,1],[-1,1)$ later on in this Section. We first prove the following Proposition:
\begin{proposition}
    \label{def1} Let $\dot{s}=F_\tau(s)$ be a $C^1$ one--parameter family, where $\tau$ varies in $I=(-1,1)$. Let $\alpha$ be a component in $Per$, and let $N_\alpha$ be a neighborhood of $\alpha$ in $M\times I$ s.t. $N_\alpha\cap (Reg_1\cup Reg_2)=\alpha$ (where $Reg_1$ and $Reg_2$ are as in Definition \ref{regular}). Then, there exist $C^1$-functions $P_\alpha:M\times I\to\mathbb{R}$, $V_\alpha:M\times I\to T(M\times I)$ s.t. the vector field $E_\alpha(s,\tau)=(F_\tau(s)+V_\alpha(s,\tau),P_\alpha(s,\tau))$ satisfies the following in $M\times I$:
    \begin{itemize}
        \item $E_\alpha$ is tangent to $\alpha$.
        \item $P_\alpha$ and $V_\alpha$ both vanish outside $N_\alpha$.
        \item The sign of $P_\alpha(s,\tau)$ in $N_\alpha$ is non-zero throughout $\alpha$, and can be chosen to be either positive or negative. Consequently, $E_\alpha$ has no fixed points in $N_\alpha$.
    \end{itemize}
\end{proposition}
\begin{proof}

We begin by proving that given some connected, metrizable, locally compact, and separable smooth manifold $\mathcal{N}$ and some connected open subset $O$ in $\mathcal{N}$, there exists a smooth function $g:\mathcal{N}\to\mathbb{R}$ whose support is exactly $O$ - and that when $O$ is compact, its $C^1$-norm can be chosen to be arbitrarily small. To this end, recall we assume $\mathcal{N}$ is metrizable, separable, smooth, and locally compact. Let $\{B_{r_i}(x_i)\}_{i\geq1}$ be a countable collection of balls with radii $r_i$ s.t. both $\cup_i B_{r_i}(x_i)=O$ and $B_{2r_i}(x_i)\subseteq O$. Now, let $\Psi$ be a partition of unity for $M$ and let $\{\psi_i\}_{i\geq1}$ be a subset of $\Psi$ defined by the following:
\begin{itemize}
    \item $supp(\psi_i)\subseteq B_{2r_i}(x_i)$ (where $supp$ denotes the support).
    \item $\psi_i(s)>0$ for all $s\in B_{r_i}(x_i)$.
    \item $\psi_i(x)\in[0,1)$ for all $x\in \overline{B_{2r_i}(x_i)}$.
    \item The collection is locally finite, i.e., for all $i$ there is a finite collection of $j$'s s.t. $supp(\psi_i)\cap supp(\psi_j)\ne\emptyset$.
\end{itemize}

Such a collection exists by the proof of Theorem 2.23 of \cite{Lee}. Now, choose some $\epsilon>0$ and define $g_\epsilon(x)=\epsilon\sum_{i\geq1}\frac{\psi_i(x)}{2^i}$. By definition, $g$ is smooth and the set where $g$ does not vanish is precisely $O$, therefore when $O$ is precompact in $\mathcal{N}$ by decreasing $\epsilon>0$ we can ensure that the $C^1$-norm of $g$ is arbitrarily small in $O$. We now choose $\mathcal{N}=M\times I$ and $O$ as $N_\alpha$, where we set $P''(s,\tau)$ as either $g_\epsilon(s,\tau)$ or $-g_\epsilon(s,\tau)$ for some $\epsilon$ - by definition, $P''$ has constant sign on $\alpha$. Note that given some $(s,\tau)\in N_\alpha$ the vectors $v_1=(0,P''(s,\tau))$ and $v_2=(F_\tau(s),0)$ are linearly independent since $P''$ is non-zero in $N_\alpha$. To continue, let $\mathcal{P}(s,\tau)$ denote the plane tangent to the surface $\alpha$ at $(s,\tau)$ - since $\alpha$ is a smooth surface per Definition \ref{regular} and Definition \ref{isolationdef}, $\mathcal{P}(s,\tau)$ changes smoothly as $(s,\tau)$ is varied. We now consider a $C^1$ vector field $\Pi:M\times I\to T(M\times I)$, satisfying the following three properties:
\begin{itemize}
    \item When $(s,\tau)\in \alpha$, $\Pi(s,\tau)=\pi(s,\tau)$ - where $\pi(s,\tau)$ denotes the projection of the vector $(F_\tau(s),P''(s,\tau))$ to the plane $\mathcal{P}(s,\tau)$.
    \item  $\Pi^{-1}(T\alpha)\cap N_\alpha=\alpha$, where $T\alpha$ is the bundle of tangent planes to $\alpha$, i.e., $T\alpha=\cup_{(s,\tau)\in \alpha}\mathcal{P}(s,\tau)\times\{(s,\tau)\}$. 
    \item $\Pi(s,\tau)=(F_\tau(s),0)$ outside of $N_\alpha$.
\end{itemize}

We now claim that for all $(s,\tau)\in \alpha$ we have $\Pi(s,\tau)\cdot(0,P''(s,\tau))>0$. To see why, note that $(F_\tau(s),0)$ is tangent to $M\times\{\tau\}$ and lies on $\mathcal{P}(s,\tau)$, hence $(F_\tau(s),P''(s,\tau))$ points into $M\times(\tau,1)$ when $P''(s,\tau)>0$ and into $M\times(-1,\tau)$ when $P''(s,\tau)<0$. As $\alpha$ is transverse to $M\times\{\tau\}$ at any $(s,\tau)\in \alpha$, it follows the plane $\mathcal{P}(s,\tau)$ is transverse to $M\times\{\tau\}$ at $(s,\tau)$. Hence, as $\pi(s,\tau)=\Pi(s,\tau)$ for all $(s,\tau)\in \alpha$, we conclude that for all initial conditions $(s,\tau)\in \alpha$ we have $\Pi(s,\tau)\cdot(0,P''(s,\tau))>0$. All in all, setting $E_\alpha(s,\tau)=\Pi(s,\tau)$ it follows that for all $(s,\tau)\in \alpha$ the vector field $E_\alpha(s,\tau)=\Pi(s,\tau)$ is tangent to $\alpha$, and has no fixed points near $\alpha$. This implies we can choose it s.t. $E_\alpha(s,\tau)\cdot(0,1)$ has constant sign throughout $N_\alpha$ - i.e., $E_\alpha$ has no fixed points in $N_\alpha$. Finally, it is also easy to see we can write $E_\alpha(s,\tau)=(F_\tau(s)+V_\alpha(s,\tau),P_\alpha(s,\tau))$ for some $C^1$-functions $P_\alpha$ and some correction term $V_\alpha$, which ensures the tangency of $E_\alpha$ to $\alpha$. 
\end{proof}
Having shown we can construct functions $P_\alpha$ and $V_\alpha$ around a cylinder of periodic orbits $\alpha\subseteq Per$, we now prove we can do the same for components $\alpha\subseteq Fix$:
\begin{proposition}
   \label{def11}  Let $\dot{s}=F_\tau(s)$ be a $C^1$ one--parameter family, where $\tau$ varies in $I=(-1,1)$. Let $\alpha$ be a component of $Fix$ and let $N_\alpha$ denote a neighborhood of $\alpha$ in $M\times I$ s.t. $N_\alpha\cap(Reg_1\cup Reg_2)=\alpha$, and $I=(-1,1)$. Similarly to Proposition \ref{def1}, there exist $C^1$-functions $P_\alpha:M\times I\to\mathbb{R}$ and $V_\alpha:M\times I\to T(M\times I)$ s.t. the vector field $E_\alpha(s,\tau)=(F_\tau(s)+V_\alpha(s,\tau),P_\alpha(s,\tau))$ satisfies the following:
    \begin{itemize}
        \item $E_\alpha$ is tangent to $\alpha$. 
        \item $P_\alpha$ and $V_\alpha$ both vanish outside $N_\alpha$.
        \item The sign of $P_\alpha(s,\tau)$ in non-zero and doesn't change as we vary $(s,\tau)$ on $\alpha$. It can be chosen to be either positive or negative.
        \item  $E_\alpha$ has no fixed points in $N_\alpha$.
        \end{itemize}
\end{proposition}
\begin{proof}
    We replicate the argument of Proposition \ref{def1} to $N_\alpha$. The function $P''$ is defined in the same way, but for $N_\alpha$ instead of $N_\alpha$ - i.e., using a partition of unity whose support lies strictly inside $N_\alpha$. Similarly, as $\alpha$ is a curve transverse to $M\times \{\tau\}$ at any $(s,\tau)\in\alpha$, there exists a $C^1$ vector field $\Pi:M\times I\to T(M\times I)$, satisfying the following:
\begin{itemize}
    \item When $(s,\tau)\in\alpha$, $\Pi(s,\tau)$ is just the projection of the vector $(F_\tau(s),P''(s,\tau))=(0,P''(s,\tau))$ to $T\alpha$, the bundle of lines tangent to $\alpha$.
    \item $\Pi$ satisfies $\Pi^{-1}(T\alpha)=\alpha$.
    \item $\Pi(s,\tau)=(F_\tau(s),0)$ outside of $N_\alpha$.
    \item The component of $\Pi$ in the $\tau$ direction is non-zero throughout $N_\alpha$.
\end{itemize}
Similarly, the vector field $\Pi$ can be rewritten as $\Pi(s,\tau)=E_\alpha(s,\tau)=(F_\tau(s)+V_\alpha(s,\tau),P_\alpha(s,\tau))$ inside $N_\alpha$, where the functions $P_\alpha,V_\alpha$ are at least $C^1$ and vanish outside of $N_\alpha$. Moreover, we can again ensure $E_\alpha$ is such that the sign of $P_\alpha$ is constant throughout $N_\alpha$ - which implies there are no fixed points for $E_\alpha$ in $N_\alpha$.
\end{proof}
\begin{remark}
At first sight, $V_\alpha$ are just correction terms ensuring the tangency of $E_\alpha$ to $\alpha\subseteq (Per\cup Fix)$. Towards the end of this Section we will give these functions a more concrete physical interpretation.
\end{remark}
The arguments above can be iterated and modified, and in fact, we will also need several simple modifications of them in order to construct the Entropy flow. For example, a similar argument to the one used in Propositions \ref{def1}, \ref{def11} implies the following useful Corollary, which we will be heavily applying throughout the definition of the Entropy flow:
\begin{corollary}
    \label{def2} 
Let $\dot{s}=F_\tau(s)$ be a $C^1$ one--parameter family, where $\tau$ varies in $I=(-1,1)$. Let $\alpha$ be a component of $Per$, let $\tau\in I$ be a parameter s.t. $\alpha\setminus M\times\{\tau\}$ has exactly two non-empty components, and let $N_\alpha$ be an isolating set as in Proposition \ref{def1}. Then, there exist $C^1$-functions $P_\alpha$ and $V_\alpha$ satisfying the following:

\begin{itemize}
    \item Let $N^1_\alpha$ and $N^2_\alpha$ denote the components of $N_\alpha\cap M\times(-1,\tau)$ and $N_\alpha\cap M\times(\tau,1)$ (respectively). Then, we can choose $P_\alpha$ and $V_\alpha$ s.t. $P_\alpha$ has opposing signs in $N^1_\alpha$ and $N^2_\alpha$, and the vector field $E_\alpha(s,\tau)=(F_\tau(s)+V_\alpha(s,\tau),P_\alpha(s,\tau))$ is tangent to $\alpha$.
    \item The flow of $E_\alpha(s,\tau)$ has exactly one periodic orbit on $\alpha$, given by the intersection $R=M\times\{\tau\}\cap \alpha$, and no fixed points in $N_\alpha$. In particular, $R=T\times\{\tau\}$ where $T$ is a periodic orbit for $F_\tau$ (see the upper image in Figure \ref{eq1}).
\end{itemize}

Similarly, let $\alpha$ be a component in $Fix$ s.t. $\alpha\setminus M\times\{\tau\}$ has two components, where $N_\alpha$ is an isolating sets as in Proposition \ref{def11}. Then, there exist $C^1$-functions $P_\alpha$ and $V_\alpha$ satisfying the following:    
    \begin{itemize}
        \item Let $N^1_\alpha$, $N^2_\alpha$ denote the components of $N_\alpha\cap M\times(-1,\tau)$, $N_\alpha\cap M\times(\tau,1)$ (respectively). Then, $P_\alpha$ is negative in $N^1_\alpha$, positive in $N^2_\alpha$, and the vector field $E_\alpha(s,\tau)=(F_\tau(s)+V_\alpha(s,\tau),P_\alpha(s,\tau))$ tangent to $\alpha$. 
        \item The vector field $E_\alpha$ will have exactly one fixed point interior to $N_\alpha$, given by $(x_\tau,\tau)\in\alpha$.  In particular, $F_\tau(x_\tau)=0$ (see the lower image in Figure \ref{eq1}). 
    \end{itemize}
\end{corollary}
\begin{proof}
    We prove the assertion for components of periodic orbits $\alpha\subseteq Per$. The assertion for components of fixed points $\alpha\in Fix$ is proven using similar arguments to those below, with Proposition \ref{def11} instead of Proposition \ref{def1}. Choose a parameter $\tau$ s.t. $\alpha\setminus M\times\{\tau\}$ has two components, $\alpha_1$ and $\alpha_2$, and let $N^1_\alpha$, $N^2_\alpha$ and $T\times\{\tau\}$ be as above, defined by $\alpha_i\subseteq N^i_\alpha$, $i=1,2$. We now use a similar argument to the one used to prove Proposition \ref{def1} to define functions $P_i,V_i$ on $N^i_\alpha$, $i=1,2$ s.t. for all $(s,\tau)\in N^i_\alpha$ the vector $(F_\tau(s)+V_i(s,\tau),P_i(s,\tau))$, is tangent to $\alpha_i$ - where $P_1$ and $P_2$ have opposing signs in $N^1_\alpha$ and $N^2_\alpha$. To continue, define the following functions:

    \begin{equation*}
        P_\alpha(s,\tau)=\begin{cases}
            P_i(s,\tau) &\text{when $(s,\tau)\in N^i_C, i=1,2$}\\
            0 &\text{otherwise}
        \end{cases}
    \end{equation*}
    
   \begin{equation*}
        V_\alpha(s,\tau)=\begin{cases}
            V_i(s,\tau) &\text{when $(s,\tau)\in N^i_C, i=1,2$}\\
            0 &\text{otherwise}
        \end{cases}
    \end{equation*}
Since both $P_i,V_i$, $i=1,2$ are at least $C^1$, so are $P_\alpha$ and $V_\alpha$. By their definitions, both $P_\alpha$ and $V_\alpha$ vanish on $T\times\{\tau\}$, which implies that for all $(s,\tau)\in T\times\{\tau\}$ we have $E_\alpha(s,\tau)=(F_\tau(s),0)$. Since $F_\tau$ is tangent to $T$ (as it is a periodic orbit for $F_\tau$) and since for all $(s,\tau)\in T\times\{\tau\}$ we have $E_\alpha(s,\tau)=(F_\tau(s),0)$ we conclude $T\times\{\tau\}$ is a periodic orbit for $E_\alpha(s,\tau)$ in $N_\alpha$.
\end{proof}

\begin{figure}[h]
\centering
\begin{overpic}[width=0.35\textwidth]{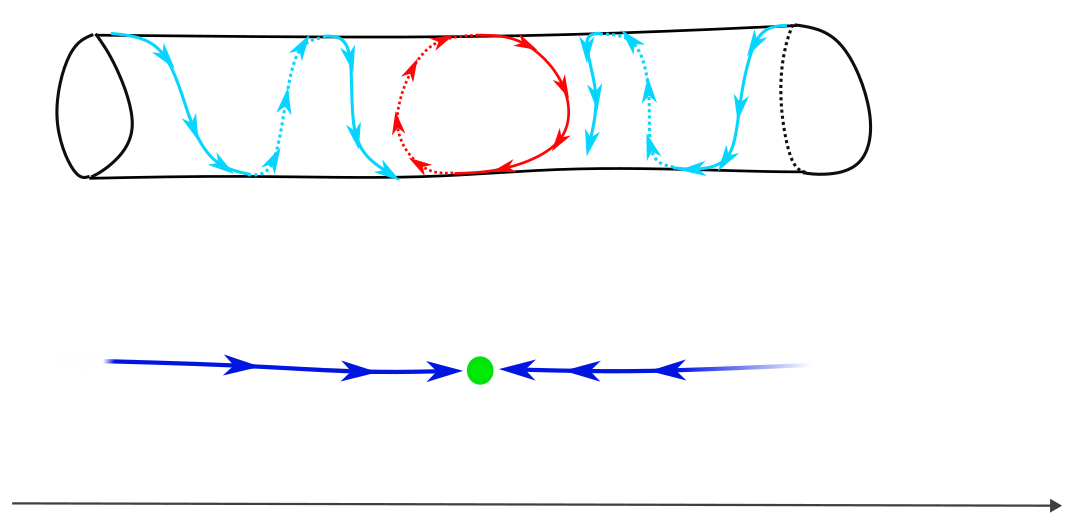}

\put(995,10){$\tau$}

\end{overpic}
\caption{\textit{On the upper image - a diagram showing the motion of $E_\alpha$ on a cylinder of periodic orbits, $\alpha\subseteq Per$, where $P_\alpha$ has opposite signs at each end of $\alpha$. The red loop corresponds to the set $M\times\{\tau\}\cap \alpha$, where both $P_\alpha$ and $V_\alpha$ vanish (hence the said loop is periodic for $F_\tau$). On the lower image there is an analogous situation for the curve $\alpha\subseteq Fix$. Since $P_\alpha$ points in opposite signs at the end of $\alpha$, there exists some fixed point $(x_\tau,\tau)$ for $E_\alpha$ on $\alpha$ (the green dot).}}
\label{eq1}
\end{figure}

Having proven Propositions \ref{def1}-\ref{def11} and Corollary \ref{def2} we now sketch how we will construct the Entropy flow. The idea is as follows: let $\{G_\alpha\}_\alpha$ denote the components of $Fix\cup Per$, or more precisely, the components of $Fix\cup Per$ \textbf{and} a collection of all isolated invariant sets that are "maximal" in some sense. As stated earlier, the components of $\{G_\alpha\}_\alpha$ should be intuitively understood to be glued to one another at bifurcation parameters. Moreover, intuitively, this collection of sets comes together with a collection of open neighborhoods $\{N_\alpha\}_\alpha$ satisfying $G_\alpha\subseteq N_\alpha$ for all $\alpha$, and $N_\alpha\cap N_\beta=\emptyset$ for all $\alpha\ne\beta$ (we will prove the existence of this collection later on, in Subsection \ref{sec3}). Our idea would be as follows: as the elements of $\{G_\alpha\}_\alpha$ are connected to one another at bifurcation orbits and fixed points, we will prescribe some required local behavior of $P_\alpha$ and $V_\alpha$ around the bifurcation orbits or fixed points connecting $G_\alpha$ to other components $G_\beta$. Following that, we use Propositions \ref{def1}, \ref{def11} and Corollary \ref{def2} (case depending) to extend this local behavior to all of $G_\alpha$ and $N_\alpha$ (see the illustration in Figure \ref{sew}). We will refer to functions $P_\alpha$ and $V_\alpha$ whose behavior around the bifurcations on $\partial G_\alpha$ coincides with our requirements \textbf{admissible} - see Definitions \ref{admissiblefixed} and \ref{admissibleper}. Finally, after proving that sets $N_\alpha$ and $N_\beta$ as described above exist, we we will sew together all these admissible functions $P_\alpha$ and $V_\alpha$ defined locally around each $N_\alpha$, thus defining the functions $P$ and $V$ needed for the global definition of the Entropy flow.\\

That being said, due to the length and technical complexity of our arguments (mostly due to the many bifurcation scenarios considered), we will divide this definition to three consecutive stages. As such, the rest of this section is organized as follows: we first describe in Subsection \ref{sec1} the required admissible behavior of $P_\alpha,V_
\alpha$ around the bifurcation fixed points connecting components $G_\alpha\subseteq Fix$ to other components of $G_\beta\subseteq (Per\cup Fix)$ (see Definition \ref{admissiblefixed}). Following that, in Subsection \ref{sec2} we lay down the analogous definition for admissible behavior on components of periodic orbits in $Per$ (see Definition \ref{admissibleper}). Finally, in Subsection \ref{sec3} we prove the existence of $\{N_\alpha\}_\alpha$, describe the extension of $P_\alpha$ and $V_\alpha$ to $N_\alpha$, extend the above to some invariant sets (and to parameter spaces other than $(-1,1)$, and tie everything together to give a global definition for the Entropy flow (see Definition \ref{generalentropy}).\\

Despite the long technicalities ahead, the thread passing through all our definitions is that we want admissible functions $P_\alpha$ and $V_\alpha$ to always push initial conditions in $M\times I$ towards a state of greater complexity along the $C^1$ curve of vector fields $\dot{s}=F_\tau(s)$ - which we will broadly interpret as $F_\tau$ having more fixed points and periodic orbits for $F_\tau$. This motivates us to conclude this Section by describing the admissible behavior around components of $Per\cup Fix$ in the absence of bifurcations. To do so, recall we consider a $C^1$ one--parameter family $\dot{s}=F_\tau(s)$, where $\tau\in(-1,1)$. Consider a component $G_\alpha\subseteq Per\cup Fix$ s.t. $G_\alpha$ corresponds to either a periodic orbit or a fixed point that does not bifurcate as we vary $\tau\in I$. When this is the case, intuitively, there is no difference in the local dynamical complexity between $F_{\tau_1}$ around $G_\alpha\cap M\times\{\tau_1\}$ compared to that of $F_{\tau_2}$ around $G_\alpha\cap M\times\{\tau_2\}$, for all $\tau_1\ne\tau_2$. As such, if $N_\alpha$ is some neighborhood of $G_\alpha$ in $M\times (-1,1)$ s.t. $N_\alpha\cap (Per\cup Fix)=\alpha$, we are motivated to prescribe the admissible behavior $P_\alpha$ and $V_\alpha$ in $N_\alpha$ as follows:
\begin{definition}
    \label{def3} Let $\dot{s}=F_\tau(s)$ be a $C^1$ one--parameter family defined over $M$, where $\tau$ varies in $(-1,1)$. Assume $G_\alpha$ is a component of $Per\cup Fix$ corresponding to either a component of periodic orbits or fixed points that do not bifurcate as $\tau$ is varied in $(-1,1)$ (in particular, for all $\tau\in(-1,1)$ $\alpha\cap M\times\{\tau\}\ne\emptyset$). Then, with previous notations and assumptions, $V_\alpha$ and $P_\alpha$ are \textbf{admissible} if they vanish identically in $N_\alpha$.
\end{definition}

\subsection{Stage $I$ - the behavior on components of $Fix$}
\label{sec1}
In this Subsection we describe the admissible behavior we require the functions $P_\alpha$ and $V_\alpha$ to have on components $G_\alpha\subseteq Fix$. We will do so on a case-by-case basis, depending on the bifurcations the fixed points in $G_\alpha$ undergo at the boundary of $G_\alpha$. To this end, we first recast several common bifurcation scenarios involving fixed points in a topological setting. Again, assume $\dot{s}=F_\tau(s)$ is a $C^1$ one--parameter family of vector fields, where $\tau$ varies in $(-1,1)$. For simplicity, we describe the scenario when the bifurcation occurs as $\tau$ crosses from $(-1,0)$ to $(0,1)$ - the symmetric case is described similarly:
\begin{itemize}
    \item \textbf{Creation} - two fixed points for $F_\tau$, $\tau<0$, denoted by $x^1_{\tau},x^2_{\tau}$ and varying continuously with $\tau$, collide and disappear at some fixed point at $x_0$ for $F_0$. One example of such a bifurcation is the saddle node bifurcation  (see  the illustration in Figure \ref{saddlepitch}). 
    \item \textbf{Splitting} - a fixed point $x_\tau$ for $F_\tau$ persists as a fixed point for $\tau<0$, and when $\tau$ enters $(0,1)$, the fixed point $x_0$ splits into finitely (or countably) many fixed points. For example, a pitchfork bifurcation corresponds to this kind of bifurcation, where a fixed point splits into three different fixed point (see the illustration in Figure \ref{saddlepitch}).
    \item \textbf{Expansion} - a fixed point $x_\tau$ for $F_\tau$ persists in $\tau<0$, and when $\tau$ crosses into $(0,1)$, one (or more) periodic orbits are born from it while the fixed point $x_\tau$ persists for $\tau>0$. For example, a Hopf bifurcation is an example of such a bifurcation (see the illustration in Figure \ref{hopff}). In this case, we assume there are no zero eigenvalues for the Jacobian of $F_0$ at the bifurcation fixed point $x_0$.
    \item \textbf{Mixed} - any other type of bifurcation not falling into the classification above. For example, a Fold-Hopf bifurcation, which can have characteristics of both Hopf and saddle node bifurcation \cite{guck}.\\
\end{itemize}

\begin{figure}[h]
\centering
\begin{overpic}[width=0.35\textwidth]{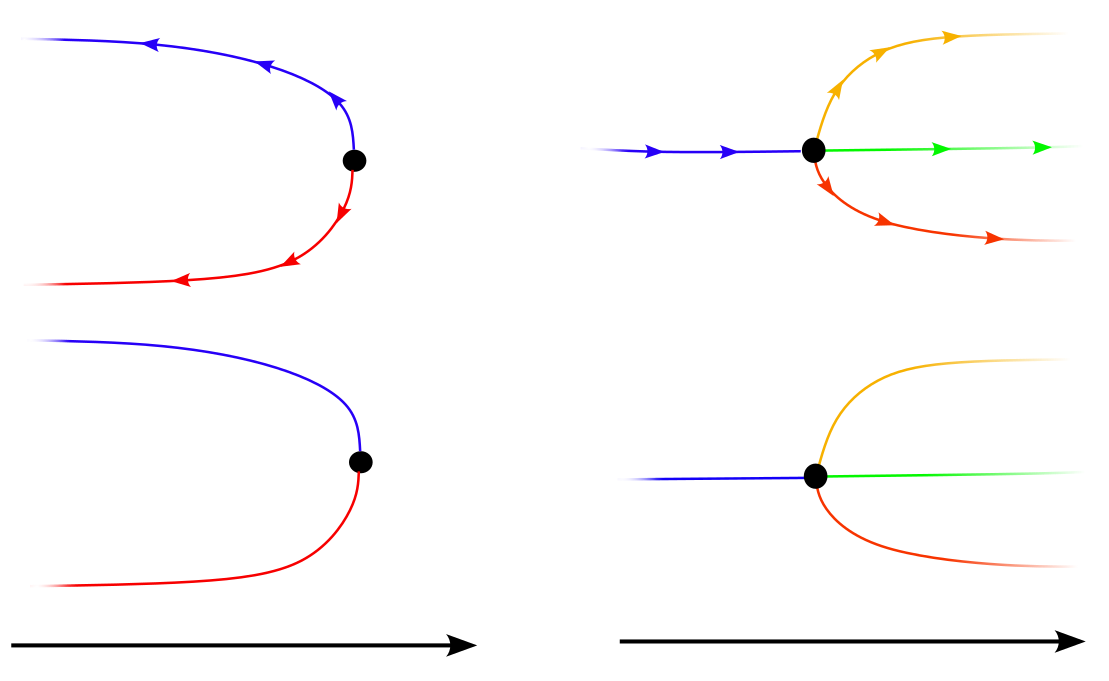}

\put(445,15){$\tau$}
\put(995,18){$\tau$}

\end{overpic}
\caption{\textit{On the lower left - a saddle node bifurcation, where two components in $Fix$ (or $Reg_2$ - see Definition \ref{regular}) collide. On the lower right - a component in $Fix$ which undergoes a pitchfork bifurcation where it splits into three distinct components. In each diagram, a different color represents different components in $Fix$ (or $Reg_2$), while the black dot represents the bifurcation point. On the upper left and upper right we see the direction of the admissible behavior (in this illustration, all fixed points are assumed to be in $Fix$).}}
\label{saddlepitch}
\end{figure}
In each of the bifurcation types above one could describe the bifurcation as a topological property of curves, which may (or may not) lie in $Reg_2$, i.e., the set of regular fixed points, that meet at some bifurcation point $x_0$ (see Definition \ref{regular}). Moreover, in the case of expansion type bifurcation, i.e., the third type, the topological scenario can be described as a curve of fixed points which is transverse to one (or more) surfaces in $M\times (-1,1)$ corresponding to the families of periodic orbits bifurcating from the fixed point $x_0$. With this geometric image in mind, consider a component $G_\alpha\subseteq Fix$ corresponding to fixed points which eventually bifurcate as $\tau$ is varied. We now prescribe the admissible behavior of the functions $P_\alpha$ and $V_\alpha$ around the bifurcation fixed points on the boundary of $G_\alpha$ (i.e., where it connects with components $G_\beta\subseteq Per\cup Fix$, where $\beta\ne\alpha$). For brevity, throughout this Subsection we will always implicitly assume the existence of open sets $N_\alpha$ and $N_\beta$, $N_\alpha\cap N_\beta=\emptyset$ for all components $G_\beta\subseteq Fix\cup Per$ connecting with $G_\alpha$ at the bifurcation fixed point $(x_0,0)$ - as stated previously, we will deal with the question of existence for these sets later on in Subsection \ref{sec3}. Also, to avoid repetition, the notation $V_\alpha$ would always denote some $C^1$ correction term s.t. for all $(s,\tau)\in G_\alpha$ the vector $(F_\tau(s)+V_\alpha(s,\tau),P_\alpha(s,\tau))$ is tangent to $G_\alpha$ (such correction terms exist by Proposition \ref{def11} and Corollary \ref{def2}). We first deal with creation type bifurcations of fixed points:
\begin{definition}
    \label{def4} Without any loss of generality, let $G_\alpha\subseteq M\times (-1,0)$ be a component of fixed points in $Fix$. Assume $(x_0,0)$ is a creation type bifurcation point which connects to some finite or countable collection of components $\{G_{\delta_i}\}_{i=1}^k\subseteq Reg_2$, where $1\leq i\leq k\leq\infty$ - by the above definition of creation type bifurcation, we have $G_{\delta_i}\subseteq M\times(-1,0)$. Then, for $P_\alpha$, $V_\alpha$, $P_{\delta_i}$ and $V_{\delta_i}$ to be admissible they must satisfy the following on $G_\alpha$ around $(x_0,0)$:
    \begin{itemize}
        \item   $P_\alpha$ always points in the direction of creation, i.e., $P_\alpha(s,\tau)$ is negative for $(s,\tau)\in \alpha$ sufficiently close to $(x_0,0)$ (see the illustration in Figure \ref{saddlepitch}). 
        \item For all $1\leq i\leq k$ s.t. $G_{\delta_i}$ is also completely isolated, we similarly require $P_{\delta_i}(s,\tau)$ to be negative on all $(s,\tau)\in\delta$ sufficiently close to $(x_0,0)$ (see the illustration in Figure \ref{saddlepitch}). 
        \item Moreover, $P_\alpha,P_{\delta_i}$ and $V_\alpha,V_{\delta_i}$ and their differentials are all required vanish at $(x_0,0)$, i.e., $(F_0(x_0)+V_r(x_0,0),P_r(x_0,0))=(F_0(x_0),0)=(0,0)$, for all $r=\alpha,\delta_i$, $1\leq i\leq k\leq\infty$. 
    \end{itemize}

The symmetric case where $G_\alpha,\{G_{\delta_i}\}_{i=1}^k\subseteq M\times(0,1)$ is given by reflecting the definition above as illustrated in Figure \ref{reflect1}.
\end{definition}
\begin{figure}[h]
\centering
\begin{overpic}[width=0.35\textwidth]{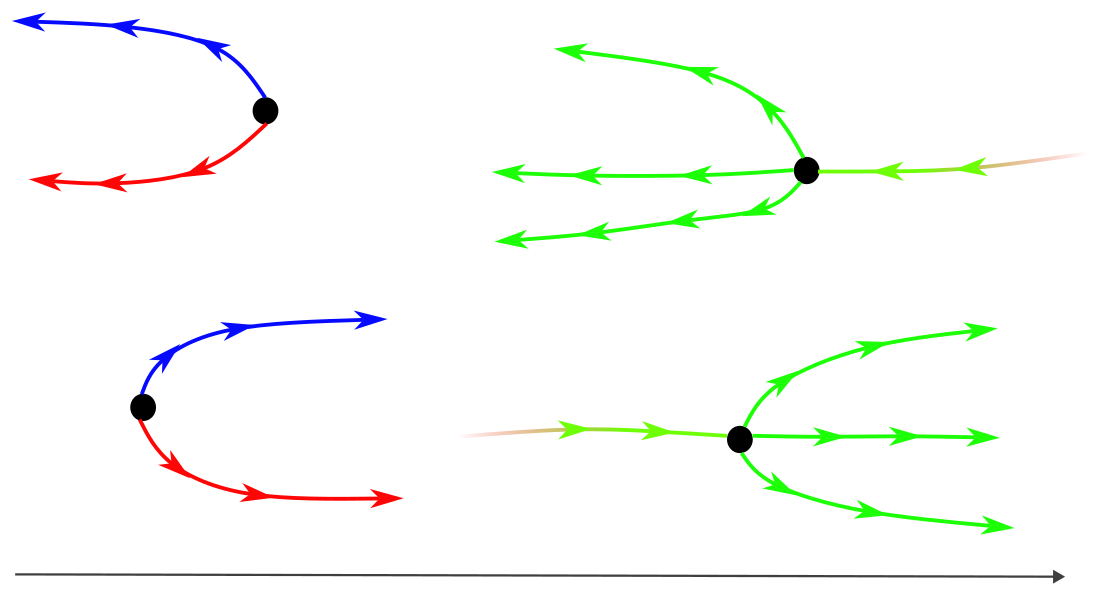}

\put(990,25){$\tau$}

\end{overpic}
\caption{\textit{On the left - a diagram for the behavior of $P_\alpha$ and $V_\alpha$ around two saddle node bifurcations (i.e., creation type) for fixed points (both point in the direction of creation). On the right - diagrams with the same motion for a pitchfork bifurcation (i.e., splitting type), where the motion is in the direction of splitting. The black dots represent bifurcation points.}}
\label{reflect1}
\end{figure}

In other words, at creation types bifurcation, we want the local behavior of $P_\alpha$ and $V_\alpha$ on $G_\alpha$ around the bifurcation to point in the direction in the parameter space of existence of more fixed points. Specifically, we want to push the trajectory of initial conditions $(s,\tau)\in M\times (-1,1)$ close to $(x_0,0)$ towards a state $F_\tau$ where the motion has a larger non-wandering set. We now replicate the same idea for splitting type bifurcations:

\begin{definition}
      \label{def41} Again, let $G_\alpha\subseteq M\times (-1,0)$ be a curve of fixed points in $Fix$. Assume $G_\alpha$ is a component in $Fix$ that connects to some bifurcation fixed point $(x_0,0)$ of the splitting type, where $G_\alpha$ connects at $(x_0,0)$ with a collection of components of $Reg_2$, $\{G_{\delta_i}\}_{i=1}^k$, $1\leq k\leq\infty$ - by definition, $G_{\delta_i}\subseteq M\times(0,1)$ for all $i$. Then, $P_\alpha$, $V_\alpha$, $P_{\delta_i}$ and $V_{\delta_i}$ are admissible if they satisfy the following around $(x_0,0)$:
    \begin{itemize}
        \item For all $(s,\tau)\in\alpha$ sufficiently close to $(x_0,0)$ we require $P_\alpha(s,\tau)$ to be positive, i.e., $P_\alpha$ will point towards $(x_0,0)$, the direction of splitting (see the illustrations in Figure \ref{saddlepitch} and Figure \ref{reflect1}). 
        \item Similarly, for all $i$ s.t. $\delta_i\subseteq Fix$, we require $P_{\delta_i}(s,\tau)$ to be positive on all $(s,\tau)\in\delta_i$ sufficiently close to $(x_0,0)$ (see the illustration in Figure \ref{saddlepitch}). 

    \end{itemize}

Moreover, in all these cases, $P_\alpha, P_{\delta_i}$ and $V_\alpha,V_{\delta_i}$ and their respective differentials are all required to vanish at $(x_0,0)$. In particular,, $(F_0(x_0)+V_r(x_0,0),P_r(x_0,0))=(F_0(x_0),0)=(0,0)$, where $r\in\{\alpha,\delta_i\}$, $1\leq i\leq k$. Again, the symmetric case where $G_\alpha\subseteq M\times(0,1)$ is given by reflection (see the illustration in Figure \ref{reflect1}).
\end{definition}

In other words, at splitting type bifurcation we again require $P_\alpha$ and $P_{\delta_i}$ to push towards a parameter $\tau$ for which $F_\tau$ has more fixed points. Having dealt with creation and splitting type fixed points, we now move on to the case of expansion type bifurcations. In this case, we prescribe the admissible behavior of $P_\alpha$ and $V_\alpha$ as described in the definition below:

\begin{definition}
\label{def42} Let $G_\alpha\subseteq M\times (-1,0)$ be a curve of fixed points in $Fix$ that connects to some expansion-type bifurcation fixed point $(x_0,0)$ on its boundary. Specifically, at $(x_0,0)$ the $G_\alpha$ connects to another component $G_\delta\subseteq Reg_2$ of fixed points, and surfaces $\mathcal{S}_1,...,\mathcal{S}_k\subseteq M\times(0,1)$ of periodic orbits in $Reg_1$ where $1\leq k\leq\infty$ (the surfaces $\mathcal{S}_1,...\mathcal{S}_k$ may or may not be in $Per$). In this case, consider the determinant $|J_0|$, where $J_0$ is the Jacobian determinant of the vector field $F_0$ at $x_0$ (per our definition of expansion bifurcations, $|J_0|\ne0$). Then, depending on the sign of $|J_0|$, $P_\alpha$, $V_\alpha$, $P_\delta$, $V_\delta$, $P_{\mathcal{S}_i}$ and $V_{\mathcal{S}_i}$ are admissible if they satisfy the following around $(x_0,0)$ (see the illustration in Figure \ref{hopff}):    

        \begin{enumerate}
            \item If $|J_0|<0$, then for all $(s,\tau)\in G_\alpha$ sufficiently close to $(x_0,0)$ we require $P_\alpha(s,\tau)$ to be positive. Similarly, whenever $G_\delta\subseteq Fix$, we require $P_\alpha(s,\tau)$ to be negative for all $(s,\tau)\in G_\delta$ sufficiently close to $(x_0,0)$. Moreover, given any $1\leq i\leq k$ s.t. $\mathcal{S}_i$ is in $Per$, we require $P_{\mathcal{S}_i}$ to be positive on all initial conditions in $\mathcal{S}_i$ sufficiently close to $(x_0,0)$ (see the left image in Figure \ref{hopff}). 
            \item If $|J_0|>0$, then, for all $(s,\tau)\in G_\alpha$ sufficiently close to $(x_0,0)$ we require $P_\alpha(s,\tau)$ to be negative. Similarly, when $G_\delta\subseteq Per$, we require $P_\delta(s,\tau)$ to be negative whenever $(s,\tau)\in G_\delta$ is sufficiently close to $(x_0,0)$. Analogously, given any $1\leq i\leq k$ s.t. $\mathcal{S}_i$ is in $Per$, we require $P_{\mathcal{S}_i}$ to be negative on all initial conditions in $\mathcal{S}_i$ sufficiently close to $(x_0,0)$ (see the right image in Figure \ref{hopff}). 

\end{enumerate}

 Again, in all these cases, $P_\alpha,P_\delta, P_{\mathcal{S}_i}$, $V_\alpha,V_\delta,V_{\mathcal{S}_i}$ and their differentials are required to vanish at $(x_0,0)$. Similarly to previous definitions, the symmetric case where $G_\alpha\subseteq M\times(0,1)$ is given by reflecting the above.
\end{definition}
\begin{figure}[h]
\centering
\begin{overpic}[width=0.5\textwidth]{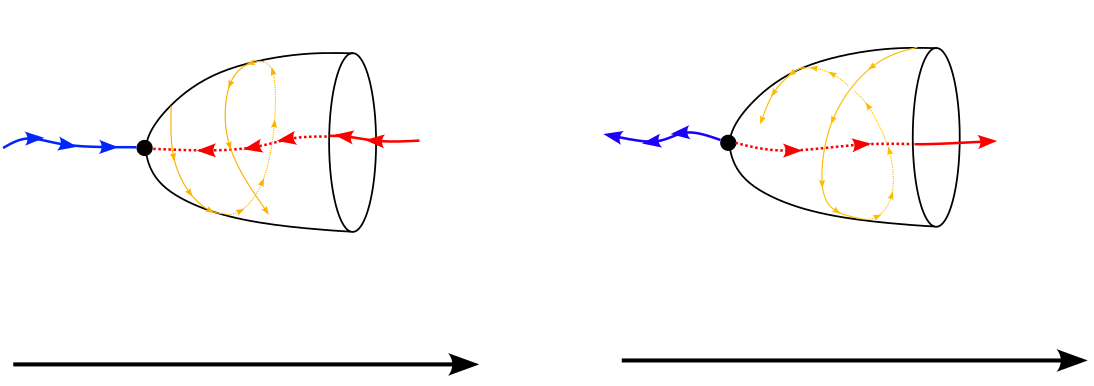}

\put(455,15){$\tau$}
\put(1005,15){$\tau$}

\end{overpic}
\caption{\textit{On the left - the local behavior of the vector field $E$ around fixed point that undergoes Hopf bifurcation when $|J_0|<0$. On the right - the local behavior when $|J_0|>0$. In each diagram, the black dot denotes $(x_0,0)$ and the blue and red curves denotes different components in $Fix$. The orange flow lines are in both cases a flow line on the surface of periodic orbit created at the bifurcation (in this illustration we assume this surface is completely isolated, i.e., that it is in $Per$). }}
\label{hopff}
\end{figure}

In other words, locally around an expansion type bifurcation we require a "saddle-like" behavior around $(x_0,0)$. Our motivation is that saddles are often associated with complex behavior, therefore, we would like the fixed points for our new flow on $M\times I$ to have saddle-like behavior which on the one hand pushes towards complexity from one direction, yet also sees the difference between different Poincare Indices of $x_0$ w.r.t. $F_0$. We now introduce the final definition of admissibility w.r.t. mixed type bifurcation:

\begin{definition}
\label{def43} Let $G_\alpha\subseteq M\times (-1,0)$ be a component of $Fix$ that connects to some mixed-type bifurcation fixed point $(x_0,0)$ on its boundary. In particular, at $(x_0,0)$ the arc $G_\alpha$ connects to other sets $\{\delta_i\}_{i=1}^k$ of fixed points (which may or may not be in $Fix$), and surfaces $\{\mathcal{S}_j\}_{j=1}^d$ of periodic orbits, where $1\leq d,k\leq\infty$ (which, again, may or may not be in $Per$). For admissibility, we require the functions $P_\alpha$, $V_\alpha$, $P_{\delta_i}$, $V_{\delta_i}$, $P_{\mathcal{S}_j}$ and $V_{\mathcal{S}_j}$  where $1\leq i\leq k$, $1\leq j\leq d$ to satisfy the following around $(x_0,0)$ (see the illustration in Figure \ref{eq2}):
\begin{itemize}
    \item For all $(s,\tau)\in G_\alpha$ sufficiently close to $(x_0,0)$, $P_\alpha$ is negative.
    \item For all $i$ s.t. $G_{\delta_i}\subseteq Fix$, if $G_{\delta_i}\subseteq M\times(-1,0)$ we require $P_{\delta_i}$ to be negative on all $(s,\tau)\in G_{\delta_i}$ sufficiently close to $(x_0,0)$. Conversely, when $G_{\delta_i}\subseteq M\times(0,1)$, we require $P_{\delta_i}$ to be positive on all $(s,\tau)\in G_{\delta_i}$ sufficiently close to $(x_0,0)$.
    \item For all $i$ s.t. $\mathcal{S}_j\subseteq Per$, if $\mathcal{S}_j\subseteq M\times(-1,0)$ we require $P_{\mathcal{S}_j}$ to be negative on all $(s,\tau)\in {\mathcal{S}_j}$ sufficiently close to $(x_0,0)$. Conversely, when $\mathcal{S}_j\subseteq M\times(0,1)$, we require $P_{\mathcal{S}_j}$ to be positive on all $(s,\tau)\in \mathcal{S}_j$ sufficiently close to $(x_0,0)$.
    \item Again, $V_\alpha$, $P_\alpha$, $V_{\delta_i}$, $P_{\delta_i}$, $V_{\mathcal{S}_j}$, $P_{\mathcal{S}_j}$ and their differentials are all required to vanish at $(x_0,0)$. 
\end{itemize}

\end{definition}
\begin{figure}[h]
\centering
\begin{overpic}[width=0.35\textwidth]{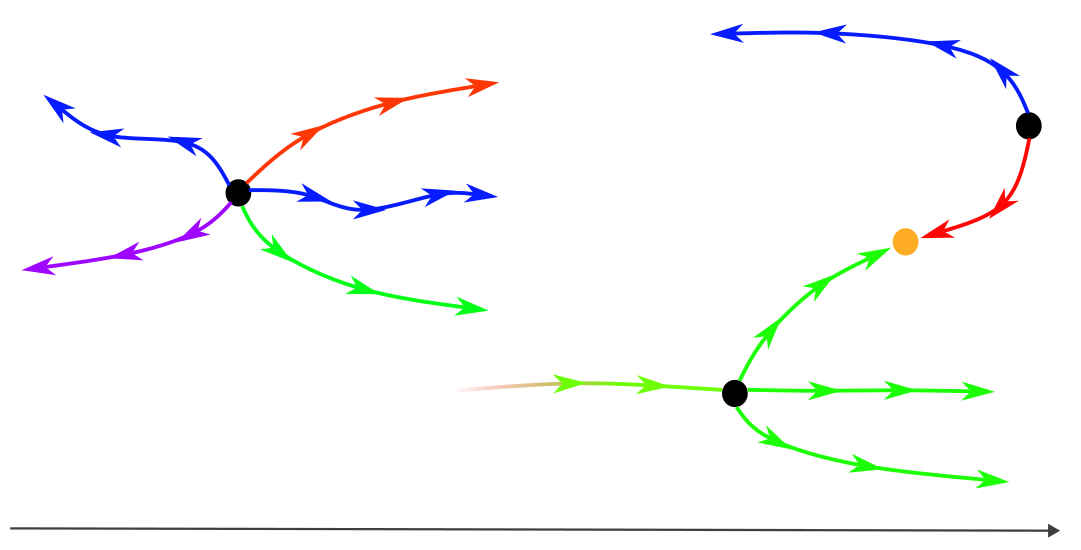}

\put(1000,10){$\tau$}

\end{overpic}
\caption{\textit{On the left - a diagram showing the admissible behavior around a mixed type bifurcation. On the right - diagrams showing the global admissible motion on a collection of components in $Fix$, glued at bifurcation points (the black dot). In this diagram, the orange dot represents an equilibrium fixed point, which exists due to the saddle node bifurcation (creation type) that connects to a pitchfork bifurcation (splitting type).}}
\label{eq2}
\end{figure}

It is easy to see that Definitions \ref{def4}-\ref{def43} all extend to the case where the bifurcation does not occur at $(x_0,0)$ but rather at some other parameters in $I=(-1,1)$. Therefore, having described the prescribed admissible behavior we require around the bifurcations, we now define explicitly what behavior we require globally of $P_\alpha$ and $V_\alpha$ on a component $G_\alpha$ of $Fix$. We do so in the following definition:
\begin{definition}
    \label{admissiblefixed} Let $\dot{s}=F_\tau(s)$ be a $C^1$ one--parameter family, where the parameter $\tau$ varies in $(-1,1)$. Let $G_\alpha$ be a component of $Fix$, let $\pi:M\times (-1,1)\to (-1,1)$ denote the projection $\pi(s,\tau)=\tau$. Set $\pi(G_\alpha)=J=(\tau_1,\tau_2)$, and parameterize $G_\alpha=\{(x_\tau,\tau)|\tau\in J\}$. Assume there exists a neighborhood $N_\alpha$ of $G_\alpha$ s.t. $N_\alpha\cap(Per\cup Fix)=G_\alpha$. Two $C^1$-functions $P_\alpha$ and $V_\alpha$ defined in $M\times (-1,1)$, whose support lies in $N_\alpha$ are said to be \textbf{admissible w.r.t. $N_\alpha$} (or just \textbf{admissible} when $N_\alpha$ is clear by context) if they satisfy the following conditions (see the illustrations in Figures \ref{eq1} and \ref{eq2}):
    \begin{itemize}
        \item As $\tau\to\tau_i$, $i=1,2$, both $P_\alpha(s,\tau)\to0$ and $V_\alpha(s,\tau)\to0$. The same  also true for their differentials, as given in Definitions \ref{def4}-\ref{def43}.
        \item The function $V_\alpha$ is a correction term s.t. for all $(s,\tau)\in G_\alpha$ the vector $(F_\tau(s)+V_\alpha(s,\tau),P_\alpha(s,\tau))$ is tangent to $G_\alpha$. 
         \end{itemize}
In addition, exactly one of the four conditions below also has to be satisfied, depending on $J$:
        \begin{itemize}
        \item If $G_\alpha$ intersects all $M\times\{\tau\}$, i.e., $J=(-1,1)$ and the fixed points on $G_\alpha$ do not bifurcate as $\tau$ is varied, then $P_\alpha$ and $V_\alpha$ are admissible precisely when both of them vanish throughout $M\times (-1,1)$ (i.e., $P_\alpha$, $V_\alpha$ are defined as in Definition \ref{def3}).
        \item Assume $J\ne (-1,1)$ and that, say, $(x_{\tau_1},\tau_1)$ is a bifurcation fixed point for the family $\dot{s}=F_\tau(s)$ of one of the types considered above, while $(x_{\tau_2},\tau_2)$ either does not belong to such a bifurcation type or doesn't exist (for example, when $x_\tau$ has no limit in $M$ as $\tau\to1$). In this scenario, $P_\alpha,V_\alpha$ are admissible provided the following holds:
        \begin{enumerate}
            \item The behavior of functions $P_\alpha$ and $V_\alpha$ on initial conditions on $G_\alpha$ near $(x_{\tau_1},\tau_1)$ is admissible with respect to Definitions \ref{def4}-\ref{def43}.
            \item $P_\alpha$ never changes its sign on $N_\alpha$ - i.e., the local behavior of $P_\alpha$ around $(x_1,\tau_1)$ determines the global behavior in $N_\alpha$.
        \end{enumerate}
The symmetric case where $(x_{\tau_2},\tau_2)$ is a bifurcation fixed points of one of the types above and $(x_{\tau_1},\tau_1)$ is not or does not exist is obtained by reflecting the definition above.
        \item Assume $J\ne (-1,1)$ and both $(x_1,\tau_1),(x_2,\tau_2)$ are bifurcation points corresponding to types above. Then, the functions $P_\alpha,V_\alpha$ are admissible provided:
        \begin{enumerate}
            \item  On initial conditions $(s,\tau)\in G_\alpha$ close to both $(x_i,\tau_i)$, $i=1,2$, the behavior of $P_\alpha$ and $V_\alpha$ is admissible with respect to Definitions \ref{def4}-\ref{def43}.
            \item If the values of of $P_\alpha$ near $(x_1,\tau_1)$ have opposing sign to its values near $(x_2,\tau_2)$ we also require that as we increase $\tau$ from $\tau_1$ to $\tau_2$, then $P^{-1}_\alpha(0)\cap G_\alpha=(x,\tau_0)$, for some $\tau_0\in (\tau_1,\tau_2)$ - otherwise, if the signs at the two ends of $G_\alpha$ are not opposed the sign of $P_\alpha$ on $G_\alpha$ is required to be constant throughout $N_\alpha$. We further require $(x,\tau_0)$, the \textbf{equilibrium point}, to be the unique fixed point for the vector field $(s,\tau)\to(F_\tau(s)+V_\alpha(s,\tau),P_\alpha(s,\tau))$ defined in $N_\alpha$. 
        \end{enumerate}

         \item          In all other possible configurations not covered by the above, we say $P_\alpha$ and $V_\alpha$ are admissible if they vanish identically throughout $N_\alpha$.
    \end{itemize}

\end{definition}
\begin{figure}[h]
\centering
\begin{overpic}[width=0.4\textwidth]{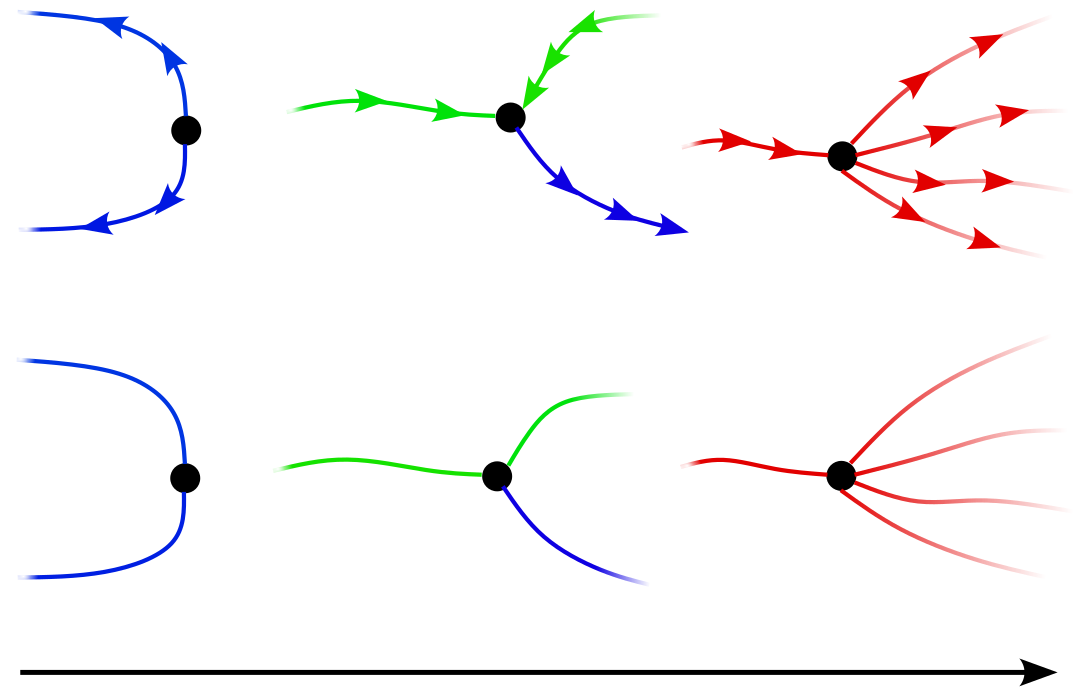}
\put(1000,15){$\tau$}
\end{overpic}
\caption{\textit{A diagram representing the local admissible behavior around creation (left), multiplying (middle) and splitting (right) bifurcations for periodic orbits. In all panels, the black dots denotes the bifurcation orbit $T_0\times\{0\}$, the lower diagram corresponds to the bifurcation diagram, while the upper one showcases the admissible behavior on the periodic orbits bordering the bifurcation. The different colors represent the components in $Per$ along which the period changes continuously as we vary $\tau$. Here, all components of periodic orbits are in $Per$.}}
\label{perbifurcations}
\end{figure}

It is immediate by Proposition \ref{def1}, Corollary \ref{def2} and Definition \ref{def3} that given $G_\alpha\subseteq Per\cup Fix$ for which a neighborhood $N_\alpha$ as above exists, there also exist admissible functions $P_\alpha$ and $V_\alpha$ w.r.t. it. It is also immediate from our definition that when $P_\alpha$ is non-zero in $N_\alpha$, there is at most one initial condition $(s,\tau)\in N_\alpha$ for which we have $(F_\tau(s)+V(s,\tau),P(s,\tau))=(0,0)$ - namely, the equilibrium point $(x_0,\tau_0)$ (if it exists - i.e., only when $P_\alpha$ has opposing signs on the ends of $G_\alpha$). That being said, before we prove the existence of such $N_\alpha$ for all components $G_\alpha\subseteq Fix$, we will first define admissibility for the analogous case for components of periodic orbits $G_\alpha$ in $Per$.

\subsection{Stage $II$ - the behavior on components of $Per$}
\label{sec2}
Having defined the behavior of $P_\alpha$ and $V_\alpha$ around fixed points, we now turn to do the same around components of periodic orbits $G_\alpha\subseteq Per$. To begin, recall $Per$ is defined as the collection of completely isolated periodic orbits (see Definition \ref{isolationdef}) for some $C^1$ family $\dot{s}=F_\tau(s)$, where $\tau$ varies in $(-1,1)$. Let $T_\tau$ be a periodic orbit for $F_\tau$ that varies continuously as $\tau$ is varied in, say, $(-1,0)$, and let $T_0$ denote a periodic orbit for $F_0$ where the said orbit undergoes a bifurcation (in particular, we assume the period of $T_\tau$ varies continuously with $\tau\in(-1,0]$). Similarly as we did for fixed point, we now classify the type of bifurcations occurring as $\tau$ crosses into $(0,1)$. Similarly to previous Subsection, we will only describe the scenario when the bifurcation occurs as $\tau$ crosses from $(-1,0)$ to $(0,1)$ - the symmetric case is described analogously:
\begin{itemize}
    \item \textbf{Creation} - $T_0$ is a collision points for (at least) two distinct periodic orbits that collide and disappear as $\tau$ crosses into $(0,1)$, i.e., the orbit $T_0$ is born at $F_0$, it splits into (at least) two periodic orbits $T_\tau$, $R_\tau$ as $\tau$ enters $(-1,0)$. For example, this occurs if $T_0$ is a saddle node bifurcation (see Figure \ref{perbifurcations}).
    \item \textbf{Multiplying} - there exists some natural $m>1$ s.t. when $\tau$ crosses into $(0,1)$ the orbit $T_0$ splits into (at least) two distinct orbits, $R_\tau$ and $D_\tau$ s.t. the following holds (see Figure \ref{perbifurcations}):
    \begin{enumerate}
        \item The period changes continuously at $T_\tau$ as varied to $R_\tau$ via $T_0$.
        \item Let $p_0$ denoted the period of $T_0$ w.r.t. $F_0$. Then, the period of periodic orbits $D_\tau$ for $\tau$ close to $0$ is approximately $mp_0$ - for example, when $m=2$, $T_0$ is a period doubling bifurcation orbit.
    \end{enumerate}
    \item \textbf{Splitting} - when $\tau$ crosses into $(0,1)$, $T_0\times\{0\}$ splits into a finite (or infinite) collection of periodic orbits, $T^1_\tau,...,T^k_\tau$, where $\tau$ varies continuously in $(0,1)$, $k\leq \infty$. For all $1\leq i\leq k$ the period varies continuously as we vary periodic orbits $T_\tau$, $\tau\in(-1,0)$ to $T^i_\tau$, $\tau\in(0,1)$ via $T_0\times\{0\}$ (see Figure \ref{perbifurcations}).
    \item \textbf{Expansion} - when $\tau$ crosses into $(0,1)$, at least one invariant two-dimensional torus is born, which persists in $M\times(0,\epsilon)$ for some $1\geq \epsilon>0$, while the orbit $T_0$ is deformed continuously to $T_\tau$, a periodic orbit for $F_\tau$, for $\tau>0$. Moreover, the periods of $T_\tau$ vary continuously with $\tau\in(-1,1)$. For example, this occurs at  Neimark-Sacker bifurcation (see Figure \ref{perbifurcation2}). 
\item \textbf{Mixed} - any other type of bifurcation of periodic orbit that does not fall completely into one of the descriptions above. For example, in a type $m$ bifurcations, originally introduced in \cite{mey} and \cite{mey2} and studied in \cite{mey3}, two multipliers cross $S^1$, which causes both the creation of an invariant torus and an $m$-multiplying bifurcation. \\
\end{itemize}

Now, let $G_\alpha$ be a component of periodic orbits in $Per$ and let $N_\alpha$ be some neighborhood of $G_\alpha$ s.t. $N_\alpha\cap(Per\cup Fix)=N_\alpha$. Similarly to what we did with curves of fixed points in $Fix$, we first describe what local behavior we require from the functions $P_\alpha$ and $V_\alpha$ around the bifurcation orbits in $\partial G_\alpha$ gluing $G_\alpha$ to components $G_\beta$ of $Per$ (note that by Definition \ref{def42} and Definition \ref{def43} we already know the behavior around components $G_\alpha\subseteq Per$ which connect to components of $Fix$ at bifurcation fixed points). After we do that we will similarly define the notion of admissibility of $P_\alpha$ and $V_\alpha$ w.r.t. $N_\alpha$ (see Definition \ref{admissibleper}). Again, for brevity, throughout this Subsection we will always implicitly assume the existence of open sets $N_\alpha$ and $N_\beta$, $N_\alpha\cap N_\beta=\emptyset$ for all components $G_\beta\subseteq Fix\cup Per$ connecting with $G_\alpha$ at $T_0\times\{0\}$. And again, to avoid repetition, throughout this Subsection, when we consider the functions $V_\alpha$ we will always implicitly assume they are correction terms s.t. for all $(s,\tau)\in G_\alpha$ the vector $(F_\tau(s)+V_\alpha(s,\tau),P_\alpha(s,\tau))$ is tangent to $G_\alpha$ - similarly to previous section, such correction terms exist by Proposition \ref{def1} and Corollary \ref{def2}. As before, we begin with creation type bifurcations:
\begin{definition}    \label{def5} Let $\dot{s}=F_\tau(s)$ be a one--parameter $C^1$ family, and let $G_\alpha\subseteq M\times(-1,0)$ be a component of $Per$ which connects at a creation bifurcation orbit $T_0\times\{0\}$ with some other surfaces of periodic orbits $\{G_{\beta_i}\}_{i=1}^k\subseteq M\times(-1,0)$, $1\leq k\leq\infty$. Then, we say the functions $P_\alpha$, $V_\alpha$, $P_{\beta_i}$ and $V_{\beta_i}$ are admissible if satisfy the following conditions around $T_0\times\{0\}$ (see the illustration in Figure \ref{perbifurcations}):
    \begin{itemize}
        \item    $P_\alpha(s,\tau)$ is negative around all initial conditions $(s,\tau)\in G_\alpha$ sufficiently close to $T_0\times\{0\}$. 
        \item Similarly, whenever $G_\beta\subseteq Per$, $P_{\beta_i}(s,\tau)$ is negative for all $(s,\tau)\in G_\beta$ sufficiently close to $T_0\times\{0\}$.
        \item Similarly to Definition \ref{def4}, $P_\alpha$, $P_{\beta_i}$, $V_{\alpha}$ and $V_{\beta_i}$ and their respective differentials are further required to vanish at $T_0\times\{0\}$.
\end{itemize}

In the symmetric case where $G_\alpha,\{G_{\beta_i}\}_{i=1}^k\subseteq M\times(0,1)$, we just reflect the above, exactly as we did in Definition \ref{def4}.
\end{definition}
\begin{figure}[h]
\centering
\begin{overpic}[width=0.35\textwidth]{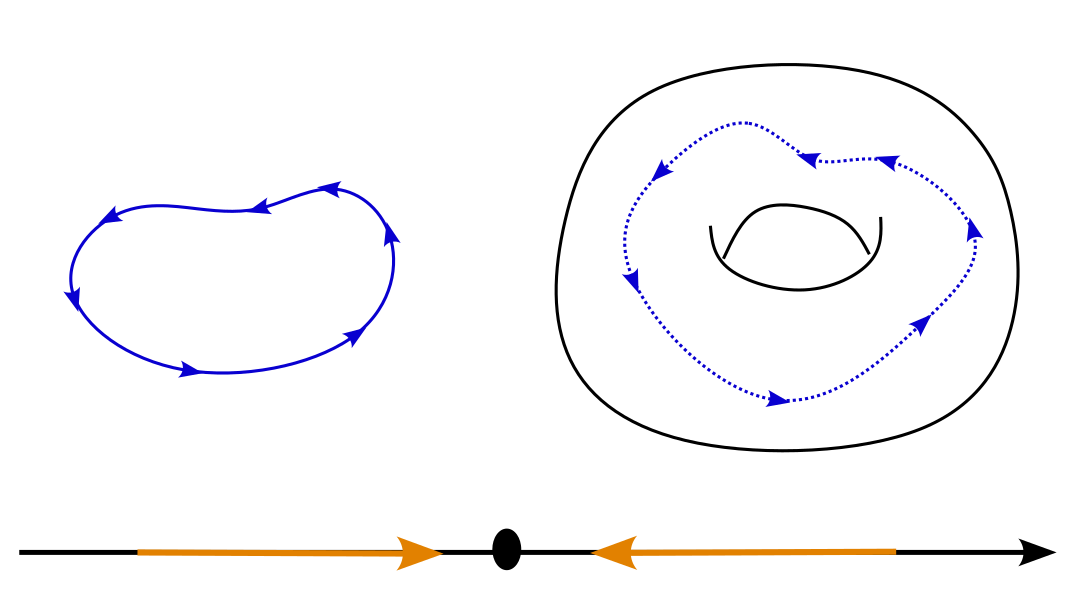}
\put(455,-20){$0$}
\put(990,50){$\tau$}
\end{overpic}
\caption{\textit{A diagram representing the admissible behavior around an expansion bifurcation. Left to $\tau=0$, there exists a periodic orbit which undergoes a supercritical Neimark-Sacker bifurcation, hence the behavior is as dictated by Definition \ref{def5}.}}
\label{perbifurcation2}
\end{figure}

Having prescribed the local behavior of $P_\alpha$ and $V_\alpha$ on $G_\alpha$ when $G_\alpha$ is connected to a creation type bifurcation orbit $T_0\times\{0\}$, we note that similarly to Definition \ref{def4}, the prescribed behavior can be described as "pushing towards the direction of creation". We now replicate the same idea for multiplying and splitting type bifurcations, using logic similar to Definition \ref{def41}:

\begin{definition}
      \label{def511} Let $G_\alpha\subseteq M\times(-1,0)$ be a component of $Per$ which connects to a bifurcation periodic orbit $T_0\times\{0\}$ whose bifurcation type is either multiplying or splitting. Assume $G_\alpha$ is glued at $T_0\times\{0\}$ with a collection of families of periodic orbits $\{G_{\beta_i}\}_{i=1}\subseteq M\times(0,1)$, where $1\leq k\leq\infty$. Then, the functions $P_{\alpha},V_\alpha$, $P_{\beta_i}$ and $V_{\beta_i}$, $1\leq i\leq k$ are admissible provided they satisfy the following around $T_0\times\{0\}$ (see Figure \ref{perbifurcations}):
    \begin{itemize}

        \item If $T_0\times\{0\}$ is a \textbf{multiplying bifurcations}, denote by $G_{\beta_1},...,G_{\beta_r}$ the surfaces of periodic orbits along which the period does not change, and by $G_{\beta_{r+1}},...,G_{\beta_k}$ the surfaces of periodic orbits along which the period is multiplied by some positive integer $m$. We require the following for admissibility:
        \begin{enumerate}
          \item On initial conditions $(s,\tau)\in G_\alpha$ sufficiently close to $T_0\times\{0\}$, we require $P_\alpha$ to be positive (i.e., so that the vector $(F_\tau(s)+V_\alpha(s,\tau),P_\alpha(s,\tau))$ points towards $T_0\times\{0\}$). 
          \item For all $1\leq i\leq r$ s.t. $G_{\beta_i}\subseteq Per$, on $(s,\tau)\in G_{\beta_i}$ sufficiently close to $T_0\times\{0\}$ we require $P_{\beta_i}$ to be negative (i.e., for such $(s,\tau)$ the vector $(F_\tau(s)+V_\alpha(s,\tau),P_\alpha(s,\tau))$ should point towards $T_0\times\{0\}$).
          \item Whenever $i$ is such that $G_{\beta_i}\subseteq Per$, $r+1\leq i\leq k$, for all $(s,\tau)\in G_{\beta_i}$ sufficiently close to $T_0\times\{0\}$ we require $P_{\beta_i}$ to be positive so that the vector $(F_\tau(s)+P_{\beta_{i}}(s,\tau),V_{\beta_i}(s,\tau)$ points away from $T_0\times\{0\}$. In other words, on such $(s,\tau)$ we point in the direction of increased oscillation.
\end{enumerate}   
    \item  In \textbf{splitting bifurcations}, $G_\alpha$ connect at $T_0\times\{0\}$ with $k$ families of periodic orbits, $G_{\beta_i}$, $1\leq i\leq k\leq \infty$, the period changes continuously as we move from $G_\alpha$ to $G_{\beta_i}$ through $T_0\times\{0\}$ (for all $i$). In this case, we require $P_\alpha$ to be positive on all $(s,\tau)\in G_\alpha$ sufficiently close to $T_0\times\{0\}$ (i.e., vector $(F_\tau(s)+V_\alpha(s,\tau),P_\alpha(s,\tau))$ points towards the splitting at $T_0\times\{0\}$). Similarly, for all $i$ s.t. $G_{\beta_i}\subseteq Per$, on all $(s,\tau)\in G_{\beta_i}$ sufficiently close to $T_0\times\{0\}$ we require $P_{\beta_i}$ to be positive - i.e., the vector $(F_\tau(s)+V_\alpha(s,\tau),P_\alpha(s,\tau))$ points in the direction of more periodic orbits.

\end{itemize}

In both cases, the functions $P_\alpha,V_\alpha, P_{\beta_i}$ and $V_{\beta_i}$ and all their respective differentials are further required to vanish as $(s,\tau)\to T_0\times\{0\}$.  Again, in the symmetric case where $G_\alpha\subseteq M\times(0,1)$ and $\{G_{\beta_i}\}_{i=1}^k\subseteq M\times(-1,0)$ we reflect the above, exactly as in Definition \ref{def41}.
\end{definition}
We now turn to define the analogue of Definition \ref{def42}. We will define admissibility in this case in a similar way, only that this time, as we are dealing with Neimark-Sacker bifurcations we will be considering families of two-dimensional invariant tori. We will need to introduce a few definitions to deal with such scenarios - to this end, given a $C^1$ curve $\dot{s}=F_\tau(s)$ we say a two-dimensional invariant torus $\mathbb{T}^2_\tau$ for $F_\tau$ on which the motion is aperiodic is  \textbf{isolated} if $\mathbb{T}^2_\tau\times\{\tau\}$ cannot be approximated by points in $M\times (-1,1)$ that lie on periodic orbits or fixed points for all vector fields in the said family. Similarly, given a connected subset $\mathcal{T}\subseteq M\times (-1,1)$, we say it is a \textbf{family of isolated tori} if the following holds:
\begin{itemize}
    \item  For all $\tau$, we have either $\mathbb{T}^2_\tau\times\{\tau\}=\mathcal{T}\cap (M\times\{\tau\})$ or $\emptyset=\mathcal{T}\cap (M\times\{\tau\})$. Moreover, $\mathbb{T}^2_\tau$ is an isolated two-dimensional invariant torus for $F_\tau$, on which the motion is aperiodic, and  $\mathbb{T}^2_\tau$ varies continuously.
    \item $\mathcal{T}$ is maximal - i.e., we cannot extend it in $M\times (-1,1)$ s.t. $\mathcal{T}$ remains connected and the two properties above still hold.\\
\end{itemize}

With these new definitions in mind, we now turn to define the admissible local behavior around expansion type bifurcations of periodic orbits:

\begin{definition}
\label{def52}    Let $\dot{s}=F_\tau(s)$ be a one--parameter $C^1$ family, and let $G_\alpha\subseteq M\times(-1,0)$ be a component of $Per$ which connects to an expansion type bifurcation periodic orbit $T_0\times\{0\}$. Specifically, assume $G_\alpha$ connects at $T_0\times\{0\}$ with a family of periodic orbits $G_\beta\subseteq M\times(0,1)$, and several families of two-dimensional invariant tori, $\mathcal{T}_1,...,\mathcal{T}_n\subseteq M\times(0,1)$, where $n\leq \infty$. Then, the functions $P_\alpha,V_\alpha$, $P_{\beta},V_{\beta}$, $1\leq i\leq n$, are admissible provided they satisfy the following behavior around $T_0\times\{0\}$ (see the illustrations in Figures \ref{RT} and \ref{RT2}):

\begin{itemize}
    \item Assume $T_0$ is an expansion bifurcation orbit s.t. $G_\alpha$ corresponds to a family of attractors which become repellers as the orbits on $G_\alpha$ are varied to orbits on $G_\beta$ via $T_0$. Then, the following is required for admissibility: 
    \begin{enumerate}
        \item for initial conditions $(s,\tau)\in G_\alpha$ sufficiently close to $T_0\times\{0\}$, we require $P_\alpha(s,\tau)$ to be positive, i.e., the vector $(F_\tau(s)+V_\alpha(s,\tau),P_\alpha(s,\tau))$ is required to point towards the bifurcation at $T_0\times\{0\}$.
        \item Similarly, we require $P_{\beta}$ to be negative on all $(s,\tau)\in G_\beta$ sufficiently close to $T_0\times\{0\}$ - i.e., again the vector $(F_\tau(s)+V_\alpha(s,\tau),P_\alpha(s,\tau))$ is required to point towards the bifurcation at $T_0\times\{0\}$ (note that as $G_\beta$ is made of repellers it is automatically in $Per$).
    \end{enumerate}
    
     \item Assume $T_0$ is an expansion bifurcation orbit where a repelling orbits on $G_\alpha$ become attractors on $G_\beta$ as $\tau$ crosses from $(-1,0)$ to $(0,1)$. Then, the following is required for admissibility:
     \begin{enumerate}
         \item On all initial conditions $(s,\tau)\in C$ sufficiently close to $T_0\times\{0\}$ we require $P_\alpha(s,\tau)$ to be negative, i.e., this time the vector $(F_\tau(s)+V_\alpha(s,\tau),P_\alpha(s,\tau))$ is required to point away from the bifurcation.
         \item  Similarly, we require $P_{\beta}$ to be positive on all $(s,\tau)\in G_\beta$ sufficiently close to $T_0\times\{0\}$, i.e., the vector $(F_\tau(s)+V_\alpha(s,\tau),P_\alpha(s,\tau))$ is required to point away from the bifurcation (symmetrically, $G_\beta\subseteq Per$ as it corresponds to a family of attractors).
     \end{enumerate}
     \item In any other case, we require $P_\alpha$ to be negative on $(s,\tau)\in G_\alpha$ near $T_0\times\{0\}$ and positive on $(s,\tau)\in G_\beta$ near $T_0\times\{0\}$ - i.e., the vector $(F_\tau(s)+V_\alpha(s,\tau),P_\alpha(s,\tau))$ is required to repel away from the bifurcation.
\end{itemize}

In all three possibilities, the functions $P_\alpha,V_\alpha$, $P_{\beta}$ and $V_{\beta}$ and their differentials are further required to vanish as $(s,\tau)\to T_0\times\{0\}$. Similarly to previous definitions, in the symmetric case when $G_\alpha\subseteq M\times(0,1)$, and $\mathcal{T}_1,...,\mathcal{T}_n,G_\beta\subseteq M\times(-1,0)$ we  reflect the definition exactly as we did in all previous cases.
\end{definition}
\begin{remark}
Note that unlike expansion bifurcations for fixed points, here we did not ascribe any admissible motion on the families of isolated invariant tori $\mathcal{T}_1,...,\mathcal{T}_n$ - we will do so in the next Subsection.
\end{remark}
We now proceed to deal with the final case - the basket case referred to earlier as the "mixed" scenario. In this case, we have a component of periodic orbits $G_\alpha\subseteq M\times(-1,0)$ that meets a collection of families of periodic orbits $G_1,...,G_d\subseteq M\times(-1,1)$ and a collection of families of invariant tori (or more generally, invariant sets) $\mathcal{T}_1,...,\mathcal{T}_k\subseteq M\times(-1,1)$ at $T_0\times\{0\}$, where $1\leq d,k\leq \infty$ (note one of these collections can be empty). Following Definition \ref{def43}, we describe the admissible behavior as follows:
\begin{definition}
\label{def53} With the notations above, the following is required for admissibility of the functions $P_\alpha,V_\alpha$, $P_1,...,P_d$, $V_1,...,V_d$:
\begin{itemize}
    \item The function $P_\alpha$ is negative on all initial conditions $(s,\tau)\in G_\alpha$ sufficiently close to $T_0\times\{0\}$. In addition, for all $1\leq i\leq d$ s.t. $G_i\subseteq Per\cap M\times(-1,0)$, $P_i$ is negative on all initial conditions $(s,\tau)\in G_i$ sufficiently close to $T_0\times\{0\}$ (in other words, on such $(s,\tau)$ the vector $(F_\tau(s)+V_\alpha(s,\tau),P_\alpha(s,\tau))$ is required to repel away from the bifurcation). 
    \item For all $1\leq i\leq d$ s.t. $G_i\subseteq Per\cap M\times(0,1)$, $P_i$ is positive on all initial conditions $(s,\tau)\in G_i$ sufficiently close to $T_0\times\{0\}$ (again, the vector $(F_\tau(s)+V_\alpha(s,\tau),P_\alpha(s,\tau))$ is required to point away from the bifurcation).
    \item Similarly to the above, we again require $P_\alpha,V_\alpha, P_{j},V_{j}$ (for all relevant $1\leq j\leq d$), and all their respective differentials to vanish at $T_0\times\{0\}$.
\end{itemize}

\end{definition}
\begin{remark}
    Similarly to previous definition, we again "ignore" the invariant tori and other persistent invariant sets $\mathcal{T}_1,...,\mathcal{T}_k\subseteq M\times(-1,1)$ glued to $T_0\times\{0\}$. We will consider them in the next Subsection.
\end{remark}
Having described the local behavior required around the bifurcation orbits gluing different components in $Per$ to one another (and possibly to invariant tori), we now sew them together. In other words, similarly to Definition \ref{admissiblefixed} we now define the admissible functions on components of $Per$:

\begin{definition}
    \label{admissibleper} Let $\dot{s}=F_\tau(s)$ be a $C^1$ one--parameter family defined over a manifold $M$ and parameterized by $\tau\in(-1,1)$. Let $G_\alpha$ be a component of $Per$ and again let $\pi:M\times (-1,1)\to (-1,1)$ denote the projection $\pi(s,\tau)=\tau$, where $\pi(G_\alpha)=J=(\tau_1,\tau_2)$. Consider a neighborhood $N_\alpha$ of $G_\alpha$ in $M\times(-1,1)$ s.t. $N_\alpha\cap(Per\cup Fix)=G_\alpha$ - two $C^1$-functions $P_\alpha$ and $V_\alpha$ defined throughout $M\times (-1,1)$ and supported in $N_\alpha$ are \textbf{admissible w.r.t. $N_\alpha$} (or just \textbf{admissible} when $N_\alpha$ is clear from context) if they satisfy the following conditions in $N_\alpha$ (see the illustration in Figure \ref{sew}):
    \begin{itemize}
        \item As $(s,\tau)\to\partial N_\alpha$, both $P_\alpha(s,\tau)\to0$ and $V_\alpha(s,\tau)\to0$. Similarly to Definition \ref{admissiblefixed}, the same is also required from their differentials as $(s,\tau)\to\partial N_\alpha$.
        \item The function $V_\alpha$ is a correction term s.t. for all $(s,\tau)\in G_\alpha$ the vector $(F_\tau(s)+V_\alpha(s,\tau),P_\alpha(s,\tau))$ is tangent to $G_\alpha$, and never vanishes on $N_\alpha$.
        \end{itemize}
        In addition, exactly one of the following four conditions also has to be satisfied:
        \begin{itemize}
        \item If there are no bifurcation orbits on $\partial G_\alpha$ w.r.t. $\dot {s}=F_\tau(s)$, $J=(-1,1)$, then $P_\alpha$, $N_\alpha$ are admissible only if both vanish throughout $M\times (-1,1)$ (i.e., $P_\alpha$, $V_\alpha$ are admissible in the sense of Definition \ref{def3}).
        \item Assume $J\ne I$, and write $T_{i}\times\{\tau_i\}=\partial G_\alpha\cap M\times\{\tau_i\}$, $i=1,2$ (these sets may or may not be empty). Further assume that $T_{1}\times\{\tau_1\}$ is a bifurcation periodic orbit of one of the types described above, and that the set $T_{2}\times\{\tau_2\}$ is either empty or $T_2$ is a bifurcation set that does not correspond to any of the bifurcation types described ($T_2$ may or may not be a periodic orbit for $F_{\tau_2}$). Then, we say $P_\alpha$ and $V_\alpha$ are admissible provided the following is satisfied:
        \begin{enumerate}
            \item $P_\alpha$ never changes its sign on $G_\alpha$.
            \item The behavior of $P_\alpha$ and $V_\alpha$ on $(s,\tau)\in G_\alpha$ sufficiently close to $T_{\tau_1}\times\{1\}$ is admissible per Definitions \ref{def4}-\ref{def53}. In other words, the bifurcation type of $T_1$ determines the global behavior on $N_\alpha$.
        \end{enumerate}
The symmetric scenario where $J\ne(-1,1)$ and $T_{2}\times\{\tau_2\}$  is a bifurcation periodic orbit of one of the types above, and the set $T_{1}\times\{\tau_1\}$ is either empty or corresponds to some bifurcation set that does not fall into any of the prescribed bifurcation types is obtained by reflecting the definition above.
        \item  Assume $J\ne (-1,1)$ and that with previous notation, both $T_{1}\times\{\tau_1\},T_{2}\times\{\tau_2\}$ are bifurcation points belonging to one of the bifurcation types described above. In this case, we say $P_\alpha$ and $V_\alpha$ are admissible provided the following holds:
        \begin{enumerate}
            \item The behavior of $P_\alpha$ and $V_\alpha$ on initial conditions $(s,\tau)\in G_\alpha$ around both $T_{\tau_1}\times\{\tau_1\}$ and $T_{\tau_2}\times\{\tau_2\}$ is admissible per Definitions \ref{def4}-\ref{def53} . 
            \item If the sign of $P_\alpha$ on $(s,\tau)\in G_\alpha$ close to $T_{1}\times\{\tau_1\}$ is opposed to its sign on $(s,\tau)\in G_\alpha$ close $T_{2}\times\{\tau_2\}$ we require $P_\alpha^{-1}(0)\cap G_\alpha$ to be given by $T\times\{\tau_0\}$, where $\tau_0\in(\tau_1,\tau_2)$ and $T$ is periodic orbit for the vector field $F_{\tau_0}$ (we will also require that $V_\alpha$ vanishes on $T\times\{0\}$) - if the said signs are not opposed, $P_\alpha$ is required to have constant sign $N_\alpha$. This condition makes $T\times\{0\}$ (whenever it exists) into a periodic orbit for the vector field defined by $(s,\tau)\to(F_\tau(s)+V_\alpha(s,\tau),P_\alpha(s,\tau))$. We will often refer to this orbit as an \textbf{equilibrium periodic orbit}.
            \end{enumerate}
            \item In all other possible configurations different from the three above, for admissibility we require $P_\alpha$ and $V_\alpha$ to vanish identically in $N_\alpha$. 
        \end{itemize}
\end{definition}
\begin{figure}[h]
\centering
\begin{overpic}[width=0.3\textwidth]{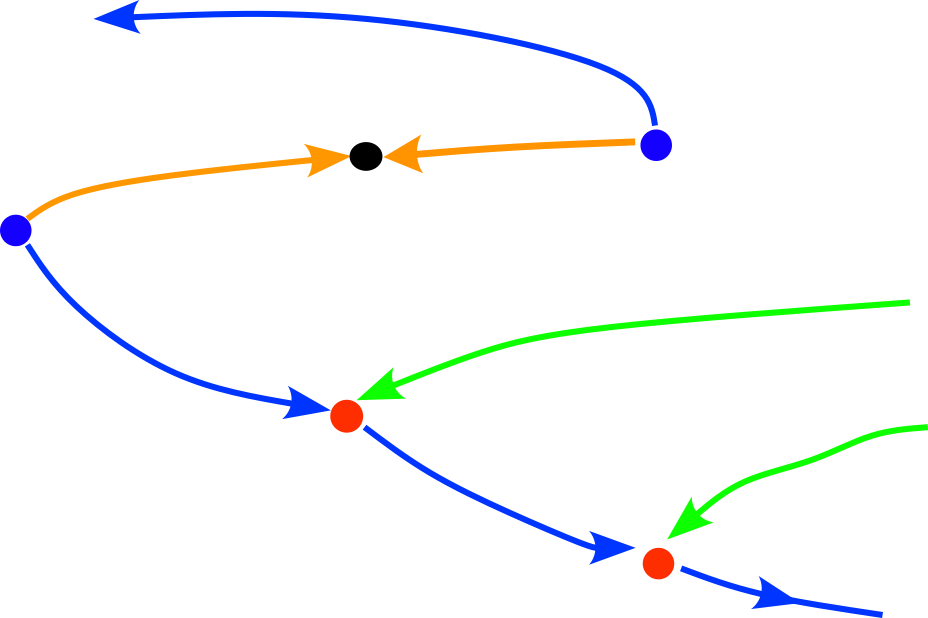}
\put(330,140){$s_3$}
\put(625,10){$s_4$}
\put(-30,340){$s_1$}
\put(735,470){$s_2$}
\put(350,420){$s_5$}
\end{overpic}
\caption{\textit{A diagram representing the behavior of the Entropy flow on a collection of components in $Per$ connected to one another at bifurcation orbits. A dot between two arrows implies that there exists a component in $Per$ connecting them, where green and orange denotes completely isolated saddles, blue denotes attracting periodic orbits. In this scenario, the blue dots at $s_1,s_2$ denote saddle node bifurcations, the red dots $s_3,s_4$ correspond to period doubling bifurcations and the black dot $s_5$ corresponds to a periodic orbit that is an equilibrium state.}}
\label{sew}
\end{figure}
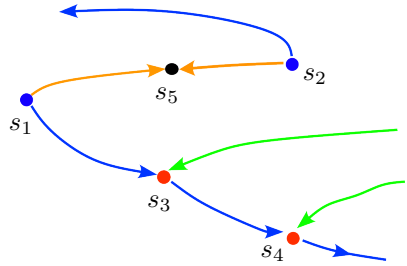
Again, it is immediate by Proposition \ref{def1}, Corollary \ref{def2} and Definition \ref{def3} that given $G_\alpha\subseteq Per$ for which a neighborhood $N_\alpha$ as above exists, there also exist admissible functions $P_\alpha$ and $V_\alpha$ w.r.t. it. Similarly to the case of admissible fixed points in Definition \ref{admissiblefixed}, it is also immediate from our definition that when $P_\alpha$ is a non-zero function in $N_\alpha$, there is at most one closed curve in $N_\alpha$ which is periodic for the vector field $(s,\tau)\to(F_\tau(s)+V_\alpha(s,\tau),P_\alpha(s,\tau))$ - namely, the equilibrium periodic orbit $T\times\{\tau_0\}$. Having defined admissibility conditions for functions $P_\alpha$ and $V_\alpha$ defined around components of $Per$, we now proceed to extend the definition of admissibility to neighborhoods of certain invariant sets, and then tie all these different definitions together to define the Entropy flow.

\subsection{Stage $III$ - the final definition}
\label{sec3}
Having considered the local bifurcations involving periodic orbits and fixed points, in this Subsection we finally tie everything together and define the Entropy flow. We first do so when the parameter space is $(-1,1)$, after which we quickly generalize it to more complicated connected one-dimensional parameter spaces. We begin by extending the admissibility definitions from isolated periodic orbits and fixed points to isolated invariant sets, after which we proceed to define $P$ and $V$, thus completing the definition of the Entropy flow. The reason we do so is because while the definitions in the previous two Subsections are tailor made for persistent local behavior (i.e., fixed points and periodic orbits), we would also like to account for global phenomena - namely, for the appearance of robust invariant sets which persist without bifurcating when the parameter $\tau$ varies in some sub-interval $J\subseteq(-1,1)$. Before we begin, we remark that even though it may not be completely clear from context, our inspiration for this extension arose from Fenichel's Theorem (see \cite{fen}), as intuitively one would expect such invariant to be normally hyperbolic w.r.t. the flow given by Equations \ref{lam1}.\\

To begin, given a $C^1$ one--parameter family $\dot{s}=F_\tau(s)$, $\tau\in(-1,1)$, assume there exists some connected open set $N_\tau\subseteq M$, varying continuously in $\tau\in (\tau_1,\tau_2)$ s.t. the following is satisfied:
\begin{itemize}
    \item $(\tau_1,\tau_2)$ is a proper subinterval of $(-1,1)$, and, setting $N_A=\cup_{\tau\in(\tau_1,\tau_2)}N_\tau\times\{\tau\}$, then both $N_A\cap {Per}=\emptyset$ and $N_A\cap\overline{Fix}=\emptyset$.
    \item For all $\tau\in(\tau_1,\tau_2)$ there exists a connected, maximal invariant set $A_\tau$ in $N_\tau$ w.r.t. $F_\tau$ which varies with $\tau$ (for example, a hyperbolic attractor). Moreover, for all $\tau,\tau'\in(\tau_1,\tau_2)$ the dynamics of $F_\tau$ on $N_\tau$ are orbitally equivalent to those of $F_{\tau'}$ on $N_{\tau'}$ - in particular, the dynamics of $F_\tau$ on $A_\tau$ are orbitally equivalent to those of $F_{\tau'}$ on $A_{\tau'}$.
    \item The interval $(\tau_1,\tau_2)$ is maximal w.r.t. the properties above, i.e., we cannot extend $(\tau_1,\tau_2)$ in $(-1,1)$ s.t. the two requirements above are both satisfied (see the illustration in Figure \ref{perbifurcation4} with a Smale Horseshoe).\\

\end{itemize}

We refer to the set $A=\cup_{\tau\in(\tau_1,\tau_2)} A_\tau\times\{\tau\}$ as a \textbf{completely isolated invariant set} - it is easy to see that hyperbolic invariant sets are an example of the above, and so are families of isolated tori (as defined in the previous Section). We would now like to define the notion of admissible behavior on the set $N_A$ in a way that pushes $(s,\tau)\in N_A$ away from $\partial N_A\cap M\times\{\tau_i\}$, $i=1,2$. We do so as follows:
\begin{definition}
    \label{def6} Let $\dot{s}=F_\tau(s)$ be a $C^1$ one--parameter family defined over a manifold $M$ of dimension $n\geq1$, and parameterized by $\tau\in(-1,1)$. Let $A$ be a completely isolated invariant set with $N_A$ as defined above. A $C^1$-function $P_A:M\times I\to \mathbb{R}$ with a support in $N_A$ is admissible if the following holds (see Figure \ref{perbifurcation4}):
    \begin{itemize}
        \item If  $-1<\tau_1<\tau_2<1$, we choose $P_A$ s.t. $P_A(s,\tau)$ is positive when $(s,\tau)\in N_A$ is sufficiently close to $M\times\{\tau_1\}$, and negative when $(s,\tau)$ is sufficiently close to $M\times\{\tau_2\}$. Moreover, with previous notations, we require there is precisely one $\tau_0\in(\tau_1,\tau_2)$ s.t. $N_{\tau_0}\times\{\tau_0\}=P_A^{-1}(0)\cap N_A$ - the \textbf{equilibrium state} for $A$, which must always includes $A_{\tau_0}\times\{\tau_0\}$.
        \item If $-1<\tau_1,\tau_2=1$, we require $P_A$ to be positive throughout the interior of $N_A$. Symmetrically, when $\tau_1=-1$ and $\tau_2<1$, we require $P_A$ to be negative throughout the interior of $N_A$.
        \item When $(\tau_1,\tau_2)=(-1,1)$, similarly to Definition \ref{def3}, we require $P_A$ to be  $0$ throughout $M\times(-1,1)$.
    \end{itemize}

Note that by the same arguments used to prove Proposition \ref{def1} and Corollary \ref{def2}, whenever $A$ and $N_A$ exist there also exist infinitely many admissible functions. Moreover, we can even choose these admissible functions to be arbitrarily $C^1$-small.
\end{definition}
\begin{remark}
    Note that in this definition we do not choose a correction term $V_A$. The reason is because there is no reason to assume $A$ has any submanifold structure in $M\times(-1,1)$, therefore it is far from clear that for all $(s,\tau)\in A$ we can find some correction term $V_A(s,\tau)$ for which the vector $(F_\tau(s)+V_A(s,\tau),P_A(s))$ is tangent to $A$. We will give a more physical interpretation of this observation towards the end of this Subsection.
\end{remark}
\begin{figure}[h]
\centering
\begin{overpic}[width=0.5\textwidth]{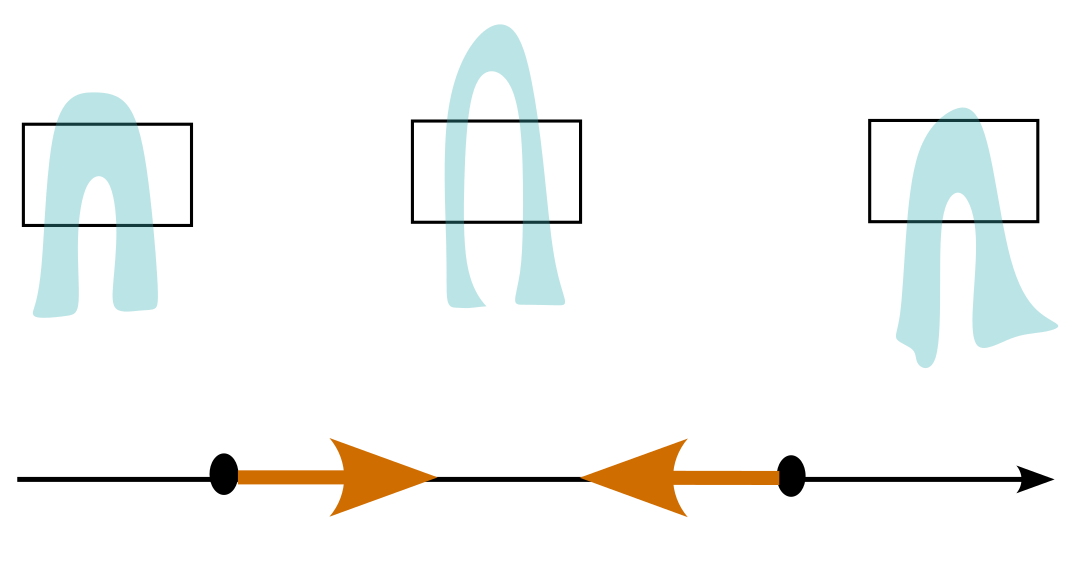}
\put(190,30){$\tau_1$}
\put(720,30){$\tau_2$}
\put(980,70){$\tau$}
\end{overpic}
\caption{\textit{A diagram representing the behavior of $E$ on an isolated invariant set. In this scenario, there exists a suspended Smale Horseshoe for all parameters inside $(\tau_1,\tau_2)$. As the said hyperbolic invariant set is isolated inside a suspended rectangle $R$, the admissible behavior for $(s,\tau)\in R\times(\tau_1,\tau_2)$ always points inside $R\times(\tau_1,\tau_2)$, represented by the orange arrows. By Definition \ref{def6} we know that in this scenario there must exist some equilibrium state in $R\times\{\tau_0\}$, for some $\tau_0\in(\tau_1,\tau_2)$.}}
\label{perbifurcation4}
\end{figure}

To justify our reasoning behind Corollary \ref{def6}, since we can choose $P_A$ to be arbitrarily $C^1$-small, when every $A_\tau$ is a hyperbolic invariant set for $F_\tau$, $\tau\in(\tau_1,\tau-2)$, by Fenichel's Theorem the flow of $(F_\tau(s),P_A(s))$ will behave "closely" to that of the flow generated by Equations \ref{lam1}. In other words, Corollary \ref{def6} allows us to both define directions on complex invariant sets, and in many cases also preserve at least some of their complex behavior. In particular, this allows us to think of the dynamics of the vector field $(s,\tau)\to (F_\tau(s),P_A(s,\tau))$ in $N_A$ as essentially a "stretched out" version of the dynamics of $F_\tau$ on $A_\tau$.\\

We are now almost ready to give a global definition for the Entropy flow, when the parameter space is the interval $I=(-1,1)$. To this end, given a $C^1$ one--parameter family $\dot{s}=F_\tau(s)$ parameterized by $\tau\in(-1,1)$, let $\{G_\alpha\}_\alpha$ denote the collection of all the components of $Per$, $Fix$, $Inv$ - where $Inv$ corresponds to the collection of completely isolated invariant sets. We now claim there exist open sets $\{N_\alpha\}_\alpha$ s.t. $G_\alpha\subseteq N_\alpha$, and when $\alpha\ne\beta$, $N_\alpha\cap N_\beta=\emptyset$. To see why, note that since we assume complete isolation, we can just construct the sets $N_\alpha$ recursively. In detail,  let $G_\alpha$ be as above - by definition, we can find some open neighborhood $N_\alpha$ of $G_\alpha$ s.t. ${N_\alpha}\cap (Per\cup Fix\cup Inv)=G_\alpha$. Since the closure of components in $Per\cup Fix\cup Inv$ intersects at most with $\partial N_\alpha$, given any other component $G_\beta\subseteq Per\cup C$, $\beta\ne\alpha$ define a similar isolating neighborhood $N_{\beta}$ s.t. $N_{\alpha}\cap N_\beta=\emptyset$. Given another component $G_\rho\subseteq Per\cup Fix$ s.t. $G_\rho\ne G_\alpha,G_\beta$, we again choose an isolating neighborhood $N_{\rho}$ s.t. both $N_{\rho}\cap N_{\beta}=\emptyset$ and $N_{\rho}\cap N_\alpha=\emptyset$. Repeating this argument ad infinitum (possibly uncountably many times), the collection $\{N_\alpha\}_\alpha$ is obtained. We will refer to this collection of open sets as an \textbf{isolating collection}.\\

 By the discussions immediately after Definitions \ref{admissiblefixed}, \ref{admissibleper} and \ref{def6} we know that every $N_\alpha$ includes the support of some admissible $C^1$-functions $P_\alpha:M\times(-1,1)\to \mathbb{R}$ and $V_\alpha:M\times(-1,1)\to TM$. We now refer to the collection $\{(N_\alpha,P_\alpha,V_\alpha)\}_\alpha$ as \textbf{admissible behavior subordinate to $\{N_\alpha\}_\alpha$}. We will adopt two conventions w.r.t. admissible behaviors:
 \begin{itemize}
     \item $V_\alpha$ is identically $0$ whenever $N_\alpha$ isolates a completely isolated invariant set in $Inv$.
     \item For all $\alpha$, whenever $P^{-1}_\alpha(0)\ne\emptyset$ then $P^{-1}_\alpha(0)=N_\alpha\cap M\times\{\tau_0\}$, for some $\tau_0\in(-1,1)$. Moreover, all the first derivatives of $P_\alpha$ and the differential of $V_\alpha$ both tend to $0$ when $(s,\tau)\to P^{-1}_\alpha(0)$.\\
 \end{itemize}
 
By Proposition \ref{def1}, Proposition \ref{def11} and Corollary \ref{def2} we know that given an isolating collection $\{N_\alpha\}_\alpha$, admissible behaviors subordinate to $\{N_\alpha\}_\alpha$ and satisfying the above exist. Given such an admissible behavior subordinate to $\{N_\alpha\}_\alpha$, let us define the following functions:

\begin{equation*}
    P(s,\tau)=\begin{cases}
        P_\alpha(s,\tau) &\text{when } (s,\tau)\in N_\alpha\\
        0 &\text{otherwise}
    \end{cases}
\end{equation*}
\begin{equation*}
    V(s,\tau)=\begin{cases}
        V_\alpha(s,\tau) &\text{when } (s,\tau)\in N_\alpha\\
        0 &\text{otherwise}
    \end{cases}
\end{equation*}

By definition, the functions $P$ and $V$ are $C^1$ throughout $M\times (-1,1)$, as $P_\alpha$ and $V_\alpha$ are supported in $N_\alpha$. Moreover, both $P$ and $V$ vanish outside of $\cup_\alpha N_\alpha$, and their differentials tend to $0$ as $(s,\tau)$ tends to either $\partial N_\alpha$ or $P^{-1}_\alpha(0)$ (the latter only when $P^{-1}_\alpha(0)\ne\emptyset$).  \\ 

Having defined the functions $V$ and $P$ we now prove the same definitions also extend to the case where $\dot{s}=F_\tau(s)$ is a $C^1$ one--parameter family parameterized by $\tau\in S^1$. That, however, is immediate - to see why, note Proposition \ref{def1}, Corollary \ref{def2} and Definition \ref{def3} trivially generalize to the case when $\tau$ varies in $S^1$, which implies Definitions \ref{admissiblefixed}, \ref{admissibleper} and \ref{def6} also trivially generalize to the same case. Similarly, consider a $C^1$ one--parameter family $\dot{s}=F_\tau(s)$ where $\tau$ varies in a parameter space $I$ diffeomorphic to either $[-1,1]$, $(-1,1]$ or $[-1,1)$. By imposing the (trivial) restriction that both $\lim_{\tau\to\partial I}P_\alpha(s,\tau)=0$, $\lim_{\tau\to\partial I}V_\alpha(s,\tau)=0$ for all $\alpha$, all our definitions and constructions above immediately extend to the case where $\tau$ varies in either a closed or a half closed interval. We may summarize this discussion as follows:

\begin{corollary}
\label{final}    Let $\dot{s}=F_\tau(s)$ be a $C^1$ one--parameter family of vector fields over $M$, parameterized by $\tau\in I$ - where $I$ is diffeomorphic to either $(-1,1),[-1,1],[-1,1),(-1,1]$ or $S^1$. Then, the sets $Per, Fix$ and $Inv$ are all well-defined, and there exists at least one isolating collection $\{N_\alpha\}_\alpha$ s.t. if $\{G_\alpha\}_\alpha$ are the components of $Per\cup Fix\cup Inv$ then for all $\alpha$, $N_\alpha\cap(Per\cup Fix\cup Inv)=G_\alpha$, and whenever $\beta\ne\alpha$, $N_\alpha\cap N_\beta=\emptyset$. Consequently, there exists at least one admissible behavior $\{(N_\alpha,P_\alpha,V_\alpha)\}_\alpha$ subordinate to $\{N_\alpha\}_\alpha$, hence the functions
\begin{equation*}
    P(s,\tau)=\begin{cases}
        P_\alpha(s,\tau) &\text{when } (s,\tau)\in N_\alpha\\
        0 &\text{otherwise}
    \end{cases}
\end{equation*}
\begin{equation*}
    V(s,\tau)=\begin{cases}
        V_\alpha(s,\tau) &\text{when } (s,\tau)\in N_\alpha\\
        0 &\text{otherwise}
    \end{cases}
\end{equation*}
are $C^1$ throughout the interior of $M\times I$ and satisfy the following:
\begin{itemize}
    \item $P$ has constant sign on $N_\alpha$, for all $\alpha$.
    \item Both $P$ and $V$ vanish outside of $\cup_\alpha N_\alpha$, and and their differentials tend to $0$ as $(s,\tau)$ tends to either $\partial N_\alpha$ or $P^{-1}_\alpha(0)$ (the latter whenever $P^{-1}_\alpha(0)\ne\emptyset$).
\end{itemize}
Moreover, when $I$ is not open and includes some of its boundary points (i.e., when it is either closed or a half closed interval), both $P$ and $V$ tend continuously to $0$ as $\tau\to\partial I$.
\end{corollary}
Therefore, all in all, for every connected one-dimensional parameter spaces $I$ and any $C^1$ family of vector fields $\dot{s}=F_\tau(s)$ defined on $M$ and parameterized by $\tau\in I$, the functions $P$ and $V$ are well-defined and are at least $C^1$ at the interior of $M\times I$. Having shown that, we finally arrive at the definition of the Entropy flow:
\begin{definition}
    \label{generalentropy} Consider a $C^1$ one-parameter of vector vector fields $\dot{s}=F_\tau(s)$ defined on $(s,\tau)\in M\times I$ - where $M$ is a connected, smooth, locally compact, metrizable manifold of finite dimension, and the parameter space $I$ is diffeomorphic to either $(-1,1),[-1,1),(-1,1]$ or $S^1$. Given any isolating collection $\{N_\alpha\}_\alpha$ and a corresponding $\{(N_\alpha,P_\alpha,V_\alpha)\}_\alpha$ subordinate to $\{N_\alpha\}_\alpha$, the corresponding \textbf{Entropy flow} is the flow generated by the vector field $E:M\times I\to T(M\times I)$, defined by the following equations:
    \begin{equation}
        \label{entropyflow} E(s,\tau)=\begin{cases}
            (F_\tau(s)+V(s,\tau),P(s,\tau))&\text{when } (s,\tau)\in \cup_\alpha N_\alpha,\\
            (F_\tau(s),0) &\text{otherwise},
        \end{cases}
    \end{equation}
    where $T(M\times I)$ is the tangent bundle, and both $P$, $V$ are as in Corollary \ref{final} - we will occasionally refer to $P$ as the \textbf{drift function in the parameter space} and to the function $V$ as the \textbf{energy transition function}. In particular, the vector field $E$ is tangent to $Per\cup Fix$m, and its dynamics in every $N_\alpha$ are admissible per either Definition \ref{admissiblefixed}, Definition \ref{admissibleper}, or Definition \ref{def6}.
\end{definition}
\begin{remark}
    As there are uncountably many isolating collections $\{N_\alpha\}_\alpha$, there also exist infinitely many Entropy flows - i.e., the Entropy flow is by no means unique. In the next Section and in Appendix \ref{quantitative} we will explore this theme further.
\end{remark}
At this point we make several remarks. The first is that our use of terms like "equilibrium state" or "Entropy flow" is not random or accidental. As stated at the Introduction, the Entropy flow as a concept is to a large extent inspired by the Second Law of Thermodynamics, i.e., that in a closed system the complexity (e.g. the Entropy) always increases. To see this heuristic connection, note that intuitively, one could interpret the vector field $F_\tau$, $\tau\in I$ as the law of motion guiding the trajectory of an initial condition $s\in M$. In addition, further note that by our construction above, the Entropy flow can be interpreted as a force gradually pushing the trajectory of $(s,\tau)$ towards $M\times\{\tau'\}$, where the motion induced by $F_{\tau'}$ is intuitively more complex  - at least when compared to $F_\tau$. As such, one could heuristically describe the Entropy flow as "the laws of motion change while in motion and become increasingly more complex as time goes" - or, put simply, the complexity can only increase. We will explore this further in Sections  \ref{basicqualitative} and \ref{turbulence}, where we will study how the Entropy flow behaves around routes to chaos and the Shilnikov homoclinic bifurcation scenario.\\

The second remark is that continuing on the above, the function $V$ can be thought of, in some sense, as the function "transferring" energy as the Entropy flow transitions between the different states of the system - where by the "states of the system" we mean the different laws of motion, $F_\tau$, present in the one-parameter family. It is for this reason we refer to it as the \textbf{energy transition function}. To illustrate, assume there is some energy function $\mathcal{E}:M\to\mathbb{R}$ defined on $M$ (for example: a Hamiltonian, a function describing the potential energy, etc). Such a function should be defined via points $s\in M$ alone and must be independent of the parameter $\tau$.  As such, for all fixed $\tau\in I$, if $s_\tau(t),t\in\mathbb{R}$ denotes the trajectory of $s$ in $M$ w.r.t. $F_\tau$, by the preservation of energy we must have $0=\frac{d\mathcal{E}(s_\tau(t))}{dt} = \nabla_s \mathcal{E}(s_\tau(t)) \cdot \frac{ds_\tau}{dt} = \nabla_s \mathcal{E} \cdot F_\tau(s)$, since this is independent of $\tau$,  $\frac{\partial\mathcal{E}(s_\tau(t))}{\partial \tau}=0$. If, however, we follow the definition of the Entropy flow and consider the additional terms $V(s,\tau), P(s,\tau)$, the result changes as follows:
\begin{equation*}
    \frac{d\mathcal{E}(s_\tau(t),\tau(t))}{dt} = \nabla_s \mathcal{E} \cdot (F_{\tau(t)}(s_\tau(t))+V(s_\tau(t),\tau(t))+\frac{\partial\mathcal{E}(s_\tau(t),\tau(t))}{\partial \tau}\cdot P(s_\tau(t),\tau(t))=\nabla_s\mathcal{E}\cdot V(s(t),\tau(t)).
\end{equation*}
In other words, $V$ controls the transition of energy. This further justifies our decision to make $V$ vanish identically on certain isolated invariant sets on which the dynamics are "fully formed" as in Definition \ref{def6} - as these sets are "complete" in some sense, no energy is needed to push them towards a greater state of complexity. That being said, at this point we emphasize the calculation above is true under the assumption such a function exists - in general, it is far from obvious it exists in a meaningful way for a general $C^1$ family $\dot{s}=F_\tau(s)$ defined over the product space $M\times I$.\\

A third remark is that one could also interpret the Entropy flow as a very specific type of a Slow-Fast system. To illustrate, let $\dot{s}=F_\tau(s)$ denote a $C^1$ one--parameter family of vector fields where the paramer $\tau$ varies in $ I$, and let us write the corresponding Entropy flow as follows:
\begin{equation*}
\begin{cases}
\dot{s} = F_\tau(s)+V(s,\tau) \\
 \dot{\tau} =P(s,\tau)
\end{cases}
\end{equation*}
Let us rewrite the equation above as follows for some $C^1$-functions $\overline{V}, \overline{P}$ and some $\epsilon>0$:

\begin{equation*}
\begin{cases}
\dot{s} = F_\tau(s)+\epsilon\overline{V}(s,\tau)=\mathcal{F}(s,\tau,\epsilon) \\
 \dot{\tau} =\epsilon\overline{P}(s,\tau)
\end{cases}
\end{equation*}
Making a change of variables, we get, again, for some $C^1$-function $Q$:
\begin{equation*}
\begin{cases}
\epsilon\dot{s} = \mathcal{F}(s,\tau,\epsilon) \\
 \dot{\tau} =Q(s,\tau)
\end{cases}
\end{equation*}
The above is the normal form for Slow-Fast systems. In particular, $\dot{s}$ would be the fast variable and $\tau$ would be the slow one. Even though we will not use tools from the theory of Slow-Fast systems to analyze the Entropy flow, one should often think of the drift function in the parameter space, $P$, as forcing a slow transition between the different states of the system, i.e., the different vector fields $\dot{s}=F_\tau(s)$, $\tau\in I$.\\ 

Before we conclude this Subsection we would like to state that in what follows next in this paper, we will only analyze the Entropy flow using qualitative tools. That being said, in the spirit of \cite{facet}, we believe the Entropy flow can also be analyzed using a more quantitative approach. We go back to this idea in Appendix \ref{quantitative}, where we attempt to indicate an approach to develop such a theory. That being said, before we begin the qualitative analysis of the Entropy flow in earnest in the next Section, we study two examples illustrating how the dynamics of the Entropy flow behave. The reason we do so is in order to show that despite its extremely technical definition, in practice describing the behavior of Entropy flows for one--parameter families $\dot{s}=F_\tau(s)$ is not as complicated as it may appear. The first example we will analyze is the van der Pol oscillator, after which we will analyze a more artificial example of a pitchfork bifurcation cascade. That being said, in addition to these examples, over the next sections we will also analyze the Entropy flows for other "real life" families of vector fields - including the Lorenz and the Rössler systems (among others).

\subsubsection{\textbf{An example - the van der Pol oscillator}}
\label{vpo}
In this Subsection we study the behavior of the Entropy flow via a simple example that showcases how its dynamics can be analyzed. Specifically, in this Subsection we will analyze the Entropy flow derived from the van der Pol oscillator. To begin, consider the following two-dimensional system originally introduced in \cite{vdp} and studied in \cite{cart}, often referred to as the van der Pol oscillator with a parameter $r$ varying in $\mathbb{R}$:
\begin{equation}
    \label{vanderpol}
    \begin{cases}
    \dot{x}=y\\
    \dot{y}=r(1-x^2)y-x
\end{cases}
\end{equation}
For simplicity, we will restrict $r$ to the interval $(-2,\infty)$ for reasons which will be soon clear. By direct computation, for all values of $r\in(-2,\infty)$ the corresponding flow has a single fixed point at the origin. Setting $J_r$ as the Jacobian matrix at the origin of the $r$-dependent vector field, it is easy to prove $J_r$ has two eigenvalues given by $\frac{r\pm\sqrt{r^2-4}}{2}$. Thus, when $r>0$, the fixed point at the origin is a source and when $-2<r<0$, it is a sink. Moreover, when $r=0$, the origin is a center which undergoes Hopf bifurcation at which the sink is opened up to a stable orbit (see the illustration in Figure \ref{vdp1}).\\

\begin{figure}[h]
\centering
\begin{overpic}[width=0.4\textwidth]{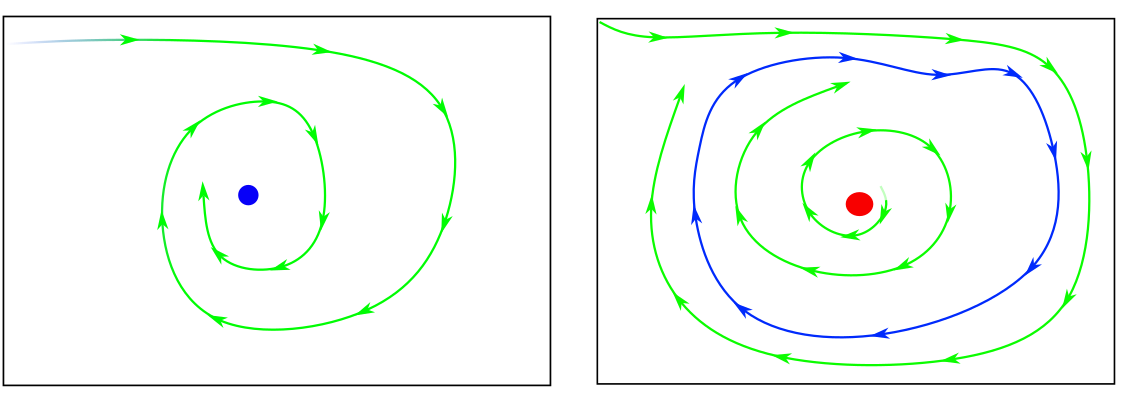}
\put(130,-25){$-2<r<0$}
\put(655,-25){$0<r<\infty$}
\end{overpic}
\caption{\textit{On the left - the situation when $r<0$, where the origin is a sink. On the right - the situation when $r>0$, after  Hopf bifurcation creates $T_r$ (the blue color).}}
\label{vdp1}
\end{figure}

Now, extend the van der Pol oscillator into an Entropy flow on $\mathbb{R}^2\times(-2,\infty)$, whose defining system of equations we denote as follows:

\begin{equation}
\label{vdpen}
    \begin{cases}
    \dot{x}=y+v_1(x,y,r)\\
    \dot{y}=r(1-x^2)y-x+v_2(x,y,r)\\
    \dot{r}=P(x,y,r)
\end{cases}
\end{equation}
We remark that in the previous notation, $V(x,y,r)=(v_1(x,y,r),v_2(x,y,r),0)$. Denote by $F_r$ the planar vector field $F_r(x,y)=(y,r(1-x^2)y-x)$. When $r>0$, by Lienard's Theorem we know that $F_r$ generates precisely one limit cycle, $T_r$ (see, for example, Theorem 3.8.1 in \cite{perko}). Moreover, as $T_r$ was created at supercritical Hopf bifurcation, we also know $T_r$ is an attracting cycle.\\

We will now proceed by quickly analyzing the dynamics of the van der Pol oscillator for $r>0$, which we will then use to analyze the Entropy flow. To this end, fix some $r>0$ - we claim that every initial condition in $\mathbb{R}^2\setminus\{0\}$ is attracted to $T_r$. To see why, note that there are no periodic orbits or fixed points in the unbounded component of $\mathbb{R}^2\setminus T_r$, which we denote by $C_u$. By Poincare-Bendixon Theorem (and Lienard's Theorem mentioned above) we conclude that since $T_r$ is attracting for the dynamics of $F_r$, the trajectory of every initial condition in $C_u$ w.r.t. $-F_r$ tends to $\infty$. Therefore, the trajectory of every initial condition in $C_u$ w.r.t. $F_r$ tends to $T_r$. Setting $C_b$ as the bounded component of $\mathbb{R}^2\setminus T_r$, since $T_r$ is stable and generated by an Hopf bifurcation, we already know that every initial condition in $C_b\setminus\{0\}$ is attracted to $T_r$ under the flow of $F_r$. All in all, when $r>0$ the orbit $T_r$ attracts all initial conditions in $\mathbb{R}^2\setminus\{0\}$.\\
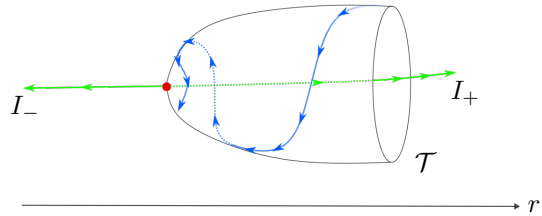
\begin{figure}[h]
\centering
\begin{overpic}[width=0.4\textwidth]{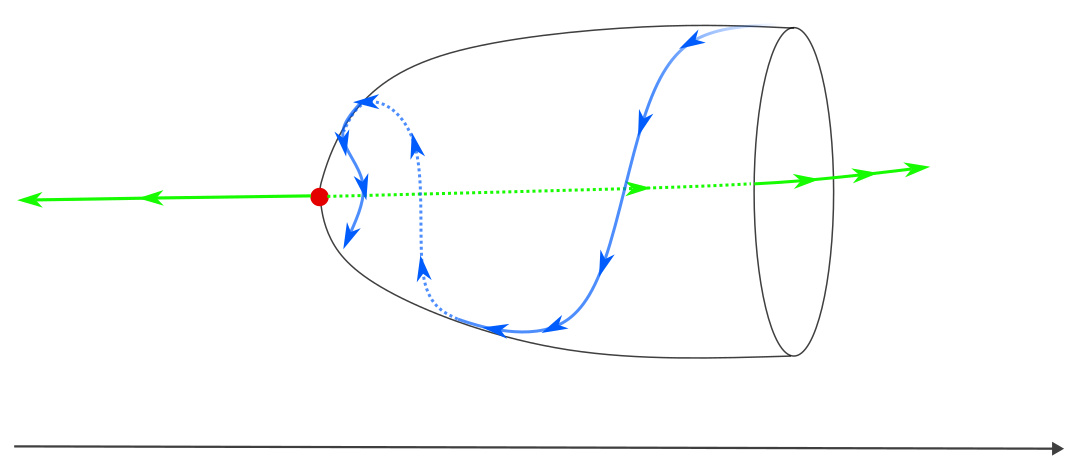}
\put(780,90){$\mathcal{T}$}
\put(850,230){$I_{+}$}
\put(-10,200){$I_{-}$}
\put(1000,10){$r$}

\end{overpic}
\caption{\textit{A scheme of the behavior of the Entropy flow for the van der Pol oscillator. On $\mathcal{T}$ the flow pushes towards $(0,0)\times\{0\}$ (the red dot), while on $I_-,I_+$ it pushes away.}}
\label{vdp2}
\end{figure}

We now use the above analysis to describe the behavior of the Entropy flow in $\mathbb{R}^2$. By definition, the set of completely isolated fixed points in $\mathbb{R}^2\times(-2,\infty)$ is given by $Fix=\{(0,0,r)|-2<r<0,0<r<\infty\}$. Similarly, set $\mathcal{T}=\cup_{r>0}T_r\times\{r\}$ - recalling $T_r$ is the only periodic orbit for $F_r$, $r>0$, as $T_r$ is an attractor for all $r>0$ it is easy to see $\mathcal{T}$ forms a surface lying in $\mathbb{R}^2\times(0,\infty)$ and that $\mathcal{T}=Per\cap \mathbb{R}^2\times(0,\infty)$. In other words, $\mathcal{T}$ is the collection of completely isolated periodic orbits for the flow in $\mathbb{R}^2\times(0,\infty)$. By the definition of the Entropy flow, we also know $(0,0)\times\{0\}$, the bifurcation point, is a fixed point for the Entropy flow. We now note the Jacobian matrix $J_0$ of $F_0$ has precisely two eigenvalues $\pm i$, which implies $|J_0|=1$. As $\mathcal{T}$ and $Fix\setminus(0,0)\times\{0\}$ are both isolated from all other periodic orbits and fixed points, by Definition \ref{admissiblefixed} (and in particular, by Definition \ref{def4}) and the definition of the Entropy flow, the following holds:
\begin{itemize}
    \item  As $(0,0)\times\{0\}$ is the only bifurcation point, it is the only fixed point for the Entropy flow in $\mathbb{R}^2\times(-2,\infty)$. 
    \item If $I_-,I_+$ are the two components of $Fix\setminus\{(0,0)\times\{0\}\}$ lying in $\mathbb{R}^2\times(-2,0)$, $\mathbb{R}^2\times(0,\infty)$ (respectively) - by the above we know fixed points on $I_-$ are sinks while those on $I_+$ are sources. Since both are composed of completely isolated fixed points, by $|J_0|>0$ we know that given any initial condition $(s,r)\in I_\pm$, the Entropy flow pushes $(s,r)$ away $(0,0)\times\{0\}$ while keeping its trajectory tangent to either $I_+$ or $I_-$ (see Figure \ref{vdp2}).
    \item Similarly, as $det|J_0|>0$ and because $\mathcal{T}$ is a surface of completely isolated periodic orbits, the trajectories of initial conditions on $\mathcal{T}$ spirals toward $(0,0)\times\{0\}$ while remaining tangent to $\mathcal{T}$ (see Figure \ref{vdp2}).\\
\end{itemize}

The above holds for all possible choices of Entropy flows for the van der Pol oscillator, as all Entropy flows must obey Definitions \ref{def3}-\ref{def6}. As such, Figure \ref{vdp2} gives a general scheme for how the Entropy flow for the van der Pol oscillator behaves around $I_-,I_+$ and $\mathcal{T}$. 
 \subsubsection{\textbf{An example - a pitchfork cascade} }
\label{pitchhh}
The dynamics of the Entropy flow for the van der Pol oscillator in the previous Subsection may appear to suggest the Entropy flow always has simple dynamics. In practice, this is not the case for most $C^1$ families of vector fields - in fact, we will see many examples throughout this paper for Entropy flows whose dynamics are anything but simple. That being said, intuitively, the reason the Entropy flow for the van der Pol oscillator is not too complicated is due to the corresponding sets $Per$ and $Fix$ having a relatively simple structure. In this Subsection we will give another example, where the sets $Per$ and $Fix$ have a more complex topology - as we will show, this would force the existence of a more complicated Entropy flow. Before we begin, we remark that despite the artificial nature of our example, it is motivated by the numerical analysis performed in \cite{KS} on the Lorenz-96 model, originally introduced in \cite{lo2}, where several consecutive pitchfork bifurcations were observed. To begin, a $C^1$ one--parameter family of vector fields $\dot{s}=F_\tau(s)$, $s\in\mathbb{R}^n$, $n>2$, $\tau\in[0,\infty)$ satisfying the following (see the illustration in Figure \ref{pitchcasc}):

 \begin{itemize}
     \item For $\tau=0$, the origin is a sink.
     \item There exists an increasing sequence $\{\tau_n\}_{n \in \mathbb{N}}$ of parameters, $\tau_n\to\infty$, defined as follows:
     \begin{enumerate}
         \item The origin persists as a sink throughout $[0,\tau_1)$. At $\tau_1$, the origin splits in a pitchfork bifurcation into one saddle, $x_{r,1}$, and two sinks, $x_{s,1}$, $x_{s,2}$. All of these fixed points persists, without bifurcating when $\tau$ varies in $(\tau_1,\tau_2)$.
         \item At $\tau_2$, the sinks $x_{s_1},x_{s_2}$ both split in a pitchfork bifurcation to two sinks and one saddle each, while the saddle $x_{r,1}$ splits into two saddles and one sink. All of these new fixed points persists without bifurcating when $\tau$ varies in $(\tau_2,\tau_3)$.
         \item At each $\tau_n$, $n>2$, the scenario repeats - i.e., all sinks split into two sinks and one saddle, and all sources split into two saddles and one sink. Again, at every consecutive stage, all these new fixed points persists without bifurcating at $(\tau_n,\tau_{n+1})$.\\
\end{enumerate}
\end{itemize}

\begin{figure}[h]
\centering
\begin{overpic}[width=0.4\textwidth]{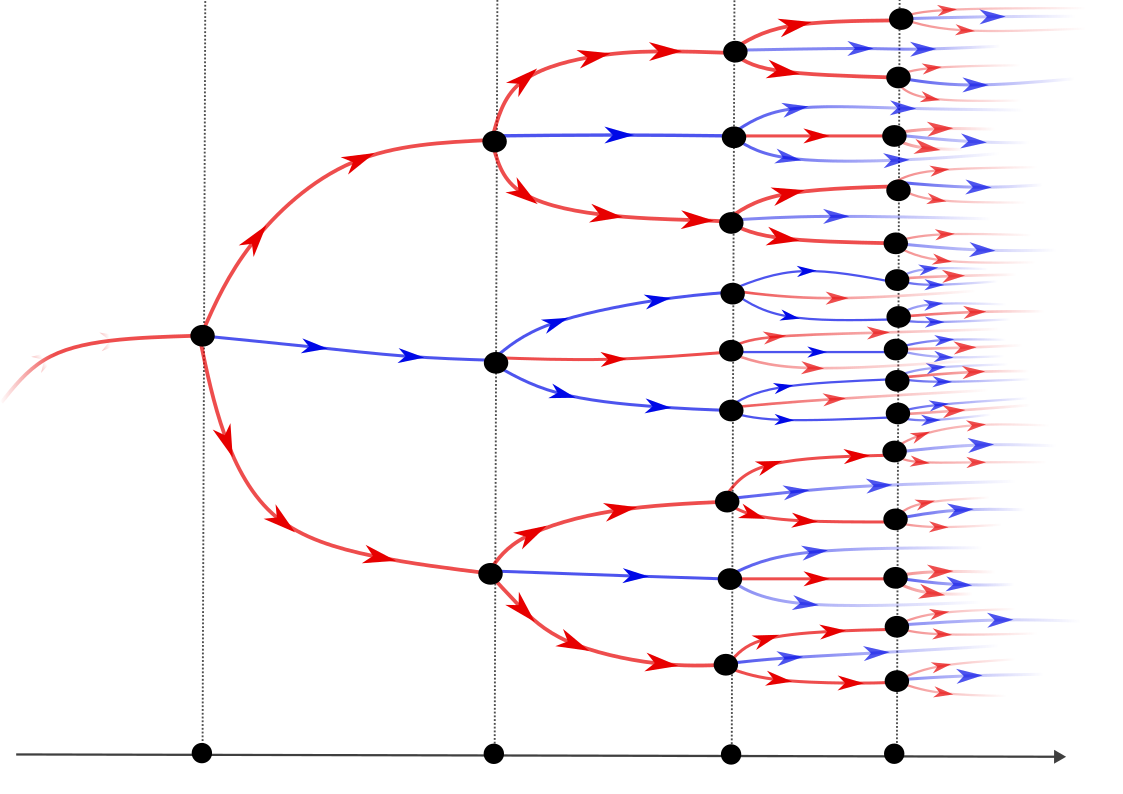}

\put(980,20){$\tau$}

\end{overpic}
\caption{\textit{A partial bifurcation diagram for $\dot{s}=F_\tau(s)$ and the pitchfork cascade. The arrows represent the directions $P$ is pointing on along the cascade (in this scenario we assume the saddles are all in $Fix$). Red denotes sinks, blue denotes sources. The dots denote the successive points,  where the bifurcations occur.}}
\label{pitchcasc}
\end{figure}

We refer to the scenario above as a \textbf{pitchfork cascade}. As all the fixed points created in this process are sinks and saddles, at least all the sinks created in the pitchfork cascade correspond to completely isolated fixed points in $Fix$. This implies that if we denote by $\mathcal{C}$ the subset of $\mathbb{R}^n\times[0,\infty)$ corresponding to the cascade above, then $\mathcal{C}\cap\mathbb{R}^n\times(\tau_n,\tau_{n+1})$ is composed of curves in $Fix$ of sinks (and possibly also saddles). By Definition \ref{admissiblefixed} (and in particular, Definition \ref{def4}) it follows that on each such curve of sinks the function $P$  is always positive - i.e., it always points towards $\mathbb{R}^n\times\{\tau_{n+1}\}$, the direction of creation of more fixed points. It is easy to see the resulting Entropy flow has complex dynamics, correlated with the complex structure of $\mathcal{C}$.

\section{Basic qualitative analysis}
\label{basicqualitative}

Having defined the Entropy flow for a $C^1$ one-parameter family of vector fields $\dot{s}=F_\tau(s)$, in this Section we perform basic qualitative analysis of the Entropy flow. In particular, we will study how the Entropy flow behaves in the three major routes to chaos - the period doubling route to chaos (see \cite{fei2} and \cite{fei}), the Ruelle-Takens-Newhouse route to chaos (see \cite{RT} and \cite{SRT}), and the Intermittency route to chaos (see \cite{PP}). This Section is organized as follows: we begin by discussing the general properties of the Entropy flow, and the behavior of its first return map around its periodic orbits. Following that, we study the Entropy flows behavior in the period doubling route to chaos (see Subsection \ref{perioddoubling}), the Ruelle-Takens-Newhouse route to chaos (see Subsection \ref{ruelletakens}), and finally, the Intermittency route to chaos (see Subsection \ref{intermittency}). As we will prove, in all of these routes to chaos the Entropy flow would have complex behavior which mirrors the original complex bifurcation structure of $\dot{s}=F_\tau(s)$. \\

To begin, let $M$ be a smooth, orientable, locally compact, metrizable manifold and let $\dot{s}=F_\tau(s)$ be a one--parameter $C^1$ family of vector fields, where $\tau$ varies in $I$ (recall that we always assume $I$ is either an open interval, a closed interval, a half closed interval, or $S^1$) . The vector field $E$ would always denote some corresponding Entropy flow, given by the formula $E(s,\tau)=(F_\tau(s)+V(s,\tau),P(s,\tau))$, $s\in M,\tau\in I$. We begin by proving the following Lemma, which connects the attracting and repelling invariant sets for $\dot{s}=F_\tau(s)$ with those of the Entropy flow:
\begin{lemma}
\label{trapping}  Assume there exists some open, connected subsets $J\subseteq I$ and $B\subseteq M$, satisfying the following:
\begin{itemize}
    \item  $\partial B$ is smooth.
    \item For all $\tau\in J$, $F_\tau$ is transverse to $\partial B$. 
\end{itemize}
Then, we can choose the Entropy flow s.t. it is transverse to $\partial (B\times J)$. In particular, for such an Entropy flow $B\times J$ would include either an attractor or a repeller, depending (respectively) on whether $F_\tau,\tau\in J$ points inside or outside of $B$ on $\partial B$.
\end{lemma}
\begin{proof}
For all $s\in\partial B$, let $N(s)$ denote the normal vector for $s$ in $M$. By definition, the normal vector to $\partial( B\times J)$ is $(N(s),0)$. By our assumption of the transversality of $F_\tau$ to $\partial B$, we know the dot product $F_\tau(s)\cdot N(s)$ is either positive or negative throughout $(s,\tau)\in \partial B\times J$ - which implies there can be no periodic orbits or fixed points for $F_\tau$ in $\partial B$. Thus, by the construction of the Entropy flow in the previous section we can construct $E$ s.t. the subset on $M\times I$ in which $V,P$ do not vanish lies away from $\partial (B\times J)$. For such a choice of an Entropy flow we will have $E(s,\tau)\cdot(N(s),0)=F_\tau(s)\cdot N(s)$. Hence, $E$ is transverse to $\partial (B\times J)$ and its behavior there depends on whether $F_\tau,\tau\in J$ points on $s\in\partial B$ inside or outside of $B$. 
\end{proof}
As technical as it may appear, Lemma \ref{trapping} shows us that in some sense, in many cases we can always find one that is "correlated" with the original dynamics. More precisely, let us assume $B$ is a basin of attraction for some attractor and that $J=I=[-1,1]$, i.e., that the attractor persists throughout the parameter space. Then, Lemma \ref{trapping} implies $\partial B\times I$ is a trapping region for the Entropy flow - i.e., that $B\times I$ includes an attracting invariant set for the Entropy flow. This idea will be useful later on, when we will study the Entropy flow around several bifurcation scenarios in Section \ref{turbulence}, and also in Subsection \ref{lorenz}, where we will study the Entropy flow of the Lorenz system. Expanding on this theme, Lemma \ref{trapping} raises another question - how different are the dynamics of the Entropy flow in $M\times I$ are from those of the original system?\\

This question is somewhat ill posed, as the Entropy flow is defined on $M\times I$, while the vector fields $F_\tau$ are defined on $M$. To make this question precise, we first recall that the Entropy flow is given as a solution to the following system of Ordinary Differential Equations in $M\times I$:

\begin{equation*}
\begin{cases}
\dot{s} = F_\tau(s)+V(s,\tau) \\
 \dot{\tau} =P(s,\tau)
\end{cases}
\end{equation*}
Let us now also consider the flow on $M\times I$ given by the following system of equations:
\begin{equation}
\label{laminar}
\begin{cases}
\dot{s} = F_\tau(s) \\
 \dot{\tau} =0
\end{cases}
\end{equation}
From now on, we will refer to the flow on $M\times I$ given by Equations \ref{laminar} as the \textbf{Laminar flow} - with the name derived from the fact that the laminar flow keeps the layers $M\times\{\tau\}$, $\tau\in I$, invariant. The previous question can now be formally and rigorously stated: how different are the dynamics of the Entropy flow from those of the Laminar flow?\\

This question probably cannot be answered in full generality. The reason is because the collection of pairs $(P,V)$ that define an Entropy flow probably forms a very complex Banach Manifold, and in general, one has to add restrictions on both $P$ and $V$ to study it (we will discuss this approach in more detail at Appendix \ref{quantitative}). That being said, one can find some global restrictions depending on the bifurcation structure - for example, in the next Subsection we will prove that given a $C^1$ family $\dot{s}=F_\tau(s)$ undergoing a period doubling cascade of attractors, the corresponding Entropy flow would always have a certain invariant set, depending only on the bifurcation structure (see Theorem \ref{major1}). Before we do so however, we study the local dynamics of the first return map of the Entropy flow around some of its periodic orbits. To begin, recall that given $\dot{s}=F_\tau(s)$ as above, the corresponding Entropy flow should be expected to have periodic orbits of the type $T\times\{\tau_0\}$, $(x_0,\tau_0)$, where $T$ is a periodic orbit for $F_{\tau_0}$ and $x_0$ is a fixed point for $F_{\tau_0}$. In detail, in the previous Section we saw such periodic orbits and fixed points can arise either as equilibrium states, or as bifurcation orbits and fixed points connecting the components of $Per\cup Fix$ - In either case, we know that if the Entropy flow is given by $(s,\tau)\to(F_\tau(s)+V(s,\tau),P(s,\tau))$, then $P$ and $V$ would both vanish on $T\times\{\tau_0\}$ and $x$ (for the details, see Definition \ref{generalentropy}, as well as Definitions \ref{admissiblefixed} and \ref{admissibleper}). We now study the local dynamics of such objects under the Entropy flow:

\begin{proposition}
    \label{centerdim} Consider a $C^1$ one--parameter family $\dot{s}=F_\tau(s)$ parameterized by $\tau\in I$. Then, given any interior point $\tau_0\in I$ and any choice of Entropy flow, we have the following:
    \begin{itemize}
        \item Assume $T$ is a periodic orbit for the vector field $F_{\tau_0}$, s.t. $T\times\{\tau_0\}$ is either connecting two different components of $Per$, or an equilibrium periodic orbit (see Definition \ref{admissibleper}). Then, the following holds:
        \begin{enumerate}
            \item If $T$ has a $k_1$-dimensional stable manifold and a $k_2$-dimensional unstable manifold w.r.t. $F_{\tau_0}$, then $T\times\{\tau_0\}$ has a stable manifold of dimension $k_1$ and an unstable manifold of dimension $k_2$ w.r.t. the Entropy flow.
            \item  If $T$ has a center manifold of dimension $k_3$ w.r.t. $F_{\tau_0}$, then $T\times\{\tau_0\}$ has a center manifold of dimension $k_3+1$.
        \end{enumerate}

        \item Similarly, let $x_{0}$ be a fixed point for $F_{\tau_0}$ s.t. $x=(x_0,\tau_0)$ is either connecting two components of $Fix$ or $(x_0,\tau_0)$ is an equilibrium point. Denote by $J_E$ the Jacobian matrix of the Entropy flow at $x$, and by $J_0$ is the Jacobian matrix of $F_{\tau_0}$ at $x_0$. Then, $J_E$ includes $J_0$ as a minor, and consequently the following holds:
        \begin{itemize}
            \item  If $x_0$ has a $k_1$-dimensional stable manifold and a $k_2$-dimensional unstable manifold, so does $x$.
            \item If $x_0$ has a center manifold of dimension $k_3$, $x$ has center manifold of dimension $k_3+1$.
        \end{itemize}
        
    \end{itemize}

\end{proposition}
\begin{proof}

We prove the assertion under the assumption $M=\mathbb{R}^n$ - as this is a local assertion, this would be enough to imply the general case. We will first prove the assertion for periodic orbits. Let $S_0\subseteq M$ denote a cross section transverse to $T$ w.r.t. the vector field $F_{\tau_0}$ at some point $s_0$. Let $f_{\tau_0}:S_0\to S_0$ denote the corresponding first-return map, and write $f_{\tau_0}(s)=(g_1(s,0),...,g_{n-1}(s,0))$, where $s=(s_1,...,s_{n-1})$ is a coordinate in $S_0$ and $n$ is the dimension of $M$. We will denote the differential of $f_{\tau_0}$ in $s_0$ by $D_{\tau_0}$ expanded in matrix form as follows:

 \begin{equation*}
D_{\tau_0}=
\begin{pmatrix}

\frac{\partial g_1}{\partial s_1}(s_0) & ... &\frac{\partial g_{1}}{\partial s_{n-1}}(s_0)\\

... & ...&...\\
\frac{\partial g_{n-1}}{\partial s_{1}}(s_0)&...&\frac{\partial g_{n-1}}{\partial s_{n-1}}(s_0)

\end{pmatrix}
\end{equation*}
 
Now, set $S=S_0\times(\tau_0-\epsilon,\tau_0+\epsilon)$. As $T\times\{\tau_0\}$ is periodic for the Entropy flow, we know that when $\epsilon>0$ is sufficiently small, $E$ is transverse to $S$. When $P$ and $V$ vanish identically in some neighborhood of $T\times\{\tau_0\}$, the first-return map $f:S\to S$ coincides with that of the Laminar flow given by Equations \ref{laminar}, which would imply its differential $D_f$ at $(s_0,\tau_0)$ given by:

 \begin{equation}
 \label{laminardiff}
D_{f}=
\begin{pmatrix}

\frac{\partial g_1}{\partial s_1}(s_0) & ... &\frac{\partial g_{1}}{\partial s_{n-1}}(s_0)&0\\

... & ...&...&...\\
\frac{\partial g_{n-1}}{\partial s_{1}}(s_0)&...&\frac{\partial g_{n-1}}{\partial s_{n-1}}(s_0)&0\\
0&...&0&1

\end{pmatrix}
\end{equation}
That is, $D_{\tau_0}$ would be a minor of $D_f$. This proves that when $P$ and $V$ both vanish around $T\times\{\tau_0\}$, there is nothing to prove. We will therefore assume $P$ and $V$ do not vanish in a neighborhood of $T\times\{\tau_0\}$. We first prove that if $D_{\tau_0}$ has a $k_1-1$-dimensional stable eigenspace and $k_2-1$-dimensional unstable eigenspace, then $D_f$ has a stable eigenspace of dimension at least $k_1-1$ and an unstable eigenspace of dimension at least $k_2-1$. This would be enough to imply $T\times\{\tau_0\}$ has a stable manifold of dimension at least $k_1$ and an unstable manifold of dimension at least $k_2$. To do so, denote by $P_0=\{(s,\tau)|V(s,\tau)=P(s,\tau)=0\}$ - per our assumption that both $P$ and $V$ vanish on $T\times\{\tau_0\}$ we know $T\times\{\tau_0\}\subseteq P_0$. Write the first return map $f$ as $(s,\tau)\to(f_1(s,\tau),...,f_{n-1}(s,\tau),p(s,\tau))$, which implies the differential $D_f$ can be expanded as follows:

\begin{equation*}
\begin{pmatrix}

\frac{\partial f_1}{\partial s_1}(s_0,\tau_0) & ... &\frac{\partial f_{1}}{\partial s_{n-1}}(s_0,\tau_0) &\frac{\partial f_1}{\partial \tau}(s_0,\tau_0)\\

... & ...&...&...\\
\frac{\partial f_{n-1}}{\partial s_{1}}(s_0,\tau_0)&...&\frac{\partial f_{n-1}}{\partial s_{n-1}}(s_0,\tau_0)&\frac{\partial f_{n-1}}{\partial \tau}(s_0,\tau_0)\\
\frac{\partial p}{\partial s_1} (s_0,\tau_0)&...&\frac{\partial p}{\partial s_{n-1}}(s_0,\tau_0) &\frac{p}{\partial\tau}(s_0,\tau_0)

\end{pmatrix}
\end{equation*}

As $P,V$ do not vanish around $T\times\{\tau_0\}$, we must have $T\times\{\tau_0\}\subseteq\overline{Per}$. Therefore, there are exactly two possibilities - that $T\times\{\tau_0\}$ lies on the boundary of some open set $O\subseteq P_0$, or that it does not. In the first possibility, we know that inside $O$ the Entropy flow coincides with the Laminar flow, where the differential for the first return map w.r.t. the Laminar flow is given by Equation \ref{laminardiff}. Therefore, by the continuity of all partial derivatives of $f_i$ and $p$, $i=1,...,n-1$, and because in $O$ the Entropy flow coincides with the Laminar flow, in the first case the differential $D_f$ is given by Equation \ref{laminardiff} - hence it includes $D_{\tau_0}$ as a minor. Consequently, the spectrum of $D_f$ includes that of $D_{\tau_0}$, therefore, if $T$ has a $k_1$-dimensional stable manifold and a $k_2$ unstable manifold w.r.t. $F_{\tau_0}$, $T\times\{\tau_0\}$ has stable and unstable manifolds of respective dimensions $k_1$ and $k_2$ w.r.t. the Entropy flow.

We now study the second case, i.e., that $T\times\{\tau_0\}$ does not lie on the boundary of any open set in $P_0$. To study this case we first recall that as the Entropy flow is given by the vector field $(s,\tau)\to(F_\tau(s)+V(s,\tau),P(s,\tau))$ where $(s,\tau)\in M\times I$. We further recall that the functions $P$ and $V$ are given per Corollary \ref{final}, i.e., there exists some isolating collection $\{N_\alpha\}_\alpha$ of mutually disjoint open sets s.t. the support of $P$ and $V$ is contained in $\cup_\alpha N_\alpha$, and the derivatives $\frac{\partial P(s,\tau)}{\partial x}$, $\frac{\partial V(s,\tau)}{\partial x}$, $x=s,\tau$ tend to $0$ as $(s,\tau)\to\partial N_\alpha$. Moreover, for all $\alpha$, $P$ vanishes on $\partial N_\alpha$ and $P^{-1}_\alpha(0)$ (see Corollary \ref{final} and the discussion preceding it). This implies that as by assumption $P$ on $T\times\{\tau_0\}$ then $T\times\{\tau_0\}\subseteq P^{-1}_\alpha(0)\cup \partial N_\alpha$, for some $\alpha$. By the requirements imposed on $P$ and $V$ in Definition \ref{generalentropy} and Corollary \ref{final} it follows the differentials of both $V$ and $P$ vanish on $T\times\{\tau_0\}$. This again implies that $D_f$ is given by Equation \ref{laminardiff}, hence its includes the spectrum of $D_{\tau_0}$. This proves that if $T$ has a $k_1$-dimensional stable manifold and a $k_2$-dimensional unstable manifold then the same is true for $T\times\{\tau_0\}$ w.r.t. the Entropy flow. The above also teaches us that in whatever case, $D_f$ is given by Equation  \ref{laminardiff}. This immediately implies that if $T$ has a $k_3$-dimensional center manifold w.r.t. $F_{\tau_0}$, then $T\times\{\tau_0\}$ has a $k_3+1$-dimensional center manifold w.r.t. the Entropy flow.

We now prove the analogous assertion fixed points. Let us denote by $J$ as the Jacobian matrix of the Laminar flow at $x=(x_0,\tau_0)$, by $J_0$ as the Jacobian matrix of $F_{\tau_0}$ at $x_0$, and by $J_E$ the Jacobian matrix of the vector field $E$ giving the Entropy flow at $x$. By direct computation we get:

\begin{equation*}
J=
\begin{pmatrix}

J_0 & 0\\

0 &0

\end{pmatrix}
\end{equation*}

In other words, $J_0$ is a matrix of $J$. Using again the fact that $P,V$ and their differentials both vanish when $(s,\tau)\in M\times I$ tend to either $\partial N_\alpha$ and $P^{-1}_\alpha(0)$ (for some $\alpha$), the exact same arguments above when applied to $J_E,J_0$ and $J$ in place of $f$, $f_{\tau_0}$ and the first return map for the Laminar flow imply that $J=J_E$. This proves that whatever the case, $J_E$ includes $J_0$ as a minor, and the assertion follows.
\end{proof}
\begin{remark}
We remark that by Theorem 1 in \cite{centerp} one can further prove that if two (or more) components $G_\alpha$ and $G_\beta$ of $Per$ are connected to each other at some periodic orbit $T\times\{\tau_0\}$, then these two components would form invariant manifolds for $T\times\{\tau_0\}$ w.r.t. the Entropy flow. Similarly, if two (or more) components $G_\alpha$ and $G_\beta$ connect at $x$, both would form invariant manifolds for $x$ w.r.t. the Entropy flow. 
\end{remark}
Proposition \ref{centerdim} may appear at first somewhat technical and unmotivated, but in fact it encodes the topology of bifurcations. To explain, recall that the functions $P$ and $V$ are allowed to to be non-zero where "interesting and distinguishable dynamics happen" - i.e., in and around the sets $Per\cup Fix$, and on the completely isolated invariant sets in $Inv$. In other words, the different ways in which we can define $P$ and $V$ correspond to the topology of the sets $Per$ and $Fix$, and their configuration in $M\times I$. In particular, the configuration of the center Manifolds of the invariant sets of the Entropy flow must be correlated with the topology of $Per$ and $Fix$, as well as the bifurcation orbits and fixed points gluing them together - consequently, one should think of the collection of possible Entropy flow as some invariant of the topology of $Per$ and $Fix$. We will return to this approach in Subsections \ref{2dtheory}, \ref{lorenz}, where we will study this aspect using the Conley index Theory.\\

That being said, as far as general qualitative analysis of the Entropy flow goes, this is as far as we go - if only because to give more global results, we need to have more structure on the sets $Per$ and $Fix$. To clarify, by "structure" we mean information on how the bifurcation orbits and fixed points glue the components of $Per\cup Fix$ to one another. Therefore, in the next Subsections we will analyze the global behavior of the Entropy flow on period doubling cascades, the Ruelle-Takens route to chaos and the Intermittency phenomena around saddle node bifurcations. As we will prove, in each one of these the dynamics of any Entropy flow would be very complex.

\subsection{The Period Doubling route to chaos}
\label{perioddoubling}
In this Subsection we study the behavior of the Entropy flow around a period doubling cascade of attractors or repellers. As we will prove, in such cascades the qualitative behavior around the cascade itself are to a large degree independent of the functions $P$ and $V$ chosen. In detail, we will prove that around such cascades for all possible choices of Entropy flows both $P$ and $V$ do not vanish, and that their behavior can be roughly described as "a drift towards complexity" (see Theorem \ref{major1} for the precise details). Such cascades of attractors and repellers appear in many examples of smooth one-parameter families that undergo transition from order into chaos - for example, such bifurcations often appear when a Smale Horseshoe is formed (see \cite{Perd}), or alternatively, around the Shilnikov bifurcation scenario (see \cite{GKP}). The results of this Subsection should be interpreted as follows - that for a large class of "real life" flows the global dynamics of the Entropy flow would deviate from those of the Laminar flow. To illustrate, at the end of this Subsection we will use these ideas to describe the Entropy flow for the Rössler system.\\

To begin, let $M$ denote a smooth orientable manifold of dimension at least $3$, and let $\dot{s}=F_\tau(s)$ be a one--parameter $C^1$ family defined on it, where the parameter $\tau$ varies in $I$. We now define the notion of a period doubling cascade - for simplicity, we will do so under the assumption $I=(-1,1)$ (the generalization to the case where $I=S^1$, or when it is either a closed or a half closed interval is trivial). To begin, let $\tau_0$ be some interior point to $(-1,1)$. Let $T_0$ be a periodic orbit for $F_{\tau_0}$. Recall that $T_0$ undergoes a \textbf{period doubling cascade} if there is some bounded, increasing sequence $\{\tau_n\}_{n\geq0}\subseteq (-1,1)$ s.t. the periodic orbit $T_0$ undergoes successive period doubling bifurcations at every $\tau_n$ (see the illustration in Figure \ref{cascade}). To continue, let $\{G_n\}_{n\geq0}$ denote the subsets of $M\times I$ corresponding to the said cascade - in other words, the sets $\{G_n\}_n$ satisfy the following:
\begin{itemize}
    \item For all $n\geq0$, $G_n\cap M\times\{\tau\}\ne\emptyset$, when $\tau\in(\tau_n,\tau_{n+1})$, $n\geq0$. Moreover, writing $T_\tau\times\{\tau\}=G_n\cap M\times\{\tau\}$, the set $T_\tau$ forms a periodic orbit for $F_\tau$ that persists without bifurcating as $\tau$ is varied in $(\tau_n,\tau_{n+1})$.
    \item $G_n$ is connected to $G_{n+1}$ at $P_{n+1}\times\{\tau_{n+1}\}$, where $P_{n+1}$ is the period doubling orbit for $F_{\tau_{n+1}}$, $n\geq1$.
    \item  Similarly to the above, let $T_\tau$ denote the periodic orbit for $F_\tau$, $\tau\in(\tau_n,\tau_{n+1})$ given by $T_\tau\times\{\tau\}=M\times\{\tau\}\cap G_n$, $n\geq1$. Letting $p_n$ denote the period of $P_n$ w.r.t. $F_{\tau_n}$ and $p_\tau$ denote the period of $T_\tau$ w.r.t. $F_\tau$, then $\lim_{\tau\rightarrow\tau_n}p_\tau=2p_n$.
    \item Moreover, for each $n\geq1$, $P_n\times\{\tau_n\}$ is the meeting point between $G_n$, $G_{n+1}$ and $D_n$ - where $D_n\subseteq M\times (\tau_n,1)$ corresponds to a family of periodic orbits thar preserve the period of $P_n$. More precisely, for all $\tau>\tau_n$ sufficiently close to $\tau_n$, $S_\tau\times\{\tau\}=D_n\cap M\times\{\tau\}$ corresponds to a periodic orbit $S_\tau$ for $F_\tau$ with a period $m_\tau$ satisfying $\lim_{\tau\downarrow\tau_n}m_\tau=p_n$ (where $p_n$ is as above).\\
\end{itemize}

\begin{figure}[h]
\centering
\begin{overpic}[width=0.3\textwidth]{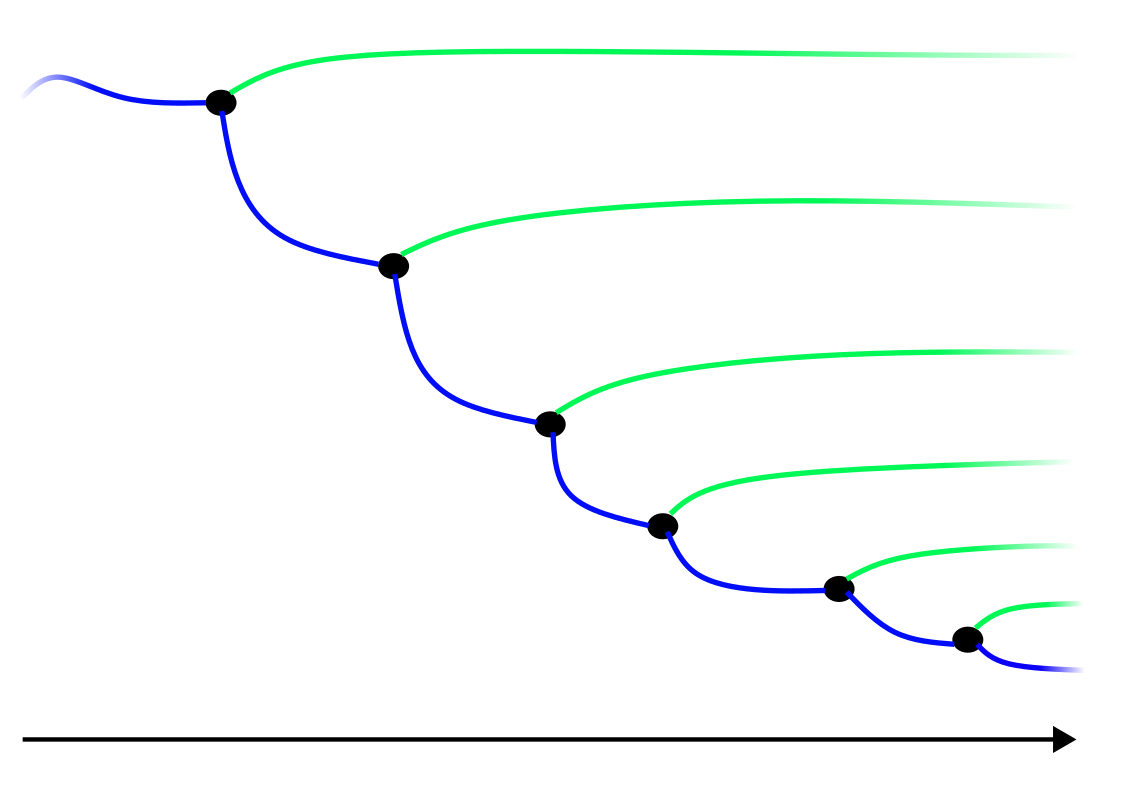}
\put(980,30){$\tau$}
\end{overpic}
\caption{\textit{The bifurcation diagram of a period doubling cascade. The successive black dots represent the period doubling orbits (i.e., $P_n\times\{\tau_n\}$), while the blue arcs represent the sets $G_n$ and the green arcs -- the sets $D_n$. The period doubling occurs on the blue arcs, along which the period is doubled successively.}}
\label{cascade}
\end{figure}

With this notation in mind, we say the cascade is a \textbf{period doubling cascade of attractors} if for every $n$ and all $\tau\in(\tau_n,\tau_{n+1})$, the corresponding orbit $T_\tau$ for $F_\tau$ on $G_n$ is an attractor. Similarly, we say the cascade is a \textbf{period doubling cascade of repellers} if the same orbits $T_\tau$ are all repellers for $F_\tau$. We now note that whenever $T_\tau$ is either an attracting or repelling orbit for $F_\tau$, there exists some fixed-point free neighborhood $O_\tau\subseteq M$ of $T_\tau$ that varies continuously with $\tau\in(\tau_n,\tau_{n+1})$, s.t. $T_\tau$ is the only periodic orbit in $O_\tau$. In other words, for all $n$ the set $O=\cup_{\tau\in(\tau_n,\tau_{n+1}}O_\tau\times\{\tau\}$ forms an open neighborhood of $G_n$ which isolates $G_n$ from all other components of $Per\setminus G_n$, $Fix\setminus O$ and $Inv\setminus O$. Since this is true for all $n$, by the definition of the set $Per$ in Definition \ref{isolationdef} we conclude that when the period doubling cascade is that of attractors or repellers, we have $\{G_n\}_n\subseteq Per$. We now prove the following universal properties of Entropy flows around period doubling cascades. To keep with the formalism above, we do so only for the case where the parameter $\tau$ varies in $I=(-1,1)$ and the sequence $\{\tau_n\}_n$ is increasing. That being said, this result trivially generalizes to the case when $I$ is either $S^1$, a closed interval, or a half closed interval, and the sequence $\{\tau_n\}_n$ is decreasing. 
\begin{theorem}
\label{major1}
Let $\dot{s}=F_\tau(s)$ be a $C^1$ one--parameter family where $(s,\tau)\in M\times (-1,1)$ and the dimension of $M$ is at least $3$. With previous notations, assume there exists a period doubling cascade of either attractors or repellers, corresponding to components $\{G_n\}_{n\geq0}\subseteq Per$, where for all $n\geq0$, $G_{n}$ is glued to $G_{n+1}$ at the period doubling bifurcation orbit $P_{n+1}\times\{\tau_{n+1}\}$. Then, for all possible Entropy flows, the following holds:

\begin{itemize}
    \item Both $G_n$ and $G_{n+1}$ lie in some invariant manifolds for $P_{n+1}\times\{\tau_{n+1}\}$.
    \item The trajectories of initial conditions on $G_n$ are  attracted to $P_{n+1}\times\{\tau_{n+1}\}$ under the Entropy flow, while those on $G_{n+1}$ are attracted under the Entropy flow to $P_{n+2}\times\{\tau_{n+2}\}$. 
\end{itemize}

Consequently, for any choice of vector field $E(s,\tau)=(F_\tau(s)+V(s,\tau),P(s,\tau))$ giving an Entropy flow, the function $P$ does not vanish identically in $M\times(-1,1)$ - i.e., there is a drift in the parameter space. Moreover, we can choose the Entropy flow s.t. for all $n\geq0$ the orbit $P_{n+1}\times\{\tau_{n+1}\}$ is a Milnor attractor for the flow on $M\times I$ in the sense that it attracts an open set $O_n\subseteq M\times I$ s.t. $P_{n+1}\times\{\tau_{n+1}\}\subseteq\partial O_n$.
\end{theorem}
\begin{proof}
Consider some Entropy flow for the $C^1$ family $\dot{s}=F_\tau(s)$ given by some vector field $E$. We first note that for all $n>1$, $\overline{G_n}\setminus G_n$ is composed of two components - $P_n\times\{\tau_n\}$ and $P_{n+1}\times\{\tau_{n+1}\}$. By Definition \ref{generalentropy}, Definition \ref{admissiblefixed} and Definition \ref{def4} we already know that for all $n$, $P_n\times\{\tau_n\}$ is a periodic orbit for the Entropy flow generated by $E$. We now recall that $P_n\times\{\tau_n\}$ and $P_{n+1}\times\{\tau_{n+1}\}$ are period doubling bifurcation orbits, hence both are multiplying type bifurcations (see the discussion at the very beginning of Subsection \ref{sec2}). Moreover, recall that as discussed above, for all $n$, $G_n\subseteq Per$. Therefore, by Definition \ref{generalentropy} and Definition \ref{def4} we conclude that because the period is doubled as $\tau$ crosses from $(\tau_n,\tau_{n+1})$ into $(\tau_{n+1},\tau_{n+2})$, the trajectories of initial conditions on $G_n$ sufficiently close to $P_{n+1}\times\{\tau_{n+1}\}$ must be attracted by the Entropy flow directed by $E$ towards $P_{n+1}\times\{\tau_{n+1}\}$. Symmetrically, the same also implies that initial conditions on $G_n$ sufficiently close to $P_n\times\{\tau_n\}$ are repelled away from $P_n\times\{\tau_n\}$ w.r.t. the same flow.

This implies that if we write $E(s,\tau)=(F_\tau(s)+V(s,\tau),P(s,\tau))$, then the function $P$ is positive on all initial conditions on $G_n$ sufficiently close to both $P_n\times\{\tau_n\}$ and $P_{n+1}\times\{\tau_{n+1}\}$ - as such, by Definition \ref{generalentropy} we know $P$ is positive on all $(s,\tau)\in G_n$, hence $E$ defines a drift in the parameter space. As $G_n$ is invariant under the Entropy flow (as all Entropy flows are tangent to $Per$) it follows the Entropy flow generated by $E$ pushes all initial conditions on $G_n$ in forward time to $P_{n+1}\times\{\tau_{n+1}\}$, and in backwards time to $P_n\times\{\tau_n\}$ (see the illustration in Figure \ref{isot12}). All in all, we have proved the trajectory of every $(s,\tau)\in G_n$ under the Entropy flow tends to $P_{n+1}\times\{\tau_{n+1}\}$ in forward time, and to $P_n\times\{\tau_n\}$ in backwards time. By the characterization of center Manifolds in \cite{centerp} we conclude that $G_n$ is an invariant manifold w.r.t. the Entropy flow directed by $E$ for both periodic orbits $P_{n+1}\times\{\tau_{n+1}\}$ and $P_n\times\{\tau_n\}$. The same arguments applied to $G_{n+1}$, $P_{n+1}\times\{\tau_{n+1}\}$ and $P_{n+2}\times\{\tau_{n+2}\}$ implies $G_{n+1}$ also lies on some invariant manifold for $P_{n+1}\times\{\tau_{n+1}\}$ w.r.t. the vector field $E$. As $E$ is an arbitrary vector field generating an Entropy flow, the first part of the Theorem is proven.

\begin{figure}[h]
\centering
\begin{overpic}[width=0.35\textwidth]{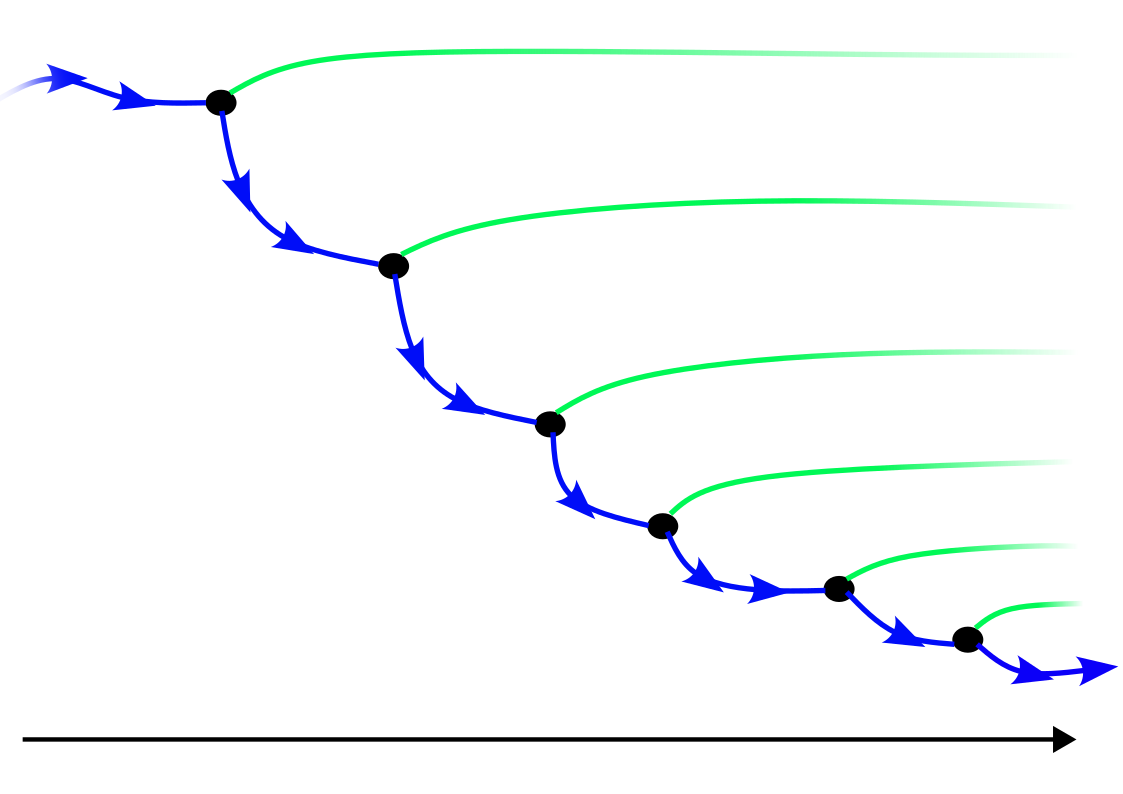}
\put(980,30){$\tau$}
\end{overpic}
\caption{\textit{A diagram showing the general behavior of an Entropy flow, sketched on the bifurcation diagram of a period doubling cascade. The arrows represent the motion along the component of orbits $G_n$ from $P_n\times\{\tau_n\}$ to $P_{n+1}\times\{\tau_{n+1}\}$, where each black dot represents $P_n\times\{\tau_n\}$, $n\geq1$.}}
\label{isot12}
\end{figure}

We now prove that we can choose the Entropy flow s.t. every $G_n$, $n\geq1$ defines an invariant set which attracts or repels an open set in $M\times I$ (in the Milnor sense, see \cite{Mil3}) - depending on whether the cascade is that of attractors or repellers. We prove this under the assumption the one--parameter family $\dot{s}=F_\tau(s)$ generates a cascade of attractors - the proof for a cascade of repellers is symmetric. To this end, given $\tau\in(\tau_n,\tau_{n+1})$, let us surround each periodic orbit $T_\tau$ s.t. $T_\tau\times\{\tau\}\subseteq G_n$ by a tubular neighborhood $\mathbb{T}_\tau\subseteq M$ homeomorphic to $S_1\times D$, where $D$ is a $d-1$-dimensional ball. We also assume $\mathbb{T}_\tau$ to satisfies the following:
\begin{itemize}
    \item $F_\tau$ is transverse to $\partial\mathbb{T}_\tau$ and points inside $\mathbb{T}_\tau$ throughout $\partial\mathbb{T}_\tau$.
    \item Every initial condition in $\mathbb{T}_\tau$ is attracted to $T_\tau$ w.r.t. $F_\tau$.
    \item $\mathbb{T}_\tau$ varies continuously with $\tau$, and the volume of $\mathbb{T}_\tau$ shrinks to $0$ whenever $\tau\to\tau_n$ or $\tau\to\tau_{n+1}$.

\end{itemize}

Note that as we assume $T_\tau$ is an attractor for $F_\tau$ all $\tau\in(\tau_n,\tau_{n+1})$, such neighborhoods exist. Note that $\mathbb{T}=\cup_{\tau\in(\tau_{n},\tau_{n+1})}\mathbb{T}_\tau\times\{\tau\}$ forms an open neighborhood of $G_n$ in $M\times I$. Per our construction, it is easy to see $\overline{Fix}\cap\overline{Per}$ lie away from $\mathbb{T}_\tau\times\{\tau\}$ for all $\tau\in(\tau_{n},\tau_{n+1})$, which implies we can choose an Entropy flow given by a vector field $E$ s.t. for all $(s,\tau)\in\partial\mathbb{T}_\tau\times\{\tau\}$ we have $E(s,\tau)=(F_\tau(s),0)$. By its construction we know $\partial \mathbb{T}$ is given by three sets: the first is $\mathcal{T}=\cup_{\tau\in(\tau_{n},\tau_{n+1})}\partial \mathbb{T}_\tau\times\{\tau\}$, while the other two are $P_{n}\times\{\tau_{n}\}$ and $P_{n+1}\times\{\tau_{n+1}\}$. The same argument used to prove Lemma \ref{trapping} now implies that $E$ is transverse to $\mathcal{T}$ and points inside $\mathbb{T}$ on it. As both $P_{n}\times\{\tau_{n}\}$ and $P_{n+1}\times\{\tau_{n+1}\}$ are periodic orbits for $E$, it is tangent to them. This proves every initial condition in $M\times I$ whose trajectory enters $\mathbb{T}$ under the Entropy flow directed by $E$ cannot escape it. As such, there exists some invariant set for the Entropy flo in $\overline{\mathbb{T}}$ - which, by definition, includes the set $G_n\cup P_n\times\{\tau_n\}\cup P_{n+1}\times\{\tau_{n+1}\}$. Moreover, this set attracts an open set $O_n$ in $M\times I$, i.e., all the initial conditions sufficiently close to $\mathcal{T}$.

We finish the proof by showing that we can choose the Entropy flow above s.t. every initial condition in $\mathbb{T}$ is attracted to $P_{n+1}\times\{\tau_{n+1}\}$. We already know this is true for initial conditions on $G_n$, so it remains to prove this for initial conditions in $\mathbb{T}\setminus G_n$. We now note that whenever $P$ is positive throughout the interior of $\mathbb{T}$, the $\dot{\tau}$ velocity is also positive. As such, given an Entropy flow $E$ as above which is transverse to $\mathcal{T}$ for which $P$ is some smooth function positive throughout $\mathbb{T}$ and tangent to $G_n$, the trajectory of $(s,\tau)\in\mathbb{T}$ would always increase in its $\tau$ coordinate. As $E$ is transverse to $\mathcal{T}$, the trajectories of such initial conditionds w.r.t. $E$ cannot escape the interior of $\mathbb{T}$. Therefore, by $\mathbb T\subseteq M\times(\tau_n,\tau_{n+1})$ we conclude the vector field $E$ forces these trajectories to accumulate on $\partial \mathbb T\cap M\times\{\tau_{n+1}\}$. Since $\partial \mathbb T\cap M\times\{\tau_{n+1}\}=P_{n+1}\times\{\tau_{n+1}\}$ we conclude $P_{n+1}\times\{\tau_{n+1}\}$ attracts all initial conditions interior to $\mathbb{T}$. As $\mathbb{T}$ is an open set in $M\times I$ which includes $P_{n+1}\times\{\tau_{n+1}\}$ on its boundary it follows $P_{n+1}\times\{\tau_{n+1}\}$ is a Milnor attractor, which concludes the proof.
\end{proof}
\begin{remark}
    It is well-known that period doubling cascades often appear together with universal scaling properties (see, for example, \cite{fei}, \cite{fei2}, \cite{Lyu}, and the references therein). Even though similar results are known for some two parameter families of flows (see \cite{GKP}), we avoid this question for the time being - we will return to it at the end of Section \ref{turbulence}, where we will discuss it in the context of the Shilnikov bifurcation scenario (see \cite{LeS}).
\end{remark}
Before we move on, we would like to note that one could ask what is the contribution of the families of periodic orbits $\{D_n\}_n$ defined above to the Entropy flow around period doubling cascades. Recall that if $G_{n+1}$ corresponds to the family of periodic orbits whose period is doubled as the parameter $\tau$ crosses from $(\tau_n,\tau_{n+1})$ to $(\tau_{n+1},\tau_{n+2})$, then $D_n$ corresponds to a family of periodic orbits in $M\times(\tau_{n+1},1)$ s.t. when we vary the periodic orbits from $G_n$ to $D_n$ via $P_{n+1}\times\{\tau_{n+1}\}$ the period changes continuously. Unfortunately, we do not have a good answer to this question. The reason for that is because $D_n$ may well correspond to a family of saddle periodic orbits, and such families need not correspond to components of $Per$. To the best of our knowledge, there is no general way to tell when $D_n\subseteq Per$ for some $n$ and when it does not. That being said, in Appendix \ref{periodappendix} we study the Conley index of the Entropy flow around period doubling cascades under the idealized assumption that $\{D_n\}_n\subseteq Per$.\\

We would now like to illustrate these ideas via a concrete example. To do so, recall that given any $a,b,c>0$, the Rössler system, originally introduced in \cite{Ross76}, is defined by the following three-dimensional vector field:
\begin{equation}
\label{rossler}
\begin{cases}
\dot{x} = -y-z \\
 \dot{y} = x+ay\\
 \dot{z}=b+z(x-c)
\end{cases}
\end{equation}
As shown numerically in \cite{Ross76} and later proven in both \cite{Zgli97} and \cite{XSYS03}, for $(a,b,c)=(0.2,0.2,5.7)$ the flow generates a chaotic attractor. As observed in many numerical studies, the Rössler system undergoes the period doubling route to chaos of attractors (see, for example, \cite{DW}, \cite{BBS}, \cite{MBKPS}, \cite{BSD} and the references therein). Based on these numerical studies, if one takes a smooth curve through the parameter space $\{(a,b,c)|a,b,c>0\}$ transverse to these period doubling cascades, every Entropy flow for the said curve would have complex behavior. Later in Section \ref{turbulence} we will return to this example as we give a global description for the dynamics of the Entropy flow associated with the Shilnikov homoclinic bifurcation scenario.

\subsection{The Ruelle-Takens-Newhouse route to chaos}

In this Subsection, we discuss the Entropy flow of the Ruelle-Takens-Newhouse route to chaos, originally introduced in \cite{RT} and \cite{SRT} (see also Chapter 6.3 of \cite{Gidea} for a survey of the topic). To introduce the Ruelle-Takens-Newhouse route to chaos, let $\dot{s}=F_\tau(s)$ be a $C^1$ one--parameter family of vector fields defined on some $n$-dimensional manifold $M$, where $n\geq3$. Without any loss of generality, throughout this Subsection we will assume the parameter $\tau$ varies smoothly in $\mathbb{R}$. We will also assume that for all $\tau\in\mathbb{R}$, $F_\tau(0)=0$, and that for $\tau\in(-\infty,0)$ the origin $0$ is a a sink. The Ruelle-Takens-Newhouse route to chaos is defined as follows (see the illustration in Figure \ref{RT}):
\begin{itemize}
    \item At $\tau=0$ the origin undergoes Hopf bifurcation and becomes an attracting periodic orbit, $T$, that persists without bifurcating until $\tau=1$. For simplicity, assume the origin persists as a source for $\tau>0$.
    \item At $\tau=1$, $T$ undergoes Neimark-Sacker bifurcation, which turns it into a two-dimensional attracting torus $\mathbb{T}^2$ on which the motion is a-periodic. The said attracting torus persists for all $\tau\in(1,2)$.
    \item As $\tau$ crosses into $(2,\infty)$ the said two-dimensional torus $\mathbb{T}^2$ eventually disintegrates at some $\tau_0\in(2,\infty)$ and becomes some bounded complex motion. \\
\end{itemize}

\begin{figure}[h]
\centering
\begin{overpic}[width=0.4\textwidth]{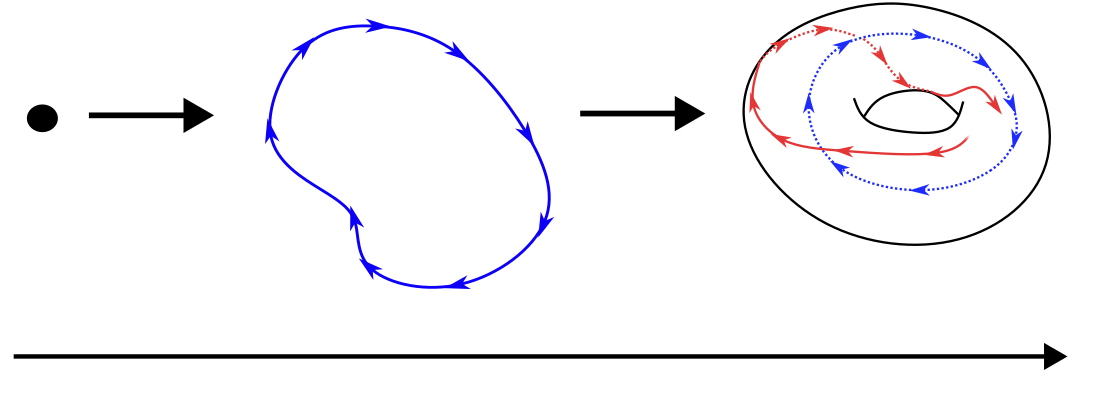}
\put(980,30){$\tau$}
\end{overpic}
\caption{\textit{The Ruelle-Takens-Newhouse scenario - a sink which bifurcates in  Hopf bifurcation into a stable orbit, which then undergoes Neimark-Sacker bifurcation and becomes an attracting torus on which the motion is aperiodic.}}
\label{RT}
\end{figure}

We now analyze the Entropy flow for this route to chaos. Let $E$ be some Entropy flow, let $\gamma_1=\{(0,\tau)|\tau<0\}$ and $\gamma_2=\{(0,\tau)|\tau\in(0,\infty)\}$ - as $0$ is a sink for $\tau<0$ and a source for $\tau>0$, it is easy to see that both $\gamma_1$ and $\gamma_2$ are components of $Fix$. Similarly, let $\mathcal{S}_1$ be the surface corresponding to the family of attracting periodic orbits created at Hopf bifurcation at $\tau=0$. As those periodic orbits undergo Neimark-Sacker bifurcation at $\tau=1$, it follows $\mathcal{S}_1\subseteq M\times(0,1)$. Moreover, as the periodic orbits $T_\tau$ given by $T_\tau\times\{\tau\}=\mathcal{S}_1\cap M\times\{\tau\}$, $\tau\in(0,1)$ are all attractors, $\mathcal{S}_1$ is a component of $Per$. The curves $\gamma_1,\gamma_2$ and the surface $\mathcal{S}_1$ are glued to one another at $(0,0)$, which is an Hopf bifurcation point (see the illustration in Figure \ref{RT2}). Let $\mathcal{T}$ denote the subset of $M\times (1,2)$ corresponding to the attracting invariant torus $\mathbb{T}^2$ - it is a three manifold isolated from $Per$ and $Fix$, as it corresponds to a family of attracting tori, which form completely isolated invariant sets. Therefore, from Definition \ref{def4}, Definition \ref{def5}, Definition \ref{def6}, and Definition \ref{generalentropy} we immediately conclude the following:
\begin{corollary}
    \label{ruelletakens} Consider a $C^1$ one--parameter family $\dot{s}=F_\tau(s)$ where $s\in M$, $\tau\in\mathbb{R}$, $n\geq3$ undergoing a Ruelle-Takens-Newhouse route to chaos. Then, any Entropy flow for the family above has the following behavior (see the illustration in Figure \ref{RT2}):
    \begin{itemize}
        \item Given any initial condition $(s,\tau)\in\gamma_i$, $i=1,2$, its trajectory w.r.t. the Entropy flow tends to the Hopf bifurcation point $(0,0)$.
        \item For all initial conditions $(s,\tau)\in\mathcal{S}_1$, their trajectory w.r.t. the Entropy flow is repelled away from $(0,0)$ and towards the Neimark-Sacker bifurcation orbit on $\partial \mathcal{S}_1$.
        \item On initial conditions $(s,\tau)\in \mathcal{T}$ for which $\tau$ is sufficiently close to $1$, the function $P$ is positive - i.e., on $\mathcal{T}$ the Entropy flow pushes away from Neimark-Sacker bifurcation.
    \end{itemize}
\end{corollary}

\begin{figure}[h]
\centering
\begin{overpic}[width=0.4\textwidth]{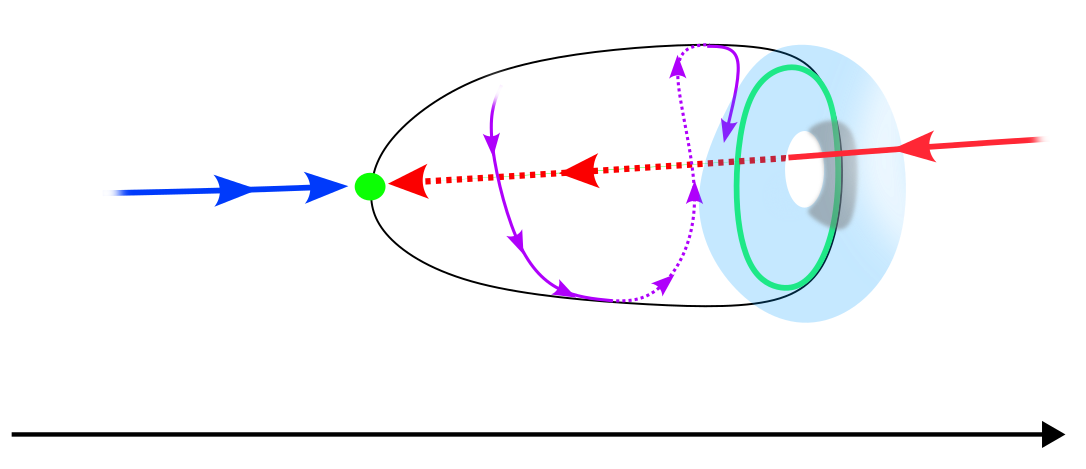}
\put(990,20){$\tau$}
\end{overpic}
\caption{\textit{A (partial) scheme describing the behavior of the Entropy flow on the Ruelle-Takens-Newhouse route to chaos, where the green dot denotes Hopf bifurcation $(0,0)$, the blue arc (and flowline) denotes $\gamma_1$ while the red one denotes $\gamma_2$. The purple flow line is tangent to $\mathcal{S}_1$ and tends to $\mathcal{T}$, which is glued  to $\mathcal{S}_1$ at  Neimark-Sacker bifurcation orbit.}}
\label{RT2}
\end{figure}

We now briefly discuss the dynamical meaning of Corollary \ref{ruelletakens}. We firstly recall the basic motivation behind the Ruelle-Takens-Newhouse route to chaos. As originally envisioned in \cite{LL}, one would expect that in an infinite dimensional dynamical system $\mathbb{T}^2$ would bifurcate again into an aperiodic motion on a $3$-torus $\mathbb{T}^3$, which would then bifurcate into an aperiodic motion on a $4$-torus, $\mathbb{T}^{4}$ and so on. In practice, as proven in \cite{RT}, in a finite dimensional manifold an aperiodic motion on $\mathbb{T}^k$, $k\geq3$ is extremely unstable and can be $C^k$-approximated by flows that are not Morse-Smale - and when $k\geq4$, by flows that are Axiom A (see Propositions 9.1, 9.2 in \cite{RT}, and \cite{SRT}). In other words, one would not expect a cascade of Hopf bifurcations to progress much beyond the Neimark-Sacker bifurcation. More precisely, one would expect a collapse of the attracting two-dimensional torus into a complex motion either before or shortly after the bifurcation expanding $\mathbb{T}^2$ to $\mathbb{T}^3$.\\

In this context, Corollary \ref{ruelletakens} should be interpreted as follows - the motion of the Entropy flow on $\mathcal{S}_1$ propels initial conditions near it towards Neimark-Sacker bifurcation, i.e., the place where $\mathbb{T}^2$ is formed. Similarly, on $\mathcal{T}$, the three manifold corresponding to the invariant torus, the Entropy flow also pushes away from Neimark-Sacker bifurcation, as illustrated in Figures \ref{RT} and \ref{RT2}. In light of the above, we may summarize our analysis by the heuristic that the Entropy flow pushes the Ruelle-Takens-Newhouse scenario towards the complex motion. At this point we remark one could probably continue this analysis and ask how exactly is the torus destroyed. In general, the mechanisms for the destruction of an invariant torus are given by the Afraimovich-Shilnikov Theorem (see either \cite{AS}, Chapter 11.7 of \cite{SSTC}, Chapter 2.1.6 of \cite{AAIS}). To the best of our knowledge, this result is not directly related to the Ruelle-Takens-Newhouse route to chaos - hence, as observed numerically in \cite{MuChu}, the Ruelle-Takens-Newhouse route to chaos could possibly be connected with the Afraimovich-Shilnikov Theorem. 

\subsection{The Intermittency route to chaos} 
\label{intermittency}

We now study behavior around another bifurcation phenomenon - the Intermittency route to chaos, originally discovered in \cite{PP}. To begin, we first recall the ideas of \cite{PP} and the description and the classification of Intermittency types, following \cite{TI} and Chapter 6.2 in \cite{Gidea}. To this end, let $\dot{s}=F_\tau(s)$ denote a $C^1$ one--parameter family of vector fields on a manifold $M$ of dimension $n\geq3$. Without any loss of generality, in this Subsection we will always implicitly assume the parameter $\tau\in\mathbb{R}$ - as our analysis is local in nature, it would apply in other scenarios as well. We say there is an \textbf{Intermittency route to chaos} if the following conditions are satisfied (see the illustration in Figure \ref{inter}):
    \begin{itemize}
        \item There exists a periodic orbit $T_0$ for $F_0$ bifurcating at $\tau=0$ in one of the three following ways (which we take as definitions for the Intermittency type):
    \begin{enumerate}
         \item \textbf{Type I Intermittency} - $T_0$ is a saddle node bifurcation, i.e., when $\tau$ crosses into $(0,\infty)$, the orbit $T_0$ splits into two periodic orbits $T_1$ and $T_2$ (generically, an attractor and a saddle). Similarly, when $\tau$ crosses into $(-\infty,0)$, $T_0$ disappears.
        \item \textbf{Type II} - $T_0$ is a Neimark-Sacker bifurcation orbit, s.t. an invariant Torus is created when $\tau$ crosses into $(0,\infty)$. $T_0$ persists as $\tau$ crosses into $(-\infty,0)$.
        \item \textbf{Type III} - $T_0$ is a "reverse" period doubling bifurcation. In other words, there exists a periodic orbit $T_\tau$ for $F_\tau$ which persists for $\tau\in(0,\epsilon)$ (for some $\epsilon>0$) and undergoes a period halving bifurcation as $\tau$ crosses into $(-\infty,0)$.
    \end{enumerate}
        
        \item Let $N$ be any small neighborhood of $T_0$ in $M$. In either one of the three cases above, when one considers the dynamics of $F_\tau$, $\tau\in(-\epsilon,0)$ in $N$ (for some small $\epsilon>0$), there are no periodic orbit in $N$, the trajectory of any initial condition $s\in N$ w.r.t. $F_\tau$ would forever alternate between the following phases:
        
        \begin{enumerate}
            \item \textbf{Laminar phases}, where the motion of $s$ w.r.t. $F_\tau$ is trapped in $N_\tau$ until leaving it and entering: 
            \item \textbf{Bursts of a-periodicity}, where the motion of the trajectory of $s$ w.r.t. $F_\tau$ moves aperiodically and erratically in $M\setminus N$ only to eventually return to $N$ and again reenter a laminar phase.\\
\end{enumerate}
    \end{itemize}
\begin{remark}
    Similarly to the period doubling route to chaos, scaling laws also appear around the Intermittency route to chaos. Such results were studied both analytically (see \cite{Kuz},  \cite{OU}) and numerically (see \cite{FKI}, \cite{PGOY}). 
\end{remark}

\begin{figure}[h]
\centering
\begin{overpic}[width=0.4\textwidth]{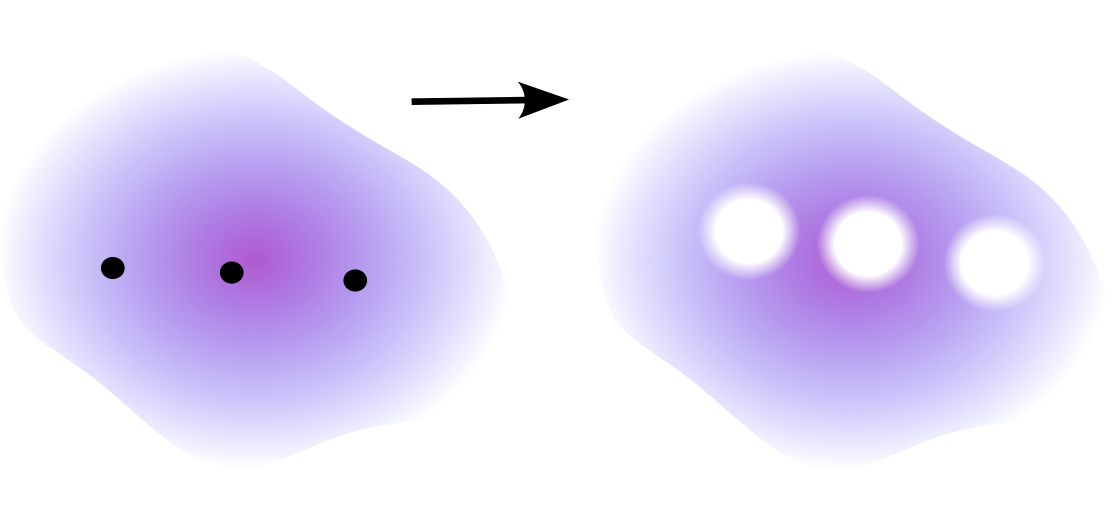}
\put(200,-30){$\tau>0$}
\put(700,-30){$\tau<0$}
\end{overpic}
\caption{\textit{An illustration representing Type $I$ intermittency - a periodic orbit (on the left) vanishes and leaves behind it a "hole" (on the right). The dynamics of initial conditions inside the "hole" are relatively ordered when they remain in the hole (the laminar phase) only to become erratic and disordered as they enter the purple region (a burst of aperiodicity).}}
\label{inter}
\end{figure}

Note the description above is not a mathematical definition per se, but rather a description of a numerical phenomenon - as such, there are several explanations to why disorder suddenly appears at saddle node bifurcations. Inspired by the results of \cite{Shil1}, \cite{den2}, \cite{kryz} and \cite{gool}, we will take the topological approach. In detail, we will study the Entropy flow at the bifurcation scenario described in \cite{Shil1} and \cite{den2} which, as we will discuss, can be used to explain how some Type I Intermittencies arise. Before moving on, we remark that even though we will not address Type II and Type III Intermittencies, provided they are generated by similar mechanisms an analogous results could potentially exist for them as well. We refer to Section 5 in \cite{den2} for the complete details, where the generalization of similar mechanisms to pitchfork and period doubling bifurcations are considered.  \\
\begin{figure}[h]
\centering
\begin{overpic}[width=0.6\textwidth]{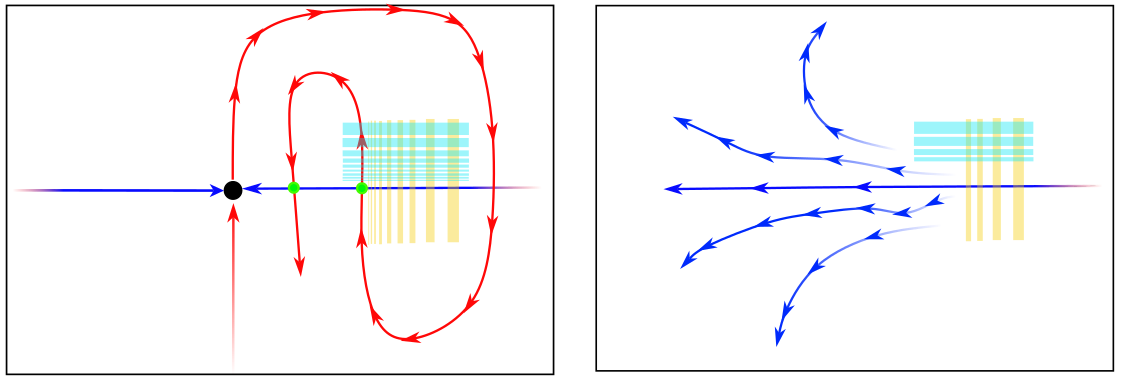}
\put(200,-20){$\tau=0$}
\put(700,-20){$\tau<0$}
\end{overpic}
\caption{\textit{On the left - a homoclinic intersection between $W^c$ and $W^s$ for the vector field $F_0$, represented on the cross-section $S_0$, where the dot represents $\{s_0\}=T_0\cap S_0$ (by Theorem 1.7 in \cite{den2}, Theorem 1 in \cite{Shil1} and \cite{gool}, this implies the existence of shift dynamics). On the right we see the remains of the invariant set after the periodic orbit $T_0$ disappears at $\tau<0$.}}
\label{homoclinic}
\end{figure}

We begin by recalling the results of \cite{Shil1}. Let $\dot{s}=F_\tau(s)$ be a $C^k$ one--parameter family of vector fields on some smooth manifold $M$, where $\tau$ varies in $\mathbb{R}$, $k\geq2$, and the dimension of $M$ is at least $3$. Let us further assume there exists a saddle node bifurcation orbit $T_0$ for the vector field $F_0$ s.t. the following occurs:
\begin{itemize}
    \item As $\tau$ crosses into $(0,\infty)$ the orbit $T_0$ splits into an attracting orbit $T_A$ and a saddle orbit $T_S$ for $F_\tau$. Conversely, $T_0$ vanishes as $\tau$ crosses into $(-\infty,0)$.
    \item $T_0$ has two invariant manifolds: a stable manifold $W^s$ and a center manifold $W^c$, which intersect transversely (see the illustration in Figure \ref{homoclinic}).
\end{itemize}

As proven by L. Shilnikov and V. Lukjanov in Theorem $1$ and Theorem $2$ of \cite{Shil1}, under these assumptions there exists some fixed neighborhood $U$ of $T_0$ in $M$ s.t. the dynamics of $F_\tau$, $\tau\in(-\infty,0)$ satisfy the following:
    \begin{itemize}
        \item There is $\epsilon>0$ s.t. for $\tau\in(-\epsilon,0]$ the invariant set of $F_\tau$ in $U$, denoted by $\Omega_\tau$, is topologically conjugate to a suspended subshift of finite type.
        \item There exist two increasing sequence of negative parameters $\{\tau_i\}_{i=0}^\infty$, $\{\tau'_i\}_{i=0}^\infty$ s.t. the following holds:
        \begin{enumerate}
            \item   $\tau_i\uparrow0$.
            \item $\tau'_i\uparrow0$, $\tau_{i}>\tau'_i>\tau_{i+1}$.
            \item The set $\Omega_\tau$ persists without bifurcating when $\tau$ varies in $[\tau_i,\tau'_{i}]$. Moreover, for such $\tau\in[\tau_i,\tau'_i]$ the set $\Omega_\tau$ is conjugate to some suspended subshift on $k_i$ symbols, denoted by $\Sigma_i$.
            \item For all $\tau\in(\tau'_i,\tau_{i+1}]$, the dynamics of $\Omega_\tau$ includes some invariant set, a strict subset of $\Omega_\tau$ conjugate to the suspension of $\Sigma_i$ - i.e., the dynamics of $\Omega_\tau$ are conjugate to the suspension of the subshift $\Sigma_i$ precisely when $\tau\in[\tau_i,\tau'_i]$.
            \item We have $\Sigma_i\subseteq\Sigma_{i+1}$, i.e., the complexity of $\Omega_\tau$ only increases as $\tau\to0$. In particular, the dynamics of $F_0$ in $\Omega_0$ are conjugate to some suspended subshift denoted by $\Sigma_\infty$ s.t. for all $i>0$ we have $\Sigma_i\subseteq\Sigma_\infty$. Moreover, the set $\Omega_0$ persists as $\tau$ is varied in $[0,\epsilon]$ (for some $\epsilon>0$).\\
        \end{enumerate}
            \end{itemize}

  From now on we will refer to such saddle node bifurcation scenarios as \textbf{homoclinic saddle nodes}. It is easy to see why Intermittency-like phenomena exist around homoclinic saddle node bifurcations: namely, as $T_A$ and its basin of attraction are destroyed at the saddle node bifurcation they are replaced with a suspended subshift of finite type. As proven in \cite{den2}, under certain genericity conditions the set $\Omega_\tau$, $\tau<0$, will be composed of a finite number of rectangles stretched on one another as illustrated in Figure \ref{homoclinic} (for a proof, see Theorem 1.7 in \cite{den2}) - in contrast, the set $\Omega_0$ would be generated by infinitely many rectangles stretched on one another and accumulating on the homoclinic intersection (and hence, also on $T_0$). In other words, as $\tau$ crosses into $(-\infty,0)$ only finitely many of these rectangles remain, which implies initial conditions $s\in U$ must have intermittent behavior as they flow close to $\Omega_\tau$ and then away from it, where $T_0$ used to be, i.e., the burst of aperiodicity and the laminar phase, respectively. With that image in mind, we now study how this behavior is seen by the Entropy flow:
\begin{theorem}
      \label{heteroclinicnets}
      Let $\dot{s}=F_\tau(s)$, $\tau\in(-\infty,\infty)$, be a $C^k$ one--parameter family where $(s,\tau)\in M\times\mathbb R$ and $k\geq2$. Assume there exists a homoclinic saddle node orbit $T_0$ for the vector field $F_0$. Then, given any neighborhood $N$ of $T_0\times\{0\}$ in $M\times\mathbb{R}$ and given any Entropy flow defined by the vector field $E(s,\tau)=(F_\tau(s)+V(s,\tau),P(s,\tau))$, the following holds:
      
      \begin{itemize}
          \item There exists a countable collection of open sets in $M\times\mathbb{R}$, $\{O_i\}_i$, accumulating on $T_0\times\{0\}$, s.t. for all $i$ the drift function $P$ takes both positive and negative values on $O_i$. In particular, the drift function $P$ is non-zero at every neighborhood of $T_0\times \{0\}$ in $M\times I$.
          \item $T_0\times\{0\}$ is the accumulation set of countably many mutually disjoint invariant sets for the Entropy flow, $\{\mathcal{I}_i\}_i$, such that on each $\mathcal{I}_i$ the dynamics of the Entropy flow are conjugate to a suspended subshift. Moreover, for all $i$ $\mathcal{I}_i\subseteq O_i$, and the closer $\mathcal{I}_i$ is $T_0\times\{0\}$ the more complicated are the dynamics on it.
      \end{itemize}
      \end{theorem}
\begin{proof}
  Consider the sets $\Omega_\tau$ described above, where $\tau\in[\tau_i,\tau'_i]$. These sets persist without bifurcating precisely in the parameter interval $[\tau_i,\tau'_i]$, and, moreover, at such parameters they form the invariant set of $F_\tau$ in $U$. We now claim the sets $\Omega_\tau\times\{\tau\}$,  $\tau\in(\tau_i,\tau'_i)$ are completely isolated from the sets $Per$ and $Fix$. As they do not include fixed points, the sets $\Omega_\tau$ cannot be approximated by fixed points in $Fix$. To prove the analogue for $Per$, we first recall that Theorem 2 in \cite{Shil1} was originally proven for some first return map w.r.t. $f_\tau$ of a cross-section $S\subseteq U$, where the invariant set was created by rectangles in $S$ being stretched over one another as in the illustration in Figure \ref{homoclinic}. By possibly replacing these rectangles by slightly smaller subrectangles inside the original ones (if necessary) we can ensure the first return map $f_\tau:S\to S$ has no periodic orbits accumulating on the boundary of the rectangles, all without changing the invariant set (see the illustration in Figure \ref{shrinking}). As $\Omega_\tau$ is trapped inside the suspension of these rectangles, all the periodic orbits intersecting $S$ for $F_\tau$ are in $\Omega_\tau$, hence for $\tau\in(\tau_i,\tau'_i)$ the set $\Omega_\tau\times\{\tau\}$ is isolated from $Per$ - all in all, it is isolated from both $Per$ and $Fix$ in $M\times I$. It  follows that for $i\geq1$ the set $\Omega_i=\cup_{\tau\in(\tau_i,\tau'_i)}\Omega_\tau\times\{\tau\}$ is a continuum of completely isolated invariant sets for $\dot{s}=F_\tau(s)$ (see the discussion before Definition \ref{def6}). 

\begin{figure}[h]
\centering
\begin{overpic}[width=0.5\textwidth]{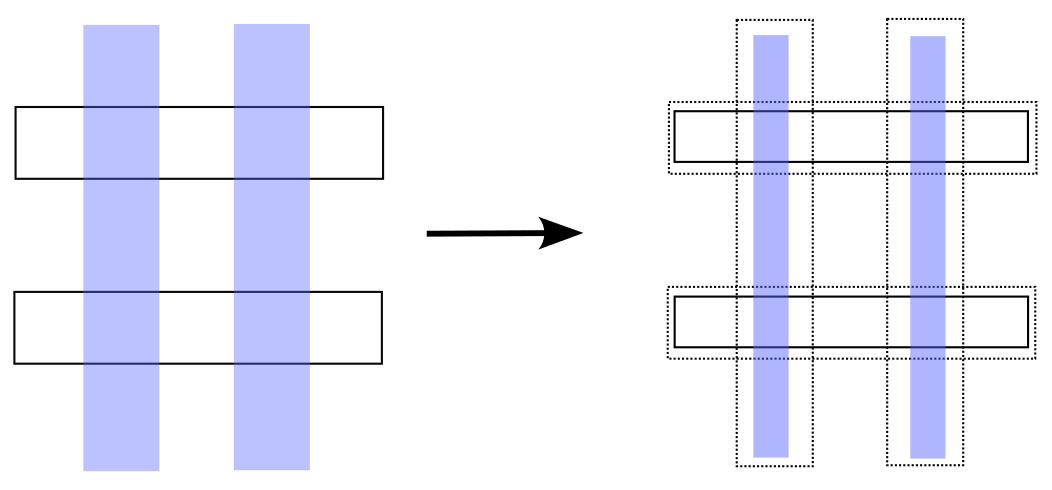}

\end{overpic}
\caption{\textit{Removing periodic orbits from the boundary - on the left we have a map of two rectangles stretched over one another, that has certain symbolic dynamics. On the right we have the same map restricted to subrectangles (the dashed lines on the right represent the original rectangles and their image). It is easy to see, the two maps have the same invariant set inside the rectangles, and if the left one has periodic orbits accumulating on the boundary of its rectangles, the one on the right does not.}}
\label{shrinking}
\end{figure}

To continue, choose some Entropy flow for $\dot{s}=F_\tau(s)$ in $M\times\mathbb{R}$ directed by the vector field $(s,\tau)\to(F_\tau(s)+V(s,\tau),P(s,\tau))$. By the definition of the Entropy flow (see Definition \ref{generalentropy}) and the definition of admissible behavior around completely isolated invariant sets (see Definition \ref{def6}) we know there exist open sets in $M\times I$, $O_i$, s.t. $\Omega_i\subseteq O_i$ and the drift function $P$ does not vanish identically in $O_i$ - conversely, by Definition \ref{def6} we also know the energy transition map, $V$, does vanish identically in $O_i$. Moreover, as $(\tau_i,\tau'_i)$ is a bounded interval on the real line we know that for $(s,\tau)\in O_i$ close to $M\times\{\tau'_i\}$ the function $P$ is negative, while on $(s,\tau)\in O_i$ close to $M\times\{\tau_i\}$ it is positive. By the definition of the Entropy flow we now conclude there exists some $\tau\in(\tau_i,\tau'_i)$ s.t. $P^{-1}(0)\cap O_i$ is non-empty, and forms an $n$-dimensional set inside the $n+1$-dimensional set $O_i$. 

As the sets $\Omega_i$, $i\geq0$ accumulate on $T_0\times\{0\}$ so do their neighborhoods $O_i$. Moreover, since the vanishing set of $P$ in $O_i$ is $n$-dimensional, it follows that on every neighborhood of $T_0\times\{0\}$ in $M\times\mathbb{R}$ the drift function $P$ is not identically $0$. This proves the first assertion of the Theorem - to conclude the proof it remains to show $T_0\times\{0\}$ is the accumulation set of countably many invariant sets $\mathcal{I}_i$ for the Entropy flow, on which its dynamics are conjugate to a suspended subshift. That, however, is immediate: consider $\mathcal{I}_i=\Omega_\tau\times\{\tau\}=P^{-1}(0)\cap\Omega_i$, for some appropriate $\tau\in(\tau_i,\tau'_i)$ (by the above it exists). Since $V$ vanishes identically in $O_i$ and because $P$ vanishes on $M\times\{\tau\}\cap O_i$, the set $\mathcal{I}_i$ is an invariant set for the Entropy flow, on which its dynamics coincide with the Laminar flow given by Equations \ref{laminar}. As such, since the dynamics of $F_\tau$ on $\Omega_\tau$ are conjugate to a suspended subshift $\Sigma_i$ w.r.t. to the flow of $F_\tau$ on $M$,  the dynamics of the Entropy flow on $\mathcal{I}_i$ are also conjugate to the suspension of $\Sigma_i$. Consequently, because the sets $\Omega_i$ accumulate on $T_0\times\{0\}$ (by their definitions) we conclude that the invariant sets $\{\mathcal{I}_i\}_i$ for the Entropy flow also accumulate on $T_0\times\{0\}$. 

Finally, as described above, the flow of $F_\tau$ on $\Omega_\tau$, $\tau\in[\tau_i,\tau'_i]$ is conjugate to the suspension of the subshift $\Sigma_i$ and  $\Sigma_i$ is a strict subset of $\Sigma_{i+1}$. This yields that the dynamics on $\Omega_\tau$ is conjugate to some strict subset of $\Omega_{\tau'}$, $\tau'\in[\tau_{i+1},\tau'_{i+1}]$. Extending this reasoning to the Entropy flow, we conclude that the dynamics of the Entropy flow $\mathcal{I}_i$ is conjugate to its dynamics on some strict subset of $\mathcal{I}_{i+1}$, which implies that the dynamical complexity of $\mathcal{I}_i$ increases as $i\to\infty$. In other words, $T_0\times\{0\}$ is the accumulation point of invariant sets for the Entropy flow of increasing dynamical complexity. As we have chosen the Entropy flow for $\dot{s}=F_\tau(s)$ arbitrarily, the Theorem now follows.
\end{proof}
Before we conclude this Subsection, we remark that Theorem \ref{heteroclinicnets} also establishes a hierarchy of complexity for the invariant sets $\{\mathcal{I}_i\}_i$ based on how close they are to $T_0\times\{0\}$, i.e., to the saddle node bifurcation orbit. We will return to the idea of hierarchies of complexity of invariant sets for the Entropy flow in Section \ref{turbulence}, where we will see it arising naturally in the Shilnikov homoclinic scenario (see Corollary \ref{turbulententropy}). Heuristically, it is easy to see why such hierarchies arise, at least in the context of homoclinic saddle nodes: namely, because certain dynamics become persistent in some parameter sub-interval (i.e., $[\tau_i,\tau'_i]$), until bifurcating and becoming more complex when $\tau$ leaves the said interval. We further remark that by the results of \cite{PY2} and Theorem 3 in \cite{Shil1}, one should generically expect the periodic orbits added as $\tau$ is varied in $(\tau'_i,\tau_{i+1})$ to arise via period doubling and saddle node bifurcations. As such, based on the above, we would expect $P$ to take non-zero values also in $\mathcal{U}=U\times(\tau'_i,\tau_{i+1})$. That being said, this would also very much depend on how the set $Per$ intersects $\mathcal{U}$, as the bifurcations there can unfold more than one way -  for details, see Theorems 3, 4 and 5 in \cite{Shil1}.

 \section{Describing the dynamics of bifurcations}
\label{bifdyn}
Having studied the behavior of the Entropy flow in the three canonical routes to chaos, in this Section we study the topological connections between bifurcations and the dynamics of the Entropy flow. Our motivation for doing so is to study how the topology, dimension of a given manifold $M$ and the parameter space constrain the behavior of the Entropy flow. To motivate why this question is interesting, let $\dot{s}=F_\tau(s)$ denote a $C^1$ one--parameter family of vector fields defined on, say, $S^2$, where $\tau$ varies in $S^1$. For discussion's sake, let us further assume that for all $\tau\in S^1$ the vector field $F_\tau$ has repelling orbits $T^\tau_1$ and $T^\tau_2$ that persist without bifurcating as we vary $\tau$ in $S^1$. Let $C_1$ and $C_2$ denote the components of $Per$ corresponding to these orbits - as we will show later on, these sets act as repellers for any Entropy flow on $S^2\times S^1$ (see Proposition \ref{decomposition}). Note that $S^2\times S^1$ is a closed three-manifold and that $C_1$ and $C_2$ are homeomorphic to two-dimensional tori embedded in $S^2\times S^1$. This heuristically implies that the dynamics of the Entropy flow on each component $C$ of $(S^2\times S^1)\setminus(C_1\cup C_2)$ is constrained by the topology of $C$ and the behavior of the Entropy flow around $C_1$ and $C_2$. In other words, the behavior of the Entropy flow on $C_1$ and $C_2$ affects the possible bifurcations the family $\dot{s}=F_\tau(s)$ can undergo inside $C$. In Subsection \ref{2dtheory} we will make this heuristic precise using the Conley index Theory (see Proposition \ref{decomposition} and Theorem \ref{bifcodesth1}).\\

This Section is organized as follows. We begin in Subsection \ref{1d}, where, for completeness, we study the Entropy flow for $C^1$ families of first-order ordinary differential equations on the real line (see Proposition \ref{general1}). Following that, in Subsection \ref{2dtheory} we study in depth how topology forcing bifurcations - which we do by applying the Conley index Theory to describe the dynamics of the Entropy flow (see Proposition \ref{decomposition} and Theorem \ref{bifcodesth1}). We conclude with Subsection \ref{lorenz}, where the Conley index is used for  the Entropy flow of the Lorenz system.

\subsection{The curious case of dimension $1$}
\label{1d}
In this Subsection, for completeness, we study the Entropy flow in the case when $M$ is a connected one-dimensional smooth manifold - i.e., $M$ is either $\mathbb{R}$ or $S^1$. Our goals here are to show the interplay between the topology of $M\times I$ to the dynamics of the Entropy flow - as such, we will restrict ourselves to the specific case when the product space $M\times I$ is $\mathbb{R}^2$, i.e., $M=I=\mathbb{R}$. In addition, we will consider the Entropy flow only for a very specific collection of $C^1$ families. To introduce it, let $\mathcal{F}$ denote the collection of $C^1$ one--parameter families $\dot{s}=F_\tau(s)$ defined for $(s,\tau)\in\mathbb{R}^2$ where the following is satisfied:
\begin{itemize}
    \item Every component $G_\alpha\subseteq Fix$ can be parameterized by $\{(x_\tau,\tau)|\tau\in J\}$ for some open connected set $J\subseteq I$. In addition, we also require $\frac{dF_\tau(x_\tau)}{ds}\ne0$, i.e., $x_\tau$ is either a sink or a source.
    \item Whenever the set $\{(x_\tau,\tau)|F_\tau(x_\tau)=0\}\setminus Fix$ is non-empty, it is composed only of saddle node and pitchfork bifurcations of attractors and repellers. In other words, all the fixed points for $\dot{s}=F_\tau(s)$ are sinks and sources that are created and/or destroyed at either saddle node or pitchfork bifurcations.\\
\end{itemize}

To (partially) heuristically justify our motivation in restricting our attention to such one--parameter families, we first recall a particular case of Thom's Theorem (see \cite{zeeman} for a proof), which we state here in a way that is suitable for our needs:
\begin{theorem}
    \label{thom} Let $F:\mathbb{R}^2\to\mathbb{R}$, $F(s,\tau)=F_\tau(s)$ be $C^\infty$. Then generically the set $M_f=\{(x,\tau)|\nabla F=0\}$ is a manifold. Moreover, setting $\pi:M_f\to\mathbb{R}$ as the projection $\pi(x,\tau)=\tau$, every singularity of $\pi$ in $M_f$ would be a fold (see the illustration in Figure \ref{fold}).
\end{theorem}
\begin{figure}[h]
\centering
\begin{overpic}[width=0.3\textwidth]{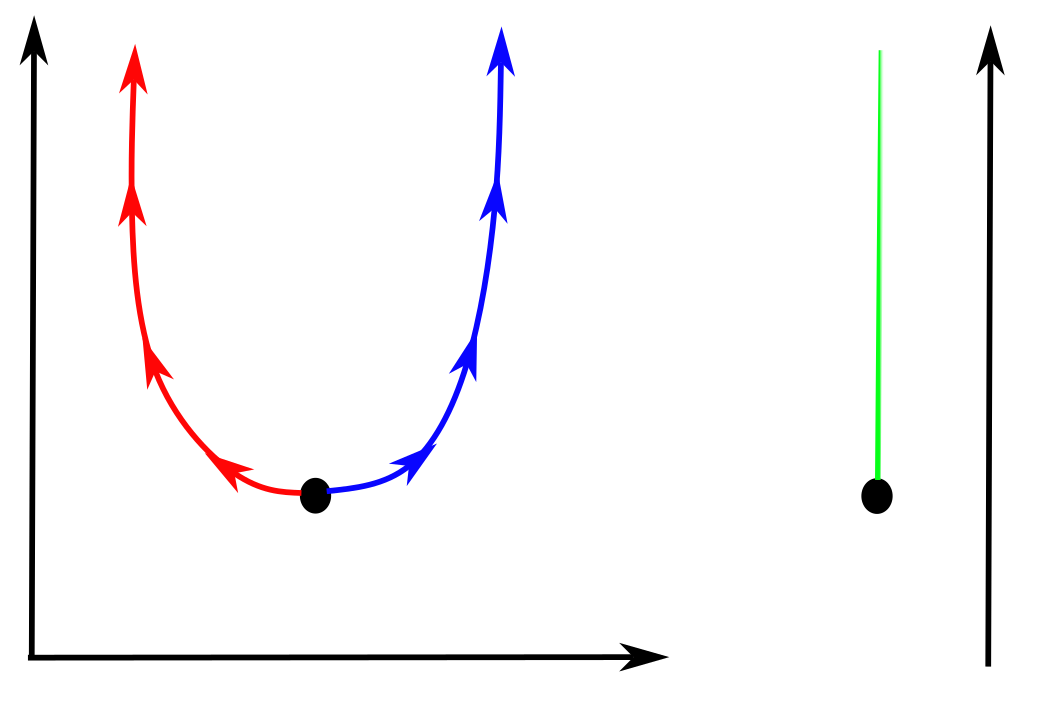}
\put(650,30){$x$}
\put(930,660){$\tau$}
\put(20,670){$\tau$}
\end{overpic}
\caption{\textit{On the left - a saddle node bifurcation and the direction of the Entropy flow on it. On the right - the projection of this set to the $\tau$ line. The black dot on the left denotes the bifurcation point, and the black dot on the right denotes the singularity.}}
\label{fold}
\end{figure}

When instead of $F:\mathbb{R}^2\to\mathbb{R}$ we write $\dot{s}=F_\tau(s)$, a fold catastrophe would correspond to a saddle node bifurcation where two fixed points collide and vanish (see Figure \ref{fold}). In addition, the condition that there exists an arc $\gamma$ parameterized by $\{(x_\tau,\tau)|\tau\in J\}$ s.t. for all $\tau\in J$ we have both $F_\tau(x_\tau)=0$ and $\frac{dF_\tau(x_\tau)}{ds}=0$ is unstable, i.e., can be removed by arbitrarily small smooth perturbations. These two facts imply that the family of $C^\infty$ families in $\mathcal{F}$ are $C^\infty$-dense in $C^\infty(\mathbb{R}^2,\mathbb{R})$. As such, because we allow $C^1$ families in $\mathcal{F}$ to have both pitchfork and fold singularities (and because $C^\infty$ maps are dense in $C^1$ maps), we conclude by Thom's Theorem that any $C^1$ one--parameter family $\dot{s}=F_\tau(s)$ where $(s,\tau)\in\mathbb{R}^2$ can be $C^1$-approximated by sequences in $\mathcal{F}$ (we will not address the question whether the set $\mathcal{F}$ is also $C^1$ generic in $C^1(\mathbb{R}^2,\mathbb{R})$). We now prove the main result of this Subsection:
\begin{proposition}
\label{general1} Consider a $C^1$ one--parameter family $\dot{s}=F_\tau(s)$ where $(s,\tau)\in\mathbb{R}^2$ and the map $F(s,\tau)=F_\tau(s)$ is in $\mathcal{F}$. Then, we can choose an Entropy flow which has no periodic orbits in $\mathbb{R}^2$. Let us denote the vector field generating the said Entropy flow by $E$ - then, given any initial condition $(s,\tau)$ its trajectory w.r.t. $E$ either tends to some fixed point for $E$ or to $\infty$ (see the illustration in Figure \ref{1dim}).
\end{proposition}
\begin{proof}

Assume first that we can find an Entropy flow on $\mathbb{R}^2$ corresponding to $\dot{s}=F_\tau(s)$ s.t. it has no periodic orbits. Let us now denote the vector field generating the said Entropy flow by $E$. By the Poincare-Bendixon Theorem we know every bounded trajectory for $E$ would either tend to a fixed point, a homoclinic or a heteroclinic cycle, or to a periodic orbit. As $E$ has no periodic orbits per our assumption,  every orbit must tend to either some fixed point $(x_0,\tau_0)$ for $E$, a homoclinic or a heteroclinic cycle, or to $\infty$. This shows that to prove the assertion there are two things to do: first, prove that we can always find some Entropy flow that has no periodic orbits, and second, prove that for this specific Entropy flow nothing can be attracted to either a homoclinic or a heteroclinic cycle.

We begin by using a topological argument. Given a component $G_\alpha\subseteq Fix$, let us parameterize it by $\{(x_\tau,\tau)|F_\tau(x(\tau))=0, \tau\in(a,b)\}$ (for some $(a,b)\subseteq\mathbb{R}$). Per our assumption, for all $\tau\in(a,b)$, all $x_\tau$ are either sinks or sources, and as we vary $\tau\in(a,b)$, no fixed point in $x_\tau$ changes its type, i.e., no sink becomes a source or vice versa. Without any loss of generality, assume all fixed points on $G_\alpha$ are sinks, and for any $\tau\in(a,b)$ let $I_\tau=(x_\tau-\epsilon_\tau,x_\tau+\epsilon_\tau)$ denote some interval s.t. initial conditions $s\in I_\tau$ are attracted to $x_\tau$ w.r.t. the ODE $\dot{s}=F_\tau(s)$. Conversely, when all the $x_\tau$ are sources we define $I_\tau$, $\tau\in(a,b)$ analogously s.t. all initial conditions $s\in I_\tau\setminus \{x(\tau)\}$ escape $I_\tau$ under the ODE $\dot{s}=F_\tau(s)$. In both cases, if, say, $a$ is finite, we choose $\epsilon_\tau$ s.t. $\lim_{\tau\to a}\epsilon_\tau=0$. Similarly, whenever $b$ is finite, we also require $\lim_{\tau\to b}\epsilon_\tau=0$. We define $N_\alpha=\cup_{\tau\in(a,b)}I_\tau\cup\{\tau\}$ - in particular, we construct $N_\alpha$ s.t. $Fix\cap N_\alpha=G_\alpha$.

We now write $Fix=\{G_\alpha\}_{\alpha}$ and let us consider the isolating collection $\{N_\alpha\}_\alpha$ constructed as described above (note that since $\dot{s}=F_\tau(s)$ is a one-dimensional ODE for all $\tau$, the sets $Per$ and $Inv$ are empty by definition). We now define an Entropy flow directed by a vector field $E$ for which the support of $P$ and $V$ is in $\cup_\alpha N_\alpha$, and on $\partial N_\alpha$ we have $E(s,\tau)=(F_\tau(x),0)$. By Lemma \ref{trapping} and the construction of $N_\alpha$ above, we know that when $G_\alpha$ is an arc of sinks, the trajectory of every initial condition $(s,\tau)\in\mathbb{R}^2$ that enters $N_\alpha$ cannot escape it. In contrast, when $G_\alpha$ is an arc of sources, the trajectory of any initial condition $(s,\tau)\in \mathbb{R}^2$, $(s,\tau)\not\in G_\alpha$ is repelled away from $N_\alpha$. This already proves that there can be no periodic orbits for $E$ intersecting the sets $\{(x,\tau)|V(x,\tau)\ne0\}$ and $\{(x,\tau)|P(x,\tau)\ne0\}$.

Consequently, if $E$ has a periodic orbit in $\mathbb{R}^2$, it must be trapped in the sets $\{(s,\tau)|E(s,\tau)=(F_\tau(s),0)\}$. By Theorem 1 in \cite{Kris} we know that if $T$ is a periodic orbit for $E$, the velocity $\{\dot{s}=0\}$ must switch signs on it - or, in other words, $T$ intersects transversely with the set $\{(s,\tau)|F_\tau(s)+V(s,\tau)=0\}$. As by the above $T$ must be trapped inside the set $\{(s,\tau)|V(s,\tau)=0\}$ we conclude that if $T$ exists, it must intersect transversely with the set $\{(s,\tau)|F_\tau(s)=0, V(s,\tau)=0, P(s,\tau)=0\}$, i.e., it must intersect a fixed point for the vector field $E$. This is a contradiction, hence we conclude the Entropy flow directed by $E$ has no periodic orbits. Therefore, to complete the proof we need to show no trajectory can accumulate on a homoclinic or a heteroclinic cycle.

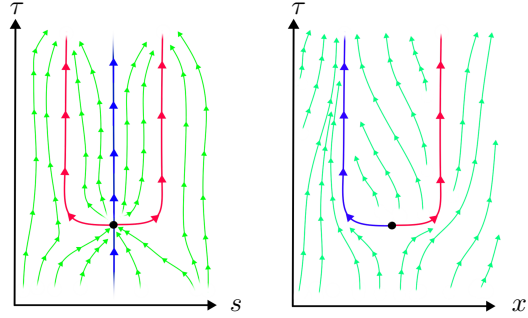
\begin{figure}[h]
\centering
\begin{overpic}[width=0.4\textwidth]{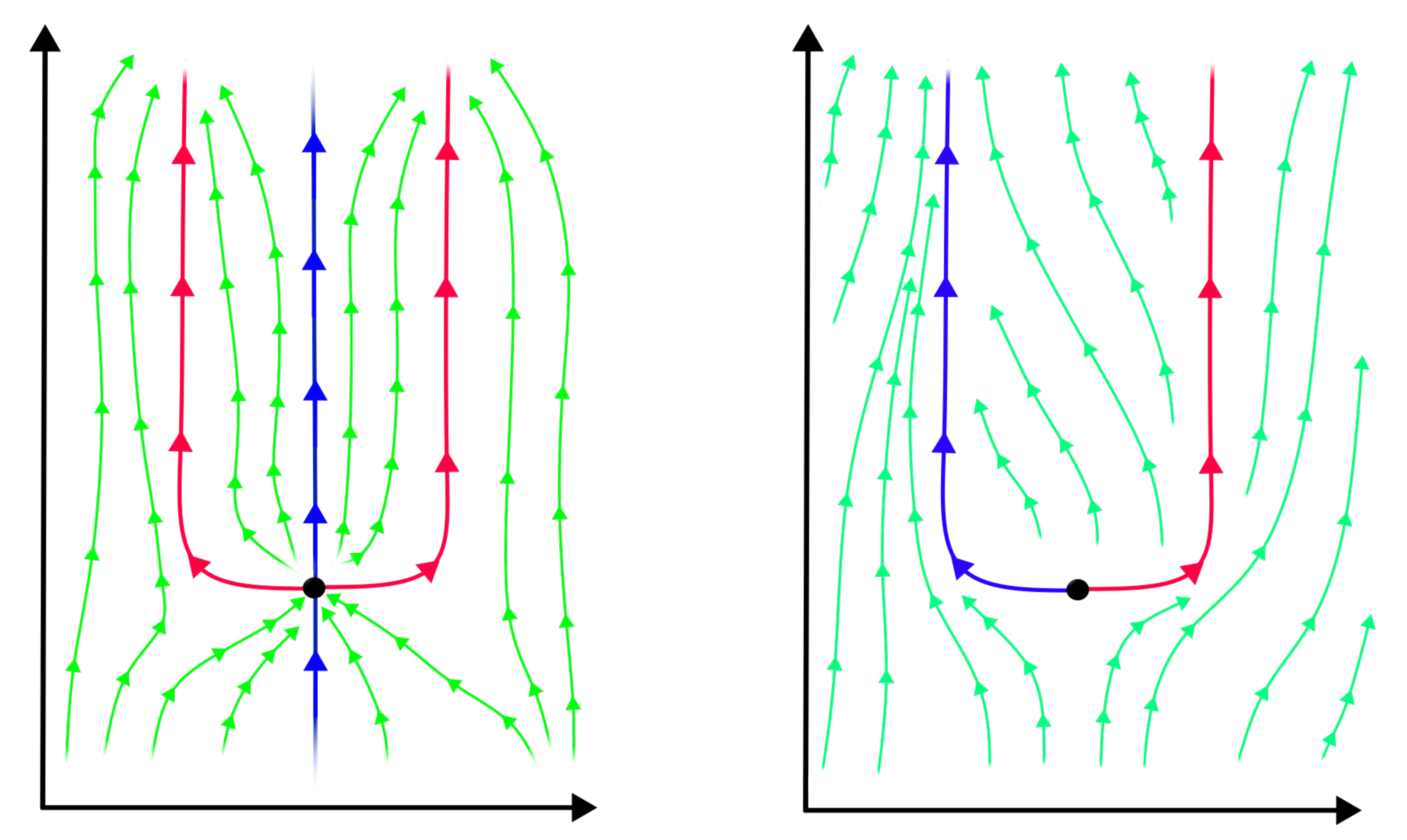}
\put(1000,10){$x$}
\put(450,10){$s$}
\put(570,590){$\tau$}
\put(20,590){$\tau$}
\end{overpic}
\caption{\textit{Diagrams of one-dimensional Entropy flows with a pitchfork bifurcation on the left and saddle node on the right. }}
\label{1dim}
\end{figure}

To do so, consider some initial condition $(s,\tau)\in\mathbb{R}^2$ that is not a fixed point for $E$, which remains bounded under $E$. By the Poincare-Bendixon Theorem and the above we know its trajectory must pass arbitrarily close to some fixed point for $E$, $(x_0,\tau_0)$, as it either tends to a fixed point or accumulates on a homoclinic/heteroclinic cycle - we now prove the trajectory of such an $(s,\tau)$ must actually tend to $(x_0,\tau_0)$. To begin,  note that by the above, if the trajectory of $(s,\tau)$ enters some $N_\alpha$ corresponding to some $G_\alpha$ that is an arc of sinks, by our construction of $N_\alpha$ we know that it cannot escape it - in which case its trajectory would tend to a fixed point. Similarly, it can only escape (and never enter) $N_\alpha$ constructed around $G_\alpha$ that are arcs of sources. Therefore, if there exists a homoclinic or a heteroclinic trajectory $\Gamma$ s.t. the trajectory of $(s,\tau)$ can accumulate on $\Gamma$, it must be located inside the set $O=\{(s,\tau)|E(s,\tau)=(F_\tau(s),0), (s,\tau)\not\in\cup_\alpha\partial N_\alpha\}$. By definition, the velocity $\dot{\tau}$ vanishes on this set - moreover, it includes all the fixed points of $E$. Since homoclinic trajectories in $\mathbb{R}^2$ cannot be confined to subsets of the type $\mathbb{R}\times\{\tau\}$, $E$ cannot have any homoclinic trajectories.

As such, if $\Gamma$ exists, it can only be a heteroclinic trajectory connecting (at least) two fixed points for $E$. This implies $\Gamma$ can be parameterized by $\{(f(t),\tau_0)|t\in[0,1]\}$ for some fixed $\tau_0$ and some monotone function $f(t)$. It also implies that in order for the trajectory of $(s,\tau)$ to accumulate on $\Gamma$, the $\tau$ velocity must change - i.e., if $(s,\tau)\not\in\Gamma$, in order for its trajectory to accumulate on $\Gamma$ it must intersect the set $\{(s,\tau)|P(s,\tau)\ne0\}$ infinitely many times. As the support of $P$ is contained in $\cup_\alpha N_\alpha$, by the argument above this cannot occur. Therefore, there does not exist a heteroclinic trajectory $\Gamma$ on which the trajectory of $(s,\tau)$ w.r.t. $E$ can accumulate. All in all, it follows that if the trajectory of $(s,\tau)$ is bounded, it must tend to some fixed point for $E$, i.e., some saddle node or pitchfork bifurcation orbit for the original family $\dot{s}=F_\tau(s)$.
\end{proof}

\subsection{When topology forces bifurcations}
\label{2dtheory}

The main takeaway from the proof of Proposition \ref{general1} is that topology can and will constrain the dynamics of the Entropy flow. In this Section we will study the same idea, but in higher dimensions. As we will see this will lead us to conclude that topology can force bifurcations. More precisely, we will show that the topological constrains imposed on the Entropy flow for $\dot{s}=F_\tau(s)$ can and will constrain its dynamics - this, in turn, constrains the bifurcations the family $\dot{s}=F_\tau(s)$ can undergo in $M$ as $\tau$ is varied. Even though our main motivation to study this problem arose from $C^1$ one--parameter families of two-dimensional vector fields on a closed surface $M$, our arguments also hold when $M$ is a closed manifold of higher dimension. This Subsection is organized as follows: we begin by making some general topological observations. Following that, we discuss how (and when) one can decompose Entropy flows by their Conley indices (see Theorem \ref{decomposition}), and use these Conley indices to describe the bifurcations the family undergoes as $\tau$ is varied (see Theorem \ref{bifcodesth1}).\\

\begin{figure}[h]
\centering
\begin{overpic}[width=0.3\textwidth]{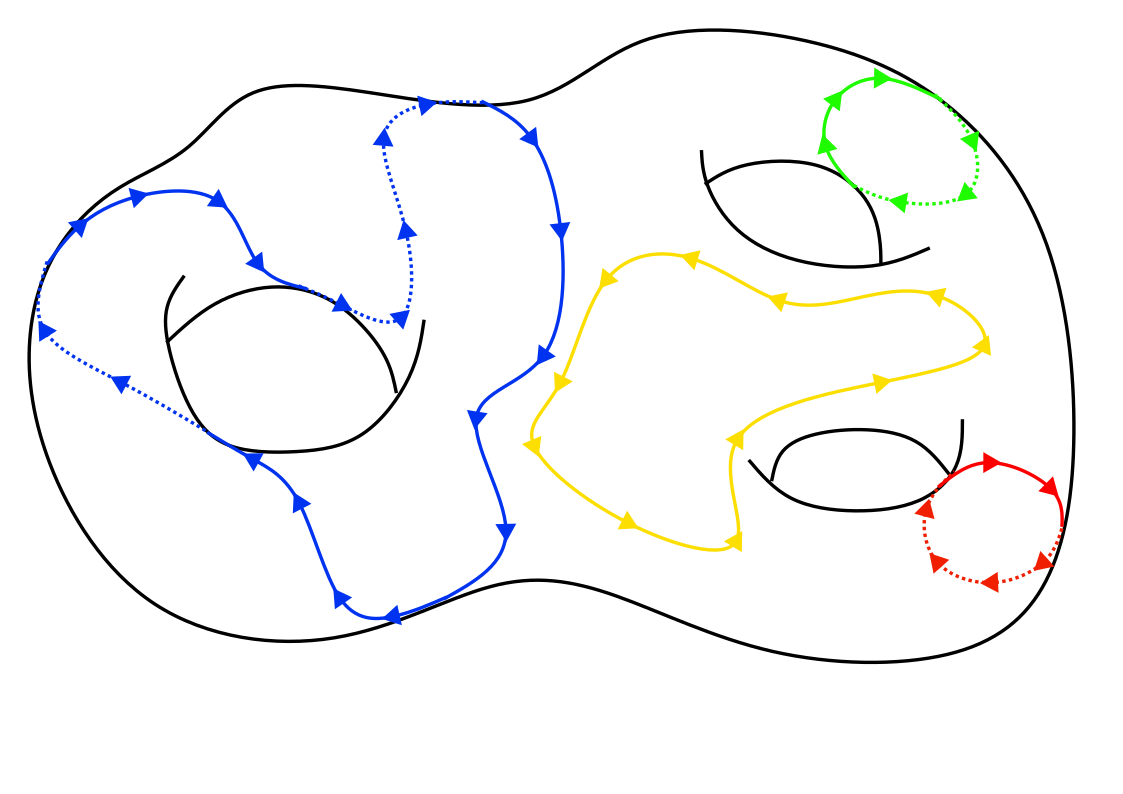}

\end{overpic}
\caption{\textit{Different periodic orbits on a sphere with three handles. No pair of these periodic orbits can be isotoped one to another - consequently, they cannot be destroyed by a bifurcation which collides them with one another. }}
\label{wind}
\end{figure}

We will go with a motivating example. Consider the scenario illustrated in Figure \ref{wind}, where $M$ is a sphere with three handles. In this setting there exist four different loops that wind differently around the handles, making them non-isotopic in $M$. Given any $C^1$ family of vector fields $\dot{s}=F_\tau(s)$, $\tau\in [-1,1]$ that realizes these four loops as periodic orbits, as we vary $\tau$, these periodic orbits cannot be destroyed by colliding with one another. One could, of course, replace the language of ambient isotopies with that of the Entropy flow. By the Kupka-Smale and Peixoto's Theorems (see \cite{Pei}, \cite{Smale} and \cite{Kup}) and one should generically expect these periodic orbits to be attractors and repellers - or, in other words, to correspond to components in $Per$, $G_1,G_2,G_3$, $G_4$.\\

Now, let $C$ be a component of $M\times[-1,1]\setminus(\cup_{i=1}^4 G_i)$ - the Entropy flow on $C$ would be a solution to differential equation with boundary condition, imposed by the behavior of the Entropy flow on $G_1,G_2, G_3$, $G_4$. Boundary values can (and often will) constrain the behavior of solutions to ordinary differential equations, and dynamical systems in general. That being said, due to the definition of the Entropy flow, one should expect these constrains to force the bifurcations of $\dot{s}=F_\tau(s)$ to behave in a certain way. We will now make this idea precise. To do so, we first recall the definition of the Conley index, following \cite{Conley1978isolated}, \cite{Salamon1985ConnectedSimpleSys}, \cite{Mischaikow}, \cite{Mischaikow2002ConleyIndex}:

\begin{definition}
\label{index} Let $S$ be an isolated invariant set for a flow $\varphi:M\times\mathbb{R}\to M$.  A pair of compact sets $(N,L)$ with $L\subset N$ are called an \emph{index pair} for $S$ if the following conditions hold:
\begin{enumerate}
    \item $N\setminus L$ is a neighborhood of $S$ and $S=\Inv_\varphi \bigl(\overline{N\setminus L}\bigr)$, where $\Inv_\varphi$ is the invariant set for $\varphi$ in $\overline{N\setminus L}$.
    \item $L$ is positively invariant relative to $N$: if $x\in L$ and $\varphi(x,[0,t])\subset N$, then $\varphi(x,t)\in L$.
    \item Every trajectory that exits $N$ in positive time first meets $L$. The set $L$ is called the \emph{exit set}.
\end{enumerate}

The \textbf{\emph{homotopy Conley index} of $S$} is $h(S):=[N/L]$, i.e., the pointed homotopy type of the quotient space $N/L$.  The \emph{\textbf{homological Conley index}} is $ CH_*(S;\mathbb{F}) := H_*(N,L;\mathbb{F})
   \cong \widetilde H_*(N/L;\mathbb{F})$, where by $H_*(N,L;\mathbb{F})$ we will always mean the sequence of homology groups $\{H_n(N,L;\mathbb{F})\}$. Moreover, when the coefficient field $\mathbb{F}$ is fixed and clear from context, we write simply $CH_*(S)$. 
\end{definition}

Our idea in this Subsection would be as follows: by construction, the Entropy flow for the one--parameter family $\dot{s}=F_\tau(s)$ is a dynamical system whose invariant sets are designed to be the bifurcations of $\dot{s}=F_\tau(s)$. Therefore, provided we can isolate these invariant sets and compute their Conley index, we can infer the bifurcation structure of $\dot{s}=F_\tau(s)$. We begin by proving the following:
\begin{proposition}
 \label{decomposition}   Let $M$ be a closed manifold of dimension $n\geq2$ and let $\dot{s}=F_\tau(s)$ denote a $C^1$ one--parameter family of vector fields on $M$, where $\tau$ varies in $S^1$. Assume there is some finite collection of periodic orbits and fixed points, $\gamma_1,...,\gamma_n$ satisfying the following:
 \begin{itemize}
     \item $\gamma_1,...,\gamma_n$ do not bifurcate as $\tau$ is varied in $S^1$.
     \item For all $i$, when $\gamma_i$ is a fixed point it is either source or a sink. Conversely, when it is a periodic orbit it is either an attractor or a repeller.
 \end{itemize}
 
Then, there exists some $r$ s.t. given any Entropy flow, its dynamics can be encoded using a finite number of homotopy Conley indices $h_1,...,h_r$ and homological Conley indices $c_1,...,c_r$. Moreover, $h_1,...,h_r$ and $c_1,...,c_r$ are independent of the choice of Entropy flow.
\end{proposition}

\begin{proof}

Recall, $Per$ denote the collection of points in $M\times I$ which lie on periodic orbits for some $F_\tau$, $\tau\in I$. Similarly, let $Fix$ denote the collection of fixed points for the vector fields on the curve. Since all the periodic orbits and fixed points $\gamma_1,...,\gamma_n$ are either attractors, repellers, sources, or sinks, each $\gamma_i$ corresponds to its own component $\mathcal{S}_i$ in $Per\cup Fix$. As $\tau$ varies in $S^1$, it follows that for all $1\leq i\leq n$, when $\gamma_i$ is a periodic orbit, the component $\mathcal{S}_i$ is a closed surface embedded in the $n+1$-dimensional manifold $M\times I$. To see why, first note that $M\times I$ is a closed manifold, being the product of two closed manifolds. Therefore, given $\{(s_n,\tau_n)\}_n\subseteq\mathcal{S}_i$ a Cauchy sequence in $M\times I$, we know $(s_n,\tau_n)$ tends to some $(s,\tau)\in M\times I$. As $\gamma_i$ does not bifurcate, it is immediate that $(s,\tau)\in\mathcal{S}_i\cap M\times I$,  which proves $\mathcal{S}_i$ is closed. A similar argument proves that for $i$ s.t. $\gamma_i$ is a fixed point, $\mathcal{S}_i$ is an embedding of $S^1$ in $M\times I$.

To continue, consider $\mathcal{M}=M\times I\setminus (\cup_{i=1}^n \mathcal{S}_i)$ (note, $\mathcal{M}$ is a compact manifold with boundary) - it is immediate that there are only finitely many components in $\mathcal{M}$. Since all the sets $\mathcal{S}_i$ are compact manifolds of dimension $1$ or $2$, by our assumption that all the $\gamma_1,...,\gamma_n$ are either sinks, sources, attractors or repellers, the same argument used to prove Lemma \ref{trapping} shows we can thicken $\mathcal{S}_i$  to open sets, $O_i$, $1\leq i\leq n$, satisfying the following:
\begin{itemize}
    \item For all $\tau$, $F_\tau$ is transverse to $\partial (M\times\{\tau\}\cap O_i)$.
    \item The number of components in $\mathcal{M}$ is the same as that in $M\times I\setminus(\cup_{i=1}^n O_i)$.
\end{itemize}

Given any Entropy flow directed by a vector field $E$, as $\gamma_1,...,\gamma_n$ do not bifurcate, we also know that the functions $P$ and $V$ both vanish around $\mathcal{S}_1,...,\mathcal{S}_n$ (see Section \ref{defsect}). This proves we can always find some $O_i$, $i=1,...,n$ s.t. $E$ is transverse to $\partial O_i$ - namely, by choosing the sets $O_i$ to be arbitrarily small thickenings of the corresponding $\mathcal{S}_i$. It is easy to see that  when we replace $E$ by a vector field $E'$ directing another Entropy flow with corresponding sets $O'_i$, the sets $O_i$ and $O'_i$ are ambient isotopic in $M\times I$, and if $N(p)$ and $N'(p')$ are the respective normal vector to $\partial O_i$ and $\partial O'_i$ at some $p,p'$, then $E(p)\cdot N(p)$ and $E'(p')\cdot N'(p')$ have the same sign for all choices of $p\in\partial O_i$ and $p'\in\partial O'_i$. Therefore, assuming $\mathcal{M}_1$ is some component of $\mathcal{M}$ bounded by $\mathcal{S}_1,...,\mathcal{S}_k$, $1\leq k\leq n$, there exists a closed manifold with boundary, $B_1$ such that $B_1=\mathcal{M}_1\setminus(\cup_{i=1}^k O_i)$. Moreover, for any other choice of Entropy flow, $E'$, there exist an analogous set $B'_1$ given by $\mathcal{M}_1\setminus(\cup_{i=1}^kO'_i)$ and  satisfying the following:
\begin{itemize}
    \item $B_1$ and $B'_1$ are isotopic in $M\times I$, and both are isotopic to $\mathcal{M}_1$.
    \item $E$ and $E'$ point the exact same way throughout $\partial B_1$ and $\partial B'_1$ (respectively).
\end{itemize}

Let us number these sets $B_1,...,B_r$ created by $M\times I\setminus(\cup_{i=1}^n\mathcal{S}_i)$.  For each \(B_j,\) $1\leq j\leq r$, define the next sets $L_j=\{(s,\tau)\in\partial B_j:\ E(s,\tau)\text{ points out of }B_j\}$. By construction, \(E\) is transverse to \(\partial B_j\), so \(B_j\) is an isolating block for the set $S_j$, the invariant set of $E$ in $B_j$. As such, the pair \((B_j,L_j)\) satisfies the index-pair axioms, which implies the Conley indices $ h(S_j)=[B_j/L_j]$ and $ CH_*(S_j;R)=H_*(B_j,L_j;R)$ are well-defined (see Definition \ref{index}). Since by the above different sufficiently small choices of \(O_i\) are ambient isotopic and the entrance and exit signs on their boundaries are unchanged, the quotient pairs $(B_j,L_j)$ are independent of our choice of $O_i$ and have the same homotopy type (as long as we define them as above). Moreover, given another Entropy flow $E'$ with corresponding quotient pairs $(B'_j,L'_j)$  defined analogously, by the above it is clear that the homotopy type is unchanged - which implies Conley indices are independent of the choice of an Entropy flow for $\dot{s}=F_\tau(s)$.
\end{proof}

\begin{figure}[h]
\centering
\begin{overpic}[width=0.27\textwidth]{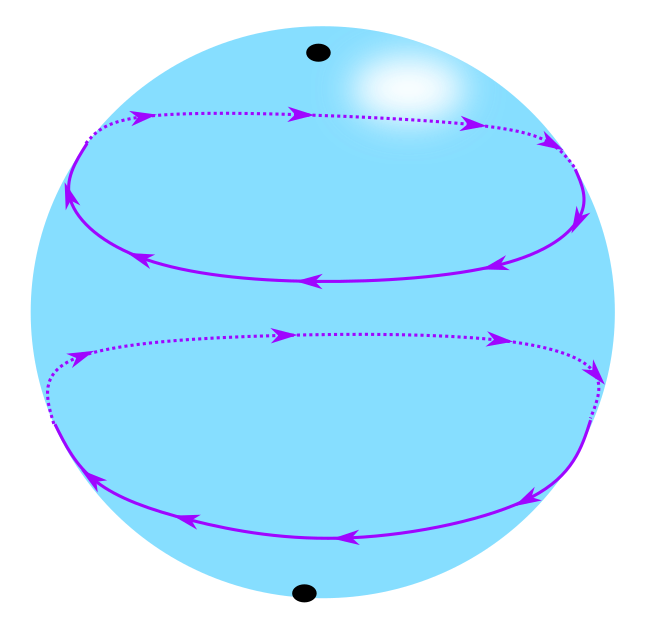}
\put(900,200){$\gamma_R$}
\put(960,630){$\gamma_A$}
\put(460,-10){$s$}
\put(460,830){$n$}
\end{overpic}
\caption{\textit{Two periodic orbits - an attractor $\gamma_A$ near the North pole $n$ and a repeller $\gamma_R$ near the South pole $s$.}}
\label{exsphere}
\end{figure}
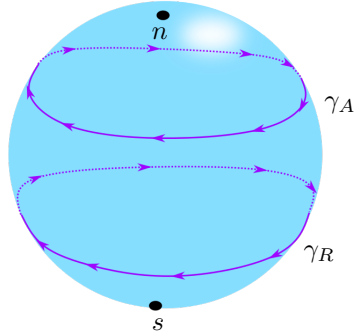

The above Propositions can be interpreted as saying the dynamics of the Entropy flow is a finite collection of Conley indices, "glued together" at the stable periodic orbits and fixed points. In addition, it shows us the role of stable non-bifurcating dynamics. Similarly to how the braid type can force the dynamics of a homeomorphism of a punctured disc to behave in a certain way (see \cite{Bo} and \cite{CH} for a survey), so do the persistent attracting and repelling sets constrain the bifurcations of $\dot{s}=F_\tau(s)$. We now give a simple example illustrating how this plays out. To this end, consider a $C^\infty$ one--parameter family of vector fields $\dot{s} = F_\tau(s), s \in S^2, \tau\in S^1$. Assume there exist two periodic orbits corresponding to two disjoint simple closed curves $\gamma_A, \gamma_R \subset S^2$, where $\gamma_A$ is an attracting cycle and $\gamma_R$ a repelling cycle s.t. both persist without bifurcating for all parameters $\tau\in S^1$ (see Figure \ref{exsphere}). Setting $\mathcal{F}=(\gamma_A\times S^1)\cup(\gamma_R\times S^1)$, we now consider components of $\mathcal{M} = (S^2\times S^1)\setminus \mathcal{F}$. After deleting the two latitude lines $\gamma_A$ and $\gamma_R$ from $S^2$, three connected components remain (see Figure \ref{exsphere}):
\begin{enumerate}
    \item  A disk region $D_N$ near the North Pole $n$, bounded by $\gamma_A$.
    \item The annular region $A$ between the two latitudes.
    \item A disk region $D_S$ near the South Pole $s$, bounded by $\gamma_R$.\\
\end{enumerate}

With a slight abuse of notation, after replacing the closed surfaces \(\gamma_A\times S^1\) and \(\gamma_R\times S^1\) by sufficiently small thickened neighborhoods transverse to the Entropy flow we denote the three complementary blocks by $\mathcal M_1=D_N\times S^1$, $\mathcal M_2=A\times S^1$ and $\mathcal M_3=D_S\times S^1$. Thus, $\mathcal M_1\simeq D^2\times S^1$, $\mathcal M_2\simeq \mathbb T^2\times[0,1]$ and $\mathcal M_3\simeq D^2\times S^1$ where $\mathbb T^2 = S^1 \times S^1$ denotes the two dimensional torus. We now compute their Conley indices - by Proposition \ref{decomposition} we know the results would apply to all possible choices of Entropy flows defined by $\dot{s}=F_\tau(s)$. To begin, since \(\gamma_A\) is attracting the boundary component corresponding to \(\gamma_A\times S^1\), it is an exit boundary for the adjacent complementary blocks w.r.t. all possible Entropy flow. Similarly, since \(\gamma_R\) is repelling the boundary component corresponding  to \(\gamma_R\times S^1\), it is an entrance boundary w.r.t any choice of Entropy flow. Therefore, because $\mathcal M_1=D_N\times S^1\simeq D^2\times S^1$ and since \(\partial\mathcal M_1\simeq S^1\times S^1\) corresponds to the attracting cycle \(\gamma_A\), $\partial\mathcal M_1$ is an exit set. Hence, by Proposition \ref{decomposition}, its homological Conley index w.r.t. every Entropy flow is given by $CH_*(\mathcal M_1) = H_*(D_N\times S^1,S^1\times S^1)$, and in more detail:
    \[
    CH_q(\mathcal M_1)
    \cong
    \begin{cases}
    \mathbb Z, & q=2,3,\\
    0, & \text{otherwise}.
    \end{cases}
    \]
Note this can also be rewritten as $H_*(D_N\times S^1,S^1\times S^1)
    \cong
    \widetilde H_*(S^2\vee S^3),$ where \(\vee\) denotes the wedge sum. We now compute the same Conley index for \(\mathcal M_2\). Recall, $\mathcal M_2=A\times S^1\simeq \mathbb T^2\times[0,1]$, hence the exit set is the boundary component corresponding to the attracting cycle \(\gamma_A\), say $L_2=\mathbb T^2\times\{1\}.$ Since the inclusion $T^2\times\{1\}\hookrightarrow \mathbb T^2\times[0,1]$ is a homotopy equivalence, the pair $(\mathbb T^2\times[0,1],\mathbb T^2\times\{1\})$ has trivial relative homology. Therefore, the Conley index of $\mathcal{M}_2$ w.r.t. every Entropy flow is given by $CH_*(\mathcal M_2)=H_*(\mathbb T^2\times[0,1],\mathbb T^2\times\{1\})=0.$    Finally, for \(\mathcal M_3\), recall $\mathcal M_3=D_S\times S^1\simeq D^2\times S^1.$ Since the boundary component corresponds to the repelling cycle \(\gamma_R\), it is entrance boundary for $\mathcal{M}_3$ w.r.t. any choice of Entropy flow. Thus, the exit set is empty and $CH_*(\mathcal M_3) = H_*(D_S\times S^1)$, and in detail:
    \[
    CH_q(\mathcal M_3)
    \cong
    \begin{cases}
    \mathbb Z, & q=0,1,\\
    0, & \text{otherwise}.\\
    \end{cases}
    \]

By the Ważewski principle (see \cite{waz}), we conclude the existence of invariant sets for the Entropy flow in both $\mathcal{M}_1$ and $\mathcal{M}_3$. Moreover, the fact  $CH_*(\mathcal M_2)=0$ has the following implications for the bifurcations the dynamics of the $C^\infty$ family $\dot{s}=F_\tau(s)$ can undergo in $A$. Intuitively, by Sotomayor's Theorem (see \cite{soto}) we would expect all periodic orbits and fixed points in $A$ to be added (or destroyed) when $\tau$ is varied by saddle node bifurcations of either periodic orbits or fixed points, where an attractor-repeller pair or a source-sink pair would be created (or annihilated). The fact $CH_*(\mathcal{M}_2)=0$ makes this intuition precise, as it says the invariant set of the Entropy flow in $\mathcal{M}_2$ is expected to be symmetric - that is, for every attractor we would expect a repeller, and vice versa. As such, $CH_*(\mathcal{M}_2)=0$ essentially encodes how anything created in $\mathcal{M}_2$ w.r.t. $\dot{s}=F_\tau(s)$ must be created "in pairs", which annihilate one another as $\tau$ is varied.\\

\begin{figure}[h]
\centering
\begin{overpic}[width=0.4\textwidth]{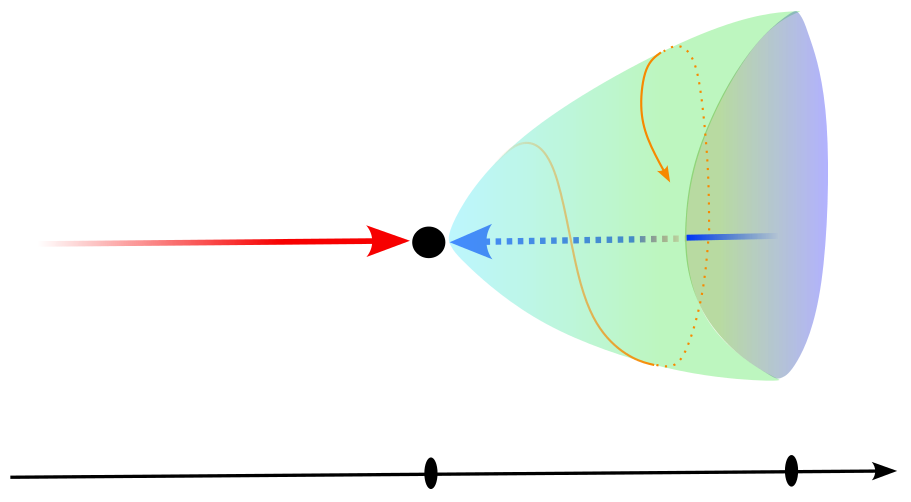}
\put(1000,30){$\tau$}
\put(460,-30){$1$}
\put(860,-30){$2$}
\end{overpic}
\caption{\textit{The red curve corresponds to the sinks, the blue -- to the sources, the surface -- to the periodic orbits generated at  Hopf bifurcation, where the arrows denote the motion of the Entropy flow on them. }}
\label{hopf}
\end{figure}

The idea behind Proposition \ref{decomposition} is that stable phenomena w.r.t. a $C^1$ family $\dot{s}=F_\tau(s)$ constrains the dynamics (and hence the bifurcations) of the Entropy flow. This leads us to ask the following - does something similar occurs when bifurcations are allowed? Specifically, does the existence of some bifurcation, for example, an Hopf bifurcation at some parameter $\tau\in I$ force another bifurcation, say, a saddle node, to appear at some $\tau'\in I$? And also, can we remove the requirement in Proposition \ref{decomposition} that the parameter space $I$ is $S^1$ and replace it with something more natural, like an interval? As we will see below, using the properties of the Entropy flow one can, to a certain extent, encode some bifurcations using the Entropy flow. Before stating these precisely, we begin with another  motivating example. Let $\dot{s}=F_\tau(s)$ denote a $C^1$ one--parameter family of vector fields on $S^2$, where, say, $\tau\in [0,2]$. We further assume the following:
\begin{itemize}
    \item For $\tau\in[0,1)$ there exists a stable sink at the origin which varies with $\tau$ (see Figure \ref{hopf}). At $\tau=0$ the origin is asymptotically stable.
    \item At $\tau=1$ the origin undergoes an Hopf bifurcation, where it becomes a source, and an attracting periodic orbit bifurcates from it, as illustrated in Figure \ref{hopf}. For simplicity, we assume the said orbit persists as an attractor for all $\tau\in(1,2]$. 
    \item For all $\tau\in[1,2]$ the periodic orbit is trapped inside some disc $\mathbb{D}\subseteq\mathbb{R}^2$ of bounded radius. \\
\end{itemize}

\begin{figure}[h]
\centering
\begin{overpic}[width=0.5\textwidth]{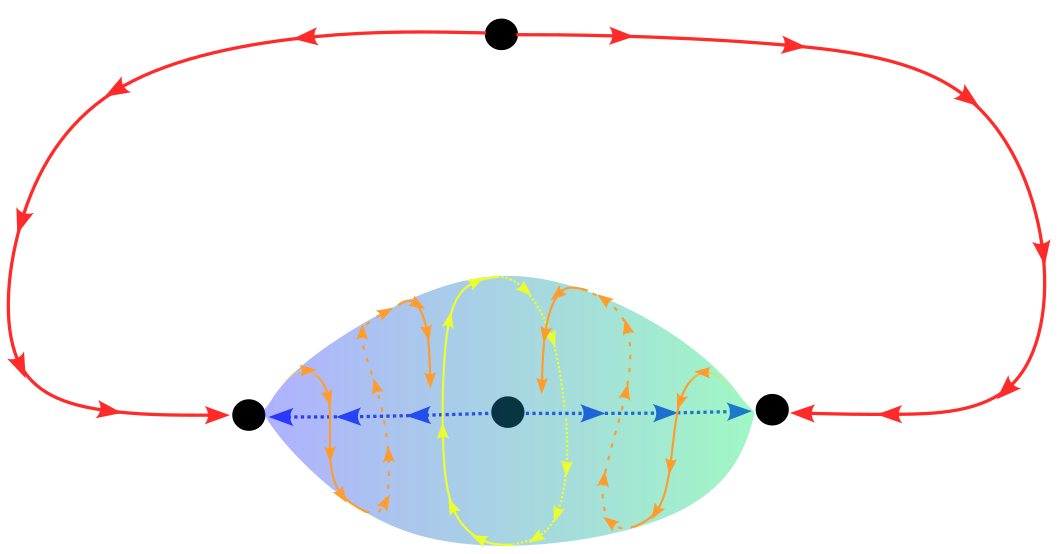}
\put(460,530){$p_3$}
\put(460,270){$T$}
\put(30,270){$W_1$}
\put(900,270){$W_2$}
\put(460,90){$p_4$}
\put(210,90){$p_1$}
\put(710,90){$p_2$}
\put(330,150){$W'_1$}
\put(550,150){$W'_2$}
\end{overpic}
\caption{\textit{The derived Entropy flow on $S^2\times S^1$, with the directions on $W=W_1\cup W_2$ and $W'_1\cup W'_2$. The orange flow lines denote flow lines on $S$, tending to the orbit $T$. The blue-green surface corresponds to $S$. }}
\label{hopf2}
\end{figure}
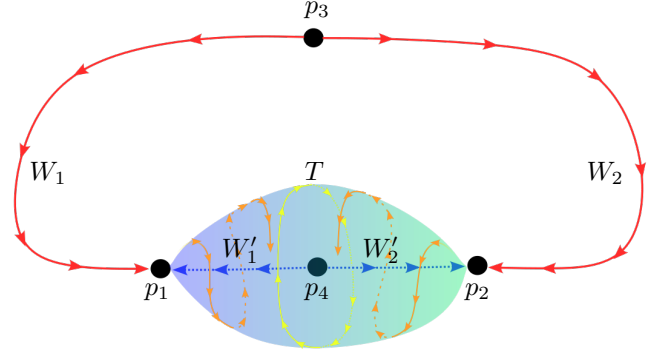

To continue, let $C$ denote the collection of components in $Per\cup Fix$ which includes the surface of periodic orbits bifurcating from the origin together with the arcs of sinks and sources - by definition, the said surface is in $Per$ while the curves are in $Fix$. Ideally, we would like to compute the Conley indices associated with these three components of $C$ for the Entropy flow. Unfortunately, this naive approach cannot work, if only because the parameter space $I=[0,2]$, an interval, makes it hard (if not impossible) to isolate $C$. This shows that in order to describe the Entropy flow in this setting using the Conley index, we must first somehow compactify everything. To this end, set $I_1\simeq[0,1]$, $I_2\simeq[0,2]$ correspond to arcs on $S^1$, and let us reparameterize $\dot{s}=F_\tau(s)$  s.t. $\tau$ now varies in $I_2$. We now smoothly extend the $C^1$ family $\dot{s}=F_\tau(s)$ to the parameter space $S^1\setminus I_2$ and create a one-parameter family on $S^2\times S^1$. We do so in a way that  everything is symmetric - i.e., for $\theta\in S^1\setminus I_2$ there exists a unique parameter $\tau\in[0,2]$ s.t. $F_\tau$ and $F_\theta$ have orbitally equivalent dynamics on $S^2$. This "reflection" makes $C$ into a subset of $\overline{C}$, where $\overline{C}$ is the set illustrated in Figure \ref{hopf2}. By its definition, the set $\overline{C}$ is composed of the following subsets:
\begin{itemize}
    \item Points $p_1=0\times\{1\}$, $p_2=0\times\{\theta_1\}$, corresponding to Hopf bifurcations, and a curve $W$ in $S^2\times (S^1\setminus I_2)$ connecting them. The curve $W$ is a curve of sinks in $Fix$ w.r.t. the described reflection. 
    \item Choose some $p_3\in W$ to serve as an equilibrium fixed point for the Entropy flow as in Definition \ref{admissiblefixed}, and let $W_1,W_2$ denote the components of $W\setminus p_3$ connecting $p_3$ to $p_1$ and $p_2$ (respectively). 
    \item Let $W'$ denote the line of fixed points connecting $p_1,p_2$, i.e., $W'$ is the curve of sources in $Fix$ w.r.t. this reflection. Again, choose some $p_4\in W'$ to serve as an equilibrium fixed point for the Entropy flow, and let $W'_1,W'_2$ denote the subarcs of $W'\setminus\{p_4\}$ connecting $p_4$ to $p_1,p_2$ (respectively).
    \item Let $S$ denote a surface homeomorphic to $S^2$ punctured at $p_1$ and $p_2$, i.e., $S$ is the surface corresponding to the periodic orbit created at the Hopf bifurcation at $p_1$ which is closed back to a fixed point at $p_2$. $S$ is glued to $W$ at $p_1$ and $p_2$, while the arc $W'$ lies strictly inside the topological ball bounded by $S$. Again, choose some closed loop $T\times\{r\}\subseteq S$ s.t. $T$ is periodic for $F_r$, $r\in S^1$, to serve as an equilibrium periodic orbit for the Entropy flow (see Definition \ref{admissibleper}).\\
\end{itemize}

By the definition of the Entropy flow we have a large degree freedom in the choice of the points $p_3$ on $W$, $p_4$ on $W'$, and the loop $T\times\{r\}$ on $S$ that serve as equilibrium states for the Entropy flow. That being said, this is where the freedom ends: per the definition of the Entropy flow, regardless of which points in $W,W'$ we choose as equilibrium fixed points or which loop we choose on $S$ to serve as an equilibrium orbit, any Entropy flow would behave on $\overline{C}$ to be as follows (see the illustration in Figures \ref{hopf2} and \ref{hopf3}):
\begin{itemize}
    \item  The point $p_3$ repels initial conditions on $W_1$, $W_2$ towards $p_1$, $p_2$ (respectively). As $p_3$ is originally a sink for some $F_\theta$, by Lemma \ref{trapping} we conclude that there exists a three-dimensional set, $C_1\subseteq M\times S^1$, homeomorphic to a ball, s.t. the Entropy flow points out of $C_1$ on two caps on $\partial C_1$ corresponding to the intersection with $W_1,W_2$. Similarly, there exists a band separating the two caps on $\partial C_1$ where the Entropy flow points into $C_1$. 
    \item The point $p_4$ is a source that repels initial conditions on $W'_1,W'_2$ asymptotically towards $p_1,p_2$ (respectively) - by Lemma \ref{trapping} we can choose some ball $A$ around it, s.t. $A$ lies strictly inside the region enclosed by $S$ and $E$ points out of $A$. Similarly, as every orbit on ${S}$ is attracting, by Lemma \ref{trapping} we can find a three-dimensional neighborhood $N$ of $S$ in $S^2\times S^1$ with $p_1,p_2$ on $\partial N$ s.t. $E$ points inside $N$ throughout $\partial N\setminus\{p_1,p_2\}$. Finally, since we assumed our original Hopf bifurcation points are asymptotically stable so are $p_1$ and $p_2$. As such, we can enclose the set $B=W_1\cup W_2\cup S\cup\{p_1,p_2,p_3\}$ in some isolating block, $C_2$, s.t. the Entropy flow points inside $C_2$ throughout $\partial C_2$ (note that $B$ is homotopic to $S^2\cup\gamma$, where $\gamma$ is a curve in $S^3\setminus S^2$ connecting the North and South poles on $S^2$).\\
\end{itemize}

The above is true for all choices of an Entropy flow in $S^2\times S^1$ chosen in this way, i.e., by using Lemma \ref{trapping} as described above. As such, we conclude the Conley indices of the invariant sets within the isolating blocks $C_1,C_2$ are well-defined. Moreover, similarly to the proof of Proposition \ref{decomposition}, these Conley indices are independent of the choice of Entropy flow - again, provided it is chosen using Lemma \ref{trapping} as above. Let us denote the respective Conley indices by $\mathcal{C}_1,\mathcal{C}_2$. We note that $\mathcal{C}_1$ can be "embedded" into $\mathcal{C}_2$, i.e., as a topological space: this embedding should be interpreted as encoding the change in the dynamics of the Entropy flow when the parameter $\tau$ varies from $[0,1)$ to $[1,2]$. More precisely, $\mathcal{C}_1$ and $\mathcal{C}_2$ should be interpreted as constraints the Hopf bifurcation induces on the dynamics of Entropy flows chosen as above. We are now going to discuss this example in more detail, after which we will generalize it to a wider setting.\\

\begin{figure}[h]
\centering
\begin{overpic}[width=0.5\textwidth]{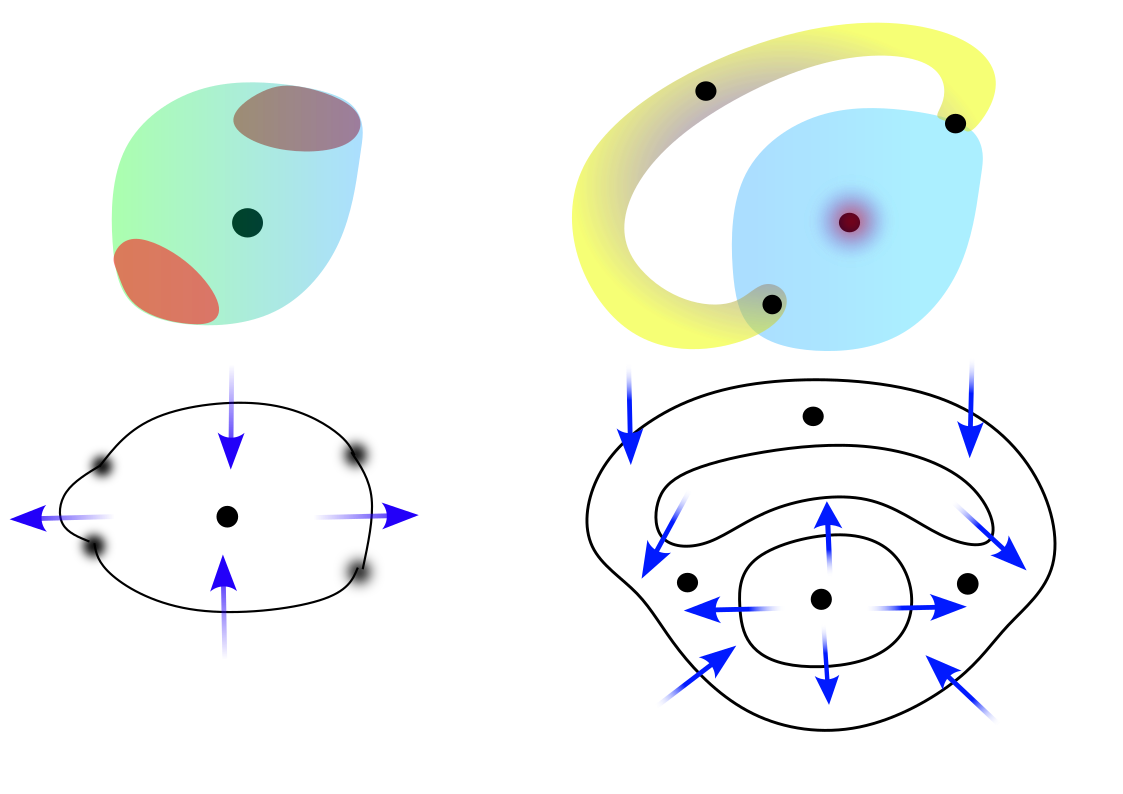}

\end{overpic}
\caption{\textit{On the upper left, $C_1$ - the red caps denote regions on $\partial C_1$, where the Entropy flow points outwards, while the cyan-green regions correspond to where it points inwards. On the downer left, there is a cross-section showing how the flow points on the boundary. Similarly, on the upper right, we see $C_2$, where the red ball in the middle centered at $p_4$ \textbf{is not} a part of $C_2$. And, on the downer right, there is a cross-section, showing the flow points inwards throughout the boundary. }}
\label{hopf3}
\end{figure}

To begin, consider the topological sphere $\Sigma = S  \cup \{p_1\} \cup \{p_2\}$ and recall $W=W_1\cup W_2\cup\{p_3\}, \; W^{'}=W^{'}_1 \cup  W^{'}_2\cup\{p_3\}$. Let $H$ denote the three-dimensional ball bounded by $\Sigma$. By definition, $p_4\in H, \; W^{'} \subset H$ and $ \; T\subset \Sigma.$ Similarly, let $B = W\cup S \cup \{p_1, p_2, p_3\}$, then $B \cong S^2 \cup \gamma$, where $\gamma$ is a curve connecting the North pole and the South pole of the $S^2$, i.e., $B \cong S^2 \vee S^1.$ To continue, let $N_0$ be the neighborhood isolating the periodic orbit $T$, let $N_1$ be the neighborhood isolating $B$ and let and $N_2$ be the neighborhood isolating $B \cup H$. By the above, $p_4 \in \operatorname{Inv}(N_2),\;  W_1^{'} \subset N_2, \;  W_2^{'} \subset N_2$ where the invariant set is w.r.t. the Entropy flow chosen as above (in previous notation, $N_1=C_2$ and $N_2=N_1\cup H$).  By our choice of Entropy flow, if it is given by a vector field $E$ then $E$ always points inside $N_i$ throughout $\partial N_i$, where $i=0,1,2$.\\

We now compute the respective Conley indices of $N_1$ and $N_2$, and study their relation to one another. We note the neighborhood \(N_0\) may be chosen to be a small solid-torus neighborhood of the periodic orbit \(T\), hence it deformation retracts onto the phase circle \(T\), i.e., $N_0$ can be collapsed homotopically to $T$, and hence to $S^1$ The neighborhood \(N_1\) deformation retracts onto \(B\), and therefore $N_1\simeq B\simeq S^2\vee S^1$, where by $\simeq$ we mean "homotopic". Here the two-dimensional homology class of \(S^2\) is represented by the Hopf sphere \(\Sigma\), while the one-dimensional homology class of \(S^1\) is represented by the loop obtained from the branch \(W\) together with a path on \(\Sigma\).  Finally, passing from \(N_1\) to $N_2$, by \(N_2=N_1\cup H\) the ball \(H\) fills the sphere $\Sigma$, it follows all the loop classes on $N_2$ persist as we map $N_1$ by the inclusion to $N_2$. It is trivial to see that since $N_2$ is obtained by "filling up" the sphere $\Sigma$ then $N_2\simeq S^1$.\\

We now recall the Conley index is computed from an index pair \((N,L)\) by collapsing the exit set \(L\) w.r.t. the flow, or, equivalently, by computing the relative homology groups \(H_*(N,L)\) (see Definition~\ref{index}). In the present example, since we chose the Entropy flow s.t. its generating vector field \(E\) points strictly into each \(N_i\) throughout its boundary, $i=0,1,2$, there is no exit set. Hence, writing the index pair $P_0=(N_0,\emptyset)$, $ P_1=(N_1,\emptyset)$ and $ P_2=(N_2,\emptyset)$ we get $CH_*(P_i)=H_*(N_i;\mathbb Z),  \;i=0,1,2.$ As stated above, there are well-defined inclusion maps $\iota_{ij}:(N_i,\emptyset)\hookrightarrow (N_j,\emptyset), \; 0\leq i<j\leq 2,$ which therefore induce homomorphisms $\Phi^*_{ij}:CH_*(P_i)\longrightarrow CH_*(P_j)$ for all $*$. Moreover, by the by functoriality of singular homology for all \(i<j<k\) these maps satisfy $\Phi^*_{ik}=\Phi^*_{jk}\circ \Phi^*_{ij}$. Using the homotopy types above, we now obtain:  $$CH_*(P_0)=H_*(S^1),\; CH_*(P_1)=H_*(S^2\vee S^1), \; CH_*(P_2)=H_*(S^1).$$
And more explicitly (where $\cong$ in the diagram below represents an isomorphism):
\[
\begin{tikzcd}[row sep=3.2em, column sep=5.5em]
CH_0(P_0)=\mathbb{Z}
  \arrow[d, "\Phi^{0}_{01}\cong"]
&
CH_1(P_0)=\mathbb{Z}\langle t\rangle
  \arrow[d, "\Phi^{1}_{01}=0"]
&
CH_2(P_0)=0
  \arrow[d]
\\
CH_0(P_1)=\mathbb{Z}
  \arrow[d, "\Phi^{0}_{12}\cong"]
&
CH_1(P_1)=\mathbb{Z}\langle w\rangle
  \arrow[d, "\Phi^{1}_{12}\cong"]
&
CH_2(P_1)=\mathbb{Z}\langle q\rangle
  \arrow[d, "\Phi^{2}_{12}=0"]
\\
CH_0(P_2)=\mathbb{Z}
&
CH_1(P_2)=\mathbb{Z}\langle w_2\rangle
&
CH_2(P_2)=0,
\end{tikzcd}
\]
where \(\mathbb{Z}\langle \eta\rangle=\{n\eta:n\in\mathbb Z\}\) denotes the free abelian group generated by some class \(\eta\). In our case, the class \(t\) represents the oriented phase circle of the periodic orbit \(T\), i.e., $S^1$. The class \(q\)  is represents the Hopf sphere \(\Sigma\), while the class \(w\) represents the loop obtained by following the curve \(W\) and closing it by a path on \(\Sigma\).  The loop represented by $w$ is not the phase circle of a single periodic orbit, but rather it is a global product-space loop created by the way the fixed-point curve $W$ and the Hopf sphere are glued together in \(M\times I\) - its image in \(H_1(N_2)\) is denoted by \(w_2\). We now justify our assertions about the maps in the diagram above:
\begin{itemize}
    \item The degree-zero maps are isomorphisms because all three neighborhoods are connected, and for any connected space \(X\) one has \(H_0(X;\mathbb Z)\cong\mathbb Z\). The inclusion of one connected neighborhood into another maps the generator (represented by any point) to the generator of the larger connected space.
    \item The map $\Phi_{01}^1:H_1(N_0)\to H_1(N_1)$ is zero because the circle \(T\) bounds a two-chain in the Hopf caps contained in \(N_1\) in the chain group \( C_2 (N_1; \mathbb Z) \), or equivalently its image is null-homologous in \(N_1\).  
    \item   The class \(w\), on the other hand, is not inherited from \(N_0\), rather, it is the new one-dimensional class of \(N_1\simeq S^2\vee S^1\). As such the map $\Phi_{12}^1:H_1(N_1)\to H_1(N_2)$ is an isomorphism, since adjoining the ball \(H\) fills only the sphere \(\Sigma\) and does not fill the loop represented by \(w\). Thus it follows  $\Phi_{12}^1(w)=\pm w_2$, where the sign depends on the orientation convention.
    \item Finally, $\Phi_{12}^2:H_2(N_1)\to H_2(N_2)$ is zero since  \(q=[\Sigma]\) becomes the boundary of the ball \(H\) in \(N_2\).\\
\end{itemize}

The kernel and cokernel of these transport maps have the following interpretation.  For a given map $\Phi_{ij}^q:CH_q(P_i)\to CH_q(P_j)$ the kernel consists of old degree \(q\) Conley classes that become trivial after the neighborhood is enlarged from \(N_i\) to \(N_j\). The cokernel is the quotient of \(CH_q(P_j)\) by the classes inherited from \(CH_q(P_i)\), so it records degree \(q\) Conley classes in the larger block that were \textbf{not} inherited from the smaller one. In our example, this can be explicitly interpreted as follows:
\[
\begin{aligned}
\ker \Phi_{01}^{1}=\mathbb{Z}\langle t\rangle
\;&:\;
\text{the periodic phase class dies;}\\
\operatorname{coker}\Phi_{01}^{1}
=
\mathbb{Z}\langle w\rangle
\;&:\;
\text{the global product-space loop is born;}\\
\operatorname{coker}\Phi_{01}^{2}
=
\mathbb{Z}\langle q\rangle
\;&:\;
\text{the Hopf sphere is born;}\\
\ker \Phi_{12}^{2}
=
\mathbb{Z}\langle q\rangle
\;&:\;
\text{the ball }H\text{ fills the Hopf sphere;}\\
\Phi_{12}^{1}(w)=\pm w_2
\;&:\;
\text{the global loop survives.}
\end{aligned}
\]
These maps between the nested pairs don't only record the appearance or disappearance of individual orbits, they also record the topology of the Hopf mechanism in the product space. We can conclude the following topological signature as follows: \textit{\textbf{$t$ dies, $q$ is born and dies, $w$ is born and survives.}} This signature expresses the collapse of the periodic orbit phase, the formation of a mirrored Hopf sphere, its filling by the interior source, and the survival of a global parameter-loop class.\\

The example above shows us that in some cases one can weaken the assumptions of Proposition \ref{decomposition} and analyze the dynamics of the Entropy flow forced by the topology of bifurcation (although that being said, unlike Proposition \ref{decomposition} this analysis does depend to some extent on the choice of the Entropy flow). We now generalize this discussion. Recall that the Entropy flow for a $C^1$ one--parameter family $\dot{s}=F_\tau(s)$, $s\in M,\tau\in I$ is a dynamical system defined on $M\times I$. Let us denote by $E$ the vector field generating the said Entropy flow, and by  $\varphi^t_E$  the corresponding flow function at time $t$. Given a compact set $N\subset M\times I$, recall we write $
\operatorname{Inv}_E(N) =\{x\in N:\varphi_E^t(x)\in N\ \text{for all }t\in\mathbb R\}$ for the maximal invariant set of \(E\) in \(N\). We now define:

\begin{definition} \label{def:entropy-index-ladder}
Let $\dot{s}=F_\tau(s)$ be a $C^1$ one parameter family on $M$ parameterized by $\tau\in I$, and let $E$ be some corresponding Entropy flow on $M\times I$. An {\textbf{entropy-adapted index filtration}} is a finite sequence of Conley index pairs $\mathcal P=\{P_i=(N_i,L_i)\}_{i=0}^m$  s.t. the following conditions hold:
\begin{enumerate}
    \item $N_i,L_i\subseteq M\times I$ for all $0\leq i<m$.
    \item $N_i\subset N_{i+1}$ for all $0\leq i<m$.
    \item $L_i=L_{i+1}\cap N_i$ for all $0\leq i<m$.
    \item Each $P_i=(N_i,L_i)$ is an index pair for the isolated invariant set $  S_i=\operatorname{Inv}_E\bigl(\overline{N_i\setminus L_i}\bigr)$, where $\operatorname{Inv}_E$ is as above.
\end{enumerate}
We write $CH_q(P_i;R):=H_q(N_i,L_i;R)$ for the degree \(q\) Conley index of index pair \(P_i\) with coefficients in a commutative ring \(R\).
\end{definition}

\begin{remark}\label{rem:index-filtration-morse}
The terminology above is meant in the classical Conley sense. The adjective ``entropy-adapted'' refers to the way the filtration is selected by the Entropy flow in the product space $M\times I$, not to a new object beyond index filtrations (see Definition 3.4 of \cite{Franzosa1986IndexFiltration}). The existence of an index filtration for a Morse decomposition is proven in Theorem 3.8 of \cite{Franzosa1986IndexFiltration}.\end{remark}

The condition \(L_i=L_{i+1}\cap N_i\) would be called \emph{\textbf{exit-set compatibility}}. It ensures that, for every $i<j$, \(L_i=L_j\cap N_i,\) and hence the inclusion of pairs $(N_i,L_i)\hookrightarrow (N_j,L_j)$ identifies the complex $C_*(N_i,L_i;R)$ with a subcomplex of $C_*(N_j,L_j;R)$. When each \(N_i\) is a trapping set for the flow (as in the example above), by definition $L_i=\emptyset$ for all $i$, so the exit-set compatibility condition is automatic. In this case an entropy-adapted index filtration is simply a nested sequence $N_0\subset N_1\subset\cdots\subset N_m$ of nested trapping sets (which are also by definition isolating neighborhoods). From now on, $E$ would always denote the vector field for some Entropy flow on $M\times I$, defined w.r.t. some $C^1$ one parameter family $\dot{s}=F_\tau(s)$ defined over some manifold $M$ of dimension at least $2$ and parameterized by $\tau\in I$. To simplify, we will often refer to $E$  as an \textbf{Entropy vector field}. We begin by proving how a Morse decomposition can be obtained from entropy-adapted index filtrations:
\begin{lemma}
Let $\mathcal P=\{P_i=(N_i,L_i)\}_{i=0}^m$ be an entropy-adapted index filtration and set $S_i=\operatorname{Inv}_E\bigl(\overline{N_i\setminus L_i}\bigr)$, where $E$ is a vector field for some Entropy flow. Then, the following holds:
\begin{itemize}
    \item $S_i\subset S_j$ for every $i<j$, and $S_i$ is an attractor in $S_j$. Consequently, $\emptyset=S_{-1}\subset S_0\subset S_1\subset\cdots\subset S_m$ is an attractor filtration.
    \item  If $M_i=S_i\cap (S_{i-1})^*_{S_i} =\{x\in S_i:\omega_E(x)\cap S_{i-1}=\emptyset\}, \;i=0,\ldots,m,$
where $(S_{i-1})^*_{S_i}$ denotes the dual repeller of $S_{i-1}$ in $S_i$ and $\omega_E$ denotes the $\omega$ limit set w.r.t. $E$, then, after deleting empty terms, the sets $M_i$ form a Morse decomposition of $S_m$ with admissible order  $M_m>M_{m-1}>\cdots>M_0$.
\end{itemize}

In particular, the difference $S_i\setminus S_{i-1}$ should not be identified with the Morse set $M_i$ defined above, as it could contain connecting orbits from $M_i$ to $S_{i-1}$.
\end{lemma}
\begin{proof}
The exit-set compatibility condition implies, for every $i<j$, that $L_i=L_j\cap N_i$. This yields $N_i\setminus L_i =N_i\setminus (L_j\cap N_i) \subset N_j\setminus L_j$, hence every complete orbit contained in $N_i\setminus L_i$ is contained in $N_j\setminus L_j$, and so $S_i\subset S_j$. We now show that $S_i$ is an attractor in $S_j$. Since $(N_i,L_i)$ is an index pair, we have $S_i\subset \operatorname{int}(N_i\setminus L_i)$. Choose an open neighborhood $U$ of $S_i$ s.t. $S_i\subset U\subset N_i\setminus L_i.$ Let $x\in U\cap S_j$. Since $x\in S_j$, both its forward and backwards trajectories w.r.t. $E$ are trapped inside $N_j\setminus L_j$. If the forward trajectory of $x$ left $N_i\setminus L_i$, then either it would hit $L_i$, or it would exit $N_i$. In the second case, the exit-set property of $(N_i,L_i)$ implies it must hit $L_i$ before leaving $N_i$. In either case the orbit hits $L_i\subset L_j$, contradicting the fact that it is contained in $N_j\setminus L_j$. All in all, the trajectory of $x$ w.r.t. $E$ is trapped in $N_i\setminus L_i$.

It follows that $\omega_E(x)\subset S_i$. Since $S_i\subset U\cap S_j$ is invariant, this proves $S_i$ attracts $U\cap S_j$ inside $S_j$, hence $S_i$ is an attractor in $S_j$, hence $\emptyset=S_{-1}\subset S_0\subset\cdots\subset S_m$ is an attractor filtration. For each $i$, the attractor-repeller decomposition of $S_i$ associated to the attractor $S_{i-1}$ gives the dual repeller $M_i=(S_{i-1})^*_{S_i}.$ Each nonempty $M_i$ is compact, invariant and isolated, therefore the standard construction of a Morse decomposition from an attractor filtration gives the family $\{M_i\}$ with order $M_m>\cdots>M_0$. In particular, any point not in one of the $M_i$ belongs to a connecting region whose alpha-limit lies in a higher layer and whose omega-limit lies in a lower layer. By applying the Lyapunov functions of the attractor-repeller pairs $(S_{i-1},M_i)$ inside $S_i$ we conclude no recurrent orbit remains outside the union of the $M_i$, i.e., the repeller pushes towards the attractor. Thus after omitting those indices \(i\) for which \(M_i = \emptyset\) (if necessary), the sets $M_i$ form a Morse decomposition of $S_m$.
\end{proof}
Motivated by the example of the Hopf bifurcation above, we now introduce the following definition:
\begin{definition}
\label{transportmap}
Given an entropy-adapted index filtration \(\mathcal P=\{P_i\}_{i=0}^m\) w.r.t. to an Entropy vector field $E$, the inclusion of pairs $P_i=(N_i,L_i)\hookrightarrow P_j=(N_j,L_j)$, $i\leq j$ induces a homomorphism $\Phi_{ij}^q:CH_q(P_i;R)\longrightarrow CH_q(P_j;R)$ for every homological degree \(q\). We refer to \(\Phi_{ij}^q\) as the {\textbf{degree \(q\) Entropy--Conley transport map}}. The collection
\[
        \mathsf{ECH}_q(E,\mathcal P;R)
        :=\big(CH_q(P_i;R),\Phi_{ij}^q\big)_{i\le j}
\]
will be defined as the \textbf{degree \(q\) Entropy--Conley transport module}.
\end{definition}

Exactly like in the case of the Hopf bifurcation above, for a fixed degree \(q\), the transport module records how degree \(q\) Conley classes behave as the Entropy flow block is enlarged along the filtration. Recalling how in our example above and recalling the kernels, images and cokernels encode the bifurcations themselves, we also introduce the following definition:
\begin{definition}
\label{birthdeathcode}
Let $E$ be some Entropy vector field with an associated degree $q$ Entropy--Conley transport module $\mathsf{ECH}_q(E,\mathcal P;R)$ as above. For a transport map $\Phi_{ij}^q:CH_q(P_i;R)\to CH_q(P_j;R)$ we define:

\[
\begin{aligned}
\operatorname{Death}_q(i,j)&:=\ker\Phi_{ij}^q,\\
\operatorname{Inherit}_q(i,j)&:=\operatorname{im}\Phi_{ij}^q,\\
\operatorname{Birth}_q(i,j)&:=\operatorname{coker}\Phi_{ij}^q.
\end{aligned}
\]

Thus \(\operatorname{Death}_q(i,j)\) consists of old degree \(q\) classes that become zero in the larger block, \(\operatorname{Inherit}_q(i,j)\) consists of classes that persist into the larger block, and \(\operatorname{Birth}_q(i,j)\) consists of new degree \(q\) classes in \(P_j\) modulo those inherited from \(P_i\).  
\end{definition}
Having generalized the ideas introduced in the example of the Hopf bifurcation, we now prove the following Lemma:
\begin{lemma}
Let \(\mathcal P=\{P_i\}_{i=0}^m\) be an entropy-adapted index filtration w.r.t. some Entropy vector field $E$. Then, for every \(q\), the assignment 
\[
        i\longmapsto CH_q(P_i;R),
        \qquad
        (i\le j)\longmapsto \Phi_{ij}^q
\]
defines a functor from the ordered set \(\{0<1<\cdots<m\}\) to \(R\)-modules.  In particular, we have $\Phi_{ii}^q=\operatorname{id}$ and $\Phi_{jk}^q\circ \Phi_{ij}^q=\Phi_{ik}^q$ for all $i\le j\le k$.
\end{lemma}

\begin{proof}
We first note the inclusion of pairs $(N_i,L_i)\hookrightarrow (N_j,L_j)$ induce an inclusion of relative singular chain complexes $C_*(N_i,L_i;R)\longrightarrow C_*(N_j,L_j;R).$ This chain map induces the homomorphism \(\Phi_{ij}^q\) on homology. By definition, the identity inclusion induces the identity chain map and the composition of inclusions induces the composition of the corresponding chain maps.  Passing to homology now gives $\Phi_{ii}^q=\operatorname{id},\;
        \Phi_{jk}^q\circ \Phi_{ij}^q=\Phi_{ik}^q$ - to see why,  this is just the usual functoriality of singular homology for pairs (see Section~2.1 in \cite{hatcher2002algebraictopology}).
\end{proof}

We continue by introducing the relative groups which record the bifurcation topology added between two levels of the filtration:
\begin{definition}
\label{transitiongroup}
Let \(P_i=(N_i,L_i)\subset P_j=(N_j,L_j)\) be two levels of an entropy-adapted index filtration defined w.r.t. some Entropy vector field $E$.  Define the \textbf{transition group} by
\[
        TR_q(i,j;R)
        :=
        H_q\big(C_*(N_j,L_j;R)/C_*(N_i,L_i;R)\big),
\]
where the \(C_*(N_j,L_j;R)/C_*(N_i,L_i;R)\) refers to the full graded quotient chain complex. Equivalently, this is the relative homology of the inclusion of pairs $(N_i,L_i)\subset (N_j,L_j)$.
\end{definition}
\begin{remark}
    Notice that when \(L_i=L_j=\emptyset\), the transition group reduces to the usual relative homology group $TR_q(i,j;R)=H_q(N_j,N_i;R).$
\end{remark}
The transition group records the relative Conley homology added as we move between two nested index pairs \(\mathcal P_i\subset \mathcal P_j\).  Algebraically, it is the homology of the quotient chain complex $C_*(\mathcal P_j;R)/C_*(\mathcal P_i;R).$ Dynamically, it represents the homological contribution of the part of the Entropy flow block which appears when passing from the smaller isolating block to the larger one. It records both the newly appearing classes in $\mathcal{P}_j$, and how these new classes are attached to the classes already present in \(\mathcal P_i\). In particular, the transition group measures the relative change in the isolated dynamics as the Entropy flow passes through the bifurcation region. It follows, the connecting homomorphism in the transition exact sequence, given by
\[
\cdots\to
CH_q(P_i;R)
\xrightarrow{\Phi^q_{ij}}
CH_q(P_j;R)
\to
TR_q(i,j;R)
\xrightarrow{\partial^q_{ij}}
CH_{q-1}(P_i;R)
\to\cdots,
\]
is a boundary operator, as it identifies which relative transition classes kill old Conley classes. We now prove the following Lemma:

\begin{lemma}
\label{exactlemma} For every inclusion \(P_i\subset P_j\) in an entropy-adapted index filtration w.r.t. some Entropy vector field $E$, there is a long exact sequence:
\[
\cdots
\to
CH_q(P_i;R)
\xrightarrow{\Phi_{ij}^q}
CH_q(P_j;R)
\to
TR_q(i,j;R)
\xrightarrow{\partial_{ij}^q}
CH_{q-1}(P_i;R)
\xrightarrow{\Phi_{ij}^{q-1}}
CH_{q-1}(P_j;R)
\to
\cdots,
\]
where $\operatorname{im}\partial_{ij}^q = \ker\Phi_{ij}^{q-1}$. That is, if a class \(\alpha\in CH_{q}(P_i;R)\) is mapped to zero by the transport map \(\Phi_{ij}^{q-1}\), then there exists a class \(\beta\in TR_q(i,j;R)\) s.t. $\partial_{ij}^q(\beta)=\alpha$. In other words, the transition class \(\beta\) records the homological mechanism by which the old class \(\alpha\) is canceled when it moves from \(P_i\) to \(P_j\).
\end{lemma}

\begin{proof}
The inclusion of relative chain complexes and exit-set compatibility defines a short exact sequence:
\[
0
\to
C_*(N_i,L_i;R)
\to
C_*(N_j,L_j;R)
\to
C_*(N_j,L_j;R)/C_*(N_i,L_i;R)
\to
0.
\]
Every short exact sequence of chain complexes induces a long exact sequence in homology. Recall the definitions $CH_q(P_i;R)=H_q(N_i,L_i;R)$, $CH_q(P_j;R)=H_q(N_j,L_j;R)$,  $TR_q(i,j;R)=H_q\big(C_*(N_j,L_j;R)/C_*(N_i,L_i;R)\big)$. By plugging them in the above we obtain the stated sequence. To conclude, note the equality $\operatorname{im}\partial_{ij}^q = \ker\Phi_{ij}^{q-1}$ follows immediately from the exactness of this sequence at \(CH_{q-1}(P_i;R)\). (this is the standard long exact sequence associated to a short exact sequence of chain complexes - see \cite{hatcher2002algebraictopology}).
\end{proof}
\begin{remark}\label{rem:exit-set-compatibility}
Put simply, Lemma \ref{exactlemma} states that whenever a Conley class becomes zero under the transport map the transition exact sequence detects this disappearance by exhibiting a relative transition class whose boundary is the old class. We further remark the short exact sequence above is not derived automatically from $N_i\subset N_j, \; L_i\subset L_j$. To see why, note the relative chain complex is defined by $C_*(N_i,L_i;R)=C_*(N_i;R)/C_*(L_i;R)$. If a chain is supported in $N_i\cap L_j$ but not in $L_i$, it could potentially define a nonzero class in $C_*(N_i,L_i;R)$ while becoming zero in $C_*(N_j,L_j;R)$, which would make the induced map $C_*(N_i,L_i;R)\rightarrow C_*(N_j,L_j;R)$ non-injective. The condition $L_i=L_j\cap N_i$ rules out this possibility, as it guarantees the inclusion of pairs identifies $C_*(N_i,L_i;R)$ with a subcomplex of $C_*(N_j,L_j;R)$, which makes the quotient complex $C_*(N_j,L_j;R)/C_*(N_i,L_i;R)$ well-defined.
\end{remark}

Lemma \ref{exactlemma} gives a criterion for detecting homologically visible bifurcations - namely, a change of Conley homology between two nested Entropy flow blocks is detected by the failure of \(\Phi_{ij}^q\) to be an isomorphism. Note this was seen in the Hopf example earlier this Section -  for example, in the map \(\Phi_{12}^2\) which send the class \(q = [\Sigma] \in CH_2 (P_1)\) to zero in \(CH_2 (P_2)\). As the transition class represented by the ball \(H\) has boundary \(\partial[H]=[\Sigma]=q\), so the boundary map records how the Hopf sphere is killed. Moreover, similarly to the said example, a nonzero kernel implies an old class becomes null-homologous in the larger block, while a nonzero cokernel means that the larger block contains a class that is not inherited from the smaller block. Consequently, the boundary map $\partial_{ij}^q:TR_q(i,j;R)\to CH_{q-1}(P_i;R)$ records how a relative transition class attaches to the older block. If the boundary of such a relative class is an old Conley class, then that old class becomes null-homologous after the transition. Summarizing our ideas, we now introduce the following definition:

\begin{definition}
\label{bifurcationcode}
Let \(E\) be an Entropy vector field on \(M\times I\) for some $C^1$ one parameter family \(\dot{s}=F_\tau(s)\) parameterized by $\tau\in I$. Let $Y\subset M\times I$ be a compact set satisfying the following:
\begin{itemize}
    \item $Y$ is an isolated invariant w.r.t. $E$.
    \item $Y\subseteq\overline{Per\cup Fix}$, i.e., $Y$ is composed of components of $Per$ and $Fix$ glued together at bifurcation orbits and fixed points.
\end{itemize}

A \textbf{bifurcation code} for \(Y\) over a coefficient ring \(R\) (typically, $\mathbb{Z}$) is an entropy-adapted index filtration $\mathcal P_Y  = \{P_{Y,i}=(N_{Y,i},L_{Y,i})\}_{i=0}^{m_Y}$ adapted to \(Y\)  (i.e., for all $i$, $N_{Y,i}\cap Y\ne\emptyset$ and $Y = \Inv_E (\overline{N_Y,m_Y \backslash L_Y,m_Y})$, together with its induced transport maps and transition groups
\[
\mathfrak B_Y(E;R)
=
\left(
\mathcal P_Y,\,
\{CH_*(P_{Y,i};R)\}_{i},\,
\{\Phi^*_{ij}\}_{i\le j},\,
\{TR_*(i,j;R),\partial_{ij}\}_{i<j}
\right),
\]
where $\Phi^q_{ij}:CH_q(P_{Y,i};R)\longrightarrow CH_q(P_{Y,j};R)$ are the maps induced by the inclusion of index pairs $P_{Y,i}\hookrightarrow P_{Y,j}.$ The ordered list $P_{Y,0}\subset P_{Y,1}\subset\cdots\subset P_{Y,m_Y}$ will be called the \textbf{bifurcation sequence} associated to \(Y\).
\end{definition}

Intuitively, one could think of a bifurcation code as a Conley-level resolution of the Entropy flow near \(Y\) (needless to say, the code depends on both the choice Entropy flow, and on the chosen index filtration). Thus it should not be interpreted as an intrinsic classification of all bifurcations in the original parameter family $\dot{s}=F_\tau(s)$, but rather, it is a tool for studying how a given Entropy flow organizes the isolated bifurcation structure near \(Y\). We now prove the following analogue of Proposition \ref{decomposition} for bifurcation codes:

\begin{theorem}
\label{bifcodesth1}
Let \(E\) be an Entropy vector field on \(M\times I\), let $Y_1,\dots,Y_k\subseteq \overline{Per\cup Fix}$ be finitely many compact subsets, and let \(R\) be a coefficient ring. Assume that each \(Y_\alpha\), \(\alpha=1,\dots,k\), admits a bifurcation code given by
\[
\mathfrak B_\alpha(E;R)
=
\left(
\mathcal P_\alpha,\,
\{CH_*(P_{\alpha,i};R)\}_{i},\,
\{\Phi_{\alpha,ij}^*\}_{i\le j},\,
\{TR_{\alpha,*}(i,j;R),\partial_{\alpha,ij}\}_{i<j}
\right),
\]
where $\mathcal P_\alpha = \{P_{\alpha,i}=(N_{\alpha,i},L_{\alpha,i})\}_{i=0}^{m_\alpha}$ is an entropy-adapted index filtration adapted to \(Y_\alpha\). Further, assume that the last isolating blocks in each filtration, $N_{\alpha,m_\alpha}, \alpha=1,\dots,k,$ are pairwise disjoint, and that their union isolates the part of the Entropy flow generated by $Y_1\cup\cdots\cup Y_k.$ Then the collection $\mathfrak B(E;R) =\{\mathfrak B_\alpha(E;R)\}_{\alpha=1}^k$ encodes the Conley-level behavior of \(E\) on this isolated region in the following sense:

\begin{itemize}
\item For every \(\alpha\), every degree \(q\), and every \(i\le j\), the transport map $\Phi_{\alpha,ij}^q:
        CH_q(P_{\alpha,i};R)
        \longrightarrow
        CH_q(P_{\alpha,j};R)$ records the degree \(q\) Conley classes inherited from the smaller block \(P_{\alpha,i}\) to the larger block \(P_{\alpha,j}\).

\item The kernel $\ker\Phi_{\alpha,ij}^q$ records the degree \(q\) Conley classes in \(P_{\alpha,i}\) which become zero after transport to \(P_{\alpha,j}\). Equivalently, these are the degree \(q\) classes which disappear when the index filtration is enlarged from \(P_{\alpha,i}\) to \(P_{\alpha,j}\).

\item The cokernel  $\operatorname{coker}\Phi_{\alpha,ij}^q$ records the degree \(q\) Conley classes in \(P_{\alpha,j}\) which are \textbf{not} inherited from \(P_{\alpha,i}\). These are the new degree \(q\) classes created by the transition from \(P_{\alpha,i}\) to \(P_{\alpha,j}\).

\item The connecting homomorphism $\partial_{\alpha,ij}^q: TR_{\alpha,q}(i,j;R) \longrightarrow CH_{q-1}(P_{\alpha,i};R)$ identifies how relative transition classes attach to the older block. More precisely, by exactness of the transition sequence, we have $\operatorname{im}\partial_{\alpha,ij}^q = \ker\Phi_{\alpha,ij}^{q-1}$. Thus the image of \(\partial_{\alpha,ij}^q\) consists precisely of the older degree \(q-1\) Conley classes which become null-homologous as we move from \(P_{\alpha,i}\) to \(P_{\alpha,j}\).
\end{itemize}

In particular, the filtration detects a Conley-level transition in degree \(q\) whenever the transport map \(\Phi_{\alpha,ij}^q\) fails to be an isomorphism.  The transition boundary maps further record which older classes become boundaries of relative transition classes in the enlarged block.
\end{theorem}

\begin{proof}
Fix an isolated set \(Y_\alpha\) - by assumption, \(Y_\alpha\) has an entropy-adapted index filtration which we denote by $\mathcal P_\alpha =\{P_{\alpha,i}=(N_{\alpha,i},L_{\alpha,i})\}_{i=0}^{m_\alpha}.$ Note that for every \(i\le j\), the inclusion of index pairs $P_{\alpha,i}\hookrightarrow P_{\alpha,j}$ induces a map on relative homology, given by
\[
        \Phi_{\alpha,ij}^q:
        H_q(N_{\alpha,i},L_{\alpha,i};R)
        \longrightarrow
        H_q(N_{\alpha,j},L_{\alpha,j};R).
\]
Or equivalently:
\[
        \Phi_{\alpha,ij}^q:
        CH_q(P_{\alpha,i};R)
        \longrightarrow
        CH_q(P_{\alpha,j};R).
\]
By the functoriality of singular homology, these maps satisfy both $\Phi_{\alpha,ii}^q=\operatorname{id}$ and $\Phi_{\alpha,j\ell}^q\circ \Phi_{\alpha,ij}^q = \Phi_{\alpha,i\ell}^q$ for all $i\le j\le \ell$. Thus, for each fixed degree \(q\), the assignments $i\longmapsto CH_q(P_{\alpha,i};R)$ and $(i\le j)\longmapsto \Phi_{\alpha,ij}^q$ is a functor from the finite ordered category \(\{0<1<\cdots<m_\alpha\}\) to the category of \(R\)-modules. 

To continue, consider a transition from $P_{\alpha,i}$ to $ P_{\alpha,j}$, where $P_{\alpha,i}\subset P_{\alpha,j}$. The long exact sequence of the triple of chain complexes, equivalently of the quotient complex $C_*(P_{\alpha,j};R)/C_*(P_{\alpha,i};R)$ yields the exact sequence:
\[
\cdots
\to
CH_q(P_{\alpha,i};R)
\xrightarrow{\Phi_{\alpha,ij}^q}
CH_q(P_{\alpha,j};R)
\to
TR_{\alpha,q}(i,j;R)
\xrightarrow{\partial_{\alpha,ij}^q}
CH_{q-1}(P_{\alpha,i};R)
\xrightarrow{\Phi_{\alpha,ij}^{q-1}}
CH_{q-1}(P_{\alpha,j};R)
\to
\cdots .
\]
While the exactness implies
\[
        \operatorname{im}\partial_{\alpha,ij}^q
        =
        \ker\Phi_{\alpha,ij}^{q-1}.
\]
Therefore, whenever an old Conley class dies under the transport map in degree \(q-1\), it appears as the boundary of a relative transition class in degree \(q\).  This proves that the boundary map identifies the transition layer that attaches to, or kills, older Conley classes.  The interpretations of kernels, images, and cokernels as death, inheritance, and birth are then precisely the algebraic meanings of the transport maps. The final claim follows immediately.
\end{proof}

\begin{remark}
Note Theorem \ref{bifcodesth1} is graded, i.e., a bifurcation code may be trivial in one degree $q$ and nontrivial in another degree $q'$. Thus, if all degree \(q\) transport maps are isomorphisms, this means the chosen filtration detects no Conley-level transition in degree \(q\).  It does not rule out births, deaths, or transition attachments in a different degree \(q'\). Therefore, we say a bifurcation is detected by the bifurcation code if such a nontrivial event occurs in at least one degree.
\end{remark}
Put simply, Theorem \ref{bifcodesth1} establishes the connection between the topology of certain bifurcations and how the Entropy flow sees them. In addition to giving concrete description how bifurcations are seen by the Entropy flow, this Theorem (much like Proposition \ref{decomposition}) should be thought of as a Sharkovskii-like result for bifurcations. To illustrate, recall the Sharkovskii Theorem (see \cite{Shar}) and the Li-Yorke Theorem (see \cite{Yor}) force the existence of certain periodic dynamics for interval maps provided certain topological conditions are met. As Theorem \ref{bifcodesth1} and Proposition \ref{decomposition} allow us to associate Conley indices with some bifurcations and analyze their structure, they induce topological constrains on the Entropy flow - and in doing so, they impose constrains on the bifurcations of periodic orbits in $Per$ and fixed points in $Fix$.\\

That being said, we would like to stress Theorem \ref{bifcodesth1} \textbf{does not} claim every bifurcation is detected, nor that the underlying invariant set for an Entropy flow is uniquely determined by its Conley index. Rather, once an isolated Entropy flow block and an adapted index filtration have been chosen, the code records how Conley classes are inherited, born, killed, or attached along the bifurcation sequence. This is the sense in which bifurcation codes provide a criterion for detecting bifurcation structure through the Entropy flow - in particular, the bifurcation code depends on the choice of the Entropy flow. This Theorem is best seen as a bifurcation analogue to the way braid types force prescribed dynamics to appear for planar homeomorphisms, in the sense that once a braid type is detected we can study what dynamics are forced by it (see \cite{CH}). Similarly, provided we can verify a certain component in $\overline{Per\cup Fix}$ exists and can be isolated, we can associate a bifurcation code with it w.r.t. some Entropy flow and study how it constrains its dynamics. This discussion leads us to:

\begin{definition}
\label{dynamicallyindependent}
Let $P_\alpha=(N_\alpha,L_\alpha), \alpha=1,\ldots,k,$ be terminal index pairs as in Theorem \ref{bifcodesth1}, and set $S_\alpha = \operatorname{Inv}_E\big(\overline{N_\alpha\setminus L_\alpha}\big)$, where $E$ is an Entropy vector field. We now define the following notions:

\begin{itemize}
    \item  Let \(N_{\mathrm{amb}}\) be an ambient isolating neighborhood containing all sets \(S_\alpha\).  A \textbf{complete \(E\)-orbit in \(N_{\mathrm{amb}}\)} is a solution curve for $E$ given by \(\gamma:\mathbb R\to N_{\mathrm{amb}}\) s.t. 
$\gamma(t)=\varphi_E(t,\gamma(0))$ for all $t\in\mathbb R$, and \(\gamma(t)\in N_{\mathrm{amb}}\) for all \(t\in\mathbb R\) (where $\varphi_E$ denotes the corresponding flow function).  
\item For \(x\in \operatorname{Inv}_E(N_{\mathrm{amb}})\), we write $\alpha_E(x) = \bigcap_{T>0}\overline{\{\varphi_E(t,x):t\le -T\}}$ and $\omega_E(x)  =  \bigcap_{T>0}\overline{\{\varphi_E(t,x):t\ge T\}}$ for the alpha- and omega-limit sets of the complete orbit through \(x\). For two isolated invariant sets \(S_\alpha,S_\beta\subset\operatorname{Inv}_E(N_{\mathrm{amb}})\), define the set of connecting points from \(S_\alpha\) to \(S_\beta\) inside \(N_{\mathrm{amb}}\) by
$$
C_E(S_\alpha,S_\beta;N_{\mathrm{amb}})
=
\left\{
 x\in \operatorname{Inv}_E(N_{\mathrm{amb}})
 \setminus(S_\alpha\cup S_\beta)
 :
 \alpha_E(x)\subset S_\alpha,
 \omega_E(x)\subset S_\beta
\right\}.$$
We say the terminal blocks \(P_1,\ldots,P_k\) are \textbf{dynamically independent} in \(N_{\mathrm{amb}}\) if for all indices $\alpha\neq\beta$ we have $C_E(S_\alpha,S_\beta;N_{\mathrm{amb}})=\emptyset$.
\item If, in addition,  $\operatorname{Inv}_E(N_{\mathrm{amb}})= \bigsqcup_{\alpha=1}^k S_\alpha$ (where \(\bigsqcup\) denotes disjoint union), then the terminal blocks are said to give a \textbf{complete independent decomposition} of the isolated Entropy flow dynamics in \(N_{\mathrm{amb}}\).
\end{itemize}

\end{definition}

We conclude our discussion of bifurcation codes for Entropy flows with the following criterion, describing when component bifurcation codes can be assembled by direct sum:

\begin{corollary}
\label{cor:independent-direct-sum}
Let \(\mathfrak B_\alpha(E;A)\), \(\alpha=1,\ldots,k\), be bifurcation codes associated to entropy-adapted index filtrations $\mathcal P_\alpha=  \{P_{\alpha,i}\}_{i=0}^{m_\alpha}$ w.r.t. some Entropy vector field $E$. Write the terminal pair of the \(\alpha\)-th filtration as $P_{\alpha,m_\alpha}=(N_\alpha,L_\alpha)$ and set $S_\alpha =
        \operatorname{Inv}_E\big(\overline{N_\alpha\setminus L_\alpha}\big).$ Assume the terminal sets \(S_1,\ldots,S_k\) give a complete independent decomposition inside an ambient isolating neighborhood \(N_{\mathrm{amb}}\), that is, $C_E(S_\alpha,S_\beta;N_{\mathrm{amb}})=\emptyset$ for all $\alpha\neq\beta,$ and $\operatorname{Inv}_E(N_{\mathrm{amb}}) = \bigsqcup_{\alpha=1}^k S_\alpha.$
Then, the Conley homology of the global isolated invariant set decomposes as follows:
\[
        CH_q\big(\operatorname{Inv}_E(N_{\mathrm{amb}});R\big)
        \cong
        \bigoplus_{\alpha=1}^k CH_q(S_\alpha;R).
\]
Equivalently, the terminal part of the global bifurcation code is the direct sum of the terminal component codes. Consequently, if there exists a connecting orbit from some \(S_\alpha\) to some \(S_\beta\) inside \(N_{\mathrm{amb}}\), the component codes alone do not encode the full isolated Entropy flow dynamics and must be enlarged to a transition block containing the connecting orbit and include the corresponding transition group and boundary map. 
\end{corollary}

\begin{proof}
Since  $\operatorname{Inv}_E(N_{\mathrm{amb}}) = \bigsqcup_{\alpha=1}^k S_\alpha$ and as there are no connecting orbits between distinct \(S_\alpha\)'s in \(N_{\mathrm{amb}}\), the isolated invariant set decomposes as a disjoint union of independent isolated invariant sets. Therefore, choose pairwise disjoint isolating neighborhoods \(U_\alpha\subset N_{\mathrm{amb}}\) for \(S_\alpha\), with index pairs $Q_\alpha=(U_\alpha,K_\alpha)$, and define $U=\bigsqcup_{\alpha=1}^k U_\alpha, \; K=\bigsqcup_{\alpha=1}^k K_\alpha.$ It follows $Q=(U,K)$ is an index pair for the disjoint union \(\bigsqcup_{\alpha=1}^k S_\alpha\).  Now, the relative singular chain complex splits as
\[
        C_*(U,K;R)
        \cong
        \bigoplus_{\alpha=1}^k C_*(U_\alpha,K_\alpha;R).
\]
While passing to homology yields:
\[
CH_q\left(\bigsqcup_{\alpha=1}^k S_\alpha;R\right)
\cong
\bigoplus_{\alpha=1}^k
CH_q(S_\alpha;R),
\]
which proves the direct sum decomposition. Consequently, if $E$ generates a connecting orbit from \(S_\alpha\) to \(S_\beta\), then \(\operatorname{Inv}_E(N_{\mathrm{amb}})\) is not the disjoint union of the terminal sets. As the connecting orbit belongs to a transition region in the ambient isolated dynamics, the global code contains additional transition data encoded by an enlarged index filtration and its transition groups and boundary maps. Thus, whenever such a connecting orbit exists the component codes alone do not encode the full Entropy flow dynamics.
\end{proof}
Before concluding this Subsection, we remark the entropy-adapted index filtration is very much in lines with the classical Conley index Theory. To see why, recall that in classical Conley index Theory one starts with a Morse decomposition of an isolated invariant set, chooses an admissible ordering, and obtains an index filtration. The associated homology index and the connection matrix then encode the homology indices of the Morse sets and the connecting orbits between them (see \cite{Franzosa1989connectionmatrixtheory}). Moreover, nested index filtrations also appear in the continuation theory of connection matrices (see \cite{nestedindexpair}), while related persistence-based approaches to changing Conley data and combinatorial bifurcations have also appeared in recent works  \cite{Mrozek2020PersistenceoftheConleyIndex}, \cite{barcodepreprint}. That being said, the point of the entropy-adapted index filtration is different, as the objects in the filtration are compatible index pairs $ P_0\subset P_1\subset\cdots\subset P_m$ chosen from the geometry of the route through bifurcations, hence it differs in the interpretation of its maps as topological records of bifurcation transitions. In other words, the connection-matrix formalism is not replaced by our construction, but on the contrary, applied to study the bifurcation structure via the Entropy flow.  \\

Finally, we remark that even though our results hold for all dimensions, in practice, our motivating examples in this Subsection arose from $C^1$ one--parameter families $\dot{s}=F_\tau(s)$ defined on closed surfaces $M$. The reason these examples work so well is because due to the Kupka-Smale Theorem, Peixoto's Theorem and Sotomayor's Theorem, for such families the structure of the set $\overline{Per\cup Fix}$ is generically "well behaved" (see \cite{Pei}, \cite{Ku}, \cite{Smale} and \cite{soto}). This leads the following question: are these ideas also well behaved in higher dimensions? Say, for flows that are \textbf{not} Morse-Smale? Unfortunately, the answer is likely to be extremely case dependent. That being said, for the sake of completeness, in the next Subsection we study the Entropy flow for the Lorenz system using the Conley index Theory. Moreover, in Appendix \ref{periodappendix} we  discuss specifically how, under certain idealized assumptions, one can apply the Conley index method to describe the behavior of the Entropy flow near period doubling bifurcations. As such, the analysis in Appendix \ref{perioddoubling} proves that even though our theory is likely to take a different shape in each dimension, it can be meaningful also for higher dimensional flows.

\subsection{Case study: the Lorenz system}
\label{lorenz}

In this Subsection we will study the Entropy flow for the Lorenz system using the Conley index. We will not use the theory we developed in the previous Subsection for reasons which will become apparent later on - that being said, our results will illustrate how one can study Entropy flows in the absence of concrete information. To begin, recall the Lorenz system (see \cite{Lo}) is given by the following equations, where $\sigma,\rho,\beta$ are positive parameters:
\begin{equation} \label{Vect}
\begin{cases}
\dot{x} = \sigma(y-x) \\
 \dot{y} = x(\rho-z)-y\\
 \dot{z}=xy-\beta z
\end{cases}
\end{equation}

Now, let $\dot{s}=F_\tau(s)$, $s\in \mathbb{R}^3$, $\tau\in  I$, denote a $C^1$ family of Lorenz systems, where $\tau$ varies in $I$, some $C^1$-curve in the parameter space. It is well-known that for all such parameters there exists an attracting ellipsoid $\Delta_\tau$ for $F_\tau$, which attracts every initial condition and includes all the fixed points and periodic orbits for the flow \cite{Lo}. We first prove  the following Lemma:
\begin{lemma}
\label{choice} Let $\dot{s}=F_\tau(s)$ be a $C^1$ one--parameter family of Lorenz systems as above and let $E$ be an Entropy vector field generating some Entropy flow corresponding to it. Then, whenever $I$ is compact, there exists a trapping set $\Delta\subseteq \mathbb{R}^3\times I$ for $E$, which attracts every initial conditions in $M\times I$. Moreover, the set $\Delta$ includes all the periodic orbits and fixed points for $E$, and we can decompose $\Delta=\cup_{\tau\in I}\Delta_\tau\times\{\tau\}$.
\end{lemma}

\begin{proof}
Let us write $E(s,\tau)=(F_\tau(s)+V(s,\tau),P(s,\tau))$ for all $(s,\tau)\in\mathbb{R}^3\times I$. We first recall that for every $\tau$, its periodic orbits are trapped in $\Delta_\tau$. By the definition of the Entropy flow, this implies that provided $\Delta_\tau$ is sufficiently large, both $P$ and $V$ vanish in $\mathbb{R}^3\times\{\tau\}\setminus\Delta_\tau\times\{\tau\}$. By the compactness of $I$ it follows we can choose $\Delta_\tau$ to vary smoothly with $\tau$ s.t. $P$ and $V$ both vanish on $X=(\mathbb{R}^3\times I)\setminus\Delta$, where $\Delta=\cup_{\tau\in I}\Delta_\tau\times\{\tau\}$. This yields that for all $(s,\tau)\in X$ we have $E(s,\tau)=(F_\tau(s),0)$, i.e., on $X$ the Entropy flow coincides with the Laminar flow (see Equations \ref{laminar}). Consequently, as $F_\tau$ is transverse to $\partial \Delta_\tau$ and points inside it, by Lemma \ref{trapping} we conclude $E$ is transverse to $\partial \Delta$ and points inside it. Moreover, since every initial condition $s\in\mathbb{R}^3$ is attracted into $\Delta$ by $F_\tau$ we conclude that for every initial condition $(s,\tau)\in X$ its trajectory eventually enters $\Delta$ and cannot escape it. Finally, since for any $\tau\in I$ all the periodic orbits of $F_\tau$ are in $\Delta_\tau$ it follows $\overline{Per\cup Fix}\subseteq\Delta$, and the conclusion follows.\end{proof}

From Lemma \ref{choice} it follows that in order to analyze the behavior of the Entropy flow for the Lorenz system, it is enough to analyze its behavior strictly inside $\Delta$. Unfortunately, as the bifurcation structure of $\overline{Per\cup Fix}$ inside $\Delta$ is likely to be extremely complicated (see, for example, \cite{BS}, \cite{BSS}, and the references therein), a direct analytic approach as we did in \ref{vanderpol} is unlikely to work. Instead, we will take an alternative approach and give a global description of the dynamics of the Entropy flow in $\Delta$ using the Conley index.\\

To begin, from now on we will always assume $I=S^1$, i.e., that the family of Lorenz systems $\dot{s}=F_\tau(s)$ is parameterized by some closed loop in the parameter space of the Lorenz system. Fix some Entropy flow for $\dot{s}=F_\tau(s)$, and as before, let $E$ denote the vector field directing it. As $\infty$ is a repeller for the Lorenz system, without any loss of generality we can extend each vector field $F_\tau$ to $S^3$ by setting a repelling fixed point at $\infty$. This allows us to trivially extend the vector field $E$ to $S^3\times S^1$. In addition, to comply with the numerical studies of the Lorenz system, from now on we will assume that we choose our parameter space s.t. for all $\tau\in S^1$ the fixed point at the origin is a real saddle with a two-dimensional stable manifold and a one-dimensional unstable manifold (see Figure \ref{b0}). We proceed by introducing the following notations:
\begin{itemize}
    \item $(S^3\times S^1)\setminus\Delta=S_\infty$. It is easy to see we can decompose $S_\infty=B_\infty\times S^1$, where $B_\infty$ is a small neighborhood of $\infty$ in $S^3$ defined s.t. for all parameters $\tau\in S^1$, $F_\tau$ points on $\partial B_\infty$ into $\Delta=S^3\setminus B_\infty$. 
    \item Similarly, let $B_0$ be a neighborhood of the saddle at the origin, s.t. on $\partial B_0$ the vector field $F_\tau,\tau\in I$ has the following behavior: there exist two caps, $B_1,B_2\subseteq \partial B_0$ on which $F_\tau$ points into $S^3\setminus B_0$, separated by an annulus $B_3$ on which $F_\tau$ points inside $B_0$ (see Figure \ref{b0}). Writing $S_0=B_0\times S^1$, it is immediate $\partial S_0$ can be decomposed into $\partial S_0=\mathbb{T}_1\cup\mathbb{T}_2\cup \mathbb{T}_3$, where $\mathbb{T}_i=B_i\times S^1$, $i\in\{1,2,3\}$.\\
\end{itemize}

\begin{figure}[h]
\centering
\begin{overpic}[width=0.45\textwidth]{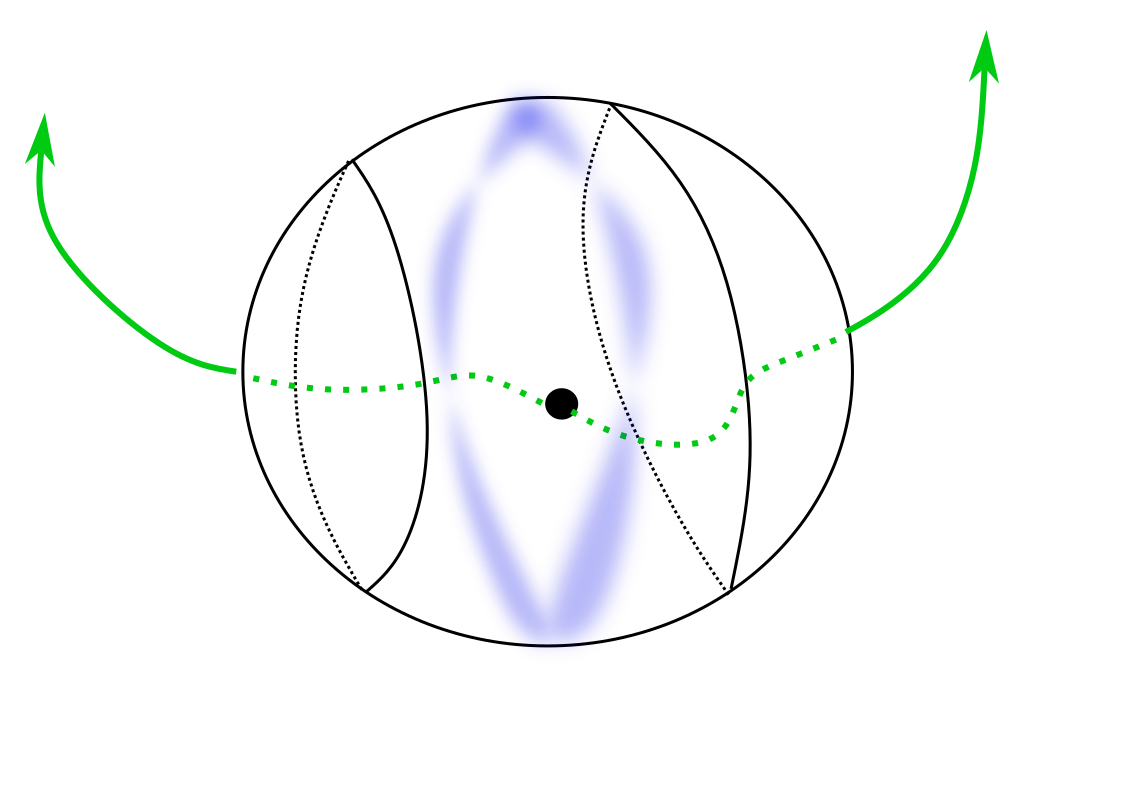}
\put(280,430){$B_1$}
\put(450,430){$B_3$}
\put(690,430){$B_2$}
\end{overpic}
\caption{\textit{The ball $B_0$ surrounding the saddle at the origin. Flow lines (including the separatrices) escape it via the caps $B_1$ and $B_2$, while flow lines enter $B_0$ via the strip $B_3$ (the blue ring on $B_3$ denotes the intersection with the two-dimensional stable manifold of the origin).}}
\label{b0}
\end{figure}

We now compute the Conley index of the set $\mathcal{B}=S^3\times S^1\setminus(S_0\cup S_\infty)$ - by decreasing the volume of the ball $B_0$ (thus decreasing the volume of $S_0$), we can ensure the invariant set of the Entropy flow in $\mathcal{B}$ includes an arbitrarily large subset of the dynamics in $\Delta$. By its definition, the space $\mathcal{B}$ is just $S^3 \times S^1$ from which we delete two four-dimensional tubes defined by $S_\infty=B_\infty \times S^1$  and $S_0 = B_0 \times S^1$. Consequently, $S^3 \setminus  (B_\infty \cup B_0) $ is homeomorphic to $ S^2 \times [0,1]$, which implies $\mathcal{B} $ is homeomorphic to $ S^2 \times [0,1] \times S^1$. Before we move on, we add an underlying assumption, inspired by the numerical studies of the Lorenz system: for all $\tau\in S^1$, the periodic orbits for $F_\tau$ are dense around the origin. We will refer to this assumption as \textbf{the saddle assumption}.\\ 

To justify this assumption, recall that as proven in \cite{TUC}, there exists an open range of parameters, denoted by $O$, where the dynamics of the Lorenz attractor are orbitally equivalent to those of the Geometric Lorenz Model (see \cite{SAB}). Moreover, recall that as proven in \cite{CI} there exists an open set of parameters where the dynamics of the first-return map for Lorenz system in $\Delta_\tau$ can be semiconjugated to Lorenz maps. Both of these conditions imply that one should expect the periodic orbits of $F_\tau$ to be dense around the saddle at the origin. It now follows by the definition of the Entropy flow in Section \ref{defsect} that under the saddle assumption, for any Entropy vector field both $P$ and $V$ vanish around $\{(0,\tau)|\tau\in S^1\}$. In particular, with previous notations, provided we choose $B_0$ sufficiently small, $P$ and $V$ would vanish on $\overline{S}_0$. We now prove the following: 

\begin{theorem}
    \label{lorenzindex} Consider a one-parameter family $\dot{s}=F_\tau(s)$ of Lorenz systems s.t. $\tau\in S^1$. Then, under the saddle assumption, for any Entropy flow there will be an attractor for the Entropy flow in $\mathbb{R}^3\times S^1$ whose homological Conley index is
        \[
CH_q(S_{\mathcal B};\mathbb Z)
\cong
\begin{cases}
\mathbb Z^2, & q=2,\\
\mathbb Z^2, & q=3,\\
0, & \text{otherwise}.
\end{cases}
\]
Consequently, every such Entropy flow has a non-trivial invariant set.
\end{theorem}
\begin{proof}
We begin by determining the exit set $\mathcal{L}$ for the Entropy flow on $ \partial \mathcal{B}$, where $\mathcal{B}$ is as defined above. By Lemma \ref{choice} we know that on $\partial S_\infty$ the Entropy flow pushes trajectories in $S_\infty$ inside $\mathcal{B}$, as it pushes them inside $\Delta$ via its boundary $\partial\Delta=\partial S_\infty$. Similarly, since $P$ and $V$ vanish on $\partial S_0$ per the saddle assumption, on $\mathbb{T}_1$ and $\mathbb{T}_2$ the Entropy flow pushes initial conditions inside $\mathcal{B}$. Finally, it follows that on $\mathbb{T}_3=B_3\times S^1$ the Entropy flow pushes initial conditions outside of $\mathcal{B}$ and into $S_0$. Summarizing our discussion, we conclude that $\mathcal{L}= \mathbb{T}_3 = B_3 \times S^1$, which implies the pair
\(
(\mathcal B,\mathcal L)
=
\big(S^3\times S^1\setminus(S_0\cup S_\infty),\; \mathbb T_3\big)
\) is an index pair for the isolated invariant set $S_{\mathcal B}=\operatorname{Inv}_E(\mathcal B\setminus \mathcal L)$. Therefore, by the definition of the homological Conley index we have:
\[
CH_q(S_{\mathcal B};\mathbb Z)
=
H_q(\mathcal B,\mathcal L;\mathbb Z).
\]
For brevity, from now on we denote this group by $CH_q(\mathcal B;\mathbb Z)$ - we now continue by computing its relative homology. Since \(
\mathcal B
\cong
S^2\times[0,1]\times S^1
\) and \(
\mathcal L
=
B_3\times S^1,
\) the interval factor is contractible, which yields $H_q(\mathcal B,\mathcal L;\mathbb Z) \cong H_q(S^2\times S^1,\;B_3\times S^1;\mathbb Z)$. This is the relative homology of the product pair $(S^2,B_3)\times S^1$, therefore, by the relative K\"unneth theorem for product pairs over $\mathbb Z$ (see, for instance, Chapter 17.2 in \cite{May1999AConciseCourseinAlgebraicTopology}), we obtain:
\[
H_q(\mathcal B,\mathcal L;\mathbb Z)
\cong
\bigoplus_{i+j=q}
H_i(S^2,B_3;\mathbb Z)\otimes H_j(S^1;\mathbb Z)
\]
with no $\operatorname{Tor}$ contribution, since the homology groups involved are free abelian. It therefore remains to compute $H_*(S^2,B_3;\mathbb Z)$. Since $B_3$ is an annulus, we have $H_0(B_3;\mathbb Z)\cong\mathbb Z$, $H_1(B_3;\mathbb Z)\cong\mathbb Z$ and $H_i(B_3;\mathbb Z)=0$ for all $i\geq2$. Hence, the long exact sequence of the pair $(S^2,B_3)$ gives:
\[
0
\longrightarrow
H_2(S^2;\mathbb Z)
\longrightarrow
H_2(S^2,B_3;\mathbb Z)
\longrightarrow
H_1(B_3;\mathbb Z)
\longrightarrow
H_1(S^2;\mathbb Z)
\longrightarrow
0.
\]

Thus, since $H_2(S^2;\mathbb Z)\cong\mathbb Z$, $H_1(B_3;\mathbb Z)\cong\mathbb Z$ and $H_1(S^2;\mathbb Z)=0$ we get $H_2(S^2,B_3;\mathbb Z)\cong\mathbb Z^2$. Moreover, the same long exact sequence proves both $H_1(S^2,B_3;\mathbb Z)=0$ and $H_0(S^2,B_3;\mathbb Z)=0$. This gives:
\[
H_i(S^2,B_3;\mathbb Z)
\cong
\begin{cases}
\mathbb Z^2, & i=2,\\
0, & i\neq 2.
\end{cases}
\]
Finally, from $H_0(S^1;\mathbb Z)\cong \mathbb Z$, $H_1(S^1;\mathbb Z)\cong \mathbb Z$ and $H_j(S^1;\mathbb Z)=0$ for all $j\geq2$, the K\"unneth formula yields:
\[
H_q(\mathcal B,\mathcal L;\mathbb Z)
\cong
\begin{cases}
\mathbb Z^2, & q=2,\\
\mathbb Z^2, & q=3,\\
0, & \text{otherwise}.
\end{cases}
\]
And consequently:
\[
CH_q(S_{\mathcal B};\mathbb Z)
\cong
\begin{cases}
\mathbb Z^2, & q=2,\\
\mathbb Z^2, & q=3,\\
0, & \text{otherwise}.
\end{cases}
\]
All in all, since the Conley index of $\mathcal B$ is nontrivial, the Wa\.zewski principle (see \cite{waz}) implies that there exists a non-empty maximal invariant set for the Entropy flow inside $\mathcal B$.
\end{proof}
\begin{figure}[h]
\centering
\begin{overpic}[width=0.4\textwidth]{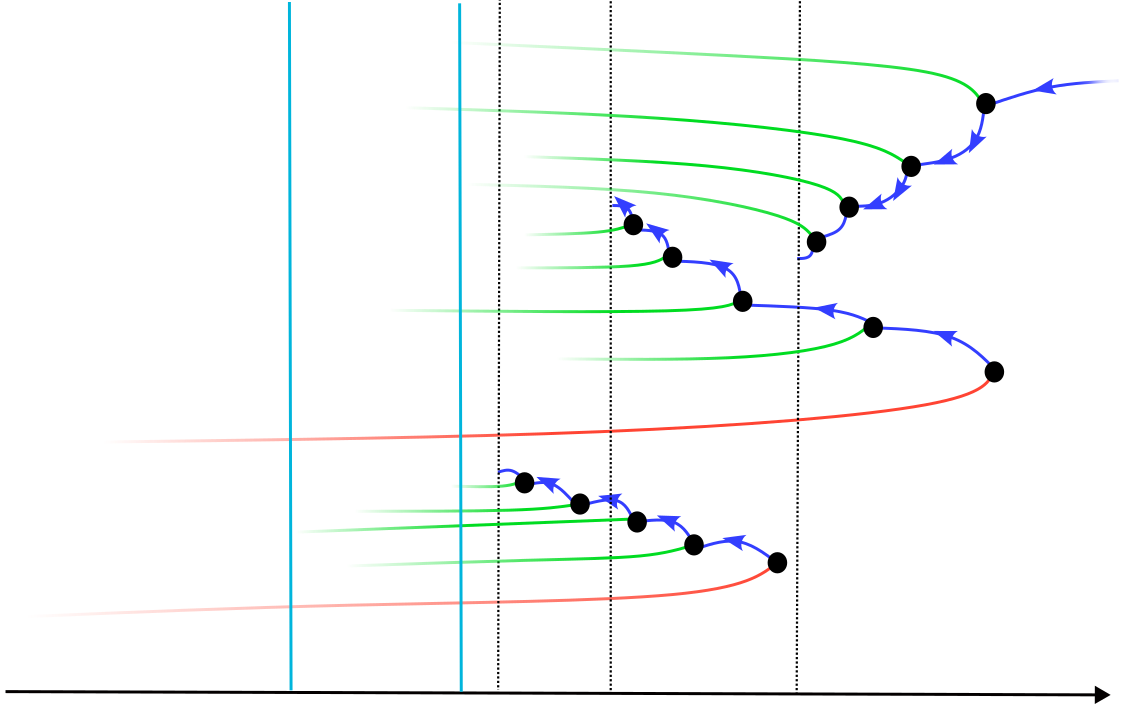}
\put(1000,0){$\rho$}
\put(380,-30){$99$}
\put(220,-30){$30.1$}

\end{overpic}
\caption{\textit{The cyan lines denotes $\rho=30.1$ and $\rho=99$. The dotted lines denote the accumulation points of the cascades on the $\rho$ line (all terminate before $\rho=99$). Blue denotes attracting orbits, while red and green arcs denote saddles with differing Mallet-Yorke Indices. The arrows on the blue arcs denote the direction $P$ points to w.r.t. $\rho$, while the gray vertical lines denote termination points for the cascades.}}
\label{largerho}
\end{figure}

At this point we note that heuristically, the two generators in the relative group $H_2(S^2,B_3;\mathbb Z)\cong \mathbb Z^2$ should be viewed as the two local exit directions from the saddle neighborhood, corresponding to the two caps $B_1$ and $B_2$. Under the usual geometric picture of the Lorenz system, these two directions are compatible with the left and right lobes of the Lorenz butterfly attractor. As such, the degree-three classes are obtained by crossing these two degree-two relative classes with the generator $[S^1]\in H_1(S^1;\mathbb Z)$ represented by one full turn around the parameter circle. That being said, we stress this interpretation is purely heuristic: while the computation above detects the relative topology of the exit geometry, it is not a canonical dynamical decomposition into the left and right lobes suspended with $S^1$.\\

As a final remark, we note that given a $C^1$ one--parameter family of Lorenz systems $\dot{s}=F_\tau(s)$ and an associated Entropy flow directed by $E(s,\tau)=(F_\tau(s)+V(s,\tau),P(s,\tau))$, the calculations above give clear answer on the question whether $P$ and $V$ vanish identically or not in $\mathbb{R}^3\times S^1$. Hence, let us first recall that as observed numerically, when the $\rho$ parameter varies in $(99.5,\infty)$ the Lorenz system exhibits period doubling cascades of attractors  (see Chapters 4, 5 in \cite{Sparbook} and the illustration in Figure \ref{largerho}). In detail, period doubling cascades of attractors appear as $\rho\downarrow 99$, while the saddles created in such bifurcations are destroyed by homoclinic bifurcations occurring in the range $\rho\in(30.1,99.5)$ - for an illustration, see Figure 5.12 in \cite{Sparbook} and Figure \ref{largerho} (for a theoretical study on the breakup of the Lorenz attractor due to these bifurcations, see \cite{BS}). Recalling Theorem \ref{major1}, we conclude (from the numerical evidence too) that there exists a $C^1$ family $\dot{s}=F_\tau(s)$ of Lorenz systems s.t. for every choice of Entropy flow for this family the functions $P$ and $V$ do not vanish. In particular, as these period doubling cascades of attractors terminate as $\tau\rightarrow \{\rho=99.5\}$, it follows that the Entropy flow pushes $(s,\tau)\in\mathbb{R}^3\times S^1$ towards the set $\mathbb{R}^3\times\{\rho=99.5\}$ around those cascades. Of course, proving such "attractors of complexity" actually occur in Entropy flows for the Lorenz system requires more information than we currently have on the bifurcation structure. That being said, in the next Section we will show that not only such attractors of complexity can occur for Entropy flows, but they actually arise naturally in the Shilnikov homoclinic bifurcation scenario too.

\section{Turbulent Entropy flows}
\label{turbulence}

In the previous two Sections we considered mostly abstract theory of Entropy flows. Most of that discussion was local in nature, as we were interested in describing the local dynamics of Entropy flows around certain bifurcation scenarios. In this Section, inspired by \cite{shil3} and \cite{turaev}, we prove just how complex the global dynamics of the Entropy flow can be. Specifically, we do so by studying the behavior of the Entropy flow around the Shilnikov bifurcation scenario. As we will prove, the complexity of the bifurcations leading up to the Shilnikov homoclinic scenario is seen by the Entropy flow qualitatively, and surprisingly, possibly also statistically. As such, this Section is organized as follows: we begin by recalling the Shilnikov homoclinic scenario, after which we discuss the bifurcation structure it induces on the periodic orbits around it (see Proposition \ref{generalizedgkp}). Following that we prove Theorem \ref{shilnikoventropy} and Corollary \ref{turbulententropy}, which together give a global description of the Entropy flow around the Shilnikov bifurcation. These two results would lead us to conclude the dynamics of the Entropy flow around the Shilnikov homoclinic bifurcation always include a certain hierarchy of invariant sets. Strangely enough, this hierarchy bears similarities (even if superficial) to complex motion in fluids, i.e., turbulence. Therefore, we conclude this Section with a brief and informal discussion of these similarities.\\

\begin{figure}[h]
\centering
\begin{overpic}[width=0.5\textwidth]{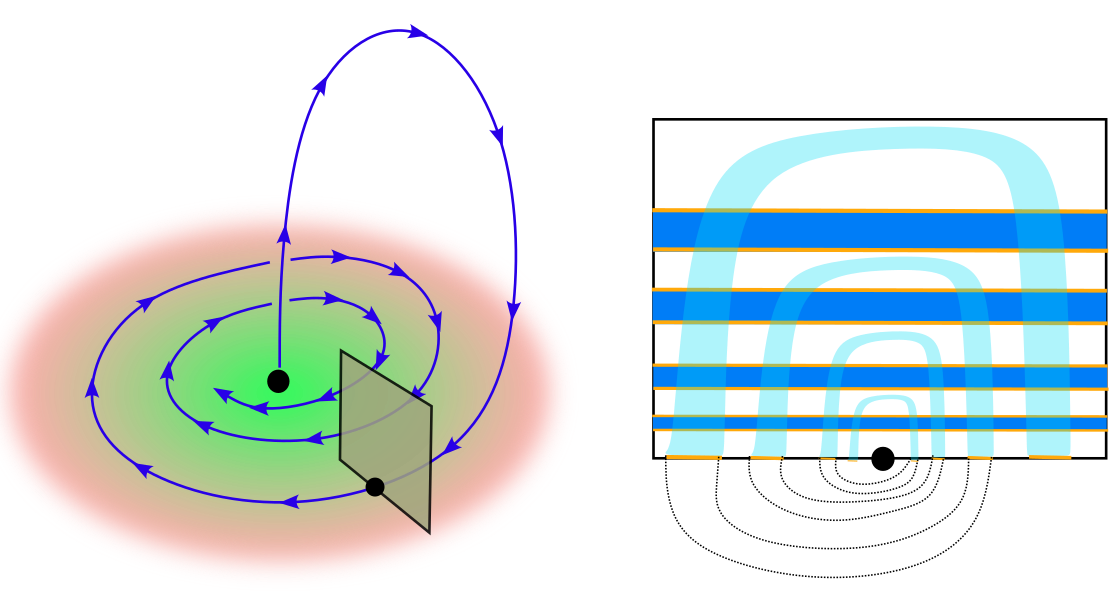}

\end{overpic}
\caption{\textit{The Shilnikov homoclinic scenario on the left and the cross-section $S$ (and its image under the first-return map) on the right (in this illustration, $n=3$). The blue rectangles denote $R_i$ and the orange arcs denote the horizontal sides mapped by the first-return map to $S\cap W^s$. The cyan Horseshoes are the images of the rectangles.}}
\label{shilnikovth}
\end{figure}

We begin by recalling the Shilnikov homoclinic scenario, originally introduced in \cite{LeS} - for simplicity, we will consider the case $M=\mathbb{R}^n$ (the generalization to manifolds will be trivial). Assume $F$ is a $C^2$ vector field of $\mathbb{R}^n$, $n\geq3$, with a fixed point at the origin. Let $J_0$ denote the Jacobian matrix of $F$ at the origin, and further assume the following holds:
\begin{itemize}
    \item $J_0$ has only one positive eigenvalue, $\gamma>0$, corresponding to one-dimensional unstable manifold $W^u$.
    \item There exist $n-1$ eigenvalues with negative real part spanning the $n-1$ dimensional stable manifold $W^s$, $\lambda_1,...,\lambda_{n-1}$, s.t.:
    \begin{enumerate}
        \item $\lambda_1=\overline{\lambda_2}$ and $\lambda_{1,2}=-\rho\pm i\omega$, both have strictly negative real part and non-zero imaginary part.
        \item For $j=3,...,n-1$, $Re(\lambda_j)<-\rho<0$.
    \end{enumerate}
    \item The \textbf{Shilnikov condition}: we have $\nu=\frac{\rho}{\gamma}<1$, i.e., the repelling force is stronger than the attracting force. $\nu$ will be referred to as the \textbf{saddle index}.
    \item There exists a homoclinic loop $\Gamma$ in $W^u$ that tends to the fixed point at the origin via the invariant manifold corresponding to the complex-conjugate pair $-\rho\pm i\omega$.\\
\end{itemize}

As proven by L. Shilnikov, we have the following result (see Theorem 13.8 in \cite{SSTC} or  \cite{LeS} and \cite{shil2}):
\begin{theorem}
    \label{shilnikov} Under the assumptions above, there exists a countable collection of hyperbolic invariant sets for the flow, $\{H_i\}_{i=1}^\infty$, accumulating on $\Gamma$ (see Figure \ref{shilnikovth}). For all $i$, $H_i$ includes countably many saddle periodic orbits. Moreover, for any $n$ there exists some $\epsilon$ s.t. if $||F-F'||_{C^1}<\epsilon$, the sets $H_1,...,H_n$ survive as we $C^2$-perturb $F$ to $F'$.
\end{theorem}
\begin{remark}
    For completeness, we remark the assumption that $W^u$ is one-dimensional can be relaxed to some extent - see \cite{shil2}, \cite{shil4} and Chapter 13.5 in \cite{SSTC} for the details.
\end{remark}
\begin{figure}[h]
\centering
\begin{overpic}[width=0.4\textwidth]{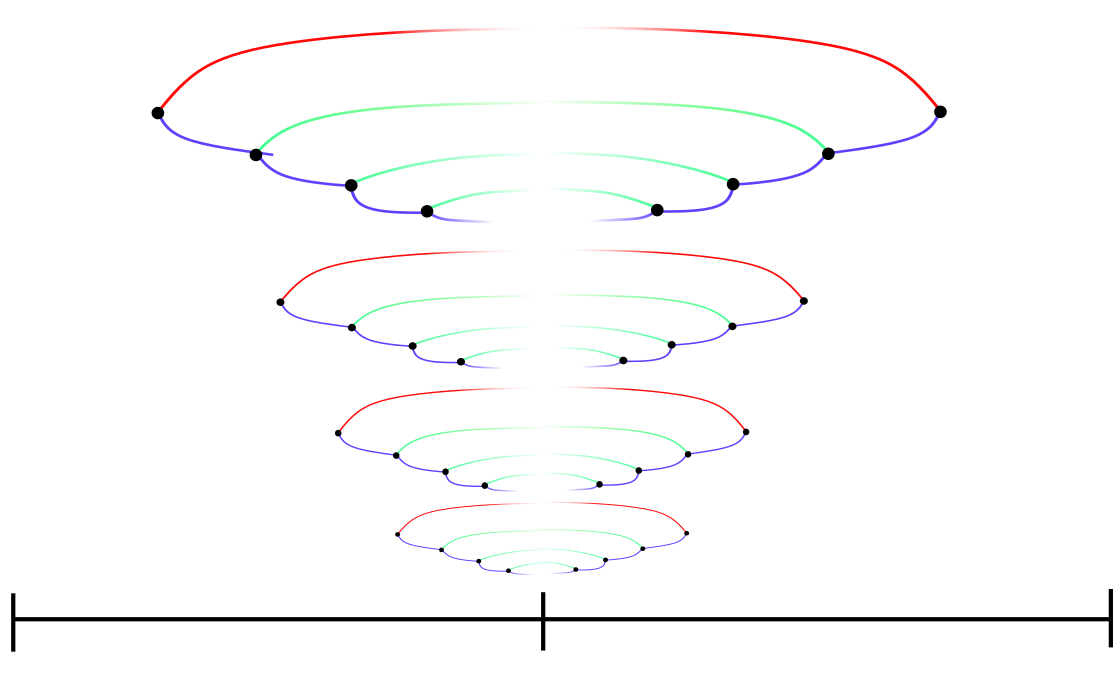}
\put(980,-10){$1$}
\put(470,-10){$0$}
\put(-40,-10){$-1$}

\end{overpic}
\caption{\textit{A (partial) bifurcation diagram around the Shilnikov parameter $0$. The black dots represent bifurcation orbits, while the red and green arcs represent saddles. The blue arcs represent periodic orbits undergoing a period doubling cascade.}}
\label{cascadeshil}
\end{figure}

We now recall several facts from the proof of Shilnikov's Theorem. Specifically, we recall there exist a cross-section $S$, transverse to $\Gamma$, a collection of $n-1$-dimensional topological cubes $\{R_i\}_i\subseteq S$ and a first-return map $f:\cup_iR_i\to S$ such that:
\begin{itemize}
    \item There exists a boundary region $W\subseteq \partial S$ that lies on $W^s$, the stable manifold for $O$ (see Figure \ref{shilnikovth}).
    \item For all $j,k,i>0$, $R_{i+j}$ is separated from $R_{i-k}$ in $S$ by $R_i$.
    \item For each $i$, $\partial R_i$ includes two $n-2$-dimensional connected regions, the \textbf{horizontal sides}, $l_{i,u}$, $l_{i,d}$ that are mapped to $W$ under $f$. Moreover, $R_i$ is a component of $S\setminus(l_{i,u}\cup l_{i,d})$ (see Figure \ref{shilnikovth}).
    \item $f(R_i)\cap R_i$ is composed of two components, that the set $H_i$ is interior to. In particular, $f:H_i\to S$ is a topological Horseshoe in the sense of \cite{KY}.
    \item Assume $T$ is a periodic orbit for $f$ lying in $H_i$ for some $i$. Then, for all sufficiently large $i$ we have:
    \begin{enumerate}
        \item There exists some $c_i\in(0,1)$ s.t. the   differential of $f$ at $T\cap S$ has one real eigenvalue with absolute value in $(1,\infty)$ of order $O(c_i^{\nu-1})$, and all other eigenvalues have absolute value of order $O(c_i^{\nu})$.
        \item The sequence $\{c_i\}_i$ can be chosen to be monotonically decreasing, i.e., for all $i>j$ we have $c_i>c_j$ and $\lim_{i\to\infty}c_i=0$.
    \end{enumerate}
\end{itemize}

For the proof, see either \cite{LeS} or Theorem 13.8 in \cite{SSTC}. To continue, let us consider some $C^2$ family of vector fields $\dot{s}=F_\tau(s)$, $s\in\mathbb{R}^3$, $\tau\in(-1,1)$ s.t. at $\tau=0$ there exists a saddle focus $O$ undergoing a Shilnikov bifurcation. As proven in \cite{GKP}, under these assumptions, as $\tau\to0$ from either $(-1,0)$ or $(0,1)$,  the periodic orbits in $\{H_i\}_i$ emerge via saddle node bifurcations and period doubling cascades, as illustrated in Figure \ref{cascadeshil}. We will need a higher dimensional version of this fact, which we now prove. Before doing so, we remark that this fact is probably well-known - but as we could not find a proof of it in the known literature, for completeness we provide an argument generalizing the results of \cite{GKP} to higher dimensions. However, our argument uses the Mallet-Yorke Index instead of the Shilnikov map used in \cite{GKP}.

\begin{proposition}
\label{generalizedgkp}    Let $\dot{s}=F_{a_1,a_2}(s)$, $a_1,a_2\in\mathbb{R}$. Denote a $C^2$ two-parameter family of vector fields on $\mathbb{R}^n$, $n\geq3$, with a fixed point at the origin. Assume $\gamma:(0,1)\to\mathbb{R}^2$ is an arc s.t. $\gamma(t)$ is a parameter at which $F_{\gamma(t)}$ generates a homoclinic trajectory to the origin s.t. for all $t\in(0,1)$ the assumptions of Shilnikov Theorem are verified for some saddle index $\nu(t)$, varying continuously in $t$ - in particular, for all $t$, $\nu(t)\in(0,1)$. Then, the following is true:
    \begin{itemize}
        \item $\gamma$ is the accumulation set of one-dimensional sets in the parameter space $\mathbb{R}^2$ corresponding to both period doubling and saddle node bifurcations.
        \item When $\frac{1}{2}<\nu(t)<1$, these period doubling cascades are period doubling cascades of attractors, and when $0<\nu(t)<\frac{1}{2}$, these are period doubling cascades of repellers.
    \end{itemize}  
\end{proposition}
\begin{proof}
Consider a $C^2$ one-parameter family $\dot{s}=F_\tau(s)$, $\tau\in(-1,1)$ with a saddle focus at $0$, s.t. at $\tau=0$ there exists a Shilnikov homoclinic bifurcation - i.e., the vector field $F_0$ generates a homoclinic trajectory $\Gamma$ to $0$ s.t. the assumptions of Theorem \ref{shilnikov} are satisfied (in particular, $\Gamma$ is destroyed as $\tau$ is varied into $(-1,0)$ or $(0,1)$). It would suffice to prove the existence of an increasing sequence $\{\tau_n\}_n\subseteq(-1,0)$ and a decreasing sequence $\{\tau'_n\}_n\subseteq(0,1)$, $\tau_n,\tau'_n\to0$, s.t. these sequences include countably many period doubling and saddle node bifurcation parameters. As period doubling and saddle node bifurcation orbits form a codimension one bifurcation sets in a two-dimensional parameter space, this will imply Proposition \ref{generalizedgkp}. Moreover, we will only prove the existence of $\{\tau_n\}_n\subseteq(-1,0)$ - the proof of existence for $\{\tau'_n\}_n\subseteq(0,1)$ is symmetric. We will also consider only when $\frac{1}{2}<\nu<1$, where $\nu$ denotes the saddle index at the origin, and prove it forces the existence of period doubling cascades of attractors. The case for period doubling cascades of repellers is proven (almost) symmetrically by considering the $C^2$ family of the inverse flows, i.e., $\dot{s}=-F_\tau(s)$.

\begin{figure}[h]
\centering
\begin{overpic}[width=0.6\textwidth]{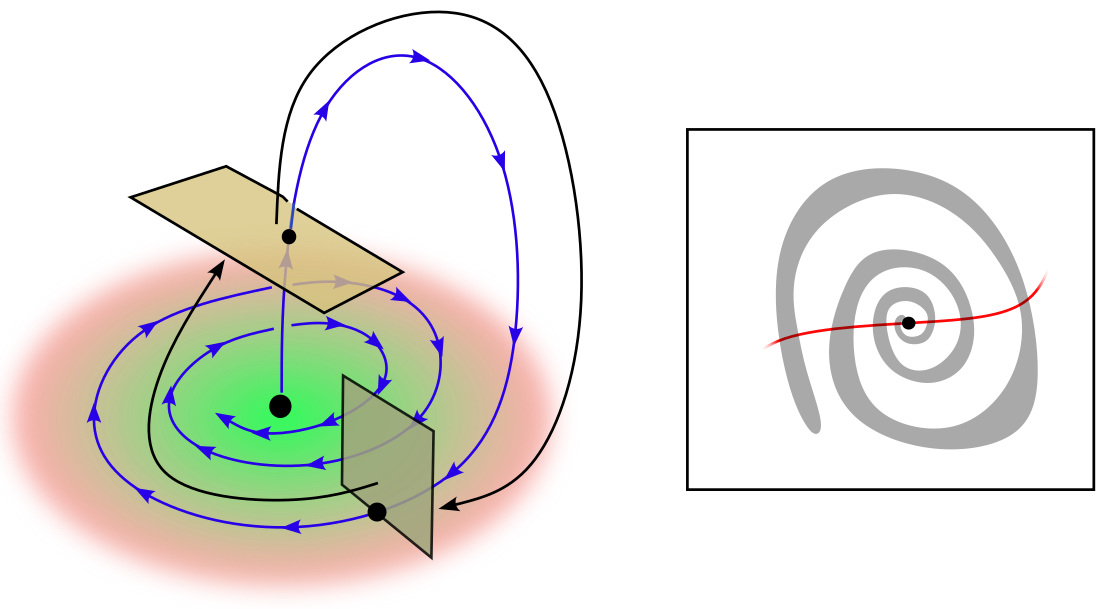}
\put(850,70){$S_1$}
\put(400,50){$S$}
\put(120,325){$S_1$}
\put(490,300){$g_1$}
\put(100,200){$g_0$}
\put(800,405){$g_0(S)$}
\end{overpic}
\caption{\textit{On the left - a depiction of the local and global maps, $g_0$ and $g_1$ (respectively), in the three-dimensional case. On the right - the cross-section $S_1$. The red curve corresponds to the intersection of $W^s$ with $S_1$, the black dot is the intersection of the homoclinic loop $\Gamma$ with $S_1$ and the shaded spiral is the image of $g_0(S)$. The map $g_1$ is defined only on the components of the spiral above the red curve.}}
\label{local}
\end{figure}
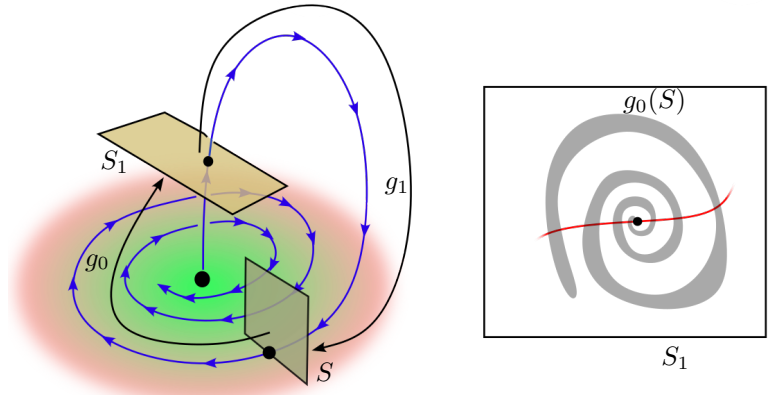

To begin, let $\Gamma$ denote the homoclinic trajectory to $0$ and let $H_i$, $i\in\mathbb{N}$, denote some hyperbolic set suspended around $\Gamma$, as given by Theorem \ref{shilnikov}. Let $T$ denote some periodic orbit inside $H_i$ and with the same notations above, let $f:R_i\to S$ denote the first-return map. To begin, for all $\tau\in(-1,0]$ write $f_\tau:R_i\to S$, as the first-return map for $R_i$ w.r.t. the vector field $F_\tau$ (for $\tau$ where it is defined) and $f_0$ corresponds to $f$, the first-return map for $F_0$. By Theorem \ref{shilnikov}, we know there exist some $\tau_i\in(-1,0)$ s.t. for all $\tau\in(\tau_i,0)$ the map $f_{\tau_i}:R_i\to S$ is a topological Horseshoe map in the sense of \cite{KY}, and that all the components of the hyperbolic invariant set $H_i$ (i.e., its periodic and aperiodic orbits) are created before $\tau$ approaches $\tau_i$ from below (note that by definition, there exists some $\tau_i\leq \tau'_i\leq0$ s.t. for all $\tau\in(\tau'_i,0)$ $H_i$ persists as a hyperbolic invariant set for $F_\tau$). Our idea is the following: we first prove there exists some $\tau^0_i<\tau_i$ s.t. $f_t:R_i\to S$, $\tau\in[\tau^0_i,0]$ is an isotopy. By a standard approximation argument from the Mallet-Yorke Orbit Index Theory, this would imply the existence of period doubling cascades. Following that, by verifying sufficient conditions from  \cite{Perd} we will conclude these period doubling cascades are period doubling cascades of attractors.

To begin, we first prove that for $\tau^0_i$ sufficiently close to $-\tau_i$, the first-return map $f_\tau:R_i\to S$, $\tau\in(\tau^0_i,0]$ is an isotopy of $C^1$-maps. We first recall that per the proof of Shilnikov's Theorem the map $f_0:\cup_i R_i\to S$, i.e., the first-return map at the Shilnikov parameter, can be written as the composition of two maps $g_1\circ g_0$ - where $g_0:S\to S_1$, the \textbf{local map}, is defined throughout $S$, and $g_1:S_1\to S$, the \textbf{global map}, is defined only on some subsets $g_0(S)$ in $S_1$ (where $S_1$ is the cross-section transverse to the homoclinic curve $\Gamma$ sketched in Figure \ref{local}). Recalling $S$ is a cross-section whose side lies on the invariant manifold corresponding to $-\rho\pm i\omega$, we know it persists as long as $0$ remains a saddle focus. Moreover, for all $\tau$ sufficiently close to $0$, $g^\tau_0:S\to S_1$, the corresponding map for $F_\tau$, is well-defined (where $g_0=g^0_0$ - see Figure \ref{local}).

\begin{figure}[h]
\centering
\begin{overpic}[width=0.6\textwidth]{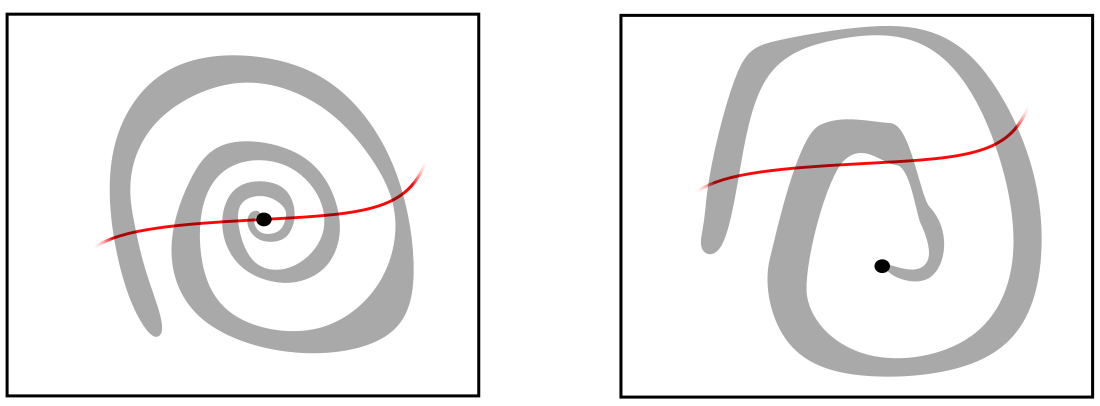}
\put(780,80){$g^\tau_0(S)$}

\put(120,325){$g_0(S)$}
\end{overpic}
\caption{\textit{On the left - the image of $g_0(S)$ in $S_1$. On the right - the image of $g^\tau_0(S)$ in $S_1$. The black dot represents $p_0$, the intersection of the separatrix of $W^u$ with $S_1$, while the red curve corresponds to the transverse intersection $W^s\cap S_1$. The black dot always represents $p_0$.}}
\label{transitionspiral1}
\end{figure}

We now recall that $g_0:S\to S_1$ maps $S$ to an $n-1$-dimensional logarithmic spiral inside $S_1$, centered at some $p_0\in S_1$ corresponding to the intersection of the homoclinic trajectory $\Gamma$ with $S_1$. In addition, for the vector field $F_0$, the point $p_0$ lies at the transverse intersection between the $n-1$ dimension stable manifold $W^s$ with $S_1$, denoted by $\gamma$ (see Figure \ref{local}). As such, when $F_0$ is perturbed to $F_\tau$, $\tau<0$, $\gamma$ persists (at least for $\tau$ sufficiently close to $0$) while $p_0$ persists as a point on $W^u$ and disconnects from $\gamma$ (for a possible scenario, see Figure \ref{transitionspiral1}). That being said, as the origin is a saddle focus, $p_0$ remains the center of the logarithmic spiral $g^\tau_0(S)$. As only finitely many of the Horseshoes survive (say, all $H_i$, $i<i_0$), we know $g^\tau_0(S)$ intersects $\gamma$ transversely in finitely many arcs, as illustrated in Figure \ref{transitionspiral1}. This implies that as $\tau\to0$, $f_\tau=g^\tau_1\circ g^\tau_0:R_i\to S$ is well-defined and maps the horizontal sides of $R_i$ to $\partial S\cap W^s$, while pushing $f_\tau(R_i)$ upwards, which eventually creates the intersection with the rectangle $R_i$ (see the illustration in Figure \ref{transitionspiral2}). This proves the suspended Horseshoe corresponding to $H_i$ is created gradually, by an isotopy $f_\tau:R_i\to S$, where $\tau$ varies in $(\tau^i_0,0)$, for some $\tau^i_0<\tau_i$. As we began with a $C^2$ one--parameter family, it follows that this isotopy is also $C^2$ (and in particular, $C^1$).

\begin{figure}[h]
\centering
\begin{overpic}[width=0.6\textwidth]{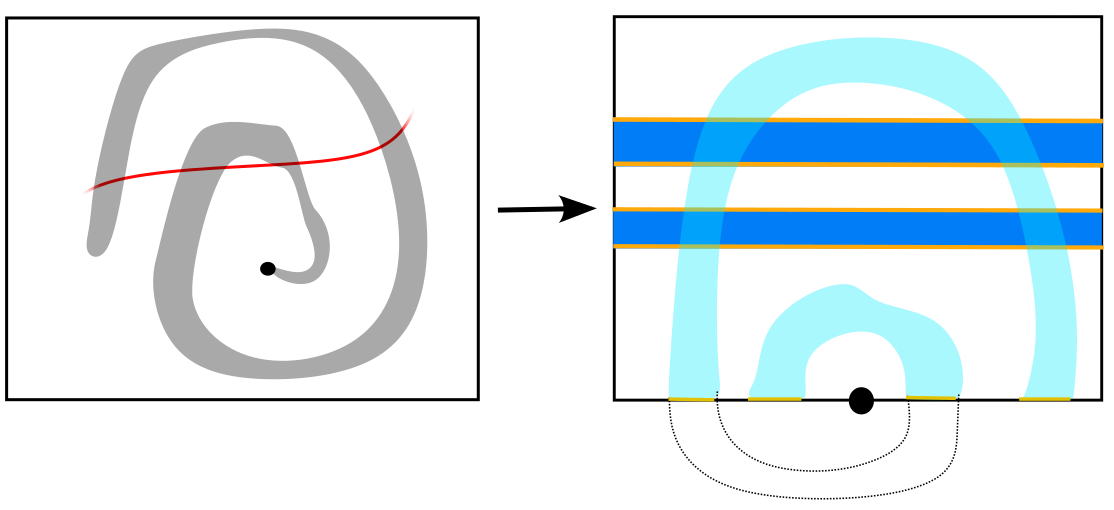}

\end{overpic}
\caption{\textit{On the left - the image of $g^\tau_0(S)$ in $S_1$, where the red curve corresponds to the transverse intersection $W^s\cap S_1$. On the right - the image under $f_\tau$ causing the formation of a Horseshoe as $\tau\to0$.}}
\label{transitionspiral2}
\end{figure}

We now $C^3$-approximate $\dot{s}=F_\tau(s)$ by a families $\dot{s}=G_\tau(s)$ that lies in the family $\mathcal{K}\subseteq C^3(\mathbb{R}^n\times(-1,1),\mathbb{R}^n)$, characterized by the following property: all the periodic orbits for $\dot{s}=G_\tau(s)$ that are created away from the fixed points are created only by period doubling and saddle node bifurcations (we can do so by the $C^3$-density of $\mathcal{K}$ - see the Appendix in \cite{PY2}). In detail, recall that given a $C^3$ family in $\mathcal{K}$, $\dot{s}=G_\tau(s)$, $\tau\in(0,1)$ and some bounded region $C\subseteq\mathbb{R}^n$ s.t. $G_\tau(s)\ne0$ for all $\tau\in(0,1)$, all the periodic orbits of $\dot{s}=G_\tau(s)$ in $C$ can be created and/or destroyed only by saddle node and period doubling bifurcations. Therefore, let $h_\tau:R_i\to S$, $\tau\in(\tau^0_i,0]$ be a $C^3$-isotopy which $C^2$-approximates $f_{\tau}:R_i\to S$ - that is, $h_\tau$ are the first-return maps for some family $\dot{s}=G_\tau(s)$ approximating $\dot{s}=F_\tau(s)$. As $R_i$ lies uniformly away from the fixed point at $0$ for all $\tau\in[\tau^0_i,0]$, all the periodic orbits that appear in it w.r.t. the isotopy $h_\tau:R_i\to S$ are created in a bounded region - as such, they can only be created by saddle node bifurcations and period doubling cascades. We now conclude the exact same argument used to prove Theorem 4.A in \cite{Perd} proves the periodic orbits for the original isotopy $f_\tau:R_i\to S$, $\tau\in[\tau^0_i,0]$ are also created by period doubling cascades. 

We now prove these cascades are all cascades of attractors - we do so by verifying a sufficient condition used to prove the same in Theorem 4.A in \cite{Perd}. To this end, recall we assume $\frac{1}{2}<\nu<1$, and that $f_0:R_i\to S$, $f_0(y,x,u)=(f_1(y,x,u),f_2(y,x,u),f_3(y,x,u))$, $u\in R^{n-2}$ can be written in the following form (see the proof of Theorem 13.8 in \cite{S}):

$$ f_1(y,x,u)=Axy^\nu\cos(\omega \log(\frac{1}{y})+\theta_{1})+o(y^\nu),$$
$$f_2(y,x,u)=C_2+Bxy^\nu\cos(\omega \log(\frac{1}{y})+\theta_{2})+o(y^\nu),$$
$$ f_3(y,x,u)=C_3\overline{u}+Cxy^\nu\cos(\omega \log(\frac{1}{y})+\theta_{3})+o(y^\nu),$$
where $A,B,C,C_2,C_3, \theta_{j}$, $j=1,2,3$ are constants dependent on the fixed point at the origin and $H_i$, and $\overline{u}$ is a vector of dimension $n-2$ that is independent of the variables $x,y$ - where $|y|\to0$ as $i$ increases. Moreover, recall that as shown in the course of proof of Shilnikov's Theorem, the differential of the first-return map $f_0:R_i\to S$ at $(x,y,u)$ is given by the following matrix around periodic orbits, at least when $i$ is sufficiently large:
\begin{equation}
\label{shilmat}
\begin{pmatrix}
    a y^{\nu-1}(1+o(1)) & o(y^\nu) & o(y^\nu)\\
  o(y^{\nu-1}) & o(y^\nu) & o(y^\nu)\\
      o(y^{\nu-1}) & o(y^\nu) & o(y^\nu)
\end{pmatrix}
\end{equation}

Where $a=a(x,y)$ is some continuous function, s.t. $|a(x,y)|$ is independent of $i$ on all sufficiently large $i$. For the details, see the proof of Theorem 13.8 in \cite{SSTC}. Thus, when $y$ is sufficiently small (i.e., when $i$ is sufficiently large), there exists just one eigenvalue of absolute value strictly larger than $1$, denoted by $\lambda_1\approx  a y^{\nu-1}(1+o(1))$, while all other eigenvalues $\lambda_2,...,\lambda_{n-1}$ have eigenvalues with absolute values in $(0,1)$ (in particular, it follows $\lambda_1=o(y^{\nu-1})$). As such,  if $T$ is a periodic orbit for $F_0$ intersecting $R_i$ at a periodic orbit of minimal period $r$, the eigenvalues of its differential are $\lambda_1=o(y^{r\nu-r})$ and $\lambda_j=o(y^{r\nu})$, $j=2,...,n-1$. As a consequence, $\lambda_1\lambda_j=o(y^{2r\nu-r})$ for all $j=2,...,n-1$. The same is true for $\lambda_j\lambda_k$ trivially when $j,k\ne1$. Now, consider the inequality $2r\nu-r>0$, that is satisfied precisely when $\nu>\frac{r}{2r}=\frac{1}{2}$, which is satisfied by $\frac{1}{2}<\nu<1$. This implies that for all $j\ne k$, and $|y|<1$, we have $\lambda_j\lambda_k=o(y^{2r\nu-r})=o(y^{2\nu-1})$. As such, since as $i\to\infty$ the coordinate $y\to 0$, there exists some $\theta_i\in(0,1)$ s.t. for all $1\leq j< k\leq n-1$, we have $|\lambda_j\lambda_k|<\theta_i<1$. 

Iterating this argument, we see it remains true for the differential $D_{f^r}$ around all periodic orbits in $R_i$ of minimal period $r$ - for all $r$. Finally, since $|y|$ and $|a(x,y)|$ are both uniformly bounded on $R_i$ independently of $i$ (at least for sufficiently large $i$), we can choose $\theta=\theta_i$ uniformly for all sufficiently large $i$. Moreover, note that for sufficiently large $i$, this inequality persists as $f_0$ is perturbed to $f_\tau$. More precisely, by the formula for the Shilnikov map and its differential above we conclude there exists some $\epsilon>0$ s.t. whenever $\tau^0_i\in(-\epsilon,0)$ the isotopy $f_\tau:R_i\to S$ satisfies that if $p$ is a periodic point of minimal period $r$ for $f_\tau$, $\tau\in[\tau^0_i,0)$ with eigenvalues $\lambda^\tau_1,...,\lambda^\tau_{n-1}$, then $|\lambda^\tau_i\lambda^\tau_j|<\theta$. In other words, we have a strong contraction condition in the Horseshoe map $f_0:R_i\to S$ for all sufficiently large $i$ (if we were to use the terminology of \cite{Perd}), with the constant $\theta$ being independent of the $R_i$. The exact same argument used to prove Theorem 4.A in \cite{Perd} now implies that all the periodic orbits along the said isotopy arise via a period doubling cascade of attractors.

In more detail (and using previous notation), as proven in Theorem 4.A in \cite{Perd} the uniform contraction condition implies all the periodic orbits for the isotopy $h_\tau:R_i\to S$, $\tau\in(\tau^0_i,0)$ that undergo the period doubling cascade are attractors, with eigenvalues uniformly bounded (in absolute value) by $\sqrt{\theta}$ - at least when the $C^3$-isotopy $h_\tau:R_i\to S_i$ is sufficiently $C^2$-close to the isotopy $f_\tau:R_i\to S$. The exact same limiting argument used at Theorem 4.A now yields that the eigenvalues for the periodic orbits undergoing the cascades for $C^2$-isotopy $f_\tau:R_i\to S$ are also bounded (in absolute value) by $\sqrt{\theta}$. As $\theta<1$ this proves these periodic orbits are attractors. In other words, we have proven that for all sufficiently large $i$, the dynamics of $F_0$ on $H_i$ are created by period doubling cascades of attractors. 
\end{proof}

\begin{remark}
In \cite{OS} it was proven that the perturbation of the Shilnikov scenario when $\nu\in(0,\frac{1}{2})$ in general leads to the creation of repelling orbits. Conversely, when $\nu\in(\frac{1}{2},1)$ the same perturbation in general leads to the creation of attracting orbits (for the precise statement - \cite{OS}). 
\end{remark}

Having proven Proposition \ref{generalizedgkp}, we see that period doubling cascades of either attractors or repellers are a common phenomenon around Shilnikov bifurcations - in all dimensions. We now recall that given a homeomorphism $f:X\to X$ on some topological space $X$, we say an invariant set $I\subseteq X$ is a \textbf{Milnor attractor} if it attracts all initial conditions $x\in O$, where $O$ has positive measure w.r.t. the Lebesgue measure (see \cite{Mil3}). $O$ is a \textbf{Milnor repeller} if it is a Milnor attractor w.r.t. $f^{-1}$. Combining this definition with Proposition \ref{generalizedgkp} and the results of Subsection \ref{perioddoubling}, we now prove:
\begin{theorem}
    \label{shilnikoventropy} Let $\dot{s}=F_t(s)$, $t\in (-1,1)$ be a $C^2$ one-parameter family of vector fields on $\mathbb{R}^n$, $n\geq3$, s.t. $F_{0}$ has a fixed point $O$ with a Shilnikov homoclinic trajectory $\Gamma$ and a saddle index $\nu\ne\frac{1}{2}$. Then, writing $E(s,\tau)=(F_\tau(s)+V(s,\tau),P(s,\tau))$ as the corresponding Entropy vector field on $\mathbb{R}^n\times(-1,1)$, the following holds:
    \begin{itemize}
        \item $P$ takes both positive and negative values in $M\times (-1,1)$ arbitrarily close to $\Gamma\times\{0\}$.
        \item When $\nu\in(\frac{1}{2},1)$, we can choose the Entropy flow s.t. the set $\Gamma\times\{0\}$ is the accumulation of countably many periodic orbits for the Entropy flow given by $\{T^\pm_i\times\{\tau^\pm_i\}\}_i$ satisfying the following:
        \begin{enumerate}
            \item  $\{\tau^+_i\}\subseteq(0,-1)$ and $\{\tau^-_i\}_i\subseteq(-1,0)$.
            \item $\lim_{i\to\infty}\tau^+_i=\lim_{i\to\infty}\tau^-_i=0$.
            \item $T^\pm_i$ is a periodic orbit for the vector field $F_{\tau^\pm_i}$, $i\in\mathbb{N}$.
            \item  Every $T^\pm_i\times\{\tau^\pm_i\}$ defines some Milnor attractor for $E$.
        \end{enumerate}    
        \item When $\nu\in(0,\frac{1}{2})$ the same holds, with the only difference being that the periodic orbits above become Milnor repellers.
        \item Let $\{H_i\}_i$ denote the hyperbolic invariant sets given by Theorem \ref{shilnikovth}. Then, we can choose the Entropy flow s.t. for all $i$, all $\epsilon>0$ and all $(s_0,0)\in H_i$, there exist initial conditions $(s,\tau)\in M\times(-1,0)$ and $(s',\tau')\in M\times(0,1)$ whose trajectory under the Entropy flow travels arbitrarily close to $(s_0,0)$. 
      
    \end{itemize}
\end{theorem}
\begin{proof}
We begin with the first assertion. By Proposition \ref{generalizedgkp} we already know that when $\nu\ne\frac{1}{2}$ and $\tau\to0$, the periodic orbits for the Horseshoes suspended around $\Gamma$ are created by period doubling bifurcations of either attractors and repellers - depending on whether $\nu\in(0,\frac{1}{2})$ or $\nu\in(\frac{1}{2},1)$ (see Figure \ref{cascadeshil}). Moreover, the same Proposition also tells us that there are two periodic sequences $\{\tau^\pm_i\}_i$, $\{\tau^+_i\}_i\subseteq(0,1)$ and $\{\tau^-_i\}_i\subseteq(-1,0)$, satisfying the following:
\begin{itemize}
    \item $\lim_{i\to\infty}\tau^+_i=\lim_{i\to\infty}\tau^-_i=0$.
    \item At $\tau^+_i$ the vector field $F_{\tau^+_i}$ has a period doubling bifurcation orbit $T^+_i$. Similarly, $F_{\tau^-_i}$ has a period doubling bifurcation orbit $T^-_i$. These period doubling bifurcations are where the period of either attractors or repellers is doubled, depending on the size of $\nu$. 
\end{itemize}

We assume these periodic orbits correspond to all period doubling bifurcations leading to the creation of the Shilnikov sets $\{H_i\}_i$ for $F_0$ - we can do so by Proposition \ref{generalizedgkp}. By definition of Entropy flow, the sets $T^\pm_i\times\{\tau^\pm_i\}\subseteq M\times\{\tau_i^\pm\}$ are all periodic orbits for any Entropy flow we choose. As these orbits correspond to period doubling bifurcations of either attractors or repellers, for each $\tau^+_i$ there is some $\tau^+_j$, for some $j$, s.t. $\tau^+_i>\tau^+_j$ and $T^+_i\times\{\tau^+_i\}$ is connected to $T^+_j\times\{\tau^+_j\}$ by a cylinder of periodic orbits $C^+_{i,j}\subseteq Per$. In particular, $C^+_{i,j}\subseteq M\times(\tau^+_j,\tau^+_i)$ and by the definition of the Entropy flow, we know $P$ is negative on $C^+_{i,j}$ (see Figure \ref{cascadeshil2}). Similarly, for each $\tau^-_i$ there is some $\tau^-_j$ for some $j$, s.t. $\tau^-_i<\tau^-_j$ and $T^-_i\times\{\tau^0_i\}$ is connected to $T^-_j\times\{\tau^0_j\}$ by a cylinder of periodic orbits $C^-_{i,j}\subseteq Per$. Moreover, $C^-_{i,j}\subseteq M\times(\tau^-_i,\tau^-_j)$ and $P$ is positive on $C^-_{i,j}$, as illustrated in Figure \ref{cascadeshil2}. As the cascades are arbitrarily close to $\Gamma\times\{0\}$, the first assertion now follows.

\begin{figure}[h]
\centering
\begin{overpic}[width=0.45\textwidth]{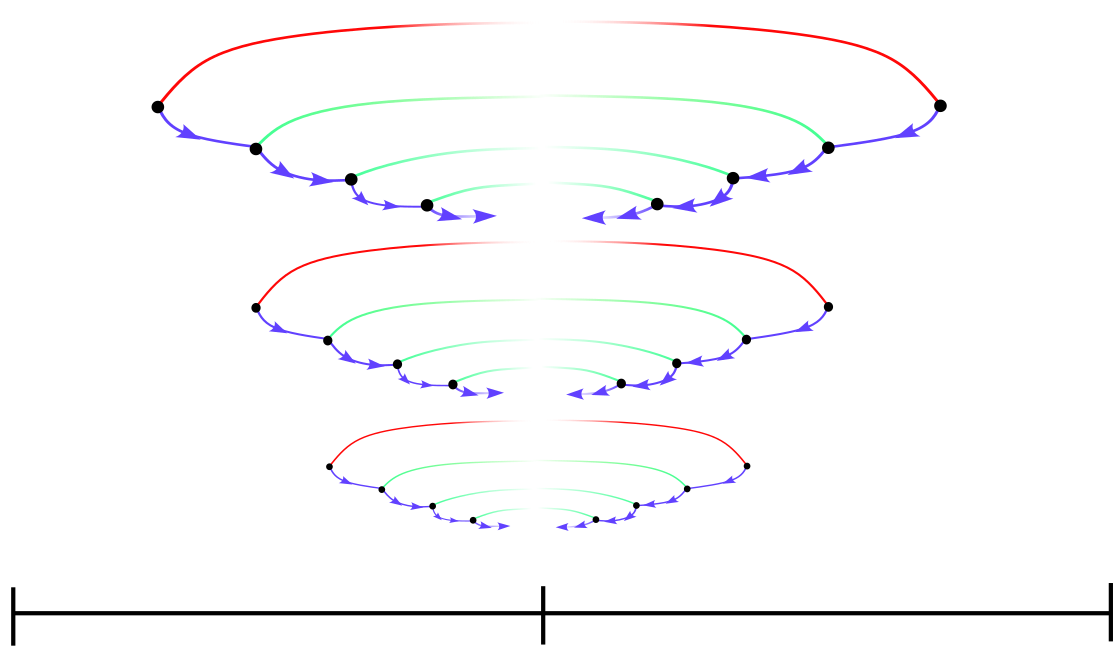}
\put(980,-15){$1$}
\put(470,-15){$0$}
\put(-40,-15){$-1$}

\end{overpic}
\caption{\textit{The directions of the Entropy flow imposed on the (partial) bifurcation diagram of period doubling cascades around the Shilnikov bifurcation, where $\nu\ne\frac{1}{2}$.}}
\label{cascadeshil2}
\end{figure}
We now prove the second assertion. To this end, assume $\nu\in(\frac{1}{2},1)$ and choose some $C^r_{i,j}$, where $r\in\{+,-\}$. By Proposition \ref{generalizedgkp}, we know that whenever $C^r_{i,j}\cap \mathbb{R}^n\times\{\tau\}\ne0$, the intersection is of the form $T\times\{\tau\}$, where $T$ is an attracting orbit for $F_\tau$. By Lemma \ref{trapping}, we know we can choose $E$ s.t. $\overline{C^r_{i,j}}$ includes some attracting invariant set, that includes $T^r_{j}\times\{\tau^r_j\}$ - and by Theorem \ref{major1}, we can choose $E$ s.t. for all $i,j$ the orbit $T^r_j\times\{\tau^r_j\}$ attracts an open neighborhood $O_{i,j}$ of $C_{i,j}$ satisfying $T^r_j\times\{\tau^r_j\}\subseteq\partial O_{i,j}$. This implies $T^r_j\times\{\tau^r_j\}$ is a Milnor attractor for all $j\in\mathbb{N}$, $r\in\{+,-\}$ as it lies on the boundary of $O$. Similarly, when $\nu\in(0,\frac{1}{2})$ a similar argument applied to the $C^2$ family $\dot{s}=-F_\tau(s)$  implies $T^r_j\times\{\tau^r_j\}$ is a Milnor repeller, and the third assertion follows.

We conclude the proof by verifying the fourth assertion, i.e., by proving we can choose the Entropy flow such that there exist initial conditions in $\mathbb{R}^n\times (-1,1)$ arbitrarily close to $H_i\times\{0\}$ that flow arbitrarily close to $H_i\times\{0\}$. To this end, choose some Entropy flow for $\dot{s}=F_\tau(s)$ and recall $H_i$ is a hyperbolic invariant set on which the dynamics are orbitally equivalent for all $F_\tau$ with $\tau\in(-\epsilon_i,\epsilon_i)$ (for some $\epsilon_i>0$, depending only on $H_i$). By the definition of the Entropy flow (see Definition \ref{entropyflow} and Definition \ref{def6}), there exist two open neighborhoods $N^-_i,N^+_i\subseteq \mathbb{R}^n\times (-1,1)$ satisfying the following:
\begin{itemize}
    \item $N^-_i\subseteq M\times(-\epsilon_i,0)$ and $N^+_i\subseteq M\times(0,\epsilon_i)$.
    \item $H_i\times(-\epsilon_i,0)\subseteq N^-_i$ and $H_i\times(0,\epsilon_i)\subseteq N^+_i$.
\end{itemize}

It is easy to see that the set $H_i\times(-\epsilon_i,\epsilon_i)$ is completely isolated in the sense of Subsection \ref{sec3}, which implies we can choose $N^r_i$, $r\in\{+,-\}$ s.t. $P$ does not vanish on both. In detail, $P$ will be strictly positive on $N^+_i$, and negative on $N^-_i$, while it would also vanish on $\partial N^+_i,\partial N^-_i$. This immediately implies some initial conditions in $N^+_i$ and $N^-_i$ flow arbitrarily close to $H_i\times\{0\}$ under the Entropy flow and for each $(s_0,0)$ we can find initial conditions in both $N^+_i$ and $N^-_i$ flowing arbitrarily close to it. \end{proof}
\begin{remark}
    Note the final assertion of Theorem \ref{shilnikoventropy} is independent of the assumption $\nu\ne\frac{1}{2}$, as it only needs the existence of sets $\{H_i\}_i$, i.e., it holds for all $\nu\in(0,1)$. 
\end{remark}
To illustrate our ideas, before moving on we consider two examples of dynamical systems where the Shilnikov scenario occurs. We begin with the following example, derived from \cite{champ}. As shown in \cite{champ}, PDEs governing the motion of water in pipes can be reduced to the following system of second order equations, where $G,u,r>0$:

\begin{equation*}
\begin{cases}
8\frac{d^2\theta}{dt}+3\frac{d^2\phi}{d^t}\cos(\phi-\theta)-3(\frac{d\theta}{dt})^2\sin(\phi-\theta)+3u^2\sin(\phi-\theta) + 3u\frac{d\theta}{dt}+ \\+ 6u\frac{d\phi}{dt} cos(\phi-\theta)+6G\sin(\theta)+18\theta-2\phi+2r+8\theta^3-(\phi-\theta-r)^3=0,\\
3\frac{d^2\theta}{d^2t}\cos(\phi-\theta)+2\frac{d^2\phi}{d^2t}+3(\frac{d\theta}{dt})^2\sin(\phi-\theta)+3u\frac{d\phi}{dt}+2G\sin(\phi)+\\
+2(\phi-\theta-r)+(\phi-\theta-r)^3=0
\end{cases}
\end{equation*}
It is easy to see these two equations can be recast as a dynamical system on $\mathbb{R}^4$. As observed numerically in \cite{champ}, there exists a unique fixed point for the flow and for $G=0,r=0.6, u\approx5.6607$ there exists a homoclinic trajectory to it, associated with period doubling cascades of attractors (see Section 4.1 in \cite{champ}).\\

We now consider the second example, which illustrates that Theorem \ref{shilnikoventropy} remains partially true also when the saddle index $\nu$ is $\frac{1}{2}$, is the Michelson system, originally introduced in \cite{Michh} as a traveling wave solution to the Kuramoto-Sivashinsky PDE. Specifically, recall that given any $c\in\mathbb{R}$, the Michelson system is defined by the following equations:

\begin{equation*} 
\begin{cases}
\dot{x} = y\\
 \dot{y} = z\\
 \dot{z}=c^2-y-\frac{x^2}{2}
\end{cases}
\end{equation*}
It is direct to show that for all parameters $c$ the flow is divergence free - i.e., the flow is volume-preserving. Now, recall that as proven in \cite{DW}, there are countably many parameters $c$ at which the Michelson system undergoes a Shilnikov homoclinic bifurcations. As the Michelson system cannot have either attractors or repellers, by Proposition \ref{generalizedgkp} it follows that at each such Shilnikov parameter we have $\nu=\frac{1}{2}$. Because of the absence of attractors and repellers we cannot apply Theorem \ref{shilnikoventropy}, or, in other words, we cannot conclude the existence of either Milnor attractors or repellers for its Entropy flow. In fact, by Theorem 6.7 in \cite{BaWil} we already know there are intervals of parameters where elliptic periodic orbits surrounded by KAM tori exist, which further complicates matters (i.e., the set $Per$ is likely to be empty due to the Birkhoff Fixed Point Theorem holding around said elliptic orbits). That being said, whenever the Michelson system undergoes a Shilnikov homoclinic bifurcation all the suspended Horseshoes $\{H_i\}_i$ exist. This implies the last implication of Theorem \ref{shilnikoventropy} (and consequently, also the first) still holds. In other words, consider an Entropy flow for the Michelson system with an associated drift function $P$. Then, if $\Gamma$ is a Shilnikov homoclinic trajectory for the Michelson system at some $c_0\in\mathbb{R}$, we can choose the Entropy flow so $P$ would take both positive and negative values in every neighborhood of $\Gamma\times\{c_0\}$. In other words, the Michelson system proves that dissipativeness is not a necessary condition for the Entropy flow to have complex behavior.\\ 

We continue our study of the Entropy flow around the Shilnikov scenario. We first expand the definition of the Shilnikov scenario in order to include the case of the inverse flow. In detail, let $\dot{s}=F_\tau(s)$ denote a $C^2$ family of vector fields on $\mathbb{R}^n\times (-1,1)$, and assume $O\in\mathbb{R}^n$ is a saddle-focus fixed point for $F_0$ with a Shilnikov homoclinic loop $\Gamma$. When $F_0$ satisfies the Shilnikov condition, by Shilnikov's Theorem, its dynamics are complex - but the same is true when $-F_0$ satisfies it. This motivates us to say $F_0$ \textbf{satisfies the Shilnikov condition} provided $|\nu|\in(0,1)$, where $\nu$ is the saddle index for $F_0$ at $0$. It is easy to see that if $-1<\nu<0$, then $-F_0$ satisfies the Shilnikov scenario - thus Theorem \ref{shilnikoventropy} immediately generalizes to the case where $\nu\in(-1,0)$, $\nu\ne-\frac{1}{2}$, with the Milnor attractors becoming Milnor repellers, and vice versa. We now differentiate between these two cases as follows (see the illustration in Figure \ref{transfer}): 
\begin{enumerate}
    \item If $F_0$ satisfies the Shilnikov condition as described in Theorem \ref{shilnikov}, we say its Index is $1$.
    \item In contrast, if $-F_0$ satisfies the Shilnikov condition, we say its Index is $-1$.\\
\end{enumerate}

With these ideas in mind, we now prove the following Corollary of Theorem \ref{shilnikoventropy}:
\begin{corollary}
\label{turbulententropy} With the above notations, let $\{H_i\}_i$ denote the hyperbolic invariant sets suspended around $\Gamma$ w.r.t. $F_0$, given by Theorem \ref{shilnikov}. Assume each $H_i$ persists as a hyperbolic invariant set in some maximal interval $(-\epsilon_i,-\epsilon_i)$, and that $|\nu|\ne\frac{1}{2}$ - where $\nu$ is the saddle index. Then, there exists a countable collection of open sets $N_i\subseteq M\times (-1,1)$ and a choice of Entropy flow on $\mathbb{R}^n\times (-1,1)$, directed by an Entropy vector field $E$ (see Equations \ref{entropyflow}), satisfying the following:
\begin{itemize}
    \item $H_i\times(-\epsilon_i,\epsilon_i)\subseteq N_i$, and moreover, the associated drift function $P$ is positive in $N_i\cap M\times(-\epsilon_i,0)$ and negative on $N_i\cap M\times(0,\epsilon_i)$.
    \item For all sufficiently large $i$, Milnor attractors (or repellers) accumulate on $\partial N_i\cap( \mathbb{R}^n\times\{\pm\epsilon_i\})$.
    \item Moreover, depending on the Index, the following holds:
    \begin{enumerate}
        \item  When the Index of $F_0$ is $1$ then for all $j>i$ and every $(s_j,0)\in H_j\times\{0\}$, there are initial conditions in $N_i$ whose trajectory w.r.t. $E$ flows arbitrarily close to $(s_j,0)$. In addition, for any choice of integers $i<i_1<...<i_k<j$ we can choose these initial conditions in $N_i$ so their trajectory flows arbitrarily close to $H_{i_1}\times\{0\},...,H_{i_k}\times\{0\}$ as it progresses towards $(s_j,0)$.
        \item  When the Index of $F_0$ is $-1$, then for all $j>i$ and every $(s_i,0)\in H_i\times\{0\}$ there exist initial conditions in $N_j$ whose trajectory w.r.t. $E$ flows arbitrarily close to $(s_i,0)$. Conversely, for any choice of integers $i<i_1<...<i_k<j$, we can choose these initial conditions in $N_j$ so their trajectory flows arbitrarily close to $H_{i_1}\times\{0\},...,H_{i_k}\times\{0\}$ as it progresses towards $(s_i,0)$.
    \end{enumerate}
\end{itemize}
\end{corollary}
\begin{proof}

The first assertion is an immediate consequence of the definition of the Entropy flow (see Section \ref{sec3}), where $\{N_i\}_i$ are just some completely isolating neighborhoods (see Definition \ref{def6}). Thus we first prove the second assertion, after which we specify the set $N_i$ precisely. We begin with the case where the Index is $-1$. In this case, $\Gamma$ is a component of the one-dimensional manifold $W^u$, which connects to the $n-1$-dimensional $W^s$ via the two-dimensional submanifold corresponding to the complex-conjugate eigenvalues. We now recall the cross-section $S$ for $F_0$ and the $n-1$-dimensional cubes $\{R_i\}_i$ on it, whose invariant set w.r.t. the first-return map $f: R_i\to S$ for $F_0$ corresponds to the intersection of $H_i$ with $S$ (see Figure \ref{shilnikovth}). By definition, $f:R_i\to S$ is an $n-1$ dimensional Smale Horseshoe map, whose vertical sides are mapped to the side of $\partial S$ intersecting the $n-1$-dimensional stable manifold $W^s$. It now follows that $f(R_i)$ also intersects all $R_j$, $j>i$ (see Figure \ref{shilnikovth}), which implies that if $T\subseteq H_j$ is a periodic orbit, the stable manifold of $T$ w.r.t. $F_0$ intersects $R_i$. In fact, as $f(R_i)$ stretches across all of the $R_j$, $j>i$, the same argument also proves that given any periodic orbit $T_i$ intersecting $R_i$ and any periodic orbit $T_j$ intersecting $R_j$, $i>j$, there exists a heteroclinic trajectory for $F_0$ connecting them. Moreover, we can choose this trajectory to pass through arbitrarily many of $R_t$, $i<t<j$.
    
\begin{figure}[h]
\centering
\begin{overpic}[width=0.45\textwidth]{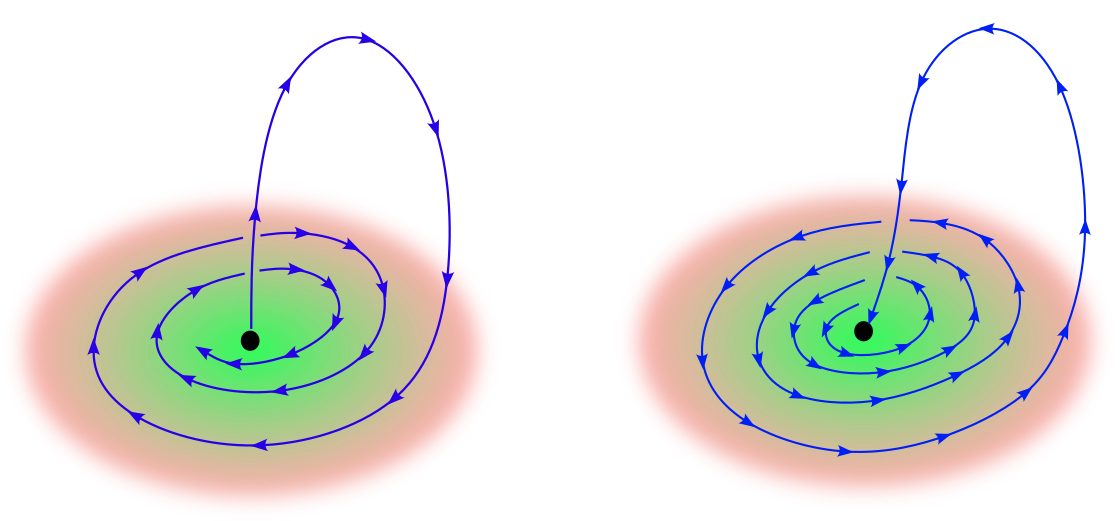}

\end{overpic}
\caption{\textit{Index $1$ homoclinic trajectory is on the left, Index $-1$ is on the right.}}
\label{transfer}
\end{figure}

We now choose an open set $N_i=\cup_{\tau\in(-\epsilon_i,\epsilon_i)}N_i^\tau\times\{\tau\}$, s.t. for all $i,\tau$, $N_i^\tau\cap S=R_i$. Let us choose an Entropy flow defined s.t. the drift function $P$ satisfies the following in $\cup_i N_i$:
\begin{itemize}
    \item $P$ is positive on $N_i^\tau\times\{\tau\}$ if $\tau\in(-\epsilon_i,\epsilon_i)$ and negative if $\tau\in(0,\epsilon_i)$ - we can do so by Definition \ref{def6}.
    \item The $C^1$-norm of $P$ in $\cup_i N_i$ can be chosen to be arbitrarily small, and moreover, we choose it so that $P$ decays uniformly together with its first derivatives as $(s,\tau)\to \cup_i N^0_i$. 
\end{itemize}

We now consider the $N$-dimensional cubes $Q_i=R_i\times(-\epsilon_i,\epsilon_i)$ and let $\mathcal{S}$ denote $S\times(-1,1)$ - by definition, $Q_i=N_i\cap\mathcal{S}$. Provided $P$ is sufficiently $C^1$-small, the first-return map $\mathcal{F}:Q_i\to\mathcal{S}$ is well-defined. Moreover, by the above we also know that $\mathcal{F}(Q_i)$ intersects $Q_j$ for all $j>i$. By the uniform decay as $(s,\tau)\to N^0_i\times\{0\}$, $(s,\tau)\in Q_i$, it follows, the closer initial conditions $(s,\tau)\in\cup_i Q_i$ are to $\cup_i H_i\times \{0\}$, the more their trajectory resembles that of the Laminar flow, i.e., the flow on $\mathbb{R}^n\times (-1,1)$ given by Equations \ref{laminar}. Combined with the existence of an invariant manifold for all periodic orbits in $H_j$ (in $N^0_i$), given any $j>i$ and any $(s_j,0)\in H_j\times\{0\}$, there exist initial conditions in $Q_i$ that are mapped by $\mathcal{F}$ arbitrarily close to every point in $H_j\times\{0\}\cap S\times\{0\}$. This implies we can find points in $Q_i$ flowing arbitrarily close to $(s_j,0)$ w.r.t. the Entropy flow - and moreover, given any finite sequence $i<i_1<..<i_k<j$, we can choose $(s_j,0)$ s.t. its trajectory passes arbitrarily close to $H_{i_1}\times\{0\},...,H_{i_k}\times\{0\}$ as it as the Entropy flow directs it towards $Q_j$.

Finally, by definition (and Theorem \ref{shilnikoventropy}), as $(-\epsilon_i,\epsilon_i)$ is the maximal interval of persistence of $H_i$, we know that $\partial N_i\cap M\times\{\pm\epsilon_i\}$ is the accumulation point for period doubling cascades. Moreover, as the sets $N_i$ are isolated from the said cascades (by their definitions),  we can apply the arguments of Theorem \ref{shilnikoventropy} s.t. when $\nu\in(0,\frac{1}{2})$, Milnor attractors accumulate on $\partial N_i\cap(\mathbb{R}^n\times\{\epsilon^\pm_i\})$ - and when $\nu\in(\frac{1}{2},1)$, Milnor repellers accumulate there instead. Hence, we have proven the Corollary for the case when the Index of $F_0$ is $-1$.
    
We now study the case of the Index $1$. In this case, the sets $N_i$, $Q_i$ and the function $\mathcal{F}$ are defined in the same way. That being said, this time the heteroclinic trajectories connecting any two periodic orbits $T_i$ in $H_i$ and $T_j$ in $H_j$ w.r.t. $F_0$, imply the opposite. Namely, given any initial condition $(s_i,0)\in H_i\times\{0\}$ and given $j>i$, by the density of periodic orbits for $F_0$ inside both $H_i,H_j$ we can always find initial conditions in $Q_j$ sufficiently close to $H_j\times\{0\}$ that are eventually iterated by $\mathcal{F}$ arbitrarily close to $H_i\times\{0\}\cap N_i\times\{0\}$ - and hence, flow arbitrarily close to $(s_i,0)$. Again, the same arguments prove that if we choose some finite sequence $i<i_1<...<i_k<j$, then we can choose those initial conditions s.t. their trajectory w.r.t. $E$ flows arbitrarily close to $H_{i_0}\times\{0\},...,H_{i_k}\times\{i_k\}$ as it progresses towards $Q_i$. The existence of Milnor attractors and repellers accumulating on $\partial N_i$ when $i$ is sufficiently large follows in a similar way. 
\end{proof}
\begin{remark}
\label{nonuniqueness} Homoclinic Shilnikov parameters in $(-1,1)$ in general accumulate on one another - for the precise statement, see Theorem 3 in \cite{g1} and Theorems 1 and 2 in \cite{t1}. As such, despite the simplicity of Figure \ref{cascadeshil2}, in practice it is likely that the Entropy flow around Shilnikov homoclinic parameters can be made significantly more complicated. 
\end{remark}
From now on we refer to an Entropy flow behaving as dictated by both Theorem \ref{shilnikoventropy} and Corollary \ref{turbulententropy} as a \textbf{turbulent Entropy flow}. Before we motivate why we chose this name, we remark that such flows can arise in "real life" systems. To exemplify, recall the Rössler system whose Entropy flow we discussed in the end of Subsection \ref{perioddoubling} (where $a,b,c>0$):
\begin{equation*}
\begin{cases}
\dot{x} = -y-z \\
 \dot{y} = x+ay\\
 \dot{z}=bx+z(x-c)
\end{cases}
\end{equation*}
As mentioned there, this system is well-known numerically to have a period doubling cascade of attractors in its three-dimensional parameter space, which accumulate on surfaces of Shilnikov homoclinic trajectories to the origin (see, for example, \cite{G}, \cite{BBS}, \cite{MBKPS} and the references therein). By \cite{MBKPS} we know that for vector fields on the said homoclinic surface in the parameter space the Index is $1$ and $\nu\in(-\frac{1}{2},0)$. By Proposition \ref{generalizedgkp}, this implies the inverse flows for such parameters has a cascade of repellers - which implies the Rössler system has a cascade of attractors. Moreover, by the numerical evidence, all of these complex dynamics occur inside an attracting invariant set: namely, the Rössler attractor, which is centered around the fixed point at the origin. As such, by Corollary \ref{turbulententropy} we can choose a one-parameter family of Rössler systems, $\dot{s}=F_\tau(s)$, $\tau\in(-1,1)$ that admits a turbulent Entropy flow. For simplicity, let us assume the homoclinic bifurcation occurs at $\tau=0$.\\

Let us now consider a turbulent Entropy flow for the Rössler system on $\mathbb{R}^3\times(-1,1)$. Let the mapping $\varphi_E:(\mathbb{R}^3\times(-1,1))\times\mathbb{R}\to\mathbb{R}^3\times(-1,1)$ denote the flow map, and let $\pi:\mathbb{R}^3\times(-1,1)\to\mathbb{R}^3$ denote the projection to $\mathbb{R}^3$. With the notations above, choose initial conditions $(s_1,\tau_1),(s_2,\tau_2)$ where $s_1,s_2,\tau_1,\tau_2$ are "close" to $0$, and $s_1,s_2$ in $\mathbb{R}^3$ are "close" to the Shilnikov homoclinic loop $\Gamma$ at $F_0$. Denote by $\gamma_i(t)=\pi(\varphi_E((s_i,\tau_i),t)$, $i=1,2$, $t>0$, i.e., the projection of the trajectories of these two initial conditions to $\mathbb{R}^3$. By Corollary \ref{turbulententropy} we can choose $(s_i,\tau_i)$, $i=1,2$ s.t. as $t$ increases, $\gamma_1(t),\gamma_2(t)$ spiral jointly for a long in some neighborhood of $\Gamma$, until being ejected out and getting pulled away towards different invariant sets $I_j,I_k$, the projections under $\pi$ of, say, $H_j\times\{0\}$ and $H_k\times\{0\}$. Moreover, the same is true for initial conditions near $I_1$ and $I_2$ - in other words, we can choose, say, $(s_2,\tau_2)$ s.t. its trajectory is first drawn towards $I_2$ for an arbitrarily long time, until it gets ejected towards $I_1$, while the trajectory of $(s_1,\tau_1)$ is sucked directly towards $I_1$ without first flowing close to $I_2$. Finally, since by both Theorem \ref{shilnikoventropy} and Corollary \ref{turbulententropy} we can find Milnor attractors accumulating arbitrarily close to the origin, it follows the behavior of $\varphi_E$ can be heuristically described as follows:

\begin{itemize}
    \item As there exists an attractor for all $\tau$ (since we assume $\dot{s}=F_\tau(s)$ to be a family of Rössler systems), by Lemma \ref{trapping} we can choose the Entropy flow s.t. it also have an attracting invariant set $\mathcal{A}$. Every slice $A_\tau=\mathcal{A}\times\mathbb{R}^3\times\{\tau\}$, $\tau\in(-1,1)$ is a Rössler attractor for $F_\tau$.
    \item There exist initial conditions $(s_1,\tau_1)$ and $(s_2,\tau_2)$ arbitrarily close to the origin $(0,0)$ that are attracted to different invariant sets inside $\mathcal{A}$ - namely, the projection of the Milnor attractors corresponding to the period doubling orbits.
    \item  In addition, there also exist initial conditions $(s,\tau)$ arbitrarily close to $(0,0)$ whose trajectory flows arbitrarily close to any $H_{i_1}\times\{0\},..., H_{i_k}\times\{0\}$, where $i_1<...<i_k$. Moreover, before drawing close to $H_{i_j}\times\{0\}$ the trajectory of $(s,\tau)$ first flows "close" to $H_{i_r}\times\{0\}$, $k\geq r>j$.\\
\end{itemize}

\begin{figure}[h]
\centering
\begin{overpic}[width=0.4\textwidth]{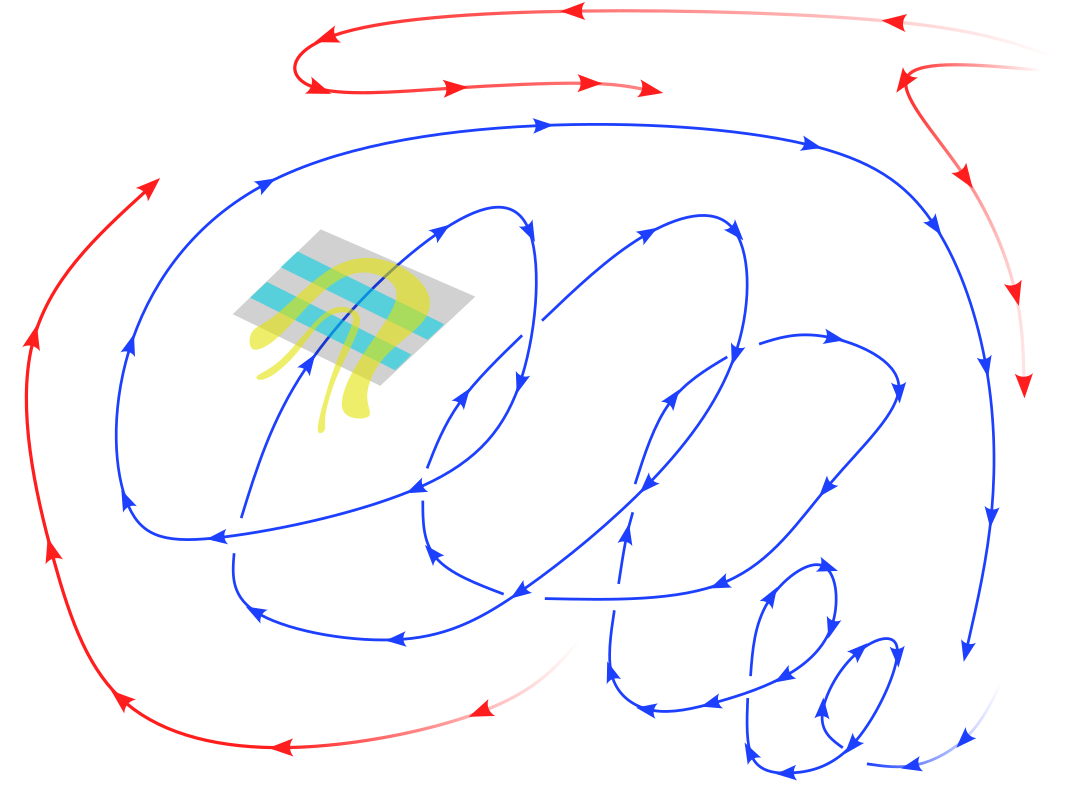}

\end{overpic}
\caption{\textit{How the flow lines for the Entropy flow for the Rössler system around Shilnikov bifurcation possibly look like under the projection to $\mathbb{R}^3$ .}}
\label{shilnikovspiral}
\end{figure}

Summarizing the above, as we project to $\mathbb{R}^3$, the trajectories of initial conditions $v_1,v_2\in\mathbb{R}^3\times (-1,1)\setminus\mathcal{A}$ flow towards $\mathcal{A}$ under the Entropy flow, their motion at first looks ordered, until they draw sufficiently close to the attractor. As the motion of the Entropy flow on $\mathcal{A}$ is erratic and as $\mathcal{A}$ includes different subsets that push initial conditions from one to the other in a certain hierarchy - namely, the sets $\{H_i\times\{0\}\}_i$ - the notion of turbulent fluids comes to mind: that is, a fluid motion that becomes increasingly disordered with time (i.e., the increase in the Reynolds number), coupled with transition of energy between the invariant sets (i.e., the eddies). To be precise, in this analogy the "eddies" for the Entropy flow are the sets $H_i\times\{0\}$, whose hierarchy is defined based on how close they are to the homoclinic loop $\Gamma\times\{0\}$. Moreover, similarly to how larger eddies dissipate energy to smaller ones, by Corollary \ref{turbulententropy} we know initial conditions near $H_j\times\{0\}$ are propelled by the Entropy flow towards $H_i\times\{0\}$, where $i<j$. This comparison, of course, could also apply in much higher dimensional flows which undergo the Shilnikov bifurcation with Index $-1$ and $\nu\in(-\frac{1}{2},0)$.\\

At present, we do not have an explanation for why this analogy arose. Therefore, inspired by the Arnold-Khesin Scheme (see \cite{ArKhe}) and the ideas of \cite{shil3}, \cite{coherent}, \cite{tturbocamb}, \cite{turbopaper}, \cite{turaev}, we conclude this Section by heuristically discussing why we believe this similarity is no mere coincidence. The first argument we make is that at least superficially, turbulent Entropy flows look physically similar to the Taylor–Proudman Theorem (see \cite{TPRotation}). The Theorem states that in a rapidly rotating fluid with slow motion and negligible viscosity, the velocity of the fluid does not change along the axis of rotation. This means that the flow becomes nearly two-dimensional, even though it exists in three-dimensional space, leading to the formation of coherent columnar structures aligned with the rotation axis. The Theorem itself is PDE-based, but since it suppresses 3D variability the flow behaves more coherently and it creates the geometric simplification that often makes ODE descriptions possible. Intuitively, the Shilnikov homoclinic loop does something analogous in phase space: it funnels trajectories towards the saddle-focus equilibrium into a spiral geometrical shape before escaping again - i.e., both involve dynamics dominated by rotational operators, hence this might make their dynamics look similar.\\

Another argument is more philosophical and statistical in nature. To explain it, we first recall Richardson's notion of turbulence (see \cite{RichCascade}), according to which a turbulent flow is composed from eddies of different sizes. These different sizes define a characteristic length scale for the eddies, which are further characterized by flow velocity scales and time scales (turnover time), which is dependent on the length scale. The large eddies are unstable and eventually break up into smaller eddies, with the kinetic energy of the initial large eddy being divided into the smaller eddies that stemmed from it. These smaller eddies undergo the same process, giving rise to even smaller eddies that inherit the energy of their predecessor eddy, and so on. In this way, the energy is passed down from the large scales of the motion to smaller scales until reaching a sufficiently small length scale s.t.  the viscosity of the fluid can effectively dissipate the kinetic energy into internal energy (see \cite{RichCascade} for the precise details). The eddies of size $l$ have some velocity difference (i.e., the magnitude over the length scale), thus the kinetic energy of eddies at scale $l$ and below is proportional to the square of the velocity difference. Theis energy cascade can be described in terms of wavenumbers, $k \sim \frac{1}{l}$. Let \(E(k)\,dk\) represent the kinetic energy per unit mass contained in eddies with wavenumbers between \(k\) and \(k + dk\) - as calculated in \cite{Tong}, using this framework we arrive to the Kolmogorov's $\frac{5}{3}$ law: $E(k) \sim k^{-5/3}$.\\

Even though we do not have a good interpretation for wavenumbers and kinetic energy in the setting of turbulent Entropy flows, one cannot rule out that such an interpretation exists (at least for some one-parameter families of flows). Therefore, at least in principle, the suspended Horseshoes given by Shilnikov's Theorem could be scaled in a way that, w.r.t. this conjectured interpretation, their suspension satisfies the Kolmogorov's $\frac{5}{3}$ law. If that were the case, the resemblance between turbulent behavior and Shilnikov's Theorem can be explained as follows: stretching $\rightarrow$ formation of elongated vortices; folding due to spiraling around the Shilnikov trajectory $\rightarrow$ vortices loop back, interact and create smaller structures; iteration of the first-return map $\rightarrow$ fractal-like hierarchy of eddies (see the illustration in Figure \ref{shilnikovspiral}). In this sense, the equilibrium system of the turbulent Entropy flow is a spiral of suspended Horseshoes that force the dynamics of nearby initial conditions in $\mathbb{R}^n\times (-1,1)$ to behave similarly. Namely, initial conditions passing close to it would appear to behave for a while as dictated by the collection of suspended Horseshoes, until (possibly) drifting away from $\mathbb{R}^n\times\{0\}$. As such, if all the suspended Horseshoes in this model are "to scale" in some sense, by a $C^k$-approximation argument we would expect that the same is true for open sets in $\mathbb{R}^3\times (-1,1)$ transported near them by the Entropy flow (and Corollary \ref{turbulententropy}). \\ 

At present, we do not know if (or how) the above heuristic can be made into a rigorous argument. That being said, as the Shilnikov scenario was proven to exist in the Lorenz-84 model (see  \cite{Broer2}, \cite{Cap}), as well as observed numerically in pipe flows (see \cite{champ}) and in some ocean circulation models (see \cite{NaLu}), we suspect there should be some underlying mathematical explanation. In fact, we believe a potential solution to this problem should be similar to the ideas presented in the GOY shell model (the Gledzer-Ohkitani-Yamada model - see  \cite{GOY}), originally designed to mimic the energy cascade in Fourier space. In detail, we believe something similar could be done to analyze the interaction between different Horseshoes (via stretching factors), which would make the analogue of energy transported between the eddies precise.\\

We will conclude this Section with the words of the horror manga artist Junji Ito, known for works like Uzumaki and Tomie, who once reflected on the unsettling power of spirals: “\textit{The spiral is not just a shape. It’s an idea that infects people. Once you see it, you can’t stop thinking about it. And the more you think, the more it twists your mind}.” In Uzumaki, this concept becomes the central motif, where spirals literally and metaphorically consume the townspeople, embodying obsession, madness, and the inescapable pull of fate. This idea forms the core of his spiral-themed horror, where fascination itself becomes a trap, consuming those who cannot look away.

 \section{Discussion}

Having studied the Entropy flow and its behavior in several scenarios, before we conclude this paper we would like to address several possible extensions and continuations of our work. The reason we do so is because we believe that despite its complexity, the Entropy flow could prove useful as a tool in more than one place. In fact, many of the ideas of this paper originated in our attempts to explore several problems, including the extension of Forcing Theory to bifurcations, studying the bifurcations of the Hopf flow, the Chaotic Hypothesis and the homoclinic spirals observed around the Shilnikov bifurcation. Therefore, in this Section we give a brief review of these problems, explaining their (possible) connection to the Entropy flow, and how we believe it could be used to solve them.\\ 

To begin, recall that as stated in the Introduction, we designed the Entropy flow in order to study two problems. The first major problem we considered was to describe the connection between topology and bifurcations, which we examined in Section \ref{bifdyn}. In the spirit of Subsection \ref{2dtheory}, we think that the main application of the Entropy flow is to inspect how does topology forces certain bifurcations of periodic orbits and fixed points. In other words, we believe the Entropy flow could help study the connection between the topology of a given manifold $M$ and the possible bifurcations of $C^1$ one--parameter families $\dot{s}=F_\tau(s)$ defined one it. That being said, based on the results of the first two authors (see \cite{us}), we believe such theories are dependent at least on the dimension of $M$, and probably also on its topology. In other words, we believe that each manifold $M$ supports a Thurston-Nielsen-like result (see \cite{Fat}), giving sufficient conditions for certain bifurcation to occur. Thus we conjecture the following Sharkovskii type result for bifurcations:

\begin{conjecture*}
    \label{conj1} Let $M$ be a smooth, orientable, metrizable manifold of dimension at least $2$. Then, there exists a class of bifurcations $\{B_\alpha\}_\alpha$ and a partial dynamical ordering on $\{B_\alpha\}_\alpha$, both depending on $M$, satisfying the following: assume a $C^1$ one-parameter family $\dot{s}=F_\tau(s)$, $\tau\in I$ which undergoes the bifurcation $B_\alpha$ at some interior $\tau_\alpha\in I$. Then, if $B_\beta<B_\alpha$ w.r.t. the dynamical ordering, there exists at least one parameter $\tau_\beta$ interior to $I$ s.t. $\dot{s}=F_\tau(s)$ undergoes the bifurcation $B_\beta$ at $\tau_\beta$.
\end{conjecture*}
We note that similar results already exist: for example, from \cite{GKP} (and Proposition \ref{generalizedgkp}) we know Shilnikov bifurcations are preceded by period doubling cascades. Similarly, our results in Section \ref{bifdyn} could also be interpreted as a step towards proving this conjecture. In fact, we suspect the class of bifurcations $\{B_\alpha\}_\alpha$ must at least include those that can be made into isolated invariant sets for the Entropy flow (in the sense of Subsection \ref{2dtheory}) - and that the partial order on $\{B_\alpha\}_\alpha$ has to be related to the Entropy flow in some way.\\

The second problem which underlies this paper is the following: given a $C^1$ one--parameter family $\dot{s}=F_\tau(s)$ which transitions from order into chaos (where $(s,\tau)\in M\times I$), construct a flow on $M\times I$ that realizes the said transition in its dynamics - i.e., make the "laws of motion" change while in motion, by pushing them continuously towards higher dynamical complexity. As the Entropy flow is a construction solving some aspects of this problem, we believe it can be applied to several questions related to such transitions. To illustrate one such example, consider the transition between the Hopf flow and the Kuperberg flow on $S^3$ (or any aperiodic flow on $S^3$ - see \cite{Kup} and \cite{KD}). To be precise, let $\dot{s}=F_\tau(s)$, $(s,\tau)\in S^3\times[0,1]$ be a $C^1$ family satisfying:
\begin{itemize}
    \item $F_0$ directs the Hopf flow on $S^3$: the flow whose trajectories are all periodic orbits of minimal period $2\pi$.
    \item The vector field $F_1$ has no periodic orbits in $S^3$.\\
\end{itemize}

As proven in Proposition 5.1 in \cite{PY}, when $\tau$ is varied from $0$ to $1$ either there exists a fixed point undergoing a Hopf bifurcation or the periods of periodic orbits become unbounded. To the best of our knowledge, there are no other results studying how the Hopf flow bifurcates as it changes into an aperiodic state. We believe the Entropy flow could be useful to study this problem, if only because we would expect (at least intuitively) that the dynamics of $F_1$ should include invariant sets acting either as Milnor attractors or repellers for the Entropy flow (with previous terminology, we would expect these invariant sets of $F_1$ to be equilibrium states for the Entropy flow). If this is the case, we believe, it should follow that the function $P$ is non-zero in some region where the transition occurs, i.e., the bifurcations belong to a certain type. \\

Another more complicated problem where we think the Entropy flow approach could be useful is the Chaotic Hypothesis, originally introduced in \cite{gal}. According to the Chaotic Hypothesis, given a chaotic attractor in, say, $\mathbb{R}^n$, $n\geq3$, its dynamics are "practically hyperbolic". More precisely, according to the Chaotic Hypothesis, one should expect the existence of an SRB measure on $A$ with time averages w.r.t. observables that converge a.e. (for the precise formulation, see \cite{gal}). Heuristically, we believe that the Entropy flow could provide a proof or at least a heuristic justification for the Chaotic Hypothesis. To begin, assume we have a smooth vector field $\dot{s}=F(s)$ defined on some smooth manifold $M$ that admits a chaotic attractor, $A$, which lies away from the fixed points of the flow. Let us further assume we can smoothly deform $F$ to a vector field $G$ in some fixed-point free neighborhood of $A$, $O$, as follows:
\begin{itemize}
    \item No new fixed points appear (or disappear) in $O$ as $F$ is smoothly deformed to $G$.
    \item $A$ is transformed to $A'$, an attractor for $G$, "with no loss of information", i.e., the dynamics of $A$ are semiconjugate to those of $A'$, i.e., $A'$ forms a "topological lower bound" for the complexity of $A$.\\
\end{itemize}

Let us now write $\dot{s}=F_\tau(s)$, $\tau\in[0,1]$, where $F_0=F$, $F_1=G$, and let $E(s,\tau)$ denote an Entropy vector field for this transition. The associated drift function $P$ and the energy transition function $V$ for this family would intuitively be non-zero only on the periodic orbits on $A$ that bifurcate and disappear as $\tau\to1$. In other words, we would expect the dynamics of $-E$, the inverse flow on $M\times[0,1]$, to push towards the dynamics of $F_1$ - i.e.,  much like $E$ pushes the dynamics towards a greater state of complexity, $-E$ would push them towards decreased complexity. When $A'$ can be chosen to be a hyperbolic attractor, the flow of $-E$ can be interpreted as a mechanism to remove the "extra" information from the attractor. As such, the flow of $E$ (or $-E$) could possibly be used to determine just how much of the dynamics of $A$ are practically hyperbolic. In particular, since hyperbolic attractors support an SRB measure (see \cite{SRB} for a survey), we believe the Entropy flow could be useful to study how much of this SRB measure survives as $A'$ is deformed back to $A$.\\

In addition to the above, there are also good reasons to consider the generalization of the Entropy flow - namely, to study the analogous construction for $C^1$ families of vector fields $\dot{s}=F_\omega(s)$, where $\omega$ varies inside some smooth, connected, $k$-dimensional manifold (typically $\mathbb{R}^k$), where $k>1$. To motivate why this is an interesting question, let us recall the spiral structures observed in the parameter space of the Rössler system (see \cite{G}, \cite{BBS}, \cite{MBKPS} and the references therein). Such structures often appear composed from two countable collections of regions in the parameter space - where one collection of regions corresponds to the existence of stable attracting orbits, while another corresponds to the existence of a chaotic attractor. Moreover, all regions in these two collections appear to spiral towards some specific point, often called a \textbf{periodicity hub}, that corresponds to some point on the Shilnikov homoclinic bifurcation set. We believe the Entropy flow could be the key to explain why these structures arise. The reason for that is because the dynamics at the periodicity hub are extremely unstable, in the sense that every small perturbation of them could throw the system off balance and launch it into either a chaotic state or stable state. This effect can even be amplified, as different periodicity hubs in the parameter space appear to accumulate on one another (see \cite{MBKPS}).\\

Spiral structures are not unique to the Rössler system and had been observed in other systems as well (see, for example, \cite{G2}) - as such, it appears likely that there exists some general mechanism underlying them. This mechanism, we believe, is related to the Entropy flow. In detail, assuming a generalization of the Entropy flow exists for higher dimensional parameter spaces, we conjecture that these spiral structures are actually the projection of flow lines for the Entropy flow to the parameter space. This conjecture is motivated by Remark \ref{nonuniqueness}, where we mentioned that due to homoclinic parameters accumulating on one another, stable states should be dense around one another. When the parameter space is two-dimensional, homoclinic parameters are curves that can have fold singularities. As proven in \cite{GKP}, one can match parameters around these curves with periodic orbits undergoing saddle node and period doubling cascades, by drawing spiral curves transverse to the period doubling and saddle node bifurcations. Because periodicity hubs often appear to lie on the fold singularities of the homoclinic curves, we believe these homoclinic fold parameters must somehow be (Milnor) attractors for the generalized of the Entropy flow, which spirals towards them.\\

Another motivation to generalize the Entropy flow to higher dimensional parameter spaces arises from Hamiltonian dynamics. From our construction of the Entropy flow (and from the examples we analyzed) it is clear that when designing it we had dissipative systems in mind, i.e., flows with attractors and repellers. It is hence natural to ask if Entropy flows can, in general, be meaningfully defined for smooth one-parameter families of Hamiltonian flows. In that particular case, the answer depends on higher dimensional Entropy flows. To illustrate, let $\dot{s}=dH_\tau(s)$ by a smooth, one--parameter families of vector fields on, say, $M\times(0,1)$, where $H_\tau:M\to\mathbb{R}$ is a Hamiltonian function for all $\tau\in(0,1)$. Let us further assume that, say, $0$ is a regular value for all $\tau$ - this implies that if $dH_{\tau,0}$ is the restriction of $dH_\tau$ to $H^{-1}_\tau(0)$, one can discuss the smooth family of vector fields $\dot{s}=dH_{\tau,0}(s)$ for $(s,\tau)\in H^{-1}_\tau(0)\times\{\tau\}$. This, of course, applies to any regular value, which intuitively says that smooth one--parameter families of Hamiltonians on a manifold of degree $2n$ can be "foliated" by smooth families of one-parameter vector fields on $2n-1$-dimensional manifolds: namely, their level sets. Or, put simply, every one--parameter family of Hamiltonian flows defines a two-parameter family of flows, at least around regular level sets. This shows that in order to study the Entropy flows of such families meaningfully, one has to develop a theory for Entropy flows with a higher dimensional parameter space. That being said, any extension of our results to $C^k$ families of Hamiltonians would probably require extension of the definitions in Section \ref{defsect} to bifurcations of quasiperiodic motions (see \cite{BHS} for a survey).\\

In addition to the above, one cannot ignore the conceptual similarity between the construction of the Entropy flow and Floer cohomology. We recall Floer cohomology (see \cite{Floer}) studies the topology of a symplectic manifold $M$ by analyzing the action functional on its loop space - i.e., the space $LM = \{\gamma: S^1 \to M \mid \gamma \text{ is smooth}\}$. Critical points of this functional correspond to periodic orbits of a Hamiltonian system and the Floer chain complex is generated by these orbits (see \cite{Floer}). The differential for this chain complex counts solutions to a version of the Cauchy–Riemann equation connecting orbits, which can be viewed as “flow lines” in the loop space. The resulting cohomology is invariant under Hamiltonian isotopies, capturing subtle intersection and dynamical information (see \cite{fuk3}). Floer cohomology thus generalizes Morse Theory from finite-dimensional manifolds to the infinite-dimensional loop space. Therefore, in a sense our construction of Entropy flow is dual to Floer cohomology, since $P(s,\tau)$ in the Entropy flow is defined as a smooth map from $M \times S^1$ to $\mathbb{R}$ (when $I=S^1$), and we have the exponential law: $Map(M \times S^1, \mathbb{R})\simeq Map(M, Map(S^1, \mathbb{R}))$. To see why, note that a function on $M \times S^1$ is the same thing as assigning to each point $x\in M$ a loop in $\mathbb{R}$ - which compares with $LM$ that appoints to each point of the circle a point in $M$ (observe that the roles are reversed). This is especially interesting because of two reasons: the first is because Floer Homology was originally inspired by Conley index Theory. The second is that Floer cohomology has both symplectic and algebraic descriptions, and understanding the interplay between them is central to homological mirror symmetry (it often focuses on $A_{\infty}$-algebras; for details, see \cite{fuk2}, \cite{fuk4}, \cite{fuk3}, \cite{fuk1}). Therefore, we expect a repackaging of our theory in a form that’s more suitable for algebraic and categorical manipulations.\\ 

To conclude, it is impossible not to bring up the name: Evald Vasilievich Ilyenkov (\foreignlanguage{russian}{Эвальд Васильевич Ильенков}), a Soviet philosopher, who connects Entropy with the purpose of Life (see, especially, "The Dialectics of the Ideal" and "Cosmology of the Spirit"). Ilyenkov proposes: "\textit{the goal of intelligent life is the active struggle against the Universe's entropic tendency -- the preservation and development of structured, meaningful forms of Being}", i.e., Life exists only as a process that continuously works against Entropy, and consciousness is the Universe's way of resisting its own decay.

\section{Appendix - Quantifying the behavior of the Entropy flow}
\label{quantitative}

Throughout this paper we studied the Entropy flow using purely qualitative means - namely, topological tools. Inspired by the end of Section \ref{turbulence}, by \cite{facet} and by the theory of the Koopman Operator (see \cite{Koopman} and \cite{Kman2}), in this Appendix we attempt to approach the Entropy flow quantitatively. That being said, as we will see, this approach also has several limitations which we will discuss below. Our main result in this Appendix is Theorem \ref{optimization}, which, in a nutshell, proves one can minimize a (compact) collection of Entropy flows, w.r.t. to several constrains.\\

To begin, consider a $C^1$ family $\dot{s}=F_\tau(s)$ defined over some smooth manifold $M$ and parameterized by some one-dimensional manifold $I$. Let us consider the family as a $C^1$ curve $\gamma:I\to\Xi(M)$, where $\Xi(M)$ denotes the Banach manifold of all $C^1$ vector fields on $M$. We denote by $\mathcal{E}(\gamma)$ the collection of all vector fields on $M\times I$ corresponding to Entropy flows that can be defined on $\gamma$. Ideally, one way to study a given Entropy flow $E\in\mathcal{E}(\gamma)$ is to study it via the Koopman Operator (see \cite{Koopman}) - namely, to convert $E$ to an infinite dimensional system and study the $L^1$-functions defined on it. This approach is known to work well for measure preserving flows. That being said, as shown throughout this paper, in general one should expect an Entropy flow to have interesting dynamics when it is dissipative - i.e., when it has attractors and repellers, which would imply the assumption $E$ is measure-preserving (not to mention ergodic) is in many interesting cases unrealistic. This motivates us to study the set $\mathcal{E}(\gamma)$ with an alternative approach.\\

To motivate, recall the notation $E(s,\tau)=(F_\tau(s)+V(s,\tau),P(s,\tau))$ and let us denote the Lie bracket of two vector field $F$ and $G$ by $[F, G]$. Moreover, given an open domain $A\subseteq M\times I$, let us denote by $|| \ast||_{1,A}$ the $C^1$-norm in $A$. Now, given $E\in\mathcal{E}(\gamma)$ and a domain $A\subseteq M\times I$  define $g:\mathcal{E}(\gamma)\to [0, +\infty]$ via
$$g(E)=||[E(s,\tau),\mathcal{L}(s,\tau)]||_{1,A},$$
where $\mathcal{L}(s,\tau)=(F_\tau(s), 1)$. In the spirit of Section \ref{turbulence}, one could interpret $g$ as measuring how far away are the dynamics of $E$ from inducing a trivial drift on the parameter space throughout $A$. The function $g$ is easily seen to be continuous - as such, if it can be minimized in $\mathcal{E}(\gamma)$, there is a certain lower bound for the complexity of all Entropy flows in $\mathcal{E}(\gamma)$. Specifically, if there exists some $c>0$ s.t. for all $E\in\mathcal{E}(\gamma)$ we have $g(E)>c$, it implies that the complexity of the Entropy flow (when measured via $g$) is intrinsic to the bifurcation structure of $\dot{s}=F_\tau(s)$. To generalize this idea, we begin with the next definition, inspired by the ideas of \cite{facet}:

\begin{definition}
    \label{measurablefunction} A non-constant continuous function $f:\mathcal{E}(\gamma)\to \mathbb{C}$ is called an \textbf{aspect}. An aspect is said to be \textbf{independent of $\mathcal{E}(\gamma)$} (or just \textbf{independent} when the context is clear) if $0=\inf\{|f(E)|, E\in\mathcal{E}(\gamma)\}$. Otherwise, we say $f$ is \textbf{dependent on $\gamma$}, or in short, \textbf{dependent}.
\end{definition}
The definition above has the following meaning - every aspect is some quantifiable notion of the complexity of the Entropy flow. If it is independent, then that aspect is meaningless, i.e., this aspect is independent of the dynamical complexity of $\mathcal{E}(\gamma)$ and is unrelated to whatever complex dynamics arises by the Entropy flow of $\gamma$. In this context, it is easy to see that $g$ is a type of aspect which measures how close is the drift function $P$ is to a constant sign and how much the energy transition map $V$ is close to $0$. Definition \ref{measurablefunction} naturally leads us to ask the following two questions:

\begin{enumerate}
    \item Do dependent aspects even exist? In other words, is the above definition not redundant?
    \item Can we somehow optimize and find which Entropy flow jointly minimizes (or maximizes) a certain collection of aspects? \\
\end{enumerate}

We are going to consider both of these questions. We first show, by an example, that there exists a curve of vector fields $\gamma$, a constant $c>0$ and an aspect $g$ s.t. for all $E\in \mathcal{E}(\gamma)$, $|g(E)|>c$. In other words, we will prove $g$ is a dependent aspect w.r.t. $\gamma$, thus showing the definition is not redundant. To this end, consider a bifurcation scenario as depicted in Figure \ref{no0}, where two attracting periodic orbits, $T_1,T_2$, are created by a saddle node bifurcation in opposing directions in the parameter space. One easily sees by the definition that for all choices of ${E}\in\mathcal{E}(\gamma)$, ${E}=(F_\tau+V,P)$, the sign of $P$ along the orbits of $T_1$ is opposite to that of the orbits of $T_2$. Hence, $[(F_\tau+V,P),(F_\tau,1)]$ is not the $0$ vector field as the flows of these vector fields do not commute. Let us define $g(E)=||[E(s,\tau),(F_\tau(s),1)]||_{1, M\times I}$, i.e., the $C^1$-norm of the Lie bracket throughout $M\times I$. We now prove the following:
\begin{lemma}
Let $g$ be the aspect defined above. Then, every manifold $M$ of dimension $3$ and up supports a $C^1$ one-parameter family $\dot{s}=F_\tau(s)$, $\tau\in[0,1]$, for which $g$ is a dependent on $\gamma$.
\end{lemma}
\begin{proof}
First, note that every manifold $M$ as above admits some $C^1$ family having periodic orbits with a bifurcation diagram as in Figure \ref{no0}. Namely, there exists a $C^1$ family which has two saddle node bifurcation orbits, $T_1$ and $T_2$ for $F_{\tau_1}$ and $F_{\tau_2}$, s.t. the following holds:

\begin{itemize}
    \item There exist two $0<\tau_1<\tau_2<1$ s.t. as $\tau$ enters $(\tau_1,1]$, $T_1$ splits into two periodic orbits: an attractor and a repeller, $T^1_A$ and $T^1_R$, which persist as attractors and repellers for $F_\tau$, $\tau\in(\tau_1,1]$.
    \item Similarly, as $\tau$ crosses into $[0,\tau_2)$, $T_2$ splits into two orbits: an attractor, $T^2_A$, and a repeller, $T^2_R$, which persist without bifurcating as periodic orbits for all $F_\tau$, $\tau\in[0,\tau_2)$.
\end{itemize}
    
Let $C^i_j$, $i=1,2$, $j=A,R$ denote the components of $Per$ that include the attractors and repellers $T^i_j$. By definition, each Entropy flow $E(s,\tau)=(F_\tau(s)+V(s,\tau),P(s,\tau))\in\mathcal{E}(\gamma)$ satisfies that $P$ is negative on $C^2_j$ and positive on $C^1_j$. Now, choose some sequence $\{E_n\}_n\in\mathcal{E}(\gamma)$ and assume by contradiction $g(E_n)\to 0$. As $g$ is continuous, it follows that the motion induced by the vector field $\mathcal{L}(s,\tau)=(F_\tau(s),1)$ comes closer and closer to commuting with the flow of $E_n$ as $n\to\infty$. In other words, writing ${E}_n(s,\tau)=(F_\tau(s)+V_n(s,\tau),P_n(s,\tau))$, the motion induced by $P_n$ on the parameter space $I$ must commute with the motion induced by the constant velocity $1$. As for all $n$ the function $P_n$ is strictly positive on $C^1_A, C^1_R$ and strictly negative on $C^2_A, C^2_R$, this cannot happen. This is a contradiction, which implies $g$ is dependent on $\gamma$, i.e., that dependent aspects exist.
\end{proof}
\begin{remark}
Intuitively, given an aspect one should not expect to find an upper bound on complexity. More precisely, taking again $g$ as an example, one can always perform a perturbation of $P$ while keeping it admissible, and ensure the $C^1$ distance between $[E,\mathcal{L}]$ and $0$ increases. 
\end{remark}

\begin{figure}[h]
\centering
\begin{overpic}[width=0.3\textwidth]{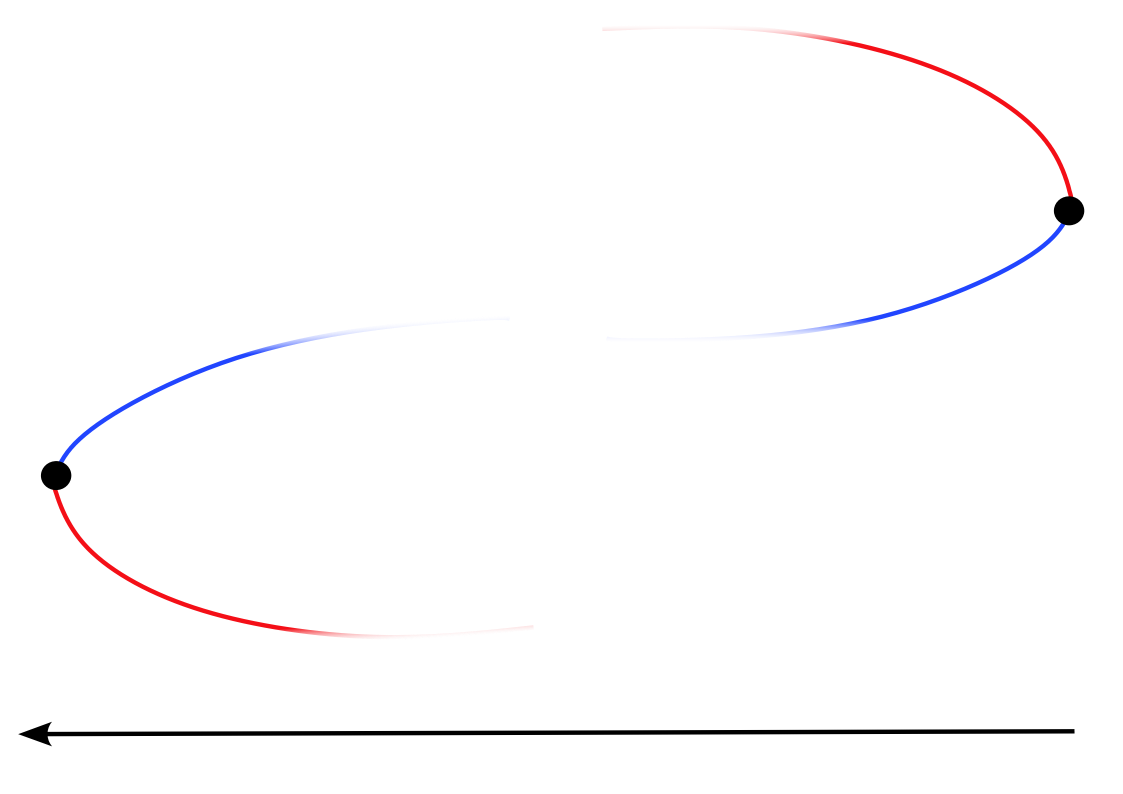}
\end{overpic}
\caption{\textit{A bifurcation with two different saddle node bifurcations pusing at opposing directions. The blue curves correspond to the attracting orbits $T_1$ and $T_2$ created at the bifurcation. }}
\label{no0}
\end{figure}

Having proven the notion of independence is not a purely abstract notion, we now proceed to show that dependent aspects w.r.t. to $\gamma$ can be optimized to some extent. To make this statement precise, given $n$ aspects $f_1,...,f_n:\mathcal{E}(\gamma)\to\mathbb{C}$ and some compact $\mathcal{C}\subseteq \mathcal{E}(\gamma)$, we say an Entropy flow $E\in\mathcal{C}$ \textbf{minimizes $f_1,...,f_n$ in relation to one another} if given some $E'\in C$ s.t. if $|f_i(E')|<|f_i(E)|$ for some $i\in\{1,...,n\}$, there exists some $j\ne i$ s.t. $|f_j({E})|<|f_j({E}')|$. In this case, $|f_1(E)|,...,|f_n(E)|$ would be referred to as the \textbf{relative minimal values}. To visualize, let us perturb ${E}$ a little inside $\mathcal{C}$ s.t. the value of $f_1$ would decrease - then, by the definition we would have to increase the value of, say, $f_2$. This definition aims to capture the following idea: if $E\in\mathcal{C}$ minimizes $f_1,...,f_n$ in relation to one another, it is the "minimal" Entropy flow, for which all the properties measured by $f_1,...,f_n$ are  small as they can be, i.e., $E$ is essentially a saddle point on $\mathcal{C}$ w.r.t. $f_1,...,f_n$. With that idea in mind, we now prove:
\begin{theorem}
    \label{optimization} Let $f_1,...,f_n$ be aspects, let $c_1,...,c_n>0$ be scalars s.t. $\{f_i\geq c_i\}$ is not empty in $P$, and let $\mathcal{C}\subseteq E(\gamma)$ be compact, connected and non-empty. Then, if for all $i=1,...,n$ the sets $T_i=\{|f_i|\geq c_i\}\cap \mathcal{C}\ne\emptyset$, the following holds:
    \begin{itemize}
        \item There exists some ${E}\in \mathcal{C}$ which minimizes $f_1,...,f_n$ in relation to one another in $\cup_{i=1}^n T_i$ - and when $f_1,...,f_n$ are all dependent on $\gamma$, we can minimize $f_1,...,f_n$ in relation to one another on $\mathcal{C}$. 
        \item  Let us denote by $\mathcal{C}_r=\mathcal{C}\cap(\cup_{i=1}^n\{c_i\leq |f_i|\leq r\})$. Then, the relative minimal values of $f_1,...,f_n$ in $\mathcal{C}_r$ vary continuously with $r$.
        
    \end{itemize}

\end{theorem}

\begin{proof} 
 We prove Theorem \ref{optimization} using the Berge Maximum Theorem. To this end, recall that given two topological spaces $X,Y$ and a multivalued map $G:Y\to 2^X$, the map $G$ is said to be \textbf{hemicontinuous} if it satisfies the following two properties:
 \begin{enumerate}
     \item \textbf{Upper hemicontinuity} - if $O\subseteq X$ is open and $G(y)\subset O$, there exists some neighborhood $O'\subseteq Y$ of $y$ s.t. for all $y'\in O'$, $G(y')\subset V$.
     \item   \textbf{Lower hemicontinuity} - if $O\subseteq X$ is open and $G(y)\cap O\ne\emptyset$, there exists a neighborhood $O'$ of $y$ s.t. for all $y'\in O'$, $G(y')\cap O\ne\emptyset$.
 \end{enumerate}
 
 It is easy to see that any continuous single-valued function $g:Y\to X$ defines a hemicontinuous map by setting $G:Y\to 2^X$, $G(y)=\{g(y)\}$. Under these definitions, we have the following result (for a proof, see \cite{Berge}, Chapter VI, "A Maximum Theorem"):
 \begin{theorem*}
     Assume $X,Y$ are topological spaces, let $f:X\times Y\to\mathbb{R}$ be continuous, and let $F:Y\to 2^X$ be a hemicontinuous map s.t. for all $y$ the set $F(y)\in 2^X$ is compact and non-empty. Set $v(y)=\sup\{f(x)|x\in F(y)\}$ and $Arg(y)=\{x\in F(y)|f(x,y)=v(y)\}$. Then
     \begin{enumerate}
         \item $v$ is well-defined and continuous in $y$.
         \item $Arg(y)$ is non-empty, compact and upper-hemicontinuous - i.e., $v$ has a proper maximum in every $F(y)$.
     \end{enumerate}
 \end{theorem*}
We apply the the Berge Theorem as follows. Choose some compact $\mathcal{C}\subseteq \mathcal{E}(\gamma)$ and set $\Theta=(c,\infty)^n$, where $c>0$ is chosen s.t. $T_i=\{c\geq |f_i|\geq c_i\}\cap C$ is non-empty for all $i$. For every Entropy flow ${E}\in E(\gamma)$, set $f({E})=\frac{1}{\sum_{i=1}^n\lambda_i |f_i({E})|}$, where $\sum_{i=1}^n\lambda_i=1$. We first claim it would suffice to maximize $f$ - to see why, assume $f$ is maximized at some ${E}\in\mathcal{ E}(\gamma)$, and let ${E}'\in\mathcal{ E}(\gamma)$ be s.t. $|f_i({E}')|\leq |f_i({E})|$ for all $i$, where at least some of the inequalities are sharp. This would immediately imply ${f({E}')}> {f({E})}$, which contradicts ${E}$ being the maximum of $f$. As such, it would suffice to maximize $f$. Therefore, we now define the following functions:
 \begin{itemize}
     \item Write $\theta=(\theta_1,...,\theta_n)$. We define $F:\mathcal{C}\times\Theta\to\mathbb{R}^+$ by setting $F({E},\theta)=f({E})$ - it is easy to see $F$ is continuous by its definition.
     \item Set $\mathcal{D}:\Theta\to 2^{\mathcal{E}(\gamma)}$ by $\mathcal{D}(\theta_1,...,\theta_n)=(\cup_{i=1}^n |f_i|^{-1}([c_i,\theta_i]))\cap \mathcal{C}$. This function is trivially hemicontinuous and per our assumption on $c_i$, $\mathcal{D}(\theta)$ is always compact and non-empty. 
\end{itemize}

Consequently, writing $f^*(\theta)=\sup\{F(x,\theta)|x\in \mathcal{D}(\theta)\}$, by Berge's Maximum Theorem, we know the set $M(\theta)=\{x\in \mathcal{D}(\theta)|F(x,\theta)=f^*(\theta)\}$ is compact and non-empty. Moreover, $f^*(\theta)$ is continuous in $\theta$. In other words, we have proven that for all $r_i>c$, $i=1,...,n$, we can maximize $f$ in $\cup_{i=1}^n |f_i|^{-1}([c_i,r_i])\cap\mathcal{C}$ - thus jointly minimizing $f_1,...,f_n$ with relation to one another in the same set. Recalling that $\mathcal{C}$ is compact and the aspects $f_1,...,f_n$ are continuous, it follows that there exist some constants $R_1,...,R_n$ for which  $T_i=|f_i|^{-1}([c_i,R_i])\cap\mathcal{C}=\{c_i\leq |f_i|\}\cap C$ - thus we can minimize $f_1,...,f_n$ in relation to one another in $\cup_{i=1}^n T_i$. Moreover, when $f_1,...,f_ n$ are dependent, we can choose $c_i$, $i=1,...,n$, s.t. $\mathcal{C}=\cup_{i=1}^n T_i$. Finally, writing $\mathcal{C}_r=\mathcal{C}\cap(\cup_{i=1}^n\{c_i\leq |f_i|\leq r\})$, by the discussion above we know that as we vary $r$, the maximum of $f$ varies continuously - as such, the relative minima of $f_1,...,f_n$ also varies continuously.
\end{proof}
\begin{remark}
Theorem \ref{optimization} is particularly inspired by the ideas of \cite{facet}, according to which chaos is the same in all systems, only observed in different ways. Moreover, recalling notion of Information Entropy introduced in \cite{Shannon}, recall it is uniquely determined (up to some positive multiplicative constant) by three axioms: \textbf{Continuity} (small changes in probabilities should only lead to small changes in uncertainty), \textbf{Maximality} (uncertainty is greatest when you are completely unsure which outcome will happen), and \textbf{Recursivity} (uncertainty can be calculated step by step without changing the overall measure). It is easy to see that the aspects $f_1,...,f_n$ lack only the Recursivity Property. \end{remark}

In the spirit of Theorem \ref{optimization}, we conclude this Appendix by providing a list of possible aspects - together with their possible interpretation w.r.t. the Entropy flow. Needless to say, our list is far from being complete and it is likely many more aspects could easily be conceived. To this end, let $M$ be a smooth manifold, let $\dot{s}=F_\tau(s)$ be a $C^1$ one--parameter family of vector fields parameterized by $\tau\in I$, and let $E(s,\tau)=(F_\tau(s)+V(s,\tau),P(s,\tau))$ denote some corresponding Entropy flow. With these notations, we list the following collection of possible aspects:

\begin{itemize} 
    \item The standard $C^1$-norm for the function $P$, $V$, or both is an aspect measuring how "fast" is the drift on the parameter space induced by $P$, or alternatively, how strong is the energy transition along the bifurcations induced by $V$.
    \item $||[E(s,\tau),(F_\tau(s),1)]||_{1,A}$, where $A$ is a domain in $M\times I$. This aspect measures how much a given Entropy flow deviates from having a "simple" drift with constant velocity in the parameter space and no energy transition (the constant $1$, of course, can be replaced by any other non-zero constant).
    \item $||[E(s,\tau),(F_\tau(s),0)]||_{1,A}$, where again $A$ is a domain in $M\times I$. This aspect measures how much a given Entropy flow deviates from being the Laminar flow given by Equations \ref{laminar}.
    \item $||P(s,\tau)+\int div(F_\tau+V(s,\tau))d\tau||_{1,O}$, where $div$ is the divergence form on $M$ and $O$ is an open set in $M\times I$. This aspect quantifies how much the Entropy flow deviates from being volume-preserving. To see why, note that when $M=R^n$, $E$ is divergence free in $O$ precisely when $\sum_{i=1}^n\frac{\partial(F_\tau(s)+V(s,\tau))}{\partial s_i}+\frac{\partial P(s,\tau)}{\partial\tau}$ vanishes identically in $O$ (where $s=(s_1,...,s_n)$).
    \item The distance from hyperbolicity (see \cite{Gromov}), which is measured by the smallest $\delta$ s.t. for any geodesic triangle in the space, each side of the triangle is within a $\delta$-neighborhood of the union of the other two sides. This notion could possibly help measure how curved are the projections to $M$ of any three flow lines for the Entropy flow in $M\times I$.
    \item The Gromov-Hausdorff distance (see \cite{GHD2}, \cite{GHD1}) or the Wasserstein distance (see \cite{WasD}) applied to $E$ all quantify how much the flow of $E$ deviates from being symplectic.\\ 
\end{itemize}

\section{Appendix - Conley indices for period doubling bifurcations}
\label{periodappendix}

Recall that in Subsection \ref{2dtheory} we proved the Conley index can be used to derive constrains on the dynamics of the Entropy flow - and consequently, on the bifurcations of the one-parameter family it corresponds to. In this Appendix, we discuss the  Conley index for the Entropy flow around a finite number of period doubling bifurcations (under certain idealized assumptions, explained below). In detail, we will prove that under these idealized assumptions the Entropy flow admits isolating blocks for a sequence of three consecutive period doubling bifurcations, and that the corresponding entropy-adapted index filtration has dyadic degree-one transport - where by \textbf{dyadic} we mean that the degree-one Conley groups are $ \mathbb Z/2$, $\mathbb Z/4$, and $\mathbb Z/8$. In particular, we will prove the inclusion-induced transport maps given by Definition \ref{transportmap} between these groups are multiplication by powers of two. \\

To begin, let $\dot{s}=F_\tau(s)$ be a $C^3$ one-parameter family defined on some three-manifold $M$ and parameterized by $\tau\in I$. Keeping up with the notations of Subsection~\ref{perioddoubling}, we will denote the vector field generating an Entropy flow for the family by \(E\) - i.e., the Entropy vector field, in the terminology of Section \ref{bifdyn}. Recall a surface $S$ embedded in $M\times I$ is a \textbf{cylinder of periodic orbits} if for all $\tau\in I$ s.t. $S\cap M\times\{\tau\}\ne\emptyset$, all components of the said intersection are periodic orbits for $F_\tau$. Let us assume that as we vary $\tau$, the system undergoes three consecutive period doubling bifurcations at parameters $\tau_1<\tau_2<\tau_3$. This implies there exist three cylinders of periodic orbits in $M\times I$, denoted by \(C_n\subset M\times I\), $n=1,2,3$. For \(i=1,2,3\), \(C_i\) denotes the regular period doubled continuation sheet between its two indicated boundary circles. Thus the interior of \(C_i\) contains no bifurcation orbit. Each $C_n$ is glued to one another at sets of the type \(P_n\times\{\tau_n\}\) - where $P_n$ are periodic orbits for the vector fields $F_{\tau_n}$ corresponding to the period doubling bifurcation. Throughout this Appendix, we write \(P_n\) for \(P_n\times\{\tau_n\}\) when this causes no confusion. The geometric picture is therefore as follows: \(C_1\) connects \(P_1\) to \(P_2\), \(C_2\) connects \(P_2\) to \(P_3\), and \(C_3\) is the next branch born at \(P_3\) (see the illustration in Figure \ref{ill:perioddoublingillustration}).\\

In particular, as we vary periodic orbits on $C_n$ via $P_{n+1}$ to $C_{n+1}$, the period w.r.t. the flow is doubled. Let \(\tau_4>\tau_3\) denote the next period doubling parameter after \(\tau_3\). Choose \( \tau_*\in(\tau_3,\tau_4),\)and let \(L=T_{\tau_*}\times\{\tau_*\}\subset C_3 \) be the corresponding periodic orbit. The circle \(L\) will serve as the terminal exit set of the index pairs constructed below. We will always assume that \(P_1,P_2,P_3\) are nondegenerate flip orbits, at each $P_n$ there exists a cylinder of periodic orbits, \(D_n\), glued to $C_{n}$ and $C_{n+1}$ at $P_n$, corresponding to periodic orbits for $\dot{s}=F_\tau(s)$ along with the preserved period (i.e., the Möbius orbits in the terminology of \cite{PY2}). Specifically, the geometric picture is as follows: $P_1$ glues $C_1$ and $D_1$, $P_2$ glues $C_1, C_2$ and $D_2$, while $P_3$ glues $C_2,C_3$ and $D_3$ (see the illustration in Figure \ref{ill:perioddoublingillustration}). To further clarify our notation, note we let \(P_n \subset \overline{D_n}\), \(P_n \subset \overline{C_n}\) but \(P_n \cap D_n =P_n\cap C_n=\emptyset\).\\ 

We now recall the standard Conley index notation (see Definition \ref{index}). For a compact set \(A\subset M\times I\) we denote the invariant set w.r.t. the vector field $E$ by $\operatorname{Inv}_E(A)=\{x\in A\mid \varphi_E(t,x)\in A\text{ for all }t\in\mathbb R\}$, where \(\varphi_E\) is the flow of \(E\). If \(\mathcal P=(N,L)\) is an index pair, then $CH_q(\mathcal P;\mathbb Z):=H_q(N,L;\mathbb Z)$. Moreover, in this Appendix the notation \(\mathcal P_k=(N_k,L_k)\) always denotes an Entropy flow index pair in \(M\times I\). For the homology calculation we use the following periodic-orbit skeleta:
\[
        \widehat C_1=P_1\cup C_1\cup P_2,
\]
\[
        \widehat C_2=P_2\cup C_2\cup P_3,
\]
\[
        \widehat C_3
=
P_3
\cup
\bigl(C_3\cap(M\times[\tau_3,\tau_*])\bigr)
\cup
L.
\]
We will use these two-dimensional sets to construct index pairs for $E$, under idealized assumptions which will be stated later on. To this end, define the sets:
\[
        X_0=L,
\]
\[
        X_1=\widehat C_3,
\]
\[
        X_2=\widehat C_2\cup_{P_3}\widehat C_3,
\]
\[
        X_3=\widehat C_1\cup_{P_2}\widehat C_2\cup_{P_3}\widehat C_3.
\]
Here \(\cup_{P_i}\) denotes gluing along the circle \(P_i\).  In this Appendix, the word skeleton will refer only to these explicitly defined sets \(X_k\), $k=0,1,2,3$ - see the illustration in Figure ~\ref{ill:perioddoublingillustration}.\\

\begin{figure}[h]
\centering
\begin{overpic}[width=0.6\textwidth]{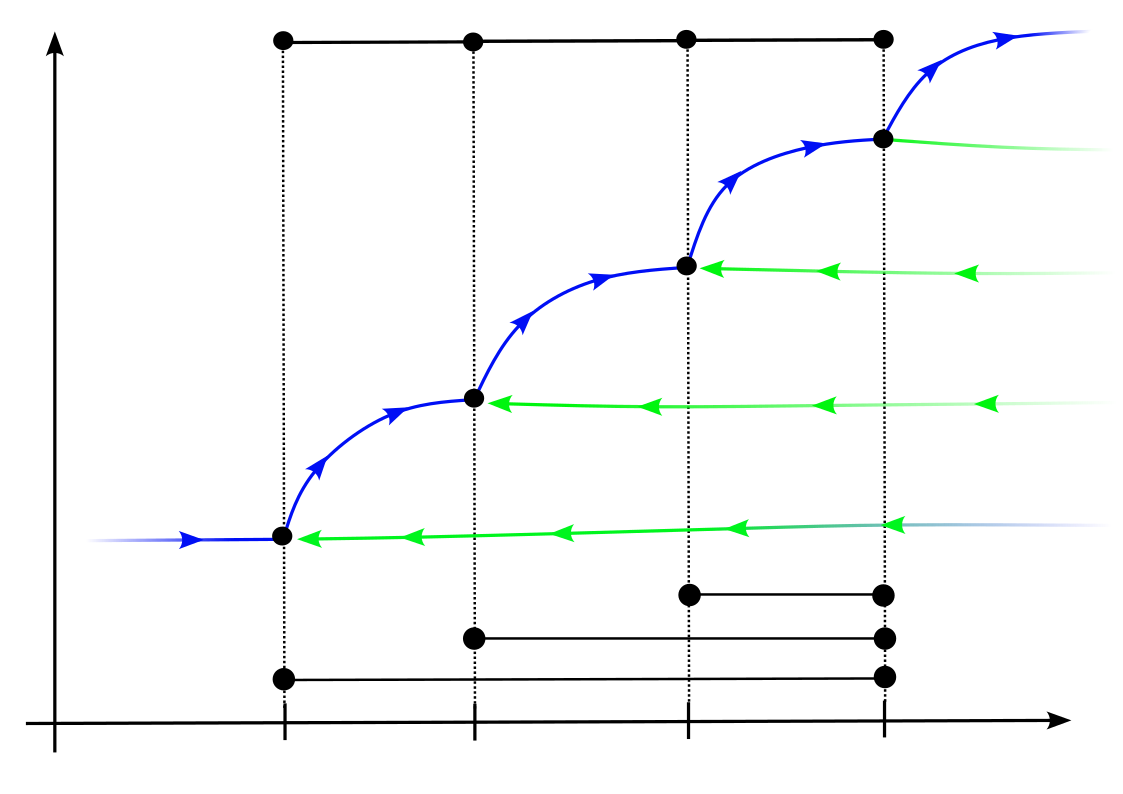}
\put(780,20){$\tau_*$}
\put(605,20){$\tau_3$}
\put(410,20){$\tau_2$}
\put(245,20){$\tau_1$}

\put(685,190){$x_1$}
\put(545,145){$x_2$}
\put(305,115){$x_3$}

\put(750,540){$L$}
\put(570,420){$P_3$}
\put(380,300){$P_2$}
\put(210,180){$P_1$}

\put(1020,230){$D_1$}
\put(1000,330){$D_2$}
\put(1020,430){$D_3$}

\put(690,600){$C_3$}
\put(500,485){$C_2$}
\put(330,365){$C_1$}

\put(690,690){$\widehat C_3$}
\put(500,690){$\widehat C_2$}
\put(330,690){$\widehat C_1$}

\put(30,690){$M$}
\put(980,60){$\tau$}

\end{overpic}
\caption{\textit{A diagram of three consecutive period doubling bifurcations, as described above.}}
\label{ill:perioddoublingillustration}
\end{figure}
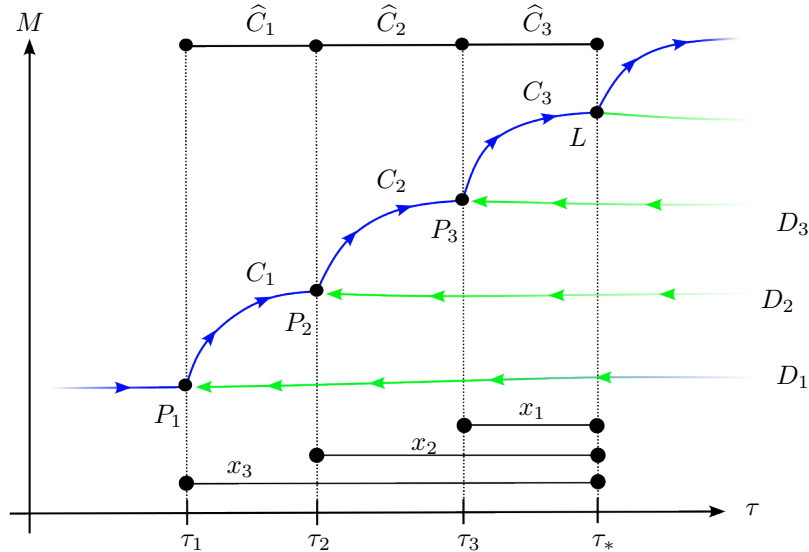

We shall use the word \textbf{flip} in the following sense: with previous notations, assume $T$ is a period doubling orbit for some vector field $F_\tau$ in the $C^3$ family $\dot{s}=F_\tau(s)$. We say $T$ is a nondegenerate flip orbit if, given a cross section $S_\tau$ transverse to $T$ at some $s_\tau$, the differential for the first-return map $f:S_\tau\to S_\tau$ at $s_\tau$ has exactly one eigenvalue that is a root of unity, $-1$ (as proven in the Appendix of \cite{PY2}, such period doubling bifurcations are $C^3$-generic). In our setting, we will always assume $P_1,P_2,P_3$ are nondegenerate flip orbits for $F_{\tau_1},F_{\tau_2}$ and $F_{\tau_3}$ respectively. The next Lemma records the only local dynamical input needed for the dyadic calculation.

\begin{lemma}
\label{lem:local-flip-attachment-Cn}

Let \(P_i\), $i=1,2,3$, be a nondegenerate flip bifurcation orbit of the $C^1$ family \(\dot s=F_\tau(s)\), and choose a regular parameter \(\tau\) on the period-doubled side of \(P_i\), and set \(
\Gamma_\tau=C_i\cap(M\times\{\tau\}).
\) Then, there is a projection from \(\Gamma_\tau\) to \(P_i\) whose degree is \(\pm2\).  Moreover, \(\widehat C_i\) admits a deformation retraction  $r_i:\widehat C_i\longrightarrow P_i$ whose restriction to \(\Gamma_\tau\) is homotopic to the phase projection above and the restriction of \(r_i\) to the outgoing circle has degree \(\pm2\) too.

\end{lemma}

\begin{proof}
Choose a point on \(P_i\), a local cross-section $S$ transverse to $P_i$ at that point, and let \(f_{\tau}:S\to S\) denote the corresponding first return map for $F_\tau$, for all $\tau$ close to $\tau_i$. At a flip bifurcation, the doubled orbit, appearing for $\tau>\tau_i$, intersects the cross-section in two points, \(x_+(\tau)\) and \(x_-(\tau)\), where $f_\tau(x_+(\tau))=x_-(\tau)$ and $f_\tau(x_-(\tau))=x_+(\tau)$. Thus, one full circuit of \(\Gamma_\tau\) consists of two first-return passages near \(P_i\). In a tubular neighborhood of \(P_i\), choose coordinates $ (\theta,y)\in S^1\times D^2$ so that $ P_i=\{y=0\}$ and \(\theta\) is the phase coordinate along \(P_i\). Now, let $\pi:S^1\times D^2\to P_i$ be the projection \(\pi(\theta,y)=(\theta,0)\). A first-return passage near \(P_i\) follows one full turn of the phase coordinate \(\theta\).  Since the period-doubled orbit \(\Gamma_\tau\) closes only after two first returns, the map $\pi|_{\Gamma_\tau}:\Gamma_\tau\to P_i$ winds twice around \(P_i\). Thus $\deg(\pi|_{\Gamma_\tau})=\pm2$, i.e., the induced map $H_1(\Gamma_\tau;\mathbb Z)\to H_1(P_i;\mathbb Z)$ is multiplication by \(\pm2\).

Since \(C_i\) contains no bifurcation orbit between \(\Gamma_\tau\) and the boundary circle of \( \widehat C_i \), the slice periodic orbits in \(C_i\) give an isotopy from \(\Gamma_\tau\) to the outgoing boundary circle of \( \widehat C_i \).  Composing this isotopy with \(r_i\) gives a homotopy between \(r_i|_{\Gamma_\tau}\) and the restriction of \(r_i\) to the outgoing boundary circle of \( \widehat C_i \). As the degree is invariant under homotopy and as \(r_i|_{\Gamma_\tau}\) is homotopic to the local phase projection, the outgoing circle also maps to \(P_i\) with degree \(\pm2\).  With compatible orientations, the degree is \(+2\).
\end{proof}

Having proven Lemma \ref{lem:local-flip-attachment-Cn}, we now introduce the following standing assumption we would make throughout this Appendix:
\begin{assumption}
\label{ass:clean-isolated-block-Cn}
For \(k=0,1,2,3\), there are compact neighborhoods \(U_k\subset M\times I\) with a terminal exit collar over \(L\), as illustrated in Figure \ref{ill:perioddoublingillustrationwithindexfiltration}.  Let \(E_k\subset U_k\) denote this terminal exit collar, and set $U_k^\circ:=\overline{U_k\setminus E_k}.$ We assume the following:
\begin{itemize}
    \item $\operatorname{Inv}_E(U_k^\circ)=\mathcal S_k\subset \operatorname{int}(U_k^\circ)$, $k=0,1,2,3$, where $\mathcal S_0=\emptyset$, $\mathcal S_1=P_3$, $  \mathcal S_2=\widehat C_2$ and $ \mathcal S_3=\widehat C_1\cup_{P_2}\widehat C_2$.
    \item Each \(\widehat C_i\) satisfies the conclusion of Lemma~\ref{lem:local-flip-attachment-Cn} with positive degree (w.r.t. some fixed orientations).
    \item If a same-period continuation branch \(D_n\) meets the chosen neighborhood, we assume that the corresponding part of \(D_n\) is a component in $Per$, and has no further bifurcation before leaving $U_k$.  By the definition of the Entropy flow on such a branch, \(E\) points along \(D_n\) towards the limiting bifurcation orbit \(P_n\) (see Definition \ref{def511}). Consequently, the \(D_n\)-side contributes neither an additional exit set nor an additional isolated invariant set to the chosen block.
\end{itemize}

We will refer to these three assumptions as \textbf{the clean isolated period doubling sheet block}.
\end{assumption}

Before moving on, we remark the equality \(\operatorname{Inv}_E(U_k^\circ)=\mathcal S_k\) is the isolation hypothesis. The second assumption is a genericity assumption, as it assumes the flip condition which is $C^3$-generic (again, we refer to the Appendix of \cite{PY2} for the complete details). The third assumption is the idealized one, as it assumes the orbits on $D_n$, which are saddles, are isolated. We stress that this final assumption cannot be relaxed - if $D_n$ is not in $Per$ (which could occur, for example, when the invariant manifolds for orbits on $D_n$ intersect transversely) the calculation below does not apply. We now prove:

\begin{lemma}
\label{lem:thin-index-pairs-Cn}
Under Assumption~\ref{ass:clean-isolated-block-Cn}, assume the surfaces \(C_n\) are composed of attracting orbits in the fixed slices \(M\times\{\tau\}\).  Then we can choose the Entropy flow s.t. there exist index pairs $\mathcal P_k=(N_k,L_k)$, $k=0,1,2,3$ which form an entropy-adapted index filtration (see Definition \ref{ill:perioddoublingillustrationwithindexfiltration}), given by:
\[
        N_0\subset N_1\subset N_2\subset N_3,
        \qquad
        L_i=L_j\cap N_i\quad\text{for }i<j.
\]
Moreover, for every \(k\), there is a homotopy equivalence of pairs $(N_k,L_k)\simeq (X_k,L).$
\end{lemma}

\begin{proof}
For each \(k\), choose \(N_k\) to be a sufficiently small neighborhood of
\(X_k\) in \(M\times I\).  Away from the period doubling bifurcation orbits \(P_1,P_2,P_3\), this neighborhood is locally a product of the corresponding
periodic-orbit sheet with a small transverse disk.  Thus, locally, $ N_k \simeq X_k\times D^2.$ To continue, for a fixed regular parameter value
\(\tau\neq \tau_1,\tau_2,\tau_3\), as $N_k$ is four-dimensional, the slice $N_k\cap (M\times\{\tau\})$ can be assumed to be a small tubular neighborhood encasing the corresponding periodic orbit in \(M\times\{\tau\}\) - in particular, because \(\dim M=3\), this slice-wise neighborhood is a solid torus (which may or may not be knotted in $M$).

Near a bifurcation circle \(P_i\), the neighborhood is chosen in \(M\times I\), not only inside the slice \(M\times\{\tau_i\}\).  We choose it to be small enough and compatible with the adjacent neighborhoods on the selected period-doubled sheets $C_n$, so that there is a projection $\pi_k:N_k\to X_k$ which collapses each transverse disk fiber $n_k=N_k\cap (M\times\{\tau\})$ to its base point in \(X_k\) - using Lemma \ref{trapping}, we can choose the Entropy vector field $E$ s.t. it is transverse to $\partial n_k\times\{\tau\}$, i.e., on $\cup_{\tau\in(\tau_1,\tau_2)}\partial n_k\times\{\tau\}$ the Entropy flow enters $N_k$. We now recall that the cylinders \(D_n\) are not included in the branch skeleton \(X_k\).  If the neighborhood \(N_k\) meets \(D_n\) near the bifurcation orbit \(P_n\), this does not create an exit set - specifically, by Assumption~\ref{ass:clean-isolated-block-Cn}, such a \(D_n\)-side piece is an entrance piece for the Entropy flow into $N_k$. Hence, the motion of $E$ on \(D_n\) contributes neither to the exit set nor to the isolated invariant set.

It follows the only exit part is the terminal exit collar over \(L\). Choose one terminal exit collar \(W_L\) over \(L\), contained in the terminal collars of the \(U_k\), and define $L_k=N_k\cap W_L.$ After shrinking \(N_k\) (if necessary), we can ensure $   N_k\subset U_k$, $ \overline{N_k\setminus L_k}\subset U_k^\circ$, and $  \mathcal S_k\subset \operatorname{int}(N_k\setminus L_k)$ for all $k=1,2,3$. Then, by Assumption~\ref{ass:clean-isolated-block-Cn} we have:
\[
        \operatorname{Inv}_E(\overline{N_k\setminus L_k})
        =
        \mathcal S_k
        \subset
        \operatorname{int}(N_k\setminus L_k).
\]
Since every boundary piece except the terminal collar is an entrance set w.r.t $E$, every positive-time orbit that leaves \(N_k\) exits through \(L_k\).  Hence \((N_k,L_k)\) is an Entropy flow index pair. It remains to identify the pair homotopy type. To this end, recall the projection \(\pi_k: N_k \to X_k\) collapses the transverse disk fibers. In local coordinates this can be written \((x,v)\in X_k \times D^2\), where the homotopy \(Q_k(t,(x,v)) = (x,(1-t)v)\)  deforms \(N_k\)  onto \(X_k\). The neighborhoods near the bifurcation orbits were chosen compatibly with the same projection, so this collapse extends over them. Since \(L_k\) is the part of \(N_k\) lying over the terminal collar \(W_L\), the same homotopy restricts to a deformation retraction of \(L_k\) onto \(L\). Therefore, $(N_k,L_k)\simeq (X_k,L).$

To complete the proof, we choose the neighborhoods recursively. First choose \(N_3\), then choose \(N_2,N_1,N_0\) small enough so that we have the chain of inclusions $ N_0\subset N_1\subset N_2\subset N_3$ (see the illustration in Figure \ref{ill:perioddoublingillustrationwithindexfiltration}). Since all exit sets are defined using the same terminal collar \(W_L\), for \(i<j\) we have: $$L_j\cap N_i
        =
        (N_j\cap W_L)\cap N_i
        =
        N_i\cap W_L
        =
        L_i.$$ Thus the pairs $\mathcal P_k=(N_k,L_k)$, $k=0,1,2,3$, illustrated in Figure \ref{ill:perioddoublingillustrationwithindexfiltration}, form an entropy-adapted index filtration.
\end{proof}
\begin{figure}[h]
\centering
\begin{overpic}[width=0.6\textwidth]{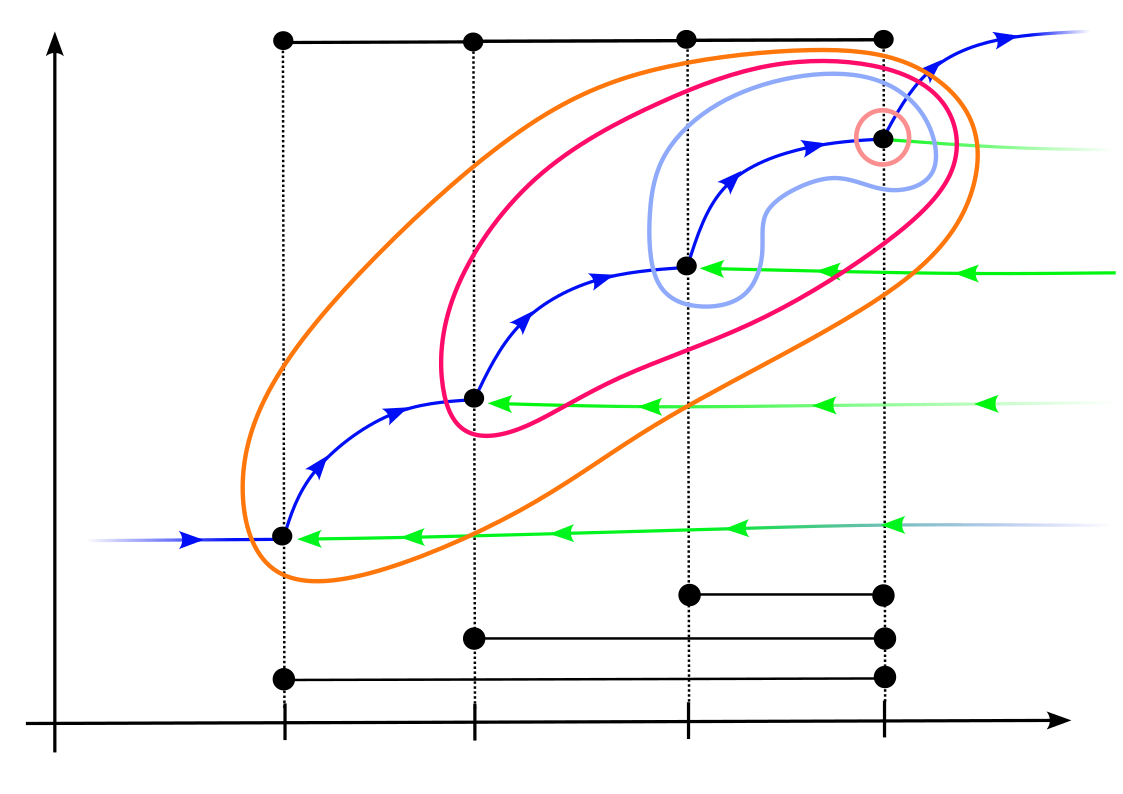}
\put(780,20){$\tau_*$}
\put(605,20){$\tau_3$}
\put(410,20){$\tau_2$}
\put(245,20){$\tau_1$}

\put(685,190){$x_1$}
\put(545,145){$x_2$}
\put(305,115){$x_3$}

\put(750,540){$L$}
\put(570,415){$P_3$}
\put(380,300){$P_2$}
\put(205,180){$P_1$}

\put(1020,230){$D_1$}
\put(1000,330){$D_2$}
\put(1020,430){$D_3$}

\put(650,580){$C_3$}
\put(500,485){$C_2$}
\put(290,345){$C_1$}

\put(720,600){$\mathcal{P}_0$}
\put(630,485){$\mathcal{P}_1$}
\put(450,370){$\mathcal P_2$}
\put(275,250){$\mathcal P_3$}

\put(690,690){$\widehat C_3$}
\put(500,690){$\widehat C_2$}
\put(330,690){$\widehat C_1$}

\put(30,690){$M$}
\put(980,60){$\tau$}

\end{overpic}
\caption{\textit{A diagram of three consecutive period doubling bifurcations with index filtration $\mathcal P_i.$}}
\label{ill:perioddoublingillustrationwithindexfiltration}
\end{figure}
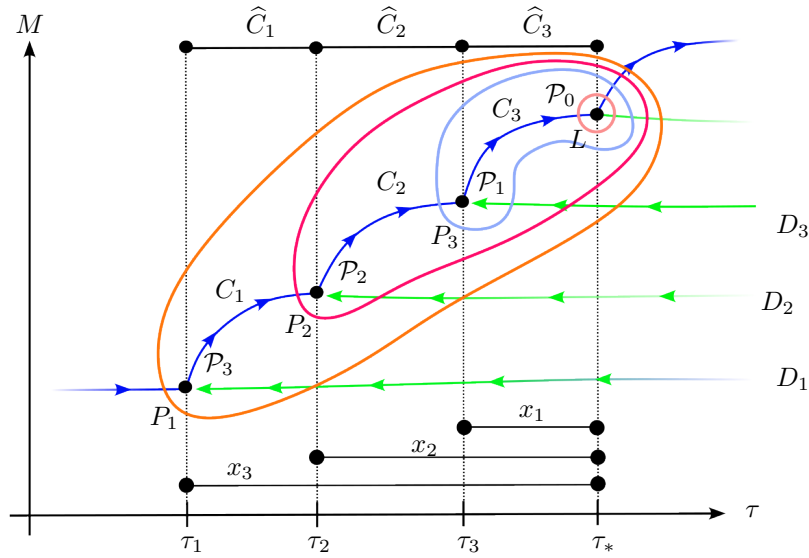


By Lemma~\ref{lem:thin-index-pairs-Cn}, we have $CH_q(\mathcal P_k;\mathbb Z) \cong    H_q(X_k,L;\mathbb Z)$. We now come to:

\begin{proposition}
\label{prop:dyadic-conley-groups-Cn}
Under the assumptions of Lemma~\ref{lem:thin-index-pairs-Cn}, for all $q$ we have $CH_q(\mathcal P_0;\mathbb Z)=0$. In contrast, for \(k=1,2,3\) we have:
\[
        CH_q(\mathcal P_k;\mathbb Z)
        \cong
        \begin{cases}
        \mathbb Z/2^k\mathbb Z, & q=1,\\
        0, & q\neq 1.
        \end{cases}
\]
\end{proposition}

\begin{proof}
It is enough to compute \(H_q(X_k,L;\mathbb Z)\).  Fix compatible orientations on the period doubling bifurcation orbits \(P_1,P_2,P_3\), and \(L\). Furthermore, let $ u_i\in H_1(P_i;\mathbb Z)\cong \mathbb Z$ be the generator determined by the orientation of \(P_i\), and let $\lambda\in H_1(L;\mathbb Z)\cong \mathbb Z$ be the generator determined by the orientation of \(L\). We note that for \(X_1=\widehat C_3\), the deformation retraction \(r_3:\widehat C_3\to P_3\) implies $H_1(X_1;\mathbb Z)\cong \mathbb Z$. Similarly, let \(c_3^{(1)}\) be the image of \(u_3\) under the inclusion \(P_3\hookrightarrow X_1\). By Lemma~\ref{lem:local-flip-attachment-Cn}, the image of \(\lambda\) under the inclusion \(L\hookrightarrow X_1\) is $2c_3^{(1)}\in H_1(X_1;\mathbb Z)$. This proves both classes are being viewed after applying the homomorphisms induced by their inclusions into \(X_1\).

For \(X_2=\widehat C_2\cup_{P_3}\widehat C_3\), the space \(X_2\) deformation retracts onto \(P_2\). Let \(c_2^{(2)}\) be the image of \(u_2\) under \(P_2\hookrightarrow X_2\). The two degree-two attachments give both $c_3^{(2)}=2c_2^{(2)}$ and  $\ell^{(2)}=2c_3^{(2)}=4c_2^{(2)}$ in $H_1(X_2;\mathbb Z)$. Conversely, for \(X_3=\widehat C_1\cup_{P_2}\widehat C_2\cup_{P_3}\widehat C_3\), the space \(X_3\) deformation retracts onto \(P_1\). Let \(c_1^{(3)}\) be the image of \(u_1\) under \(P_1\hookrightarrow X_3\). Similarly, the three degree-two attachments imply both $c_2^{(3)}=2c_1^{(3)}$ and $c_3^{(3)}=2c_2^{(3)}=4c_1^{(3)}$,therefore we have $\ell^{(3)}=2c_3^{(3)}=8c_1^{(3)}$ in $H_1(X_3;\mathbb Z)$. All in all, for \(k=1,2,3\), the inclusion \(L\hookrightarrow X_k\) induces the map:
\[
        H_1(L;\mathbb Z)\cong \mathbb Z
        \xrightarrow{\ \times 2^k\ }
        H_1(X_k;\mathbb Z)\cong \mathbb Z.
\]

The long exact sequence of the pair \((X_k,L)\) now implies the needed. To see why, note \(X_k\) has the homotopy type of a circle, so \(H_2(X_k;\mathbb Z)=0\). Both \(L\) and \(X_k\) are connected, so \(H_0(L;\mathbb Z)\to H_0(X_k;\mathbb Z)\) is an isomorphism. Moreover, since multiplication by \(2^k\) is injective on \(\mathbb Z\), we get $H_2(X_k,L;\mathbb Z)=0$ and $H_1(X_k,L;\mathbb Z)
        \cong
        \operatorname{coker}
        \left(
        \mathbb Z\xrightarrow{\times 2^k}\mathbb Z
        \right)
        \cong
        \mathbb Z/2^k\mathbb Z$, while all other relative homology groups vanish. The case \(k=0\) is \((X_0,L)=(L,L)\), so again, all relative groups vanish.
\end{proof}

Recalling the transition groups and transition maps (see Definitions \ref{transitiongroup} and  \ref{transportmap}, respectively) and that they encode the topological change in sequences of bifurcations, we now prove:

\begin{proposition}
\label{prop:dyadic-transport-maps-Cn}
Under the assumptions of Lemma \ref{lem:thin-index-pairs-Cn}, for \(k=1,2,3\), let  $q_k:H_1(X_k;\mathbb Z)\longrightarrow H_1(X_k,L;\mathbb Z)$ be the quotient map in the long exact sequence of the pair \((X_k,L)\), let $g_1=c_3^{(1)},\; g_2=c_2^{(2)},\; g_3=c_1^{(3)}$ be the generators of \(H_1(X_k;\mathbb Z)\) used in the proof of Proposition~\ref{prop:dyadic-conley-groups-Cn},
and choose $\alpha_k=q_k(g_k)\in CH_1(\mathcal P_k;\mathbb Z).$ Then the following holds:

\begin{itemize}
    \item $ CH_1(\mathcal P_k;\mathbb Z)   =    \langle \alpha_k\mid 2^k\alpha_k=0\rangle$. 
    \item For \(1\leq i<j\leq 3\), the inclusion-induced maps $\Phi^1_{ij}:CH_1(\mathcal P_i;\mathbb Z) \longrightarrow CH_1(\mathcal P_j;\mathbb Z)$ satisfy the relations $\Phi^1_{ij}(\alpha_i)=2^{j-i}\alpha_j$.
\end{itemize}

Equivalently, setting $\theta_k:\mathbb Z/2^k\mathbb Z
        \longrightarrow
        CH_1(\mathcal P_k;\mathbb Z)$ , $\theta_k(a+2^k\mathbb Z)=a\alpha_k$, then \(\Phi^1_{ij}\) corresponds under \(\theta_i,\theta_j\) to the homomorphism $m_{ij}:\mathbb Z/2^i\mathbb Z
        \longrightarrow
        \mathbb Z/2^j\mathbb Z$, $m_{ij}(a+2^i\mathbb Z)
        =
        2^{j-i}a+2^j\mathbb Z$.
\end{proposition}

\begin{proof}
Let $q_k:H_1(X_k;\mathbb Z)\longrightarrow H_1(X_k,L;\mathbb Z)$ denote the homomorphism induced by the inclusion of pairs $(X_k,\varnothing)\hookrightarrow (X_k,L)$. Thus \(q_k\) is the quotient map appearing in the long exact sequence of the pair \((X_k,L)\). Recall from the proof of Proposition~\ref{prop:dyadic-conley-groups-Cn} that we have the identities $H_1(X_1;\mathbb Z)=\mathbb Z\langle c_3^{(1)}\rangle$, $H_1(X_2;\mathbb Z)=\mathbb Z\langle c_2^{(2)}\rangle$ and $H_1(X_3;\mathbb Z)=\mathbb Z\langle c_1^{(3)}\rangle$. To continue, set $  g_1=c_3^{(1)}$, $
        g_2=c_2^{(2)}$, and $ g_3=c_1^{(3)}$ - by definition, we have:       $$\alpha_k=q_k(g_k)\in H_1(X_k,L;\mathbb Z)
        \cong CH_1(\mathcal P_k;\mathbb Z).$$
Let $\iota_{ij}:X_i\hookrightarrow X_j$ denote the inclusion. Since \(L\subset X_i\subset X_j\), this is an inclusion of pairs $(X_i,L)\hookrightarrow (X_j,L)$, and the induced map on relative homology is the model for $ \Phi^1_{ij}:CH_1(\mathcal P_i;\mathbb Z)\to
        CH_1(\mathcal P_j;\mathbb Z)$ under the pair homotopy equivalences $(N_k,L_k)\simeq (X_k,L).$ By the naturality of the long exact sequence of a pair, we have the equality $\Phi^1_{ij}\circ q_i =  q_j\circ (\iota_{ij})_*$. In particular, for \(i=1\) and \(j=2\) the inclusion \(X_1\hookrightarrow X_2\) maps the class \(g_1=c_3^{(1)}\), represented by the circle \(P_3\), to the class of \(P_3\) inside \(X_2\) - where by the degree-two attachment along \(\widehat C_2\), we have  $(\iota_{12})_*g_1=c_3^{(2)}=2c_2^{(2)}=2g_2$. Therefore, we arrive at $\Phi^1_{12}(\alpha_1) = \Phi^1_{12}(q_1(g_1)) = q_2((\iota_{12})_*g_1) = q_2(2g_2) = 2\alpha_2.$

Similarly, the inclusion \(X_2\hookrightarrow X_3\) maps \(g_2=c_2^{(2)}\), represented by the circle \(P_2\), to the class of \(P_2\) inside \(X_3\). By the degree-two attachment along \(\widehat C_1\), we similarly get $(\iota_{23})_*g_2=c_2^{(3)}=2c_1^{(3)}=2g_3.$  Hence, $\Phi^1_{23}(\alpha_2) = 2\alpha_3.$ In addition, for \(i=1\) and \(j=3\), by either composing the two previous maps or by using the inclusion \(X_1\hookrightarrow X_3\), it immediately follows that $(\iota_{13})_*g_1
        =
        c_3^{(3)}
        =
        2c_2^{(3)}
        =
        4c_1^{(3)}
        =
        4g_3 $ and $\Phi^1_{13}(\alpha_1)=4\alpha_3$. These are exactly the formulas required, i.e., $\Phi^1_{ij}(\alpha_i)=2^{j-i}\alpha_j$, $1\leq i<j\leq3$. Finally, from the identification $\theta_k:\mathbb Z/2^k\mathbb Z\longrightarrow
        CH_1(\mathcal P_k;\mathbb Z)$, $\theta_k(a+2^k\mathbb Z)=a\alpha_k$, the previous formula becomes $\theta_j^{-1}\circ \Phi^1_{ij}\circ \theta_i
        \bigl(a+2^i\mathbb Z\bigr)
        =
        2^{j-i}a+2^j\mathbb Z.$
\end{proof}
\begin{remark}
To dispel any doubts, note the expression $\theta_j^{-1}\circ \Phi^1_{ij}\circ \theta_i
        \bigl(a+2^i\mathbb Z\bigr)
        =
        2^{j-i}a+2^j\mathbb Z$ is well defined: if \(a'=a+2^i b\) and  $2^{j-i}a' = 2^{j-i}a+2^j b,$ then \(2^{j-i}a'\), \(2^{j-i}a\) determine the same element of $\mathbb Z/2^j\mathbb Z$.
\end{remark}

Propositions \ref{prop:dyadic-conley-groups-Cn} and \ref{prop:dyadic-transport-maps-Cn} mean that the degree-one transport module of this entropy-adapted index filtration can be written as follows:
\[
        0
        \longrightarrow
        \langle \alpha_1\mid 2\alpha_1=0\rangle
        \xrightarrow{\ \alpha_1\mapsto 2\alpha_2\ }
        \langle \alpha_2\mid 4\alpha_2=0\rangle
        \xrightarrow{\ \alpha_2\mapsto 2\alpha_3\ }
        \langle \alpha_3\mid 8\alpha_3=0\rangle .
\]
The two displayed maps above are injective and not surjective, i.e., in the terminology of Subsection \ref{2dtheory} (and in particular, Definition \ref{birthdeathcode}), there is no degree-one Death between adjacent nonzero levels. The adjacent degree-one Birth groups are:
\[
        \operatorname{Birth}_1(1,2)
        =
        \operatorname{coker}\Phi^1_{12}
        \cong
        \mathbb Z/2\mathbb Z,
\]
\[
        \operatorname{Birth}_1(2,3)
        =
        \operatorname{coker}\Phi^1_{23}
        \cong
        \mathbb Z/2\mathbb Z
\]
and
\[
        \operatorname{Birth}_1(0,3)
        =
        CH_1(\mathcal P_3;\mathbb Z)
        \cong
        \mathbb Z/8\mathbb Z,
\]
where the last group denotes the Birth group for the transition from the empty initial level to the full three-step block. Moreover, using the transition groups of Section \ref{2dtheory}, the long exact sequence of the triple gives $TR_1(i,j;\mathbb Z)\cong \operatorname{coker}\Phi^1_{ij}$ which yields:
\[
        TR_1(1,2;\mathbb Z)\cong \mathbb Z/2\mathbb Z,
        \qquad
        TR_1(2,3;\mathbb Z)\cong \mathbb Z/2\mathbb Z
\]
and
\[
        TR_1(1,3;\mathbb Z)\cong \mathbb Z/4\mathbb Z,
        \qquad
        TR_1(0,3;\mathbb Z)\cong \mathbb Z/8\mathbb Z,
\]
where all other transition groups vanish.\\

At this point, we would like to elaborate on the dynamical meaning of the dyadic transition groups. The generator \(\alpha_k\) represents the phase class of the \(k\)-step period doubling block after the terminal exit circle \(L\) has been collapsed. In contrast, the relation \(2^k\alpha_k=0\) records that the terminal orbit \(L\) projects through the successive flip attachments as a \(2^k\)-fold phase cover of the bifurcation periodic orbit circle. Since \(L\) is a part of the exit set, this \(2^k\)-fold phase relation becomes a torsion in the relative Conley group. Similarly, the transport map $\Phi^1_{ij}(\alpha_i)=2^{j-i}\alpha_j$ can be interpreted as follows: when the isolated block is enlarged from level \(i\) to level \(j\), the phase class already present at level \(i\) is not destroyed, but rather it is re-expressed as a \(2^{j-i}\)-fold phase class inside the deeper period doubling branch. Hence, the adjacent maps are injective - that is, no degree-one phase class dies in this clean attracting model. That being said, these maps are not surjective, as each period doubling adds one new binary phase layer which is not detected at the previous level. Consequently, the adjacent groups $TR_1(1,2;\mathbb Z)\cong \mathbb Z/2\mathbb Z,\; TR_1(2,3;\mathbb Z)\cong \mathbb Z/2\mathbb Z$ measure the new parity information created at each single flip. In particular, the group $TR_1(1,3;\mathbb Z)\cong \mathbb Z/4\mathbb Z$ records the two-step dyadic refinement, while $TR_1(0,3;\mathbb Z)\cong \mathbb Z/8\mathbb Z$ records the full three-step phase refinement of the block. Thus, under Assumption \ref{ass:clean-isolated-block-Cn}, the transition groups are a Conley index signature of the successive twofold phase attachments along the period doubling route. The vanishing of the other transition groups reflects the absence of additional dynamical phenomena inside chosen isolated block, i.e., by $D_n\subseteq Per$, $n=1,2,3$. We further highlight that the dyadic torsion is an invariant of the entire finite transition block, not of an individual attracting periodic orbit. It arises because the terminal circle \(L\) maps into \(X_k\) with degree \(2^k\), while the relevant Conley group is the relative group \(H_1(X_k,L;\mathbb Z)\).\\

Before concluding the Appendix, we remark that the above calculations assume the normal fibers to the skeleton are entrance fibers, which forces \((N_k,L_k)\simeq (X_k,L)\) (this holds due to our choice of Entropy flow in Lemma \ref{lem:thin-index-pairs-Cn}). If the isolated block has an additional unstable normal bundle \(E^u\to X_k\) of rank \(u\), the standard Thom space model for the Conley index replaces \((X_k,L)\) by the corresponding Thom pair. If in addition \(E^u\) is orientable, the Thom isomorphism would give $CH_q(\mathcal P_k;\mathbb Z) \cong H_{q-u}(X_k,L;\mathbb Z).$ Thus, the group \(\mathbb Z/2^k\mathbb Z\) would appear in degree \(u+1\) rather than degree \(1\). In other words, in both cases the Conley index records both the dyadic attachment along the skeleton and the transverse unstable directions.

\section{Statements}
\paragraph{\textbf{Author contribution statement}}
All three authors made an equal contribution to this study and all three comply with the COPE regulations.\\
\paragraph{\textbf{Data availability statement}}
No data was used for this paper.\\
\paragraph{\textbf{Competing interests statement}}
On behalf of all authors, the corresponding author states there is no conflict of interest and the authors did not receive support from any organization for the submitted work. Moreover, the authors have no financial or proprietary interests in any material discussed in this article.\\
\paragraph{\textbf{Funding statement}}
The authors did not receive funding from any organization for the submitted work.\\

\newpage
\printbibliography
\end{document}